\documentclass{article}
\usepackage{fullpage}
\usepackage{color}
\usepackage{amsmath}
\usepackage{graphicx}

\usepackage{hyperref}       

\usepackage{relsize}
\usepackage{amsthm}
\usepackage{tikz}

\definecolor{Blue}{rgb}{0,0,1}
\definecolor{Red}{rgb}{1,0,0}
\definecolor{Green}{rgb}{0,1,0}
\newcommand{\KTC}[1]{{\textcolor{Blue}{KTC: #1}}}
\newcommand{\bds}[1]{{\boldsymbol{#1}}}

\newcommand{\defeq}{:=}

\newcommand{\domain}{\Omega}
\newcommand{\domaini}{\domain_i}
\newcommand{\domainArg}[1]{\domain_{#1}}
\newcommand{\boundary}{\Gamma}
\newcommand{\boundaryi}{\boundary_i}
\newcommand{\nsubdomains}{{n_\domain}}
\newcommand{\nconstraints}{n_{\bar A}}
\newcommand{\nconstraintsROM}{n_{A}}

\newcommand{\residual}{\bds{r}}
\newcommand{\residualArg}[1]{\residual_{#1}}
\newcommand{\residuali}{\residualArg{i}}

\newcommand{\energyCriterion}{\upsilon}
\newcommand{\viscosity}{\nu}
\newcommand{\stateSnapshotMatrix}{\bds{X}}
\newcommand{\columnPivot}{\bds{P}}
\newcommand{\Qmat}{\bds{Q}}
\newcommand{\Rmat}{\bds{R}}
\newcommand{\Qmati}{\Qmat_i}
\newcommand{\Rmati}{\Rmat_i}

\newcommand{\state}{\bds{x}}
\newcommand{\stateArg}[1]{\state_{#1}}
\newcommand{\statei}{\stateArg{i}}
\newcommand{\stateDummy}{\bds{w}}
\newcommand{\stateDummyTwo}{\bds{y}}
\newcommand{\stateSol}{\state}

\newcommand{\stateApprox}{\tilde\state}
\newcommand{\stateApproxArg}[1]{\stateApprox_{#1}}
\newcommand{\stateApproxi}{\stateApproxArg{i}}
\newcommand{\stateApproxInterior}{\tilde\state^\domain}
\newcommand{\stateApproxInteriorArg}[1]{\stateApproxInterior_{#1}}
\newcommand{\stateApproxInteriori}{\stateApproxInteriorArg{i}}
\newcommand{\stateApproxInteriorik}{\tilde\state^{\domain(k)}_i}
\newcommand{\stateApproxBoundary}{\tilde\state^\boundary}
\newcommand{\stateApproxBoundaryArg}[1]{\stateApproxBoundary_{#1}}
\newcommand{\stateApproxBoundaryi}{\stateApproxBoundaryArg{i}}
\newcommand{\stateApproxBoundaryik}{\tilde\state^{\boundary(k)}_i}

\newcommand{\stateInterior}{\bds{x}^{\Omega}}
\newcommand{\stateInteriorArg}[1]{\stateInterior_{#1}}
\newcommand{\stateInteriori}{\stateInteriorArg{i}}

\newcommand{\stateOptimal}{\bds{x}^{\star,2}}
\newcommand{\stateOptimalInf}{\bds{x}^{\star,\infty}}
\newcommand{\stateOptimalInfInteriori}{\bds{x}^{\Omega,\star,\infty}_i}
\newcommand{\stateDummyInteriori}{\bds{w}^{\Omega}_i}
\newcommand{\stateDummyInteriorTwoi}{\bds{y}^{\Omega}_i}

\newcommand{\stateBoundary}{\bds{x}^{\Gamma}}
\newcommand{\stateBoundaryArg}[1]{\stateBoundary_{#1}}
\newcommand{\stateBoundaryi}{\stateBoundaryArg{i}}

\newcommand{\stateDummyBoundaryi}{\bds{w}^{\Gamma}_i}
\newcommand{\stateDummyBoundaryTwoi}{\bds{y}^{\Gamma}_i}

\newcommand{\projection}{\bds{P}}
\newcommand{\projectionRes}{\projection^\residual}
\newcommand{\projectionResArg}[1]{\projectionRes_{#1}}
\newcommand{\projectionResi}{\projectionResArg{i}}

\newcommand{\projectionStateInterior}{\projection^\Omega}
\newcommand{\projectionStateInteriorArg}[1]{\projectionStateInterior_{#1}}
\newcommand{\projectionStateInteriori}{\projectionStateInteriorArg{i}}

\newcommand{\projectionStateBoundary}{\bds{P}^\Gamma}
\newcommand{\projectionStateBoundaryArg}[1]{\projectionStateBoundary_{#1}}
\newcommand{\projectionStateBoundaryi}{\projectionStateBoundaryArg{i}}

\newcommand{\rbGlobal}{\bds{\Phi}}
\newcommand{\rbArg}[1]{\rbGlobal_{#1}}
\newcommand{\rbi}{\rbArg{i}}

\newcommand{\rbBoundary}{\bds{\Phi}^\boundary}
\newcommand{\rbBoundaryArg}[1]{\rbBoundary_{#1}}
\newcommand{\rbBoundaryi}{\rbBoundaryArg{i}}
\newcommand{\rbBoundaryiTmp}{\bar{\bds{\Phi}}^\boundary_i}
\newcommand{\rbInterior}{\bds{\Phi}^\domain}
\newcommand{\rbInteriorArg}[1]{\rbInterior_{#1}}
\newcommand{\rbInteriori}{\rbInteriorArg{i}}

\newcommand{\redState}{\hat {\bds x}}

\newcommand{\redStateArg}[1]{\redState_{#1}}
\newcommand{\redStatei}{\redStateArg{i}}
\newcommand{\redStateArgk}[1]{\redState^{(k)}_{#1}}
\newcommand{\redStateArgkp}[1]{\redState^{(k+1)}_{#1}}

\newcommand{\redStateDummy}{\hat {\stateDummy}}
\newcommand{\redStateDummyArg}[1]{\redStateDummy_{#1}}
\newcommand{\redStateDummyi}{\redStateDummyArg{i}}

\newcommand{\redStateBoundary}{\redState^\boundary}
\newcommand{\redStateBoundaryArg}[1]{\redStateBoundary_{#1}}
\newcommand{\redStateBoundaryArgk}[1]{\redState^{\boundary(k)}_{#1}}
\newcommand{\redStateBoundaryArgkp}[1]{\redState^{\boundary(k+1)}_{#1}}
\newcommand{\redStateBoundaryi}{\redStateBoundaryArg{i}}
\newcommand{\redStateBoundaryik}{\redStateBoundaryArgk{i}}
\newcommand{\redStateBoundaryikp}{\redStateBoundaryArgkp{i}}
\newcommand{\redStateDummyBoundary}{\redStateDummy^\boundary}
\newcommand{\redStateDummyBoundaryArg}[1]{\redStateDummyBoundary_{#1}}
\newcommand{\redStateDummyBoundaryi}{\redStateDummyBoundaryArg{i}}

\newcommand{\redStateInterior}{\redState^\domain}
\newcommand{\redStateInteriorArg}[1]{\redStateInterior_{#1}}
\newcommand{\redStateInteriori}{\redStateInteriorArg{i}}

\newcommand{\redStateInteriorArgk}[1]{\redState^{\domain(k)}_{#1}}
\newcommand{\redStateInteriorArgkp}[1]{\redState^{\domain(k+1)}_{#1}}
\newcommand{\redStateInteriorik}{\redStateInteriorArgk{i}}
\newcommand{\redStateInteriorikp}{\redStateInteriorArgkp{i}}

\newcommand{\redStateDummyInterior}{\redStateDummy^\domain}
\newcommand{\redStateDummyInteriorArg}[1]{\redStateDummyInterior_{#1}}
\newcommand{\redStateDummyInteriori}{\redStateDummyInteriorArg{i}}

\newcommand{\ndof}{n}

\newcommand{\ndofi}{n_i}
\newcommand{\ndofInteriori}{n^{\Omega}_i}
\newcommand{\ndofInteriorArg}[1]{n^{\Omega}_{#1}}
\newcommand{\ndofBoundaryi}{n^{\Gamma}_i}
\newcommand{\nresi}{n^\residual_i}

\newcommand{\nrb}{p}
\newcommand{\nrbArg}[1]{\nrb_{#1}}
\newcommand{\nrbi}{\nrbArg{i}}
\newcommand{\nrbInterior}{\nrb^\domain}
\newcommand{\nrbInteriorArg}[1]{\nrbInterior_{#1}}
\newcommand{\nrbInteriori}{\nrbInteriorArg{i}}
\newcommand{\nrbBoundary}{\nrb^\boundary}
\newcommand{\nrbBoundaryArg}[1]{\nrbBoundary_{#1}}
\newcommand{\nrbBoundaryi}{\nrbBoundaryArg{i}}
\newcommand{\nrbPort}{\nrb^p}
\newcommand{\nrbPortArg}[1]{\nrbPort_{#1}}

\newcommand{\constraintMat}{\bds{\bar A}}
\newcommand{\constraintMatArg}[1]{\constraintMat_{#1}}
\newcommand{\constraintMati}{\constraintMatArg{i}}

\newcommand{\constraintMatROM}{\bds{A}}
\newcommand{\constraintMatROMArg}[1]{\constraintMatROM_{#1}}
\newcommand{\constraintMatROMi}{\constraintMatROMArg{i}}
\newcommand{\identity}{\bds{I}}

\newcommand{\hessianIntInt}{\bds{H}^{\domain\domain}}
\newcommand{\hessianIntIntArg}[1]{\hessianIntInt_{#1}}
\newcommand{\hessianIntInti}{\hessianIntIntArg{i}}

\newcommand{\hessianIntBound}{\bds{H}^{\domain\boundary}}
\newcommand{\hessianIntBoundArg}[1]{\hessianIntBound_{#1}}
\newcommand{\hessianIntBoundi}{\hessianIntBoundArg{i}}

\newcommand{\hessianBoundInt}{\bds{H}^{\boundary\domain}}
\newcommand{\hessianBoundIntArg}[1]{\hessianBoundInt_{#1}}
\newcommand{\hessianBoundInti}{\hessianBoundIntArg{i}}

\newcommand{\hessianBoundBound}{\bds{H}^{\boundary\boundary}}
\newcommand{\hessianBoundBoundArg}[1]{\hessianBoundBound_{#1}}
\newcommand{\hessianBoundBoundi}{\hessianBoundBoundArg{i}}

\newcommand{\hessianSubdomBF}[1]{\bds{H}_{#1}}

\newcommand{\searchDirInteriorArgk}[1]{\bds{p}^{\domain(k)}_{#1}}

\newcommand{\searchDirInteriorik}{\searchDirInteriorArgk{i}}

\newcommand{\searchDirLagrange}{\bds{p}^{\lagrangeROM}}
\newcommand{\searchDirLagrangek}{\bds{p}^{\lagrangeROM(k)}}

\newcommand{\searchDirBoundaryArgk}[1]{\bds{p}^{\boundary(k)}_{#1}}

\newcommand{\searchDirBoundaryik}{\searchDirBoundaryArgk{i}}

\newcommand{\searchDirSubdomBFk}[1]{\bds{p}^{(k)}_{#1}}

\newcommand{\redResBoundary}{\hat{\bds{r}}^\boundary}
\newcommand{\redResBoundaryArg}[1]{\redResBoundary_{#1}}
\newcommand{\redResBoundaryi}{\redResBoundaryArg{i}}

\newcommand{\redResInterior}{\hat{\bds{r}}^\domain}
\newcommand{\redResInteriorArg}[1]{\redResInterior_{#1}}
\newcommand{\redResInteriori}{\redResInteriorArg{i}}

\newcommand{\redResSubdom}[1]{\hat{\bds{r}}_{#1}}

\newcommand{\unitvec}{\bds{e}}
\newcommand{\unitvecArg}[1]{\unitvec_{#1}}

\newcommand{\nullspacemat}{\bar{\bds{N}}}
\newcommand{\nullspacematArg}[1]{\nullspacemat_{#1}}
\newcommand{\nullspacemati}{\nullspacematArg{i}}
\newcommand{\nullspaceCoords}{\hat{\bds{x}}^{\text{null}}}
\newcommand{\nullspaceCoordsDummy}{\hat{\bds{w}}^{\text{null}}}

\newcommand{\trialSpace}{\mathcal S}
\newcommand{\trialSpaceInterior}{\mathcal S_{\text{I/B}}}
\newcommand{\trialSpaceFull}{\mathcal S_{\text{F}}}

\newcommand{\inverseTmp}{P}
\newcommand{\inverseLipschitz}{\kappa_\ell}
\newcommand{\lipschitz}{\kappa_u}

\newcommand{\lagrangeSymbol}{\lambda}
\newcommand{\lagrangeDummySymbol}{\gamma}

\newcommand{\lagrangeROM}{\bds{\lagrangeSymbol}}
\newcommand{\lagrangeROMk}{\bds{\lagrangeSymbol}^{(k)}}
\newcommand{\lagrangeROMkp}{\bds{\lagrangeSymbol}^{(k+1)}}
\newcommand{\lagrangeROMDummy}{\bds{\lagrangeDummySymbol}}
\newcommand{\lagrangian}{L}
\newcommand{\zero}{\bds{0}}
\newcommand{\argmin}[1]{\underset{#1}{\text{argmin}}}
\newcommand{\range}[1]{\text{Ran}(#1)}
\newcommand{\rank}[1]{\text{rank}(#1)}
\newcommand{\subdomainPorts}[1]{P(#1)}
\newcommand{\portsSubdomains}[1]{Q(#1)}

\newcommand{\nports}{{n_p}}

\newcommand{\ndofPorts}[1]{n^p}
\newcommand{\ndofPortsArg}[1]{n_{#1}^p}
\newcommand{\ndofPortsj}{\ndofPortsArg{j}}
\newcommand{\projectionStatePort}[2]{\projection^{#1}_{#2}}
\newcommand{\projectionStatePortAll}[1]{\projection^{#1}}
\newcommand{\nconstraintsPort}[1]{n^\text{pair}_{#1}}
\newcommand{\card}[1]{|#1|}
\newcommand{\RRstar}[1]{\mathbb{R}_\star^{#1}}

\newcommand{\sampleMat}{\bds{Z}}
\newcommand{\sampleMatArg}[1]{\bds{Z}_{#1}}
\newcommand{\sampleMati}{\sampleMatArg{i}}

\newcommand{\sampleMatStateBoundary}{\bds{Z}^\Gamma}
\newcommand{\sampleMatStateBoundaryArg}[1]{\sampleMatStateBoundary_{#1}}
\newcommand{\sampleMatStateBoundaryi}{\sampleMatStateBoundaryArg{i}}
\newcommand{\sampleMatRes}{\bds{Z}^r}
\newcommand{\sampleMatResArg}[1]{\sampleMatRes_{#1}}
\newcommand{\sampleMatResi}{\sampleMatResArg{i}}
\newcommand{\sampleMatStateInterior}{\bds{Z}^\Omega}
\newcommand{\sampleMatStateInteriorArg}[1]{\sampleMatStateInterior_{#1}}
\newcommand{\sampleMatStateInteriori}{\sampleMatStateInteriorArg{i}}
\newcommand{\hyperMat}{\bds{B}}
\newcommand{\hyperMatArg}[1]{\bds{B}_{#1}}
\newcommand{\hyperMati}{\hyperMatArg{i}}
\newcommand{\nhyperMat}{n^B}
\newcommand{\nhyperMatArg}[1]{\nhyperMat_{#1}}
\newcommand{\nhyperMati}{\nhyperMatArg{i}}
\newcommand{\nsample}[1]{n^z}
\newcommand{\nsampleArg}[1]{n_{#1}^z}

\newcommand{\nsampleBoundary}[1]{n^{z,\Gamma}}
\newcommand{\nsampleBoundaryArg}[1]{n_{#1}^{z,\Gamma}}
\newcommand{\nsampleBoundaryi}{\nsampleBoundaryArg{i}}
\newcommand{\nsampleInterior}[1]{n^{z,\Omega}}
\newcommand{\nsampleInteriorArg}[1]{n_{#1}^{z,\Omega}}
\newcommand{\nsampleInteriori}{\nsampleInteriorArg{i}}
\newcommand{\nsampleRes}[1]{n^{z,r}}
\newcommand{\nsampleResArg}[1]{n_{#1}^{z,r}}
\newcommand{\nsampleResi}{\nsampleResArg{i}}
\newcommand{\rbRes}[1]{\bds{\Phi}^r}
\newcommand{\rbResArg}[1]{\bds{\Phi}_{#1}^r}
\newcommand{\rbResi}{\rbResArg{i}}

\newcommand{\nrbResArg}[1]{\nrb_{#1}^r}
\newcommand{\nrbResi}{\nrbResArg{i}}
\newcommand{\nrbNull}{\nrb^\text{null}}

\newcommand{\testFunctionArg}[1]{\boldsymbol C^{#1}}
\newcommand{\nportTestFunction}[1]{n_{#1}^c}

\newcommand{\sampleij}[2]{\xi_{#1}^{#2}}

\newtheorem{proposition}{Proposition}
\newtheorem{remarked}{Remark}

\newcommand{\param}{\bds{\mu}}
\newcommand{\params}{\param}

\newcommand{\paramComp}{\param}
\newcommand{\paramCompArg}[1]{\mu_{#1}}
\newcommand{\paramDomain}{\mathcal{D}}
\newcommand{\nparams}{n_{\param}}
\newcommand{\innatseq}[1]{=1,\ldots,#1}
\newcommand{\natseq}[1]{\{1,\ldots,#1\}}

\newcommand{\trainParam}[1]{\bds{\mu}^{#1}_{\rm train}}
\newcommand{\numTrainSample}{n_{\rm train}}
\newcommand{\trainSample}{
\{\trainParam{j}\}_{j=1}^{\numTrainSample}
}

\newcommand{\tempComp}{u}

\newcommand{\vel}{\bds{u}}
\newcommand{\velComp}{u}
\newcommand{\velCompArg}[1]{\velComp_{#1}}

\newcommand{\stateComp}{\boldsymbol{\mathsf{x}}}
\newcommand{\stateCompArg}[1]{\mathsf x_{#1}}

\newcommand{\flopsComputeOneEntryResiduali}{ c^r_i }
\newcommand{\flopsComputeOneEntryJacobiani}{ c^J_i }
\newcommand{\averageNNZperRowInteriori}{ w^{\domain}_i }
\newcommand{\averageNNZperRowBoundaryi}{ w^{\boundary}_i }

\newcommand{\sizeLinSysPrduMono}{ s_{dM} }
\newcommand{\bandwidthLinSysPrduMono}{ w_{dM} }

\newcommand{\sizeLinSysPrduMonoSubdomBF}{ s^s_{dM} }
\newcommand{\bandwidthLinSysPrduMonoSubdomBF}{ w^s_{dM} }

\newcommand{\setOfSnapshotsDummy}{\bds{X}}
\newcommand{\ndofDummy}{n}
\newcommand{\numSnapshotDummy}{m}
\newcommand{\nrbDummy}{p}
\newcommand{\rbDummy}{\bds{\Phi}}
\newcommand{\rbDummyArg}[1]{\bds{u}_{#1}}
\newcommand{\UDummy}{\bds{U}}
\newcommand{\SigmaDummy}{\bds{\Sigma}}
\newcommand{\VDummy}{\bds{V}}

\newcommand{\setOfGlobalResidualSnapshots}{\bds{X}^{\residual}_{g}}

\newcommand{\ndofPortjSubdomi}[2]{n^p_{#1,#2}}

\newcommand{\numWorkColumnsPhiRes}[1]{ n^c_{#1} }
\newcommand{\numUnknownPerNode}{\gamma}

\newcommand{\sampleSetOfSampleMesh}[1]{ \mathfrak{s}_{#1} }
\newcommand{\additionalNumNodesOfSampleMesh}[1]{ n^a_{#1} }
\newcommand{\counterOfNumWorkBasisVectorOfSampleMesh}[1]{ n^b_{#1} }
\newcommand{\numGreedyIterOfSampleMesh}[1]{ n^{\rm it}_{#1} }
\newcommand{\maxNumRHSOfSampleMesh}[1]{ n^{\rm RHS}_{#1} }
\newcommand{\minNumWorkBasisVectorPerIterOfSampleMesh}[1]{ n^{cj,\min}_{#1} }
\newcommand{\minNumSampleNodeToAddPerIterOfSampleMesh}[1]{ n^{aj,\min}_{#1} }
\newcommand{\numWorkBasisVectorForThisIterOfSampleMesh}[1]{ n^{cj}_{#1} }
\newcommand{\numSampleNodeToAddForThisIterOfSampleMesh}[1]{ n^{aj}_{#1} }

\newcommand{\randomMat}{\bds{C}}

\newcommand{\numNewtonIterationArg}[1]{k_{#1}}

\numberwithin{equation}{section}

\newcounter{lemctr}
\newcounter{pfctr}
\newcounter{remctr}
\newcounter{propctr}
\newcounter{proposctr}
\newcommand{\RR}[1]{\ensuremath{\mathbb{R}^{ #1 }}}

\usepackage{amsmath,amssymb}
\usepackage[ruled,noend]{algorithm2e}

\title{Domain-decomposition least-squares Petrov--Galerkin (DD-LSPG)\\ nonlinear model
reduction}
\author{Chi Hoang\thanks{Extreme-scale Data Science and Analytics Department, Sandia National Laboratories, Livermore, CA 94550 (ckhoang@sandia.gov). Sandia National Laboratories is a multimission laboratory managed and operated by National Technology and Engineering Solutions of Sandia, LLC, a wholly owned subsidiary of Honeywell International, Inc., for the U.S. Department of Energy's National Nuclear Security Administration under contract DE-NA-0003525. } 
\and 
Youngsoo Choi\thanks{Lawrence Livermore National Laboratory, Livermore, CA 94550
(choi15@llnl.gov). Lawrence Livermore National Laboratory is operated by
Lawrence Livermore National Security, LLC, for the U.S. Department of Energy,
National Nuclear Security Administration under Contract DE-AC52-07NA27344
(LLNL-JRNL-812648).}
\and 
Kevin Carlberg\thanks{Mechanical Engineering and Applied Mathematics, University of Washington, Seattle, WA 98195 (ktcarlb@uw.edu).}}

\usepackage{comment}
\definecolor{orange}{rgb}{1,0.5,0}
\definecolor{Blue}{rgb}{0,0,1}
\definecolor{Red}{rgb}{1,0,0}
\definecolor{Green}{rgb}{0,1,0}
\definecolor{Bronze}{rgb}{0.8,0.5,0.2}
\definecolor{Violet}{rgb}{0.54,0.17,0.89}
\newcommand{\YC}[1]{#1}
\newcommand{\CH}[1]{#1}

\usepackage{cleveref}		
\usepackage{subcaption}		
\newcommand{\red}[1]{\textcolor{red}{#1}}   	
\newcommand{\blue}[1]{\textcolor{blue}{#1}}   	
\usepackage[inline]{enumitem}	
\usepackage{ulem}		

\usepackage{algorithmic}

\usepackage{array}
\newcolumntype{C}[1]{>{\centering\arraybackslash}m{#1}}

\begin{document}
\date{}
\maketitle


\begin{abstract}
  A novel domain-decomposition least-squares Petrov--Galerkin (DD-LSPG)
  model-reduction method applicable to parameterized systems of nonlinear
  algebraic equations (e.g., arising from discretizing a parameterized
  partial-differential-equations problem) is proposed. In contrast with previous
  works, we adopt an algebraically non-overlapping decomposition strategy rather
  than a spatial-decomposition strategy, which facilitates application to
  different spatial-discretization schemes. Rather than constructing a
  low-dimensional subspace for \textit{the entire state space} in a monolithic
  fashion, the methodology constructs separate subspaces for the different
  subdomains/components characterizing the original model. During the offline
  stage, the method constructs low-dimensional bases for the interior and
  interface of subdomains/components. During the online stage, the approach
  constructs an LSPG reduced-order model for each subdomain/component (equipped
  with hyper-reduction in the case of nonlinear operators), and enforces strong
  or weak compatibility on the `ports' connecting them. We propose several
  different strategies for defining the ingredients characterizing the
  methodology: (i) four different ways to construct reduced bases on the
  interface/ports of subdomains, and (ii) different ways to enforce
  compatibility across connecting ports. In particular, we show that the
  appropriate compatibility-constraint strategy depends strongly on the basis
  choice.  In addition, we derive \textit{a posteriori} and \textit{a priori}
  error bounds for the DD-LSPG solutions.  Numerical results performed on
  nonlinear benchmark problems in heat transfer and fluid dynamics that employ
  both finite-element and finite-difference spatial discretizations demonstrate
  that the proposed method performs well in terms of both accuracy and
  (parallel) computational cost, with different choices of basis and
  compatibility constraints yielding different performance profiles.

\end{abstract}


\section{Introduction}



Many tasks in computational science and engineering are \textit{many query} in
nature, as they require the repeated simulation of a parameterized large-scale
computational model. Model reduction has become a popular approach to make
such tasks tractable. Most of such techniques first perform an
``offline'' training stage that simulates the computational model for multiple
input-parameter instances;
then, during an ``online'' deployed stage, these techniques reduce the
dimensionality and complexity of the original computational model at arbitrary
input-parameter instances by performing a projection process of the original
computational model onto a low-dimensional
subspace or manifold. 


While such reduced-order models (ROMs) have demonstrated success in many
applications, challenges arise when applying model reduction either to
\textit{extreme-scale models} or to \textit{decomposable systems}, i.e.,
systems composed of well-defined components. In the former case, the
extreme-scale nature of the original computational model renders the offline
training simulations infeasible. In the latter case, the many-query task often
involves design, wherein components are swapped or their interconnecting
topology is modified; in this case, the state space characterizing the
original computational model changes substantially between queries, rendering
training simulations (which assume a fixed state space)
challenging.


To date, researchers have developed several methods to enable model reduction
for decomposable systems. During the offline stage, these approaches construct
a unique reduced basis for each component; during the online stage, they
formulate a
reduced-order model for the full system using
domain-decomposition approaches that enforce solution compatibility along
component interfaces. Most approaches to date have been developed for
parameterized \textit{linear} partial differential equations (PDEs). 




Reduced basis element (RBE) methods, which comprise a family of domain-decomposition reduced-order model (DDROM) techniques, are applicable to \textit{linear} PDEs \cite{maday2002reduced,maday2004reduced, iapichino2012reduced, huynh2013static_a, eftang2012adaptive, iapichino2016reduced}. Maday et al.\ \cite{maday2002reduced, maday2004reduced} proposed the very first work of this family; this approach combines the reduced-basis (RB) method with domain decomposition (DD), using full-subdomain bases\footnote{Note that ``full-subdomain bases'' here include all degrees of freedom (DOFs) of a subdomain: both interior and interface DOFs.} and ``gluing'' the subdomain interfaces weakly via Lagrange multipliers. The full-subdomain bases are built in the offline stage, while in the online stage a saddle point problem \cite{benzi2005numerical} is solved to compute the solution for any input-parameter instance. The reduced basis hybrid method (RBHM), which was proposed later by Iapichino and coworkers \cite{iapichino2012reduced}, modifies the RBE by including the finite element (FE) coarse solutions in the reduced bases (in the online stage) to recover the nonzero normal stress component of the final solution. 
The RBE and RBHM were employed to solve the steady Stokes problem with applications in cardiovascular networks \cite{maday2002reduced}, \cite{iapichino2012reduced}. In the reduced-basis--domain-decomposition--finite-element (RDF) method \cite{iapichino2016reduced}, the
same authors proposed to separate the global DOFs into all subdomain interior DOFs and ``skeleton'' DOFs, then approximate all subdomain interior DOFs by RB method. The unknowns in the final reduced linear system comprise the generalized coordinates associated with all subdomain interiors and FE degrees of freedom on the skeleton. Similar in concept, the static condensation reduced basis element (SCRBE) method proposed by Huynh et al.\ \cite{huynh2013static_a, huynh2013static_b} decomposes the ``skeleton'' DOFs further into ``port'' DOFs on each subdomain, where a subdomain can have multiple nonoverlapping ports. SCRBE employs a primal-Schur domain-decomposition method 
to assemble and solve the resulting system. In particular, Ref.\ \cite{huynh2013static_a}
carefully constructs interface bases to represent
\textit{all possible variations} of the solution on the skeleton of the global
domain. While this is a robust and comprehensive approach to compute the
skeleton solution, it also incurs a high computational cost: the dimension of the Schur-complement system is 
equal to the number of FE degrees of
freedom across all ports, which can remain large scale for fine spatial
discretizations. To address this, Ref.~\cite{eftang2012adaptive} applies ``adaptive port reduction'' to reduce the number of port degrees of freedom and hence the dimensionality and cost of solving the Schur-complement system. \CH{While the majority of the work on RBE deals with linear PDEs, we are aware that there is at least one work that deals with nonlinear PDEs \cite{lovgren2006reduced}.}



Besides the RBE family mentioned above, researchers have developed other DDROM methods to solve
parameterized linear PDEs in the context of multiscale heterogeneous materials
analysis. These methods include the multiscale reduced basis method (MsRBM)
\cite{nguyen2008multiscale}, $\rm FE^2$-based model order reduction method
\cite{he2020situ},
the localized reduced basis multiscale method
(LRBMS) \cite{kaulmann2011new, ohlberger2015error},
the reduced basis localized orthogonal decomposition method (RB-LOD) \cite{abdulle2015reduced}, the reduced basis method for
heterogeneous domain decomposition (RBHDD)
\cite{martini2015reduced} and recently the ArbiLoMod method
\cite{buhr2017arbilomod}. In addition, we are also aware of the use of DDROM in the work of graphic community, for example (not a comprehensive list), \cite{barbivc2011real, kim2012physics} deal with nonlinear problems while \cite{yang2013boundary, peiret2019schur} handle linear problems. The work \cite{teng2015subspace} solves nonlinear problems using a FOM-ROM hybrid approach that will be described in next paragraph.




While some DDROM techniques have been applied to nonlinear PDEs, most of these
techniques are multiscale in nature, meaning that they apply a ROM to only a
\textit{subset} of the physical domain, and apply the high-fidelity model
elsewhere; compatibility between the ROM and high-fidelity-model solutions is
enforced using non-overlapping domain decomposition methods and some
multiscale homogenization assumptions \cite{he2020situ}. For example, in
the work by Buffoni and coworkers \cite{buffoni2009iterative}, the authors
implemented the overlapping classical Schwarz method (using Dirichlet--Neumann iterations
\cite{toselli2006domain}) and divided the computational domain into two
subdomains. The high-fidelity-model subdomain is discretized using a standard
method (e.g., finite difference, finite element), while the ROM subdomain
employes a snapshot-based proper orthogonal decomposition (POD) technique
\cite{sirovich1987turbulence} with subdomain bases. Solution compatibility on
the interface holds weakly through the enforcement of continuity of normal
derivatives of the trace of the solutions on the interface. In another work by
Kerfriden et al.\ \cite{kerfridengoury2012}, the authors used a primal-Schur
domain-decomposition method combined with a snapshot-based POD ROM subdomain
to solve nonlinear fracture-mechanics problems. In particular, the approach
approximates the interior DOFs of \textit{linear subdomains} with snapshot-POD
(further reduction with the hyper-reduction technique DEIM
\cite{chaturantabut2010nonlinear} due to nonaffine parameter dependence) and use a full-order model (FOM) on
nonlinear damaged subdomains. The Schur-complement system is formed by enforcing strong (i.e., node pairwise) compatibility between ROM and FOM subdomains and condenses out only the generalized coordinates characterizing the ROM subdomains, rendering the Schur-complement system high-dimensional. With similar
FOM/ROM hybrid idea, the DD-POD method \cite{corigliano2015model} uses the
Gravouil--Combescure domain-decomposition approach \cite{gravouil2001multi} to
solve elastic--plastic structural dynamics problems. The method divides the
domain of interest into subdomains; during the online stage, a plastic check
is performed on each subdomain to determine whether ROM or FOM approximations
will be implemented in that subdomain. Again, full-subdomain bases are used in the
linear-elastic subdomains and weak compatibility constraints are used on the interface. Baiges and coworkers \cite{baiges2013domain} used a
primal-dual monolithic approach to solve incompressible Navier--Stokes
equations with overlapping domain decomposition. The approach also comprises a
FOM/ROM hybrid wherein the physical domain is decomposed into FOM, ROM and
overlapping subdomains. The ROM subdomains use full-subdomain bases and are further
hyper-reduced by a discrete variant \cite{baiges2013explicit} of the best
point interpolation method, while the overlapping/interface regions enforce
velocity continuity, which corresponds to a weak compatibility constraint.


This work aims to overcome several shortcomings of existing works. \CH{First, \textit{most} available DDROM methods for nonlinear PDEs employed a hybrid ROM/FOM approach; a ``complete ROM'' methodology appears to be missing for nonlinear problems (except the work \cite{lovgren2006reduced}, to the best of our knowledge).} Second, most previously developed DDROM methods were applied to self-adjoint problems and thus constrained optimization problems could be derived from a Galerkin-projection perspective; the extension of many methods to non-self-adjoint problems is unclear. Finally, most of the above approaches (with the exception of SCRBE \cite{huynh2013static_a, huynh2013static_b}) employ ``full-subdomain'' bases with support over both interior and interface degrees of freedom. Such bases only are generally compatible only with weak constraints (see, e.g., \cite{maday2002reduced, iapichino2012reduced, buffoni2009iterative, corigliano2015model}), which precludes an equivalent global solution due to non-uniqueness of the solution on the interfaces. To address these shortcomings, this work is characterized by the following novel features, \YC{which, we believe, are valuable steps toward addressing the challenges arose from the nonlinear extreme-scale models (although we do not demonstrate our numerical results on a extreme-scale problem)}:

\begin{itemize}	 
\item We consider parameterized systems of nonlinear algebraic equations, and
	adopt an algebraically non-overlapping decomposition strategy rather than a
		spatial-decomposition strategy, which facilitates application to models
		derived using different discretization methods.
\item We develop a ``complete ROM'' approach that applies model reduction to
	all degrees of freedom characterizing the nonlinear algebraic system; thus it is not a ROM/FOM hybrid.
\item We formulate a constrained optimization problem for the global problem by
  equipping the least-squares Petrov--Galerkin (LSPG\footnote{\CH{For communities other than model reduction one, LSPG and `minimum residual' are completely equivalent.}}) \cite{bui2008model,
    legresley2006application, carlbergbou-mosleh2011, carlberg2017galerkin,
    carlberg2018conservative, choi2019space} projection (with hyper-reduction
    \cite{carlberg2013gnat, choi2020sns}) with interface-compatibility
    constraints. We employ a sequential quadrating programming (SQP) method to
    solve the resulting optimization problem. Critically, this formulation is
    valid for both self-adjoint and non-self-adjoint problems.
\item We propose four different subdomain basis types, including
	the classical ``full-subdomain'' type and three ``interface/boundary'' 
	types: port, skeleton, and full-interface. Consequently, the
	characterization of the solution on the interfaces  has much greater
	flexibility than in previous contributions.
\item Support for both strong and weak compatibility constraints on the
	interfaces for all basis types. In particular, we show that the best choice
		for compatibility constraints is strongly dependent on the subdomain-basis
		type (i.e., weak compatibility is best for full-subdomain and full-interface bases;
		strong compatibility is best for port and skeleton bases).
\item Both \textit{a posteriori} and \textit{a priori} error bounds for the
		method, which illustrate how the error on each subdomain and port can be
		bounded using global quantities.
\item \CH{Bottom-up (or subdomain) training (to be distinguished with top-down training bases above) is proposed (although still simple and not yet mature) and pave the way toward handling nonlinear extreme-scale models and decomposable systems.}
\item Numerical experiments on benchmark problems in heat transfer and fluid
	dynamics that employ both finite-element and finite-difference
		discretizations that systematically assess the effect of all method
		parameters on accuracy and computational cost, lending deep insights into
		the performance aspects of the proposed methodology.
\end{itemize}


The paper is structured as follows. Section \ref{sec_DDFOM_formulation}
formulates the full-order model and algebraically non-overlapping
decomposition that characterizes our domain-decomposition strategy.
Section \ref{subsect_DDLSPG_formulation} describes the proposed DD-LSPG
framework, including the two proposed choices for subdomain reduced bases
(Sections \ref{sec:intBound} and \ref{sec:fullSub}), and strong vs. weak
compatibility constraints (Section \ref{subsect_constraint_types}).
Section \ref{sec_solver} describes the proposed SQP solver used to numerically
solve the constrained optimization problem characterizing DD-LSPG projection,
its particularization to the two types of subdomain reduced bases (Sections
\ref{sec_sqp_interior} and \ref{sqpFull}) 
and its serial/parallel costs (Section \ref{sec_cost}).
Section \ref{sec_basis_construction} describes the offline algorithms for
constructing interior/boundary bases (Section \ref{sec:constrIntBound}) and 
full-subdomain bases (Section \ref{sec:constFull}). Section
\ref{sec_error_analysis_strong} derives \textit{a posteriori} and \textit{a
priori} error bounds for the method. Section \ref{sec_numerical_results}
reports numerical experiments on a benchmark problem in heat transfer that
employs a finite-element discretization (Section \ref{subsect_heat_equation})
and a benchmark problem in fluid dynamics that employs a finite-difference
discretization (Section \ref{sec:burg}).
Finally, Section \ref{sec_conclusions} concludes the paper.

\section{Domain-decomposition formulation}\label{sec_DDFOM_formulation}

We consider the (high-fidelity) full-order model to be expressed as a parameterized system of nonlinear algebraic equations
\begin{equation}\label{eq_originalAlgebraic}
\residual(\stateSol;\params) = \zero,
\end{equation}
where the residual $\residual:\RR{\ndof}\times \paramDomain\rightarrow\RR{\ndof}$
is nonlinear in (at least) its first argument,
$\params\in\paramDomain\subseteq \RR{\nparams}$ denotes the parameters, and 
$\stateSol:\paramDomain\rightarrow\RR{\ndof}$ denotes the state, which is implicitly defined as the
solution to Eq.~\eqref{eq_originalAlgebraic} given an instance of the
parameters. 
Such problems arise, for
example, after applying spatial discretization to a stationary
PDE problem; because we take Eq.~\eqref{eq_originalAlgebraic} to be our
full-order model, our methodology is \textit{spatial-discretization agnostic}. 
For notational simplicity, we suppress all dependence on the parameters
$\param$ until needed in Section \ref{sec_basis_construction}.

We consider an \textit{algebraic decomposition} of this problem into
$\nsubdomains(\leq\ndof)$ `subdomains' such that
the residual satisfies
\begin{equation}\label{eq_algebraicDecomposition}
\residual:\stateDummy \mapsto
\sum_{i=1}^\nsubdomains[\projectionResi]^T\residuali(\projectionStateInteriori\stateDummy,\projectionStateBoundaryi\stateDummy),\quad
\forall \stateDummy\in\RR{\ndof}.
\end{equation}
Here, 
$\residuali:\RR{\ndofInteriori}\times\RR{\ndofBoundaryi}\rightarrow\RR{\nresi}$
with 
$\residuali:(\stateDummyInteriori,\stateDummyBoundaryi)\mapsto\residuali(\stateDummyInteriori,\stateDummyBoundaryi)$
denotes the $i$th subdomain residual, $\projectionResi\in\{0,1\}^{\nresi\times
\ndof}$ denotes $i$th the residual sampling matrix, 
$\projectionStateInteriori\in\{0,1\}^{\ndofInteriori\times \ndof}$ denotes
the $i$th interior-state sampling matrix, and 
$\projectionStateBoundaryi\in\{0,1\}^{\ndofBoundaryi\times \ndof}$ denotes
the $i$th interface-state sampling matrix; each sampling matrix comprises
selected rows of the $\ndof\times\ndof$ identity matrix. The residual sampling
matrix is such that the decomposition is \textit{algebraically
non-overlapping}, i.e., $\projectionResArg{i}[\projectionResArg{j}]^T=\zero$ for
$i\neq j$ and
$\sum_{i=1}^\nsubdomains \nresi = \ndof$. Further, the interior-state sampling
matrix satisfies
$\projectionStateInteriorArg{i}[\projectionStateInteriorArg{j}]^T=\zero$ for
$i\neq j$; this implies that there is no overlap between the interior states
associated with different subdomains. Thus, the operators
$\projectionStateInteriorArg{i}$ and $\projectionStateBoundaryArg{i}$,
$i\innatseq{\nsubdomains}$ are determined from the sparsity patterns of the
sampled Jacobians $\projectionResi\frac{\partial\residual}{\partial\stateDummy}$,
$i\innatseq{\nsubdomains}$.
We define the total number of degrees of freedom for each subdomain as $\ndofi \defeq
\ndofInteriori + \ndofBoundaryi$; note that $\ndofi\geq\nresi$. 

If we set $\stateInteriori\defeq \projectionStateInteriori\state\in \RR{\ndofInteriori}$ and
$\stateBoundaryi\defeq \projectionStateBoundaryi\state\in\RR{\ndofBoundaryi}$, then from
Eqs.~\eqref{eq_originalAlgebraic}--\eqref{eq_algebraicDecomposition}, the solution
for each subdomain $\statei\defeq(\stateInteriori,\stateBoundaryi)$ satisfies
\begin{equation} 
\residuali(\stateInteriori,\stateBoundaryi) =\zero, \quad i\innatseq{\nsubdomains}
\end{equation} 
along with compatibility conditions that enforce consistency across the
boundary states for different subdomains. To reason about these compatibility
conditions, we define a set of $\nports$ `ports'; the $j$th port is
characterized by $\ndofPortsArg{j}\leq\ndof$ states that are shared across a
fixed set of
subdomains denoted by $\subdomainPorts{j}\subseteq\natseq{\nsubdomains}$. Then, the compatibility conditions
can be expressed as
\begin{equation} \label{eq_portConstraints}
\projectionStatePort{j}{i}\stateBoundaryArg{i} = 
\projectionStatePort{j}{\ell}\stateBoundaryArg{\ell},\quad
i,\ell\in\subdomainPorts{j}, \ j\innatseq{\nports},
\end{equation} 
where the port sampling matrix
$\projectionStatePort{j}{i}\in\{0,1\}^{\ndofPortsArg{j}\times\ndofBoundaryi}$
comprises selected rows of the $\ndofBoundaryi\times \ndofBoundaryi$ identity matrix. For a given
subdomain $i$, we require the ports to be non-overlapping such that
$\projectionStatePort{j}{i}[\projectionStatePort{\ell}{i}]^T = \zero$ for
$j,\ell\in\portsSubdomains{i}$ and $j\neq\ell$
and
$\sum_{j\in\portsSubdomains{i}}\ndofPortsArg{j} = \ndofBoundaryi$,
where
we have defined the set of ports associated with subdomain $i$ as
$\portsSubdomains{i}\defeq\{j\ |\
i\in\subdomainPorts{j}\}\subseteq\natseq{\nports}$.
We note that for a given port $j$,
although the number of \textit{total} pairwise compatibility conditions 
 arising from
Eq.~\eqref{eq_portConstraints}
is
$k\choose 2$, the number of \textit{unique} pairwise compatibility conditions is
only $\nconstraintsPort{j}\defeq\card{\subdomainPorts{j}}-1$.
Using this formulation, the full-order model \eqref{eq_originalAlgebraic} can be
recast in decomposed form as
\begin{align} \label{eq_globaldecomposed}
\begin{split}
	&\residuali(\stateInteriori,\stateBoundaryi) =\zero,\quad i=1,\ldots,\nsubdomains\\
	&\sum_{i=1}^\nsubdomains\constraintMati \stateBoundaryi = \zero,
\end{split}
\end{align} 
where $\constraintMati\in\{-1,0,1\}^{\nconstraints\times \ndofBoundaryi}$ with $\nconstraints = \sum_{j=1}^\nports\nconstraintsPort{j}\ndofPortsArg{j}$ denote the
constraint matrices associated with port-compatibility conditions
\eqref{eq_portConstraints}. Note that
Eqs.~\eqref{eq_globaldecomposed} comprise $\sum_{i=1}^\nsubdomains\nresi +
\nconstraints = \ndof + \nconstraints$ equations in $\sum_{i=1}^\nsubdomains\ndofInteriori +
\sum_{i=1}^\nsubdomains\ndofBoundaryi$ unknowns; as there exists a unique solution to
these equations\footnote{\CH{With the assumption that equation \eqref{eq_originalAlgebraic} is well-posed.}}, we have
$
\ndof + \nconstraints \geq\sum_{i=1}^\nsubdomains\ndofInteriori +
\sum_{i=1}^\nsubdomains\ndofBoundaryi.
$

For illustration, Figure \ref{fig_divided_models} shows an example of a
decomposition using $\nsubdomains=4$ subdomains and $\nports=5$ global ports
for the case of a full-order model derived from discretizing  a PDE in two
spatial dimensions using a residual operator with a 9-point stencil. Figure
\ref{fig_SP_model_nodes_definition} shows the degrees of freedom and residual
elements associated with subdomain $\domain_{1}$.

\begin{figure}[h!] 
\centering 
\includegraphics[width=12cm]{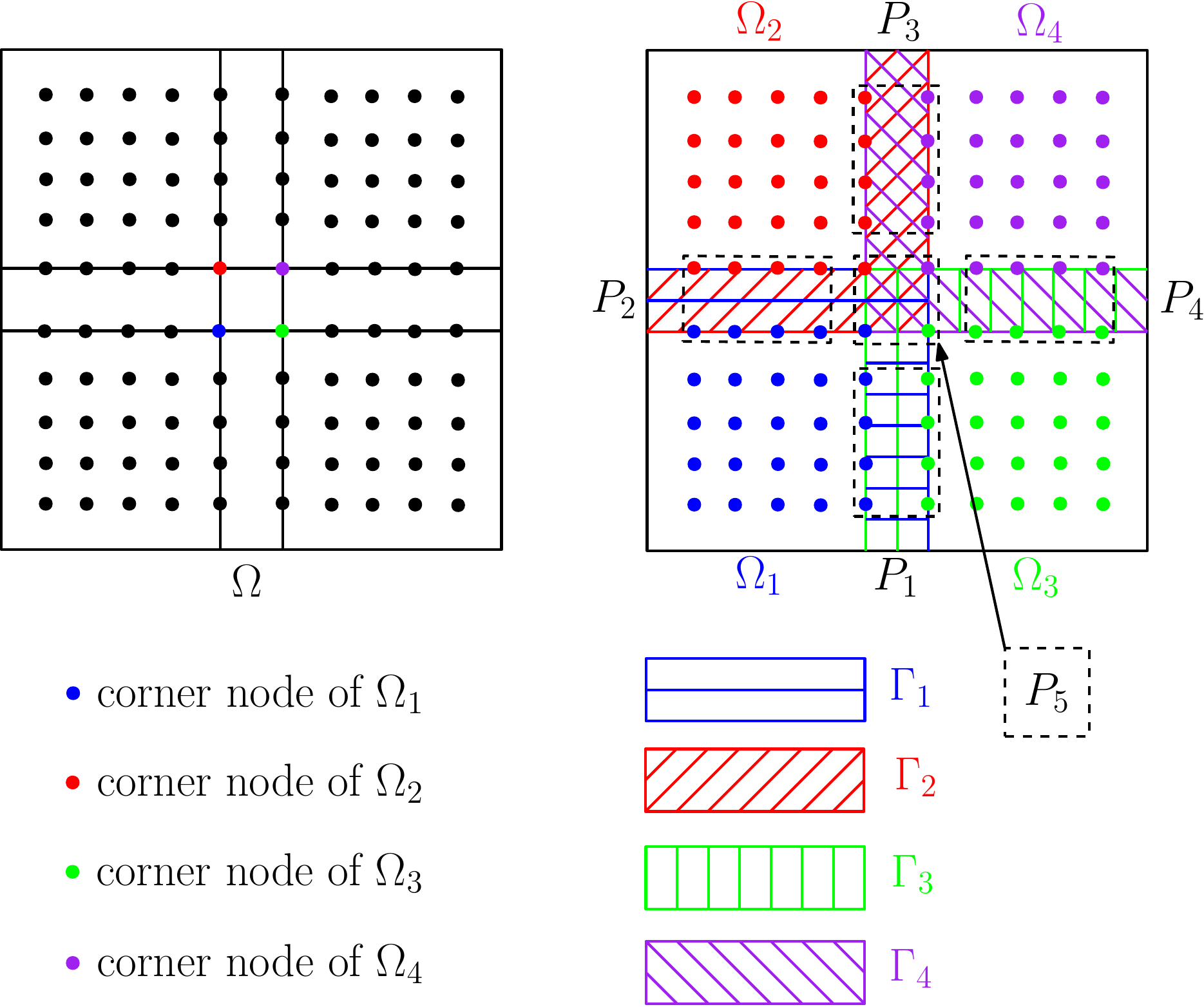} 
\caption{Domain-decomposition example: full-order model 
	derived from discretizing  a PDE in two spatial dimensions using a residual
	operator with a 9-point stencil. (Left) Residual and corner nodes. (Right) Each colored point is associated with a residual for the associated subdomain, $\nsubdomains = 4$ subdomains, $\nports = 5$ ports, $\subdomainPorts{1} = \{1,3\}$, $\subdomainPorts{2} = \{1,2\}$,
$\subdomainPorts{3} = \{2,4\}$, $\subdomainPorts{4} = \{3,4\}$,
$\subdomainPorts{5} = \{1,2,3,4\}$; and $\portsSubdomains{1}=\{1,2,5\}$,
	$\portsSubdomains{2}=\{2,3,5\}$, $\portsSubdomains{3}=\{1,4,5\}$,
	$\portsSubdomains{4}=\{3,4,5\}$. Ports $P_1$, $P_2$, $P_3$, $P_4$ are
  \YC{two-component} ports, and port $P_5$ is \YC{four-component} port. \CH{(See Appendix~\ref{subsubsect_GNAT_greedy_algorithm} for more details about corner nodes.)}
	}  
\label{fig_divided_models} 
\end{figure} 

\begin{figure}[h!] 
\centering 
\includegraphics[width=10cm]{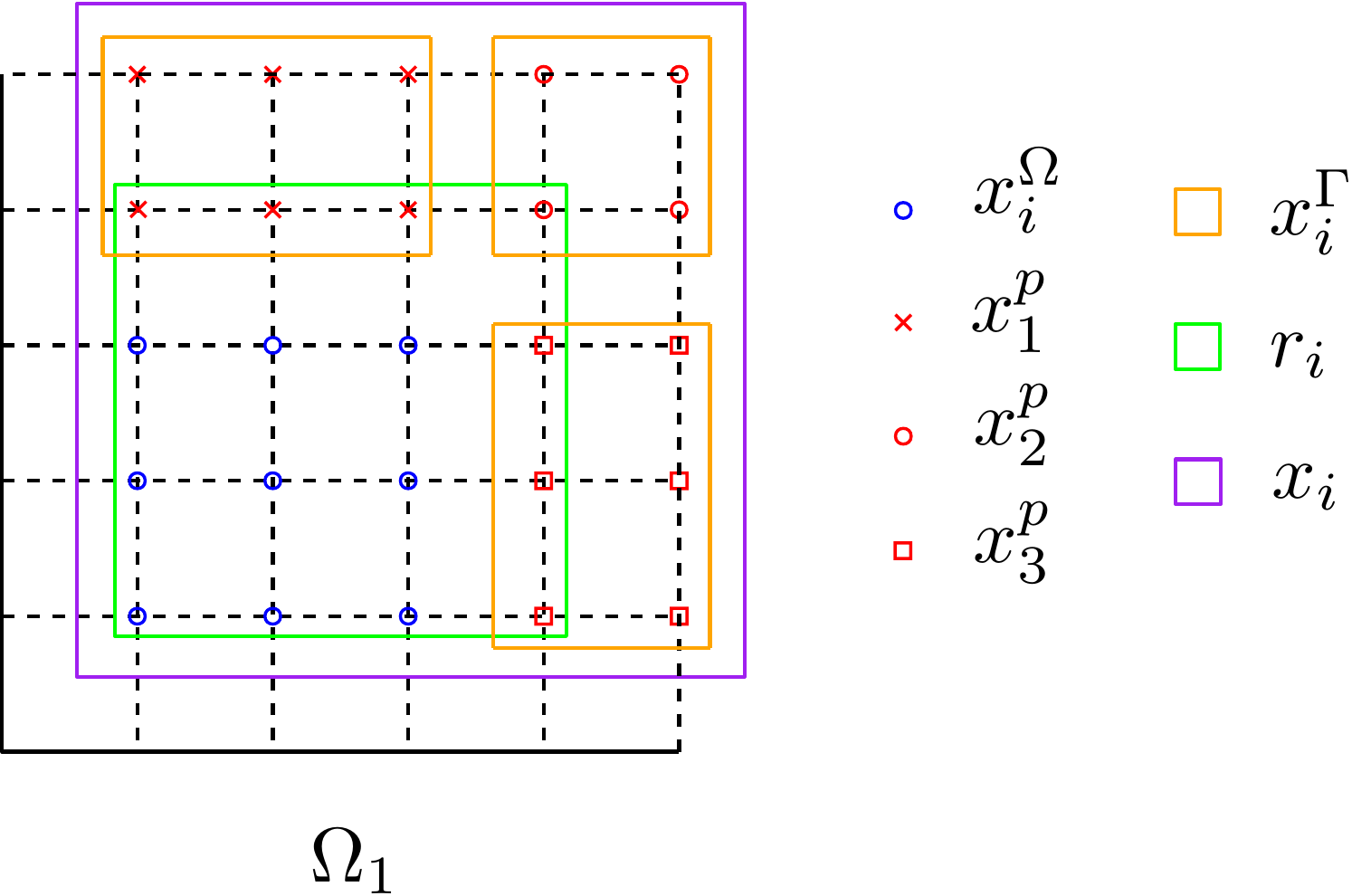} 
\caption{Domain-decomposition example: full-order model derived from discretizing  a PDE in two spatial dimensions using a residual operator with a 9-point stencil. This figure considers the bottom-left subdomain $\domainArg{1}$ of Fig.~\ref{fig_divided_models}: interior DOFs (blue circles), interface DOFs (nodes included in yellow rectangle), residual DOFs (nodes included in green rectangle) and subdomain DOFs (nodes included in purple rectangle). \CH{The remaining boundary nodes correspond to given specified boundary conditions}.}  \label{fig_SP_model_nodes_definition} 
\end{figure} 


\section{Domain decomposition least-squares Petrov--Galerkin (DD-LSPG) projection} \label{subsect_DDLSPG_formulation}

We now consider applying least-squares Petrov--Galerkin (LSPG) model reduction
\cite{bui2008model,legresley2006application,carlbergbou-mosleh2011,carlberg2017galerkin}
in the domain-decomposition setting presented in Section \ref{sec_DDFOM_formulation}.   

\subsection{Interior/boundary bases} \label{sec:intBound}
Assume we have constructed interior reduced bases
$\rbInteriori\in\RRstar{\ndofInteriori\times\nrbInteriori}$ with
$\nrbInteriori\leq\ndofInteriori$, $i=1,\ldots,\nsubdomains$ and interface reduced
bases $\rbBoundaryi\in\RRstar{\ndofBoundaryi\times\nrbBoundaryi}$
with $\nrbBoundaryi\leq\ndofBoundaryi$, $i=1,\ldots,\nsubdomains$, where
$\RRstar{m\times n}$ denotes the non-compact Stiefel manifold: the set of
full-column-rank $m\times n$ real-valued matrices;
Section~\ref{sec_basis_construction} described proposed approaches for
constructing these bases.
We then approximate the
solution on the $i$th subdomain in the associated $\nrbi$-dimensional trial
subspace with
$\nrbi=\nrbInteriori+\nrbBoundaryi$
as 
$
	\statei \approx \stateApproxi \equiv (\stateApproxInteriori,
	\stateApproxBoundaryi)=(\rbInteriori\redStateInteriori,
	\rbBoundaryi\redStateBoundaryi) \in \range{\rbInteriori} \times \range{\rbBoundaryi}\subseteq\RR{\ndofInteriori}\times\RR{\ndofBoundaryi}.$

We formulate the domain-decomposition LSPG (DD-LSPG) reduced-order model by minimizing the sum of squared
residual norms over these trial subspaces subject to (possibly weak)
port-compatibility conditions, i.e., we compute $(\redStateInteriori,
\redStateBoundaryi)$, $i\innatseq{\nsubdomains}$ as the solution to the
optimization problem
\begin{align} \label{eq_globalOpt}
\begin{split} 
\underset{(\redStateDummyInteriorArg{i},\redStateDummyBoundaryArg{i}),\,
	i=1,\ldots,\nsubdomains}{\text{minimize}}\quad&\frac{1}{2}\sum_{i=1}^\nsubdomains
\|\hyperMati\residuali(\rbInteriori\redStateDummyInteriori, \rbBoundaryi\redStateDummyBoundaryi)\|_2^2\\
\text{subject to}\quad &\sum_{i=1}^\nsubdomains
\constraintMatROMi \rbBoundaryi \redStateDummyBoundaryi = \zero.
\end{split} 
\end{align} 
Here, $\constraintMatROMi\in\RR{\nconstraintsROM\times\ndofBoundaryi}$,
$i=1,\ldots,\nsubdomains$ denote
constraint matrices 
(see Section \ref{subsect_constraint_types} for how this can be derived
from the constraint matrices $\constraintMati$, $i=1,\ldots,\nsubdomains$)
with
$\nconstraintsROM\leq\nconstraints$ denotes the number of constraints
incurred by port compatibility,  and
$\hyperMatArg{i}\in\RR{\nhyperMatArg{i}\times\nresi}$,
$i=1,\ldots,\nsubdomains$  with
$\nhyperMatArg{i}\leq\nresi$ denotes a matrix that enables the
subdomain residuals to be minimized in any
weighted $\ell^2$-(semi)norm.

In particular, we focus on three choices of matrix $\hyperMatArg{i}$:
$\hyperMatArg{i} = \identity$ as in ``standard'' LSPG, $\hyperMati=\sampleMati$ in
the case of collocation hyper-reduction \cite{legresley2006application,astrid2008missing,ryckelynck2005priori}, and
$\hyperMati=(\sampleMati\rbResi)^+\sampleMati$ in the case of gappy POD
hyper-reduction \cite{carlbergbou-mosleh2011, carlberg2013gnat,
sirovichOrigGappy} (\CH{see Appendix~\ref{subsect_GNAT_offline} for more details}).
Here,
$\sampleMati\defeq[\unitvecArg{\sampleij{i}{1}}\ \ldots\
\unitvecArg{\sampleij{i}{\nsampleArg{i}}}]^T\in\{0,1\}^{\nsampleArg{i}\times
\nresi}$ comprises selected rows of the $\nresi\times\nresi$ identity matrix,
$\unitvecArg{i}$ denotes the $i$th Kronecker vector, and
$\{\sampleij{i}{1},\ldots,\sampleij{i}{\nsampleArg{i}}\}\subseteq\{1,\ldots,\nresi\}$
denotes the indices of the residual elements sampled by the operator. On the
other hand,  $\rbResi\in\RRstar{\nresi\times\nrbResi}$,
$i=1,\ldots,\nsubdomains$ denote reduced bases for the residual and the
superscript + denotes the Moore--Penrose pseudoinverse. For the
pseudoinverses to correspond to left inverses, the matrices
$\sampleMati\rbResi$, $i=1,\ldots,\nsubdomains$ must have full column rank,
which in turn necessitates $\nsampleArg{i}\geq\nrbResi$,
$i=1,\ldots,\nsubdomains$.
Note that
$\nhyperMati = \nsampleArg{i}$ for collocation hyper-reduction and $\nhyperMati = \nrbResi$ in
the case of gappy POD hyper-reduction. 
We note that hyper-reduction is required to ensure that
computing the ROM solution incurs an $\ndof$-independent operation count.

\subsection{Full-subdomain bases}\label{sec:fullSub}
We also consider a variation on this
formulation corresponding to the 
case of classical full-subdomain reduced bases \cite{maday2002reduced, iapichino2012reduced,
buffoni2009iterative, corigliano2015model}. In this case, each subdomain
is equipped with a single reduced basis $\rbi\in\RRstar{\ndofi\times\nrbi}$
whose columns can have support over both interior and interface degrees
of freedom
such that $\rbInteriori=
\projectionStateInteriori\rbi\in\RR{\ndofInteriori\times\nrbInteriori}$ and
$\rbBoundaryi=
\projectionStateBoundaryi\rbi\in\RR{\ndofBoundaryi\times\nrbBoundaryi}$ with
$\nrbInteriori=\nrbBoundaryi=\nrbi$; note that the reduced bases
$\rbInteriori$ and $\rbBoundaryi$ need not have full column rank individually.
Approximating the solution on the $i$th subdomain as 
$
	\statei\approx\stateApproxi \equiv (\stateApproxInteriori,
	\stateApproxBoundaryi)=(\rbInteriori\redStatei,
	\rbBoundaryi\redStatei)$, the resulting DD-LSPG model computes solutions
	$\redStatei$, $i\innatseq{\nsubdomains}$ as the solution to the optimization
	problem
\begin{align} \label{eq_globalOpt_one_arg}
\begin{split} 
	\underset{\redStateDummyArg{i},\,i\innatseq{\nsubdomains}}{\text{minimize}}\quad&\frac{1}{2}\sum_{i=1}^\nsubdomains
\|\hyperMati\residuali(\rbInteriori \redStateDummyArg{i}, \rbBoundaryi \redStateDummyArg{i})\|_2^2\\
\text{subject to}\quad &\sum_{i=1}^\nsubdomains
\constraintMatROMi \rbBoundaryi \redStateDummyArg{i} = \zero.
\end{split} 
\end{align} 

We again consider the choices of $\hyperMatArg{i} = \identity$,
$\hyperMati=\sampleMati$, and $\hyperMati=(\sampleMati\rbResi)^+\sampleMati$.

For both Problems \eqref{eq_globalOpt} and
\eqref{eq_globalOpt_one_arg}, 
the
effective number of degrees of freedom in the resulting ROM corresponds to
$\nrb= \sum_{i=1}^\nsubdomains\nrbi - \rank{\constraintMatROM}$. Here, we have 
defined the reduced constraint matrix as
$\constraintMatROM\defeq[ \constraintMatROMArg{1} \rbBoundaryArg{1}\ \cdots\
\constraintMatROMArg{\nsubdomains} \rbBoundaryArg{\nsubdomains}] $. This
result holds
because the null
space of the operator $\constraintMatROM$ defines the effective subspace over
which unconstrained minimization takes place.

\subsection{Strong versus weak compatibility constraints}\label{subsect_constraint_types}

Recall that the constraint matrices $\constraintMati$,
$i=1,\ldots,\nsubdomains$ are derived by enforcing the degrees of freedom on
each port to be consistent across shared subdomains according to
Eq.~\eqref{eq_portConstraints}. We can effectively reduce the number of
constraints by weakening this notion of consistency through
enforcing a zero inner product between the difference between port solutions
and a collection of prescribed test functions, i.e.,
\begin{equation} \label{eq_portConstraints}
\testFunctionArg{j}\projectionStatePort{j}{i}\stateBoundaryArg{i} = 
\testFunctionArg{j}\projectionStatePort{j}{\ell}\stateBoundaryArg{\ell},\quad
i,\ell\in\subdomainPorts{j}, \ j\innatseq{\nports},
\end{equation} 
where $\testFunctionArg{j}\in\RR{\nportTestFunction{j}\times
\ndofPortsArg{j}}$ with $\nportTestFunction{j}\leq\ndofPortsArg{j}$ 
and $\rank{\testFunctionArg{j}} = \nportTestFunction{j}$
denotes
the matrix of test functions. 
If we assemble a constraint matrix from (weak) compatibility conditions
\eqref{eq_portConstraints} in the same manner that the original constraint
matrices $\constraintMati$, $i=1,\ldots,\nsubdomains$ were assembled from
(strong) compatibility conditions~\eqref{eq_portConstraints}, we obtain the
constraint matrices
$\constraintMatROMi\in\RR{\nconstraintsROM\times\ndofBoundaryi}$,
$i=1,\ldots,\nsubdomains$ with $\nconstraintsROM =
\sum_{j=1}^\nports\nconstraintsPort{j}\nportTestFunction{j}$ that have the
structure $\constraintMatROMi=\randomMat \constraintMatArg{i}$ for some
$\randomMat\in\RR{\nconstraintsROM\times \nconstraints}$, with
$\randomMat=\identity$ in the case of strong compatibility constraints
\eqref{eq_portConstraints}.
\YC{We generate $\testFunctionArg{j}$ from normal distribution, e.g., using {\it
randn} function in MATLAB. This matrix is widely used to sample from large data
and obtain optimal averaging effect. For example, randomized SVD introduced in
\cite{halko2011finding} uses exactly the same random matrix. Our numerical
examles shows that this choice is effective.}

\begin{remarked}\label{rem:globalSol}
Critically, the case of weak compatibility constraints (i.e.,
	$\nportTestFunction{j}<\ndofPortsArg{j}$) admits discrepancies between the
	restrictions of subdomain solutions to the $j$th port, i.e., Eq.~\eqref{eq_portConstraints}
	will not hold in general. This phenomenon
	precludes the existence of a `global solution', as ports may not have a
	uniquely computed solution.  While this may appear to be deleterious to the
	accuracy of the computed DD-LSPG solution, we show in the numerical
	experiments that this relaxation is critical to obtain accurate solutions
	when neighboring components have incompatible bases on the associated port;
	this occurs in particular for interface and
	full-subdomain bases. For such bases, enforcing strong compatibility
	constraints effectively causes the subdomains to generate the trivial
  solution on the associated ports, yielding poor overall solution accuracy,
  even if the solution on the ports is uniquely defined. \YC{In summary, the
  existence of a solution depends heavily on the compatibility of the port
  bases. If the incompatible port bases are  generated, then a weak
  constraint is necessary for the existence of a solution. On the other hand, if
  the compatible port bases are used, then both strong and weak constraints
  ensure the existence of a solution. }
\end{remarked}

\section{Sequential quadratic programming solver}\label{sec_solver}

Problems \eqref{eq_globalOpt}--\eqref{eq_globalOpt_one_arg} can be classified as a \textit{nonlinear least-squares problems with linear equality
constraints}. As such, they are well-suited to solution with a sequential quadratic programming (SQP) method, which in this case is equivalent to applying Newton's method to the Karush--Kuhn--Tucker (KKT) necessary conditions for optimality. This section describes this
solution approach.

\subsection{Interior/boundary bases} \label{sec_sqp_interior}
We begin by defining the Lagrangian associated with problem \eqref{eq_globalOpt}
\begin{equation} 
\lagrangian:(\redStateDummyInteriorArg{1},\redStateDummyBoundaryArg{1},\ldots,\redStateDummyInteriorArg{\nsubdomains},\redStateDummyBoundaryArg{\nsubdomains},\lagrangeROMDummy)\mapsto
\frac{1}{2}\sum_{i=1}^\nsubdomains
\|\hyperMati\residuali(\rbInteriori\redStateDummyInteriori,\rbBoundaryi\redStateDummyBoundaryi)\|_2^2
+ \sum_{i=1}^\nsubdomains
\lagrangeROMDummy^T\constraintMatROMi\rbBoundaryi\redStateDummyBoundaryi,
\end{equation} 
where $\lagrangeROMDummy\in\RR{\nconstraintsROM}$ denotes Lagrange multipliers.
The KKT conditions arise from
stationarity of the Lagrangian, i.e., the DD-LSPG ROM solution
$(\redStateInteriorArg{1},\redStateBoundaryArg{1},\ldots,\redStateInteriorArg{\nsubdomains},\redStateBoundaryArg{\nsubdomains},\lagrangeROM)$
satisfies
\begin{align}\label{eq_KKT} 
	\begin{split}
	&\frac{\partial \lagrangian}{\partial
	\redStateDummyInteriori}(\redStateInteriorArg{1},\redStateBoundaryArg{1},\ldots,\redStateInteriorArg{\nsubdomains},\redStateBoundaryArg{\nsubdomains},\lagrangeROM)
	= 
\redResInteriori(\redStateInteriori,\redStateBoundaryi)=
	\zero,\quad i=1,\ldots,\nsubdomains\\
	&\frac{\partial \lagrangian}{\partial
	\redStateDummyBoundaryi}(\redStateInteriorArg{1},\redStateBoundaryArg{1},\ldots,\redStateInteriorArg{\nsubdomains},\redStateBoundaryArg{\nsubdomains},\lagrangeROM)
		=\redResBoundaryi(\redStateInteriori,\redStateBoundaryi,\lagrangeROM)
=\zero,\quad i=1,\ldots,\nsubdomains
	\\
	&\frac{\partial \lagrangian}{\partial
	\lagrangeROMDummy}(\redStateInteriorArg{1},\redStateBoundaryArg{1},\ldots,\redStateInteriorArg{\nsubdomains},\redStateBoundaryArg{\nsubdomains},\lagrangeROM)=\sum_{i=1}^\nsubdomains
\constraintMatROMi\rbBoundaryi\redStateBoundaryi =\zero,
	\end{split}
\end{align} 
where we have defined
\begin{align}
	\begin{split}
	&\redResInteriori:(\redStateDummyInteriori,\redStateDummyBoundaryi)\mapsto[\rbInteriori]^T
	\frac{\partial\residuali}{\partial\stateInteriori}(\rbInteriori\redStateDummyInteriori
	, \rbBoundaryi\redStateDummyBoundaryi)^T[\hyperMati]^T\hyperMati
	\residuali(\rbInteriori\redStateDummyInteriori
, \rbBoundaryi\redStateDummyBoundaryi)\\
	&\redResBoundaryi:(\redStateDummyInteriori,\redStateDummyBoundaryi,\lagrangeROMDummy)\mapsto[\rbBoundaryi]^T\frac{\partial\residuali}{\partial \stateBoundaryi}(\rbInteriori\redStateDummyInteriori
, \rbBoundaryi\redStateDummyBoundaryi)^T
[\hyperMati]^T\hyperMati
	\residuali(\rbInteriori\redStateDummyInteriori
, \rbBoundaryi\redStateDummyBoundaryi) +
	[\rbBoundaryi]^T\constraintMatROMi^T\lagrangeROMDummy
	\end{split}
\end{align}
for $i=1,\ldots,\nsubdomains$.
Applying Newton's method with a Gauss--Newton Hessian approximation to solve
the system of nonlinear algebraic equations~\eqref{eq_KKT}
yields the SQP iterations for $k=0,\ldots,K$
\begin{gather} \label{eq_SQP}
	\small
\begin{split}
\left[
	\begin{array}{cccccc}
	\hessianIntIntArg{1}(\redStateInteriorArgk{1},\redStateBoundaryArgk{1}) 
	&
	\hessianIntBoundArg{1}(\redStateInteriorArgk{1},\redStateBoundaryArgk{1}) 
	& \ldots & \zero & \zero & \zero\\
	\hessianBoundIntArg{1}(\redStateInteriorArgk{1},\redStateBoundaryArgk{1}) 
	& \hessianBoundBoundArg{1}(\redStateInteriorArgk{1},\redStateBoundaryArgk{1}) 
	& \ldots & \zero & \zero & [\rbBoundaryArg{1}]^T\constraintMatROMArg{1}^T \\
	\vdots & \vdots & \ddots & \vdots & \vdots & \vdots\\
	\zero & \zero & \ldots &
	\hessianIntIntArg{\nsubdomains}(\redStateInteriorArgk{\nsubdomains},\redStateBoundaryArgk{\nsubdomains})
	& \hessianIntBoundArg{\nsubdomains}(\redStateInteriorArgk{\nsubdomains},\redStateBoundaryArgk{\nsubdomains}) 
	& \zero\\
	\zero & \zero & \ldots &
	\hessianBoundIntArg{\nsubdomains}(\redStateInteriorArgk{\nsubdomains},\redStateBoundaryArgk{\nsubdomains}) & \hessianBoundBoundArg{\nsubdomains}(\redStateInteriorArgk{\nsubdomains},\redStateBoundaryArgk{\nsubdomains})
	& [\rbBoundaryArg{\nsubdomains}]^T\constraintMatROMArg{\nsubdomains}^T\\
	\zero & \constraintMatROMArg{1}\rbBoundaryArg{1} & \ldots &\zero & \constraintMatROMArg{\nsubdomains}\rbBoundaryArg{\nsubdomains} & \zero
	\end{array}
\right]
\left[
	\begin{array}{c}
	\searchDirInteriorArgk{1}\\
	\searchDirBoundaryArgk{1}\\
	\vdots\\
	\searchDirInteriorArgk{\nsubdomains}\\
	\searchDirBoundaryArgk{\nsubdomains}\\
	\searchDirLagrangek
	\end{array}
\right] = \\
 -
\left[
	\begin{array}{c}
	\redResInteriorArg{1}(\redStateInteriorArgk{1},\redStateBoundaryArgk{1})\\
	\redResBoundaryArg{1}(\redStateInteriorArgk{1},\redStateBoundaryArgk{1})\\
	\vdots\\
	\redResInteriorArg{\nsubdomains}(\redStateInteriorArgk{\nsubdomains},\redStateBoundaryArgk{\nsubdomains})\\
	\redResBoundaryArg{\nsubdomains}(\redStateInteriorArgk{\nsubdomains},\redStateBoundaryArgk{\nsubdomains})\\
	\sum_{i=1}^\nsubdomains
	\constraintMatROMi\rbBoundaryi\redStateBoundaryi
	\end{array}
\right] 
\end{split}
\end{gather} 
  
\noindent where 
\begin{align}\label{eq_Hessian}
	\begin{split}
\hessianIntInti:(\redStateDummyInteriori,\redStateDummyBoundaryi)&\mapsto[\rbInteriorArg{i}]^T\frac{\partial\residualArg{i}}{\partial \stateDummyInteriori}(\rbInteriorArg{i}\redStateDummyInteriorArg{i}
,
\rbBoundaryArg{i}\redStateDummyBoundaryArg{i})^T [\hyperMati]^T\hyperMati\frac{\partial\residualArg{i}}{\partial \stateDummyInteriori}(\rbInteriorArg{i}\redStateDummyInteriorArg{i}
,
\rbBoundaryArg{i}\redStateDummyBoundaryArg{i})\rbInteriorArg{i}  \\
\hessianIntBoundi:(\redStateDummyInteriori,\redStateDummyBoundaryi)&\mapsto[\rbInteriorArg{i}]^T\frac{\partial\residualArg{i}}{\partial \stateDummyInteriori}(\rbInteriorArg{i}\redStateDummyInteriori
,
\rbBoundaryArg{i}\redStateDummyBoundaryArg{i})^T[\hyperMati]^T\hyperMati\frac{\partial\residualArg{i}}{\partial \stateDummyBoundaryi}(\rbInteriorArg{i}\redStateDummyInteriorArg{i}
,
\rbBoundaryArg{i}\redStateDummyBoundaryArg{i})\rbBoundaryArg{i}  \\
\hessianBoundInti:(\redStateDummyInteriori,\redStateDummyBoundaryi)&\mapsto[\rbBoundaryArg{i}]^T\frac{\partial\residualArg{i}}{\partial \stateDummyBoundaryi}(\rbInteriorArg{i}\redStateDummyInteriorArg{i}
,
\rbBoundaryArg{i}\redStateDummyBoundaryArg{i})^T[\hyperMati]^T\hyperMati\frac{\partial\residualArg{i}}{\partial\stateDummyInteriori}(\rbInteriorArg{i}\redStateDummyInteriorArg{i}
,
\rbBoundaryArg{i}\redStateDummyBoundaryArg{i})\rbInteriorArg{i} \\
\hessianBoundBoundi:(\redStateDummyInteriori,\redStateDummyBoundaryi)&\mapsto[\rbBoundaryArg{i}]^T\frac{\partial\residualArg{i}}{\partial \stateDummyBoundaryi}(\rbInteriorArg{i}\redStateDummyInteriorArg{i}
,
\rbBoundaryArg{i}\redStateDummyBoundaryArg{i})^T[\hyperMati]^T\hyperMati\frac{\partial\residualArg{i}}{\partial \stateDummyBoundaryi}(\rbInteriorArg{i}\redStateDummyInteriorArg{i}
,
\rbBoundaryArg{i}\redStateDummyBoundaryArg{i})\rbBoundaryArg{i}
	\end{split}
\end{align}
for $i=1,\ldots,\nsubdomains$.
We can then update the solution as 
\begin{align} \label{eq_prdumono_update}
	\begin{split}
	\redStateInteriorikp &= \redStateInteriorik + \alpha^{(k)}
\searchDirInteriorik,\quad i=1,\ldots,\nsubdomains  \\
	\redStateBoundaryikp &= \redStateBoundaryik + \alpha^{(k)}
\searchDirBoundaryik,\quad i=1,\ldots,\nsubdomains  \\
	\lagrangeROMkp &= \lagrangeROMk + \alpha^{(k)}
\searchDirLagrangek,
	\end{split}
\end{align}
where $\alpha^{(k)}$ is a step length that can be computed, e.g., via line search.
We note that the sparse block structure of SQP iterations~\eqref{eq_SQP} admit
interesting parallel-solution strategies, which is the subject of future work.
\YC{We also note that the Gauss--Newton approximation is widely used for the
solution process of nonlinear problems due to its practicality, i.e., no need
to compute an exact Hessian and often achieve a quadratic convergence rate
although its convergence is not guaranteed \cite{mascarenhas2014divergence}. Our
numerical examples show that the Gauss--Newton method works well for the
problems considered in this paper.}

\subsection{Full-subdomain bases}\label{sqpFull}

Analogously to Section \ref{sec_sqp_interior}, the Lagrangian associated
with problem
\eqref{eq_globalOpt_one_arg} is defined as
\begin{equation} 
\lagrangian:(\redStateDummyArg{1},\ldots,\redStateDummyArg{\nsubdomains},\lagrangeROMDummy)\mapsto
	\frac{1}{2}\sum_{i=1}^\nsubdomains \|\hyperMati\residuali(\rbInteriori
	\redStateDummyArg{i},\rbBoundaryi \redStateDummyArg{i})\|_2^2 +
	\sum_{i=1}^\nsubdomains \lagrangeROMDummy^T \constraintMatROMi \rbBoundaryi \redStateDummyArg{i},
\end{equation} 
where the KKT system can be derived from stationarity of the Lagrangian such
that the DD-LSPG ROM solution
$(\redStateArg{1},\ldots,\redStateArg{\nsubdomains},\lagrangeROM)$ satisfies
\begin{align}\label{eq_KKT_global} 
	\begin{split}
	&\frac{\partial \lagrangian}{\partial
	\redStateDummyi}(\redStateArg{1},\ldots,\redStateArg{\nsubdomains},\lagrangeROM)
	= 
	\redResSubdom{i}(\redStatei,\lagrangeROM)=
	\zero,\quad i=1,\ldots,\nsubdomains\\
	&\frac{\partial \lagrangian}{\partial
	\lagrangeROM}(\redStateArg{1},\ldots,\redStateArg{\nsubdomains},\lagrangeROM)=\sum_{i=1}^\nsubdomains
\constraintMatROMi\rbBoundaryi\redStateBoundaryi =\zero,
	\end{split}
\end{align} 
 where we have defined
\begin{align}\label{eq_KKT_subdom_BFs} 
\begin{split}
	\redResSubdom{i}:(\redStateDummyArg{i},\lagrangeROMDummy)\mapsto&
[\rbInteriori]^T \frac{\partial\residuali}{\partial\stateDummyInteriori}(\rbInteriori\redStateDummyArg{i}, \rbBoundaryi\redStateDummyArg{i})^T
	[\hyperMati]^T\hyperMati\residuali(\rbInteriori \redStateDummyArg{i}, \rbBoundaryi \redStateDummyArg{i}) \\
	&+ [\rbBoundaryi]^T\frac{\partial\residuali}{\partial \stateDummyBoundaryi}(\rbInteriori \redStateArg{i}, \rbBoundaryi \redStateArg{i})^T
	[\hyperMati]^T\hyperMati
	\residuali(\rbInteriori \redStateArg{i}, \rbBoundaryi \redStateArg{i}) +
	(\rbBoundaryi)^T\constraintMatROMi^T\lagrangeROMDummy
\end{split}
\end{align}
for $i=1,\ldots,\nsubdomains$.
Applying Newton's method with a Gauss--Newton Hessian approximation to solve
the system of nonlinear algebraic
equations~\eqref{eq_KKT_global}
yields the SQP iterations for $k=0,\ldots,K$
\begin{gather} \label{eq_SQP_subdom_BF}
\left[ 
\begin{array}{cccc}
\hessianSubdomBF{1}(\redStateArgk{1}) & \ldots & \zero &  (\rbBoundaryArg{1})^T\constraintMatROMArg{1}^T \\
	\vdots & \ddots & \vdots & \vdots\\
\zero & \ldots & \hessianSubdomBF{\nsubdomains}(\redStateArgk{\nsubdomains}) & (\rbBoundaryArg{\nsubdomains})^T\constraintMatROMArg{\nsubdomains}^T  \\
	\constraintMatROMArg{1}\rbBoundaryArg{1} & \ldots & \constraintMatROMArg{\nsubdomains}\rbBoundaryArg{\nsubdomains} & \zero
\end{array} \right]
\left[ \begin{array}{c}
\searchDirSubdomBFk{1}\\
	\vdots\\
\searchDirSubdomBFk{\nsubdomains}\\
	\searchDirLagrange
\end{array} \right] = - 
\left[ \begin{array}{c}
\redResSubdom{1}(\redStateArgk{1}) \\
	\vdots\\
\redResSubdom{\nsubdomains}(\redStateArgk{\nsubdomains})\\
	\sum_{i=1}^\nsubdomains \constraintMatROMi \rbBoundaryi \redStateArgk{i}
\end{array} \right], 
\end{gather} 
where 
\begin{align} \label{eq_hessianSubdomBF}
	\begin{split}
\hessianSubdomBF{i}:\redStateDummyArg{i} &\mapsto [\rbInteriorArg{i}]^T \frac{\partial\residualArg{i}}{\partial \stateDummyInteriori}(\rbInteriorArg{i}\redStateDummyArg{i},
\rbBoundaryArg{i}\redStateDummyArg{i})^T[\hyperMati]^T\hyperMati \frac{\partial\residualArg{i}}{\partial \stateDummyInteriori}(\rbInteriorArg{i}\redStateDummyArg{i},
\rbBoundaryArg{i}\redStateDummyArg{i})\rbInteriorArg{i}  \\
	&+ [\rbInteriorArg{i}]^T\frac{\partial\residualArg{i}}{\partial \stateDummyInteriori}(\rbInteriorArg{i}\redStateDummyArg{i} ,
	\rbBoundaryArg{i}\redStateDummyArg{i})^T [\hyperMati]^T\hyperMati\frac{\partial\residualArg{i}}{\partial \stateDummyBoundaryi}(\rbInteriorArg{i}\redStateDummyArg{i} ,
	\rbBoundaryArg{i}\redStateDummyArg{i})\rbBoundaryArg{i}  \\
&+ [\rbBoundaryArg{i}]^T \frac{\partial\residualArg{i}}{\partial \stateDummyBoundaryi}(\rbInteriorArg{i}\redStateDummyArg{i}, \rbBoundaryArg{i}\redStateDummyArg{i})^T [\hyperMati]^T\hyperMati\frac{\partial\residualArg{i}}{\partial\stateDummyInteriori}(\rbInteriorArg{i}\redStateDummyArg{i} , \rbBoundaryArg{i}\redStateDummyArg{i}) \rbInteriorArg{i} \\
	&+ [\rbBoundaryArg{i}]^T \frac{\partial\residualArg{i}}{\partial \stateDummyBoundaryi}(\rbInteriorArg{i}\redStateDummyArg{i} ,
\rbBoundaryArg{i}\redStateDummyArg{i})^T [\hyperMati]^T\hyperMati\frac{\partial\residualArg{i}}{\partial \stateDummyBoundaryi}(\rbInteriorArg{i}\redStateDummyArg{i} ,
\rbBoundaryArg{i}\redStateDummyArg{i}) \rbBoundaryArg{i}
\end{split}
\end{align}
for $i=1,\ldots,\nsubdomains$.
We can then update the solution as 
\begin{align}\label{eq_prdumono_update_subdomBF}
\begin{split}
	\redStateArgkp{i} &= \redStateArgk{i} + \alpha^{(k)}
\searchDirSubdomBFk{i},\quad i=1,\ldots,\nsubdomains \\
	\lagrangeROMkp &= \lagrangeROMk + \alpha^{(k)} \searchDirLagrangek,
\end{split}
\end{align}
where $\alpha^{(k)}$ is a step length that can be computed, e.g., via line search.

\section{Online algorithm and computational cost}\label{sec_cost}
We now describe the computational cost of executing the online stage.
To make this precise, we first introduce the sampling operators
$\sampleMatResi\in\{0,1\}^{\nsampleResi\times \ndofi}$,
$\sampleMatStateInteriori\in\{0,1\}^{\nsampleInteriori\times \ndofInteriori}$,
and $\sampleMatStateBoundaryi\in\{0,1\}^{\nsampleBoundaryi\times
\ndofBoundaryi}$, $i=1,\ldots,\nsubdomains$, which comprise selected rows of
the $\ndofi\times\ndofi$, $\ndofInteriori\times
\ndofInteriori$, and $\ndofBoundaryi\times \ndofBoundaryi$ identity matrices,
respectively. These matrices are those of the prescribed structure that
satisfy 
\begin{align}
	\hyperMati
[\sampleMatResi]^T\sampleMatResi
	\residuali([\sampleMatStateInteriori]^T\sampleMatStateInteriori
	\stateDummyInteriori,
	[\sampleMatStateBoundaryi]^T\sampleMatStateBoundaryi\stateDummyBoundaryi) =
	\hyperMati\residuali(
	\stateDummyInteriori,
	\stateDummyBoundaryi)
	,\quad
	\forall \stateDummyInteriori\in\RR{\ndofInteriori},\
	\stateDummyBoundaryi\in\RR{\ndofBoundaryi},
\end{align}
with the fewest number of rows.
In particular, $\sampleMatStateInteriori$ and $\sampleMatStateBoundaryi$
sample
the degrees of freedom
associated with nonzero columns of the Jacobians
$\hyperMati\frac{\partial\residuali}{\partial\stateDummyInteriori}$ and
$\hyperMati\frac{\partial\residuali}{\partial\stateDummyBoundaryi}$,
respectively.

Note that for standard LSPG (i.e., $\hyperMatArg{i} = \identity$), we have
simply
$\sampleMatResi=\identity$,
$\sampleMatStateInteriori=\identity$, and $\sampleMatStateBoundaryi=\identity$
with 
$\nsampleResi= \ndofi$,
$\nsampleInteriori= \ndofInteriori$, and $
\nsampleBoundaryi= \ndofBoundaryi$.
For collocation (i.e., $\hyperMati=\sampleMati$) and
gappy POD (i.e., $\hyperMati=(\sampleMati\rbResi)^+\sampleMati$), we have 
$\sampleMatResi = \sampleMati$, $\nsampleResi = \nsampleArg{i}$,
$\nsampleInteriori\ll \ndofInteriori$ and $
\nsampleBoundaryi\ll \ndofBoundaryi$ if $\nsampleArg{i}\ll\ndofi$ and the 
Jacobians
$\frac{\partial\residuali}{\partial\stateDummyInteriori}$ and
$\frac{\partial\residuali}{\partial\stateDummyBoundaryi}$ are sparse.

Algorithms \ref{alg_psi_assembly} and \ref{alg_psi_solving} describe the
assembly and solve steps required within each SQP iteration, respectively,
while Tables \ref{tab_DDROMprduM_assembly_cost} and
\ref{tab_DDGNATprduM_solving_cost} report the associated floating-point
operation counts.

Algorithm \ref{alg_psi_assembly} and Table~\ref{tab_DDROMprduM_assembly_cost}
show that Steps 1--3 of the online assembly can be parallelized across the
subdomains, while Step 4 requires a reduction across subdomains. Further,
these illuminate that
$\nsampleInteriori,\nsampleBoundaryi,\nsampleResi\ll\ndofi$ are necessary in
order to achieve an $\ndof$-independent online operation count; this is
precisely what is provided by hyper-reduction.  Here,
$\flopsComputeOneEntryResiduali$ and $\flopsComputeOneEntryJacobiani$ denote
the average number of floating point operations required to evaluate one entry
of the $i$th residual and one row of the $i$th Jacobian matrix, respectively,
while $\averageNNZperRowInteriori$ and $\averageNNZperRowBoundaryi$ denote the
average number of non-zeros per row of the $i$th Jacobian for the interior and
interface, respectively. Note that the operation counts for assembly are
identical for the interior/boundary bases and full-subdomain bases cases.

Algorithm \ref{alg_psi_solving} and Table \ref{tab_DDGNATprduM_solving_cost}
report the steps and associated computational costs associated with the solve
and update for each SQP iteration. 
Importantly, we see that the online solve and update depend only on the
dimensions of the reduced bases and constraint matrices; as such, they are
independent of the quantities
$\nsampleInteriori$, $\nsampleBoundaryi$, and $\nsampleResi$ and are thus
unaffected by hyper-reduction.
Also, here we observe noticeable differences in
the operation counts associated with the interior/boundary bases and
full-subdomain bases cases. In the case of interior/boundary bases, using a specialized Cholesky-based ${\rm LDL^T}$ factorization \cite{lubin2012parallel}, the solve
cost for system \eqref{eq_SQP} is $\frac{1}{3} (\sizeLinSysPrduMono)^3$ with system dimension $\sizeLinSysPrduMono=\sum_{i=1}^\nsubdomains \nrbInteriori +
\sum_{i=1}^\nsubdomains \nrbBoundaryi + \nconstraintsROM$\footnote{For reference, a better solver based on antitriangular factorization for saddle point matrices that was proposed recently \cite{pestana2014antitriangular, rees2018comparative} has the solving cost of only $8mn^2 - 2m^2(n+m/3)$ flops where $n=\sum_{i=1}^\nsubdomains (\nrbInteriori + \nrbBoundaryi)$ and $m=\nconstraintsROM$.}. In contrast, in the case of full-subdomain bases, the
solve cost for system \eqref{eq_SQP_subdom_BF} is $\frac{1}{3} (\sizeLinSysPrduMonoSubdomBF)^3$ with dimension system $\sizeLinSysPrduMonoSubdomBF=\nsubdomains \nrbArg{i} + \nconstraintsROM$. Thus, we expect the solve cost to be less expensive for the full-subdomain
cases when similar reduced-basis dimensions are employed. 

\begin{center}
\begin{minipage}[t]{\textwidth}
\vspace{0pt}  
\begin{algorithm}[H]\label{alg_psi_assembly}
	\caption{Interior/boundary bases: Online assembly at each SQP iteration}
	1: \textbf{Parallel}: Compute the required elements of the state
	$(\sampleMatStateInteriori\stateApproxInteriorik,
	\sampleMatStateBoundaryi\stateApproxBoundaryik)=(\sampleMatStateInteriori\rbInteriori\redStateInteriorik,
	\sampleMatStateBoundaryi\rbBoundaryi\redStateBoundaryi)$ for $i=1,\ldots,
	\nsubdomains$\;
	2: \textbf{Parallel}: Compute residual
	$\hyperMati\residuali(\sampleMatStateInteriori\stateApproxInteriorik	,
	\sampleMatStateBoundaryi\stateApproxBoundaryik)$, Jacobians
	$\hyperMati\frac{\partial\residuali}{\partial\stateDummyInteriori}(\sampleMatStateInteriori\stateApproxInteriorik,\sampleMatStateBoundaryi\stateApproxBoundaryik)\rbInteriori$,
	$\hyperMati\frac{\partial\residuali}{\partial\stateDummyBoundaryi}(\sampleMatStateInteriori\stateApproxInteriorik,\sampleMatStateBoundaryi\stateApproxBoundaryik)\rbBoundaryi$,
	constraint $\constraintMatROMi\rbBoundaryi\redStateBoundaryi$, and 
	$[\rbBoundaryi]^T[\constraintMatROMi]^T\lagrangeROMk$
	for $i=1,\ldots,
	\nsubdomains$\;
	3: \textbf{Parallel}: Using these quantities, compute terms that appear in
	the SQP system\;
\begin{itemize} 
	\item \textbf{Interior/boundary bases}: compute
	$\redResInteriori(\redStateInteriorik,\redStateBoundaryik)$,
	$\redResBoundaryi(\redStateInteriorik,\redStateBoundaryik,\lagrangeROMk)$, 
$\hessianIntInti(\redStateInteriorik,\redStateBoundaryik)$, 
$\hessianIntBoundi(\redStateInteriorik,\redStateBoundaryik)$,
$\hessianBoundInti(\redStateInteriorik,\redStateBoundaryik)$, and 
$\hessianBoundBoundi(\redStateInteriorik,\redStateBoundaryik)$ for $i=1,\ldots,
	\nsubdomains$.
\item \textbf{Full-subdomain basis}: compute
	$\redResSubdom{i}(\redStateInteriorik,\redStateBoundaryik)$,
	$\hessianSubdomBF{i}(\redStateInteriorik,\redStateBoundaryik)$,
 for $i=1,\ldots,
	\nsubdomains$\;
\end{itemize}
	4: \textbf{Serial}: Reduce constraints over subdomains
$\sum_{i=1}^\nsubdomains
	\constraintMatROMi\rbBoundaryi\redStateBoundaryi$;
\end{algorithm}
\end{minipage}%

\qquad 

\begin{minipage}[t]{\textwidth}
\vspace{0pt}
\begin{algorithm}[H]\label{alg_psi_solving}
	\caption{Interior/boundary bases: Online solve and update at each SQP iteration}
	5: \textbf{Serial}: Solve the SQP system \eqref{eq_SQP} or
	\eqref{eq_SQP_subdom_BF}\;
	6: \textbf{Serial}: Update the boundary and interface states via Eq.\
	\eqref{eq_prdumono_update} or \eqref{eq_prdumono_update_subdomBF}\;
	7: \textbf{Serial}: Update the Lagrange multipliers via Eq.\
	\eqref{eq_prdumono_update} or \eqref{eq_prdumono_update_subdomBF}\;
\end{algorithm}
\end{minipage}
\end{center}

\begin{table}[h!]
\center
	\caption{Algorithm \ref{alg_psi_assembly}  operation count} \label{tab_DDROMprduM_assembly_cost}
{\begin{tabular}{c c l}
	\hline \rule{0pt}{2.5ex}
\begin{tabular}{@{}c@{}}
	Algorithm \ref{alg_psi_assembly}\\
	step 
\end{tabular}
	& Parallel/serial & Floating point operation count \\ [0.2ex]
	\hline \rule{0pt}{2.5ex}
	\hspace{-3pt}1 & Parallel & $ 2\nsampleInteriori
	\nrbInteriori + 2\nsampleBoundaryi \nrbBoundaryi$ for the $i$th subdomain \\
	\hline
	2 & Parallel & 
\begin{tabular}{@{}l@{}}
$\nsampleResi \flopsComputeOneEntryResiduali + \nsampleResi \flopsComputeOneEntryJacobiani + 
	2\nsampleResi
	\averageNNZperRowInteriori\nrbInteriori + 
2\nsampleResi
	\averageNNZperRowBoundaryi\nrbBoundaryi+4\nconstraintsROM\nrbBoundaryi +$
	\\ $		\mathtt{is\_dense}(\hyperMati)
	\left(2\nhyperMatArg{i}\nsampleResi(1 + \nrbInteriori+\nrbBoundaryi)\right)
	$ \\for the $i$th subdomain
\end{tabular}
  \\
	\hline
	3 &Parallel  &
\begin{tabular}{@{}l@{}}
	$ 
	2\nrbInteriori\nhyperMatArg{i}+
	2\nrbBoundaryi\nhyperMatArg{i} + \nrbBoundaryi+
		(\nhyperMatArg{i})^2\left(\left(\nrbBoundaryi\right)^2+2\nrbBoundaryi\nrbInteriori
		+ \left(\nrbInteriori\right)^2\right)$ \\
		for the $i$th subdomain
\end{tabular}
		\\
	\hline
	4 & Serial & $2\nsubdomains\nconstraintsROM$  \\
	\hline
\end{tabular}}
\end{table}

\begin{table}[h!]
\center
\caption{Algorithm \ref{alg_psi_solving} operation count} \label{tab_DDGNATprduM_solving_cost}
{\begin{tabular}{c c l}
	\hline \rule{0pt}{2.5ex}
	Algorithm \ref{alg_psi_solving} step & Parallel/serial & Floating point operation count   \\ [0.2ex]
	\hline \rule{0pt}{2.5ex}
	\hspace{-3pt}5 & Serial & 
\begin{tabular}{@{}l@{}}
	\textbf{Interior/boundary bases}:  $\frac{1}{3} \left(\nsubdomains \nrbInteriori +
	 \nsubdomains \nrbBoundaryi + \nconstraintsROM \right)^3$ \\
	 \textbf{Full-subdomain bases}: $\frac{1}{3} \left(\nsubdomains \nrbArg{i} + \nconstraintsROM \right)^3$  
\end{tabular}
	\\
	\hline
	6& Serial & $ 2\nsubdomains (\nrbInteriori+\nrbBoundaryi)$ for the $i$th subdomain\\
	\hline
	7 & Serial& $2\nconstraintsROM $ \\
	\hline
\end{tabular}}
\end{table}

\section{Basis construction}\label{sec_basis_construction}

This section describes how the different proposed reduced bases can be
constructed assuming that a full-system state-snapshot matrix
$\stateSnapshotMatrix\defeq\left[
	\state(\trainParam{1})\ \cdots\ 
\state(\trainParam{\numTrainSample})\right]\in\RR{\ndof\times\numTrainSample}$ with
$\trainSample\subseteq\paramDomain$ has been
precomputed during an ``offline'' training stage. \CH{Section~\ref{sect_subdom_training} describes a specific approach to construct the snapshots from subdomain/component training snapshots alone, which will be necessary for truly extreme-scale models and decomposable systems.} Algorithm \ref{alg_SVD} lists
the widely-used proper orthogonal decomposition (POD) algorithm that we employ
to construct all proposed reduced bases in this work. Throughout,
$\energyCriterion\in[0,1]$ is the ``energy criterion'' used to determine the
applied truncation.

\begin{center}
\begin{minipage}[t]{16cm}
\vspace{0pt}
\begin{algorithm}[H]\label{alg_SVD}
\begin{algorithmic}[1]
	\caption{\texttt{POD}: Proper orthogonal decomposition}
	\REQUIRE Snapshots $\setOfSnapshotsDummy \in \RR{\ndofDummy \times
	\numSnapshotDummy}$, energy criterion $\energyCriterion\in[0,1]$
	\ENSURE Reduced-basis matrix $\rbDummy \in \RR{\ndofDummy \times \nrbDummy}$ 
	\STATE Compute (thin) singular value decomposition: $\setOfSnapshotsDummy = \UDummy \SigmaDummy \VDummy^T$, 
	\STATE Set $\rbDummy = \left[ \rbDummyArg{1}\ \cdots\ \rbDummyArg{\nrbDummy}
	\right]$, where 
	$\nrbDummy=\min_{i\in\Gamma(\energyCriterion)}$,
	$\Gamma(\energyCriterion)\defeq\{i\,|\,\sum_{j=1}^i\sigma_j/\sum_{k=1}^{\numSnapshotDummy}\sigma_k\geq
	1-\energyCriterion\}$.
	Here, $\UDummy \equiv \left[ \rbDummyArg{1}\ \cdots,
	\rbDummyArg{\numSnapshotDummy} \right]$ and
	$\SigmaDummy\equiv\mathrm{diag}(\sigma_1,\ldots,\sigma_{\numSnapshotDummy})$. 
\end{algorithmic}
\end{algorithm}
\end{minipage}%
\end{center}

\subsection{Interior/boundary bases}\label{sec:constrIntBound}
We first describe various approaches to constructing interior/boundary bases.
\begin{itemize} 
	\item \textbf{Interior bases}.
To compute interior bases
$\rbInteriori\in\RRstar{\ndofInteriori\times\nrbInteriori}$,
$i=1,\ldots,\nsubdomains$, we simply execute Algorithm \ref{alg_SVD} with
snapshots isolated to subdomain interiors such that
$\rbInteriori=\texttt{POD}(\projectionStateInteriori\stateSnapshotMatrix,\energyCriterion)$,
		$i=1,\ldots,\nsubdomains$.
	\item \textbf{(Boundary) Full-interface bases}.
Analogously, we compute full-interface bases by 
executing Algorithm \ref{alg_SVD} with
snapshots isolated to subdomain interfaces such that
$\rbBoundaryi=\texttt{POD}(\projectionStateBoundaryi\stateSnapshotMatrix,\energyCriterion)$, $i=1,\ldots,\nsubdomains$.
\item \textbf{(Boundary) Port bases}.
To compute port bases, we first compute reduced bases for each port by executing
Algorithm \ref{alg_SVD} with snapshots
		$\projectionStatePort{j}{\ell}\projectionStateBoundaryArg{\ell}\stateSnapshotMatrix$,
$j=1,\ldots,\nports$ with any $\ell\in\subdomainPorts{j}$, and compute the
		resulting interface bases by assembling the appropriate port bases as
$\rbBoundaryi=[\texttt{POD}(\projectionStatePort{q_i^1}{i}\projectionStateBoundaryi\stateSnapshotMatrix,\energyCriterion)\
\cdots\
\texttt{POD}(\projectionStatePort{q_i^{|\portsSubdomains{i}|}}{i}\projectionStateBoundaryi\stateSnapshotMatrix,\energyCriterion)]$,
$i=1,\ldots,\nsubdomains$,
		where $\portsSubdomains{i}\equiv\{q_i^j\}_j$.
	\item \textbf{(Boundary) Skeleton bases}. This approach first computes a reduced basis
		for the ``skeleton'', which is the union of subdomain interfaces, and
		subsequently isolates that basis to each subdomain's interface while
		ensuring full column rank on that interface. More precisely, we 
		compute $\rbBoundaryiTmp =
		\texttt{POD}((\identity-\sum_{i=1}^{\nsubdomains}[\projectionStateInteriori]^T\projectionStateInteriori)\stateSnapshotMatrix,\energyCriterion)$
		followed by a rank-revealing QR factorization with column pivoting
		$\rbBoundaryiTmp\columnPivot = \Qmati\Rmati$, and finally set
		$\rbBoundaryi = \left[\bds{q}_i^1\ \cdots\
		\bds{q}_i^{\nrbBoundaryi}\right]$, where 
$\Qmati\equiv[\bds{q}_i^1\ \cdots\
		\bds{q}_i^{\numTrainSample}]$ and
		$\mathrm{rank}(\rbBoundaryiTmp) =
		\nrbBoundaryi(\leq\numTrainSample)$.
		We note that skeleton bases require full-system snapshots and thus are not
		generally practical for either extreme-scale systems nor for decomposable
		systems, as both of these scenarios in practice preclude the ability to
		collect full-system
		snapshots; nevertheless, because this work does not directly consider
		subsystem/component-based training, we include this approach in the
		present work.
\end{itemize}
\subsection{Full-subdomain basis}\label{sec:constFull}
The full-subdomain-basis approach computes reduced bases
		that have support over all degrees of freedom for their respective
		subdomains, and subsequently isolates this basis to the subdomain interior
		and interface. That is, we compute $\rbi =
		\texttt{POD}\left(\begin{bmatrix}\projectionStateInteriori\\
		\projectionStateBoundaryi\end{bmatrix}\stateSnapshotMatrix,\energyCriterion\right)$
		and set $\rbInteriori =
		\begin{bmatrix}\identity_{\ndofInteriori}&
		\bds{0}_{\ndofBoundaryi}\end{bmatrix} \rbi$ and 
$\rbBoundaryi =
			\begin{bmatrix}\bds{0}_{\ndofInteriori\times\ndofInteriori}&
			\identity_{\ndofBoundaryi\times\ndofBoundaryi}\end{bmatrix} \rbi$, where
			$\identity_n$ and $\bds{0}_n$ denote the $n\times n$ identity and zero matrices,
			respectively.

\section{\textit{A posteriori} and \textit{a priori} error bounds}\label{sec_error_analysis_strong}
For notational simplicity, this section omits explicit parameter dependence;
results can be interpreted as holding for any arbtrary parameter instance
$\params\in\paramDomain$.
We begin by stating assumptions that will be employed in subsequent analysis.
\begin{itemize}
	\item[A1]\label{ass:strong} The DD-LSPG ROM employs strong constraints, i.e.,
		$\constraintMatROMi= \constraintMatArg{i}$,
$i=1,\ldots,\nsubdomains$.
\end{itemize}
Under Assumption A1,
the DD-LSPG ROMs can be converted to 
unconstrained minimization problems. First, we introduce the null-space matrix
$\nullspacemat\in\RR{{\sum_{i=1}^\nsubdomains\nrbBoundaryi\times \nrbNull}}$ 
with $\nrbNull\defeq\sum_{i=1}^\nsubdomains
\nrbBoundaryi-
\rank{\constraintMat}$,
which satisfies
$\constraintMat\nullspacemat=\zero$ with $\constraintMat\defeq[ \constraintMatArg{1} \rbBoundaryArg{1}\ \cdots\
\constraintMatArg{\nsubdomains} \rbBoundaryArg{\nsubdomains}]$. 
Because strong constraints enforce a global solution such that
	Eq.~\eqref{eq_portConstraints} holds (e.g., see Remark \ref{rem:globalSol}),
	DD-LSPG yields a `global' solution $\stateApprox\in\RR{\ndof}$ satisfying 
	\begin{equation}\label{eq:globalSolution}
\stateApproxInteriori=\projectionStateInteriori\stateApprox,\
	\stateApproxBoundaryi=\projectionStateBoundaryi\stateApprox,\quad
	i=1,\ldots,\nsubdomains.
	\end{equation}

Now, the interior/boundary-basis problem \eqref{eq_globalOpt} is equivalent to
the unconstrained minimization problem
wherein
$\redStateInteriori $, $i\innatseq{\nsubdomains}$ and 
$\nullspaceCoords$ comprise the solution to the problem
\begin{align} \label{eq_globalOptUncon}
\begin{split} 
\underset{(\redStateDummyInteriorArg{i}),\,
	i=1,\ldots,\nsubdomains,\, \nullspaceCoordsDummy}{\text{minimize}}\quad&\frac{1}{2}\sum_{i=1}^\nsubdomains
\|\hyperMati\residuali(\rbInteriori\redStateDummyInteriori,
	\rbBoundaryi\nullspacemati\nullspaceCoordsDummy)\|_2^2,
\end{split} 
\end{align} 
with 
	\begin{equation}\label{eq:gloalStateInteriorBoundary}
	\statei\approx\stateApproxi \equiv (\stateApproxInteriori,
	\stateApproxBoundaryi)=(\rbInteriori\redStateInteriori,
	\rbBoundaryi\nullspacemati\nullspaceCoords),
	\end{equation}
where
		$\nullspacemati\in\RR{\nrbBoundaryi\times\nrbNull}$ denotes the $i$th row
		block of $\nullspacemat$.

Note that Problem \eqref{eq_globalOptUncon} can be expressed equivalently as
computing $\stateApprox$ that satisfies
\begin{align} \label{eq_globalOptUnconSpace}
\begin{split} 
\underset{\stateDummy\in\trialSpaceInterior}{\text{minimize}}\quad&\frac{1}{2}\sum_{i=1}^\nsubdomains
\|\hyperMati\residuali(\projectionStateInteriori\stateDummy,
	\projectionStateBoundaryi\stateDummy)\|_2^2,
\end{split} 
\end{align} 
where the trial subspace $\trialSpaceInterior\subseteq\RR{\ndof}$ is defined
as
	\begin{align}
	\begin{split}
		\trialSpaceInterior \defeq &
		\{\stateDummy\in\RR{\ndof}\,|\, \exists
		\redStateDummyInteriori\in\RR{\nrbInteriori},\,
		i=1,\ldots,\nsubdomains,\ \text{and}\
		\nullspaceCoordsDummy\in\RR{\nrbNull}\ \text{s.t.}\\
		&
		\projectionStateInteriori\stateDummy =
		\rbInteriori\redStateDummyInteriori,\ \projectionStateBoundaryi\stateDummy
		=\rbBoundaryi\nullspacemati\nullspaceCoordsDummy,\,i=1,\ldots,\nsubdomains\}\subseteq\RR{\ndof}.
	\end{split}
	\end{align}

	Similarly, A1 admits
	conversion
of the full-subdomain-basis problem \eqref{eq_globalOpt_one_arg} to an unconstrained minimization problem
wherein
$\nullspaceCoords$ comprises the solution to the problem
\begin{align} \label{eq_globalOpt_uncon}
\begin{split} 
	\underset{\nullspaceCoordsDummy}{\text{minimize}}\quad&\frac{1}{2}\sum_{i=1}^\nsubdomains
\|\hyperMati\residuali(\rbInteriori \nullspacemati\nullspaceCoordsDummy,
	\rbBoundaryi \nullspacemati\nullspaceCoordsDummy)\|_2^2.
\end{split} 
\end{align} 
with 
	\begin{equation}\label{eq:gloalStateFullSubdomain}
	\statei\approx\stateApproxi \equiv (\stateApproxInteriori,
	\stateApproxBoundaryi)=(\rbInteriori\nullspacemati\nullspaceCoords,
	\rbBoundaryi\nullspacemati\nullspaceCoords),
	\end{equation}
	which can be expressed equivalently as
computing $\stateApprox$ that satisfies
\begin{align} \label{eq:gloalStateFullSubdomainSpace}
\begin{split} 
\underset{\stateDummy\in\trialSpaceFull}{\text{minimize}}\quad&\frac{1}{2}\sum_{i=1}^\nsubdomains
\|\hyperMati\residuali(\projectionStateInteriori\stateDummy,
	\projectionStateBoundaryi\stateDummy)\|_2^2,
\end{split} 
\end{align} 
where the trial subspace $\trialSpaceFull\subseteq\RR{\ndof}$ is defined
as
		\begin{align}
	\begin{split}
		\trialSpaceFull \defeq &
		\{\stateDummy\in\RR{\ndof}\,|\, \exists
			\nullspaceCoordsDummy\in\RR{\nrbNull}\ \text{s.t.}\
		\projectionStateInteriori\stateDummy =
		\rbInteriori\nullspacemati\nullspaceCoordsDummy,\
		\projectionStateBoundaryi\stateDummy
		=\rbBoundaryi\nullspacemati\nullspaceCoordsDummy,\,i=1,\ldots,\nsubdomains\}\subseteq\RR{\ndof},
	\end{split}
	\end{align}
	
We now introduce two more assumptions needed for the error bounds.
\begin{itemize}
	\item[A2]\label{ass:invLip} The residual is inverse Lipschitz continuous in
		the $\ell^2$-norm, i.e., there exists $\inverseLipschitz>0$ such that
		\begin{gather}
			\left(\sum_{i=1}^\nsubdomains
\|\residuali(\stateDummyInteriori,
	\stateDummyBoundaryi) - 
\residuali(\stateDummyInteriorTwoi,
	\stateDummyBoundaryTwoi)
			\|_2^2\right)^{1/2} \geq 
			\inverseLipschitz
\|\stateDummy - \stateDummyTwo\|_2,\quad\forall
\stateDummy,\,\stateDummyTwo\in\RR{\ndof}
			\end{gather}
		 with $
\stateDummyInteriori\defeq\sampleMatStateInteriori\stateDummy$, 
$
\stateDummyBoundaryi\defeq\sampleMatStateBoundaryi\stateDummy$,
$
\stateDummyInteriorTwoi\defeq\sampleMatStateInteriori\stateDummyTwo$, and
$
		\stateDummyBoundaryTwoi\defeq\sampleMatStateBoundaryi\stateDummyTwo$,
		$i=1,\ldots,\nsubdomains$.
	\item[A3]\label{ass:inv} The $\hyperMat^T\hyperMat$-norm and the
		$\ell^2$-norm of the residual are equivalent over all elements of the
		trial subspace such that there exists $\inverseTmp>0$ such that
		\begin{gather}
			\left(\sum_{i=1}^\nsubdomains
\|\hyperMati\residuali(\projectionStateInteriori\stateDummy,
			\projectionStateBoundaryi\stateDummy)\|_2^2\right)^{1/2} \geq \inverseTmp
			\left(\sum_{i=1}^\nsubdomains
\|\residuali(\projectionStateInteriori\stateDummy,
			\projectionStateBoundaryi\stateDummy)\|_2^2\right)^{1/2},\quad\forall
			\stateDummy\in\trialSpace\subseteq\RR{\ndof},
		\end{gather}
		where $\trialSpace = \trialSpaceInterior$ in the case of interior/boundary
		bases and $\trialSpace = \trialSpaceFull$ in the case of full-domain
		bases.
		\end{itemize}
\begin{proposition}[\textit{A posteriori} error bound]\label{prop:aposter}
	Under Assumptions A1--A3, the error in the DD-LSPG ROM approximate solution for both
	interior/boundary bases and full-subdomain bases can be bounded as
	\begin{equation} \label{eq:PropOneIn}
\max(\max_{i\in\{1,\ldots,\nsubdomains\}}\|\stateInteriori-\stateApproxInteriori\|_2,
\max_{j\in\{1,\ldots,\nports\}}
	\|\projectionStatePortAll{j}\state -
	\projectionStatePortAll{j}\stateApprox\|_2
	)
	\leq
	\|\state-\stateApprox\|_2
	\leq	\frac{1}{\inverseTmp\inverseLipschitz}
	\left(\sum_{i=1}^\nsubdomains
\|\hyperMati\residuali(\stateApproxInteriori,
	\stateApproxBoundaryi)\|_2^2\right)^{1/2},
\end{equation} 
	where $\stateApprox\in\RR{\ndof}$ is the `global' solution satisfying
	Eq.~\eqref{eq:globalSolution} and 
where $\projectionStatePortAll{j}\in\{0,1\}^{\ndofPortsj\times\ndof }$,
	$j=1,\ldots,\nports$
comprises selected rows of the identity matrix that extract the global degrees of
	freedom associated with the $j$th port.
\end{proposition}
\begin{proof}
Leveraging the norm-equivalence relation $\|\cdot\|_\infty\leq\|\cdot\|_2$,
	the `global' solution relations \eqref{eq:globalSolution}, sequentially invoking A2 and A3, and noting that 
$\residuali(\stateInteriori,
	\stateBoundaryi)=0$, $i=1,\ldots,\nsubdomains$
	yields
	\begin{equation}
		\max_{i\in\{1,\ldots,\nsubdomains\}}\|\stateInteriori-\stateApproxInteriori\|_2 \leq \|\state - \stateApprox\|_2\leq
\frac{1}{\inverseLipschitz}
		\left(\sum_{i=1}^\nsubdomains
\|\residuali(\stateApproxInteriori,
	\stateApproxBoundaryi)\|_2^2\right)^{1/2}
		\leq\frac{1}{\inverseTmp\inverseLipschitz}
		\left(\sum_{i=1}^\nsubdomains
\|\hyperMati\residuali(\stateApproxInteriori,
	\stateApproxBoundaryi)\|_2^2\right)^{1/2},
	\end{equation}
	which is valid for both interior/boundary bases and full-subdomain bases
	according to A3.
	On the invocation of the $\ell^\infty$-norm, we have decomposed the state
	vector into segments: one for each group of interior degrees of freedom, and
	one for each port.
\end{proof}

We now introduce another assumption that will be employed to derive \textit{a
priori} error bounds.
\begin{itemize}
	\item[A4]\label{ass:Lip} The residual is Lipschitz continuous in the
		$\hyperMat^T\hyperMat$-norm, i.e., there exists $\lipschitz>0$ such that
		\begin{gather}
			\left(\sum_{i=1}^\nsubdomains
\|\hyperMati\residuali(\stateDummyInteriori,
	\stateDummyBoundaryi) - 
\hyperMati\residuali(\stateDummyInteriorTwoi,
	\stateDummyBoundaryTwoi)
			\|_2^2\right)^{1/2} \leq \lipschitz
			\|\stateDummy-\stateDummyTwo\|_2,\quad
			\forall\stateDummy,\,\stateDummyTwo\in\RR{\ndof}.
			\end{gather}
\end{itemize}
\begin{proposition}[\textit{A priori} error bound with respect to the
	$\ell^2$-optimal approximation error]\label{prop:aprioriOne}
	Under Assumptions A1--A4, the error in the DD-LSPG ROM approximate solution
	for both interior/boundary bases and full-subdomain bases can be bounded in
	terms of the $\ell^2$-optimal approximation error as
	\begin{equation} 
	\|\state-\stateApprox\|_2
	\leq\frac{\lipschitz}{\inverseTmp\inverseLipschitz}
	\min_{\stateDummy\in\trialSpace}
\|\state-\stateDummy\|_2
\end{equation} 
	where  $\trialSpace=\trialSpaceInterior$ in the case of interior/boundary
	bases and $\trialSpace=\trialSpaceFull$ in the case of full-subdomain bases.
\end{proposition}
\begin{proof}
	We have from Eq.~\eqref{eq:gloalStateInteriorBoundary},  and
	Problems \eqref{eq_globalOptUnconSpace} and
	\eqref{eq:gloalStateFullSubdomainSpace}, and
Inequality
	\eqref{eq:PropOneIn} that
\begin{align} 
\begin{split} 
	&\frac{1}{\inverseTmp\inverseLipschitz}
	\left(\sum_{i=1}^\nsubdomains
\|\hyperMati\residuali(\stateApproxInteriori,
	\stateApproxBoundaryi)\|_2^2\right)^{1/2}
	 = 
\frac{1}{\inverseTmp\inverseLipschitz}
	 \min_{\stateDummy\in\trialSpace}
\left(\sum_{i=1}^\nsubdomains
\|\hyperMati\residuali(\projectionStateInteriori\stateDummy,
	\projectionStateBoundaryi\stateDummy)\|_2^2\right)^{1/2}\\
	&\leq\frac{1}{\inverseTmp\inverseLipschitz}
\left(\sum_{i=1}^\nsubdomains
\|\hyperMati
\residuali(\projectionStateInteriori\stateOptimal,
	\projectionStateBoundaryi\stateOptimal)
\|_2^2\right)^{1/2}
	\leq\frac{\lipschitz}{\inverseTmp\inverseLipschitz}
\|\state-\stateOptimal\|_2.
\end{split}
\end{align} 
Defining
	$\stateOptimal\defeq\arg\min_{\stateDummy\in\trialSpace}\|\state-\stateDummy\|_2$
	and
	combining this with Proposition \ref{prop:aposter} yields the desired result.
\end{proof}

\begin{proposition}[\textit{A priori} error bound with respect to the
	$\ell^\infty$-optimal approximation error over subdomains]
	Under Assumptions A1--A4,
\begin{align} 
\begin{split} 
	\max(&\max_{i\in\{1,\ldots,\nsubdomains\}}\|\stateInteriori-\stateApproxInteriori\|_2,
\max_{j\in\{1,\ldots,\nports\}}
	\|\projectionStatePortAll{j}\state -
	\projectionStatePortAll{j}\stateApprox\|_2
	)\\
	&\leq\frac{\lipschitz\sqrt{\nsubdomains+\nports}}{\inverseTmp\inverseLipschitz}
	\min_{\stateDummy\in\trialSpace}\max(\max_{i\in\{1,\ldots,\nsubdomains\}}\|\stateInteriori-\stateDummyInteriori\|_2,\max_{j\in\{1,\ldots,\nports\}}
	\|\projectionStatePortAll{j}\state -
	\projectionStatePortAll{j}\stateDummy\|_2
	),
\end{split} 
\end{align} 
	where  $\trialSpace=\trialSpaceInterior$ in the case of interior/boundary
	bases and $\trialSpace=\trialSpaceFull$ in the case of full-subdomain bases.
\end{proposition}
\begin{proof}
	Analogously to the proof of Proposition \ref{prop:aprioriOne}
we have from Eq.~\eqref{eq:gloalStateInteriorBoundary},  and
	Problems \eqref{eq_globalOptUnconSpace} and
	\eqref{eq:gloalStateFullSubdomainSpace}, and
Inequality
	\eqref{eq:PropOneIn} that
\begin{align} 
\begin{split} 
	&\frac{1}{\inverseTmp\inverseLipschitz}
	\left(\sum_{i=1}^\nsubdomains
\|\hyperMati\residuali(\stateApproxInteriori,
	\stateApproxBoundaryi)\|_2^2\right)^{1/2}
	 = 
\frac{1}{\inverseTmp\inverseLipschitz}
	 \min_{\stateDummy\in\trialSpace}
\left(\sum_{i=1}^\nsubdomains
\|\hyperMati\residuali(\projectionStateInteriori\stateDummy,
	\projectionStateBoundaryi\stateDummy)\|_2^2\right)^{1/2}\\
	&\leq\frac{1}{\inverseTmp\inverseLipschitz}
\left(\sum_{i=1}^\nsubdomains
\|\hyperMati
\residuali(\projectionStateInteriori\stateOptimalInf,
	\projectionStateBoundaryi\stateOptimalInf)
\|_2^2\right)^{1/2}
	\leq\frac{\lipschitz}{\inverseTmp\inverseLipschitz}
\|\state-\stateOptimalInf\|_2\\
	&\leq
	\frac{\lipschitz\sqrt{\nsubdomains+\nports}}{\inverseTmp\inverseLipschitz}
\max(\max_{i\in\{1,\ldots,\nsubdomains\}}\|\stateInteriori-\stateOptimalInfInteriori\|_2,\max_{j\in\{1,\ldots,\nports\}}
	\|\projectionStatePortAll{j}\state -
	\projectionStatePortAll{j}\stateOptimalInf\|_2
	)
\end{split}
\end{align} 
	where, on the invocation of the $\ell^\infty$-norm, we have decomposed the state
	vector into segments: one for each group of interior degrees of freedom, and
	one for each port. Defining
	$\stateOptimalInf$ as satisfying the minimization problem
\begin{equation} 
	\stateOptimalInf\in\underset{\stateDummy\in\trialSpace}{\mathrm{argmin}}\
	\max(\max_{i\in\{1,\ldots,\nsubdomains\}}\|\stateInteriori-\stateDummyInteriori\|_2,\max_{j\in\{1,\ldots,\nports\}}
	\|\projectionStatePortAll{j}\state -
	\projectionStatePortAll{j}\stateDummy\|_2
	)
\end{equation} 
	and combining with Proposition \ref{prop:aposter} yields the desired result.
\end{proof}


\section{Numerical experiments}\label{sec_numerical_results}

This section reports numerical experiments that assess the performance of the
proposed DD-LSPG method on two benchmark problems. 
\YC{We do not attempt to demonstrate our methods in either extreme-scale or
decomposable problems. Instead, we focus on the performance of our algorithm on
small problems in order to investigate the effect of some key model parameters,
such as constraint type, basis type, hyper-reduction, truncation levels, on
accuracy and speed in both weak and strong scaling.}

We compare the following methods: \\
\begin{itemize} 
\item 
	\textbf{FOM}. This model corresponds to the full-order model, i.e., the
solution satisfying Eq.~\eqref{eq_originalAlgebraic} (equiv.\ Eq.~\eqref{eq_globaldecomposed}). \\
\item \textbf{DD-LSPG}. This model corresponds to the unweighted LSPG ROM,
	i.e., the solution satisfies
		\eqref{eq_globalOpt}--\eqref{eq_globalOpt_one_arg} with $\hyperMatArg{i} =
		\identity$, $i=1,\ldots,\nsubdomains$. \\
\item \textbf{DD-GNAT}. This model corresponds to the GNAT ROM, i.e., the
	solution satisfying \eqref{eq_globalOpt}--\eqref{eq_globalOpt_one_arg} with
		$\hyperMati=(\sampleMati\rbResi)^+\sampleMati$, $i=1,\ldots,\nsubdomains$.
		Algorithm \ref{alg_Greedy_GNAT} is used to construct the sampling matrices
		$\sampleMati$, $i=1,\ldots,\nsubdomains$. 
\end{itemize}

We assess the accuracy of any ROM solution $\stateApprox(\param)$ as follows
\begin{equation}\label{eq_rel_err_definition}
{\rm relative \; error} = \sqrt{ \frac{1}{\nsubdomains} \sum_{i=1}^{\nsubdomains} \frac{\|\stateApproxi(\param)-\statei(\param)\|^2_2}{\|\statei(\param)\|^2_2} },
\end{equation}
and we measure its computational cost in terms of the wall time incurred by
the ROM simulation relative to that incurred by the FOM simulation; the
speedup is the reciprocal of the relative wall time. All timings are obtained
by performing calculations in Matlab R2018b on an 2x6-Core Intel Xeon 2.93GHz with 64 GB RAM of
memory. Reported timings comprise the average over five simulations.

\subsection{Parameterized heat equation}\label{subsect_heat_equation}

\subsubsection{Global finite-element discretization}

\begin{figure}[h!]
\centering
\begin{subfigure}[b]{0.45\textwidth}
\includegraphics[width=8.1cm]{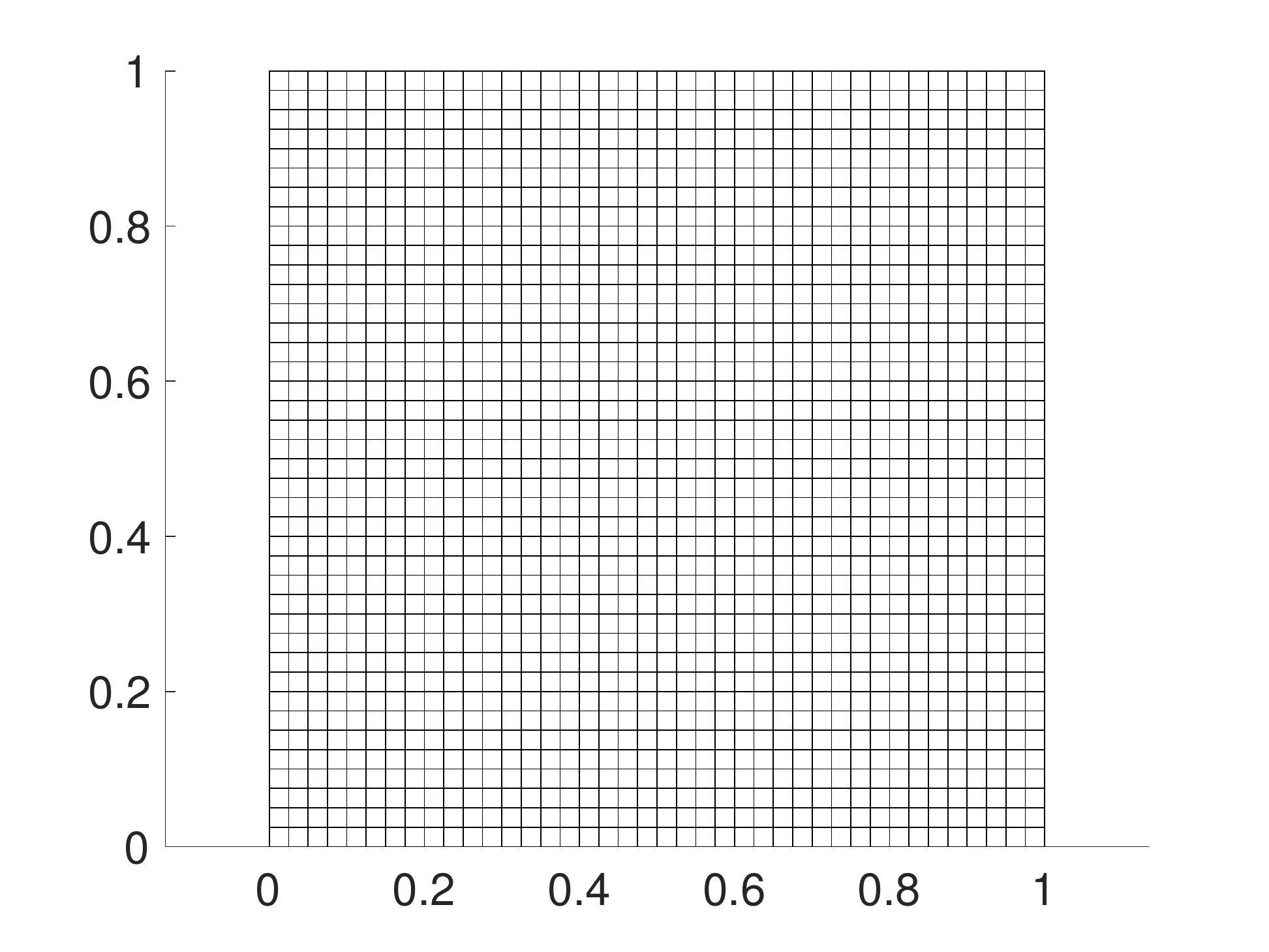}
\caption{``Coarse'' mesh: 40x40 elements}
\label{fig_ex1_globMesh_40x40}
\end{subfigure}
~
\begin{subfigure}[b]{0.45\textwidth}
\includegraphics[width=8.1cm]{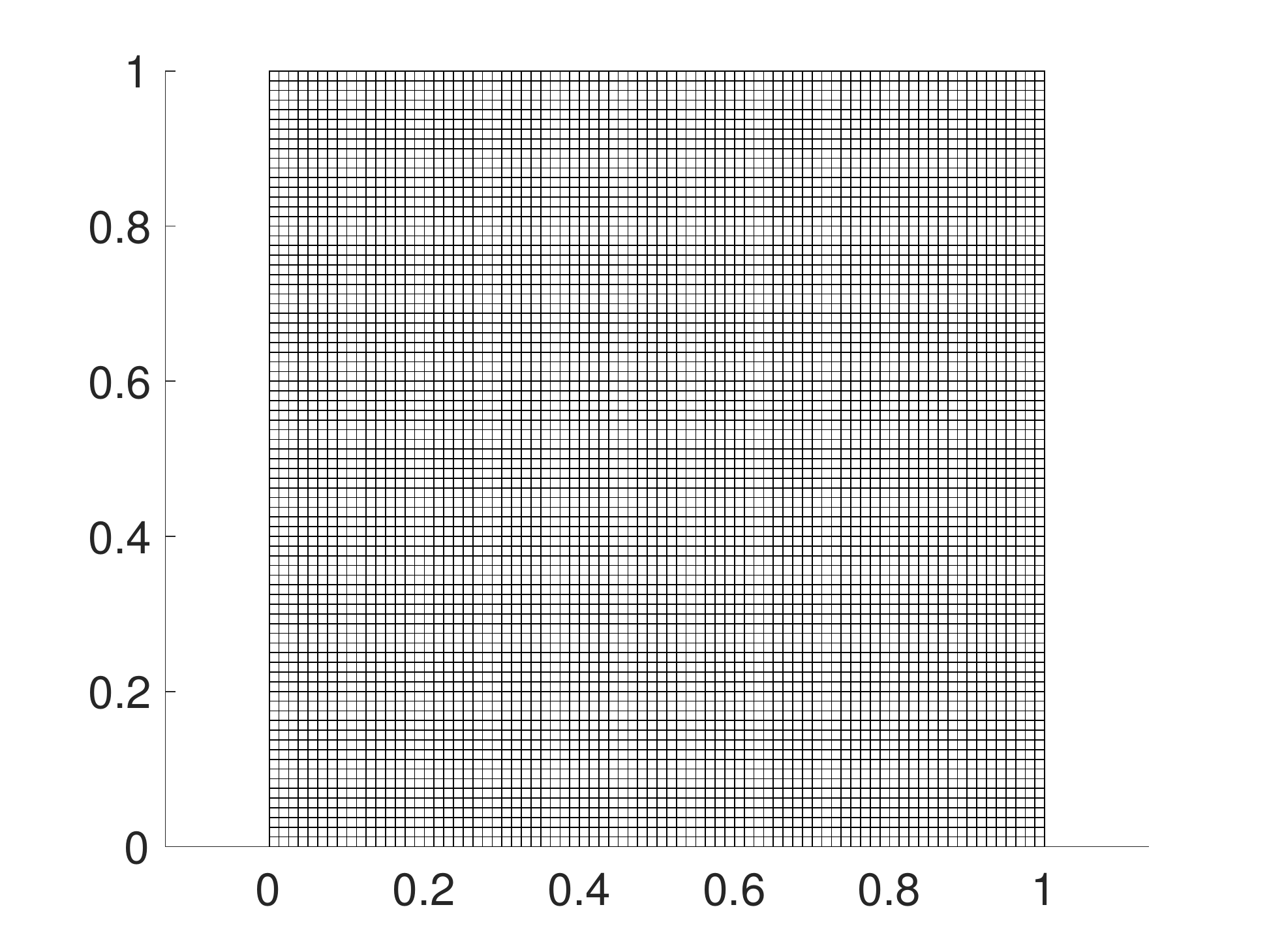}
\caption{``Fine'' mesh: 80x80 elements}
\label{fig_ex1_globMesh_80x80}
\end{subfigure}
\caption{Heat equation, two global meshes used for discretization.}\label{fig_ex1_globMesh}
\end{figure}

\begin{figure}[h!]
\centering
\begin{subfigure}[b]{0.45\textwidth}
\includegraphics[width=8.1cm]{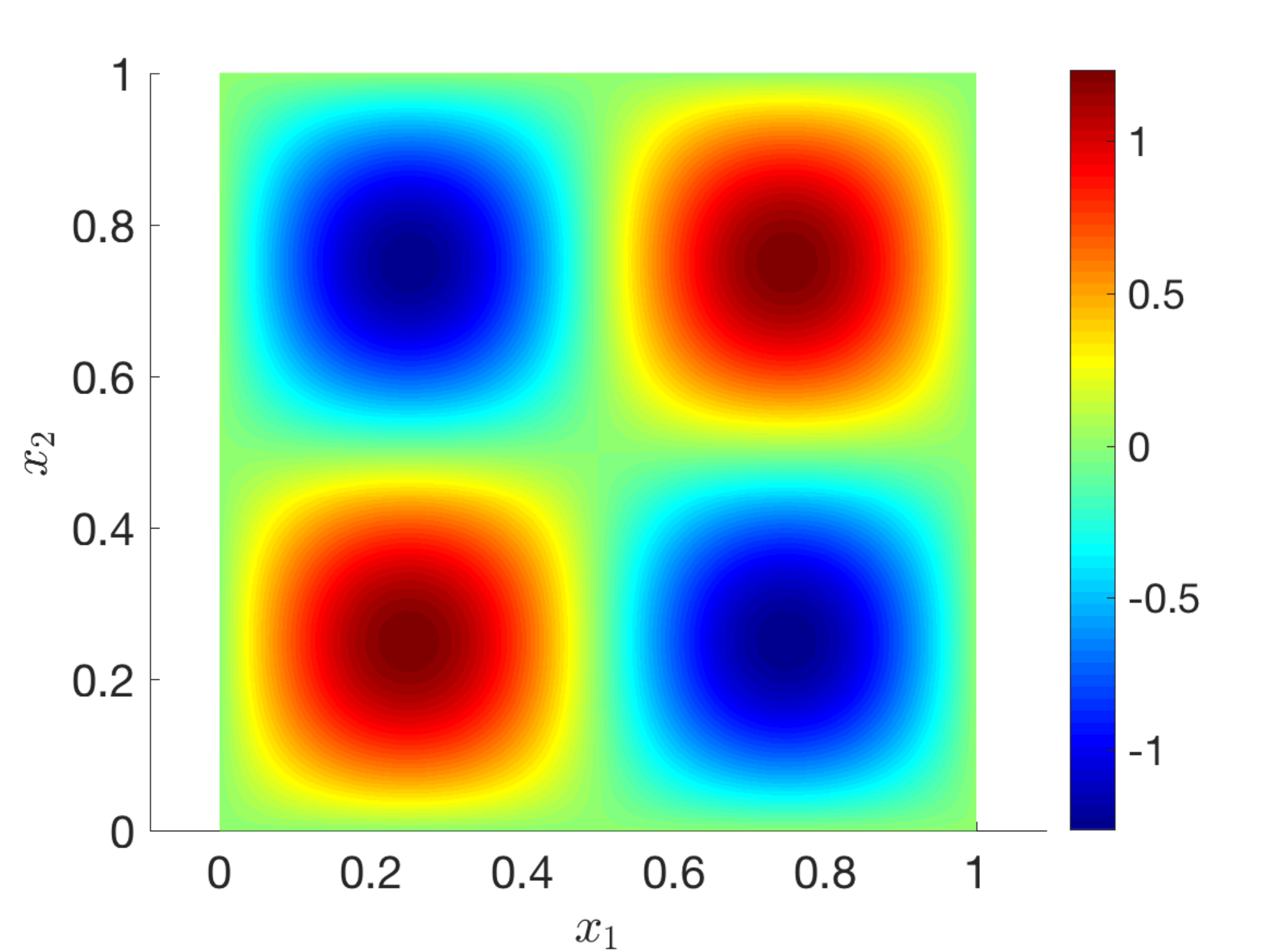}
\caption{$\paramComp=(1,1)$}
\label{fig_ex1_FEMsol_1_1}
\end{subfigure}
~
\begin{subfigure}[b]{0.45\textwidth}
\includegraphics[width=8.1cm]{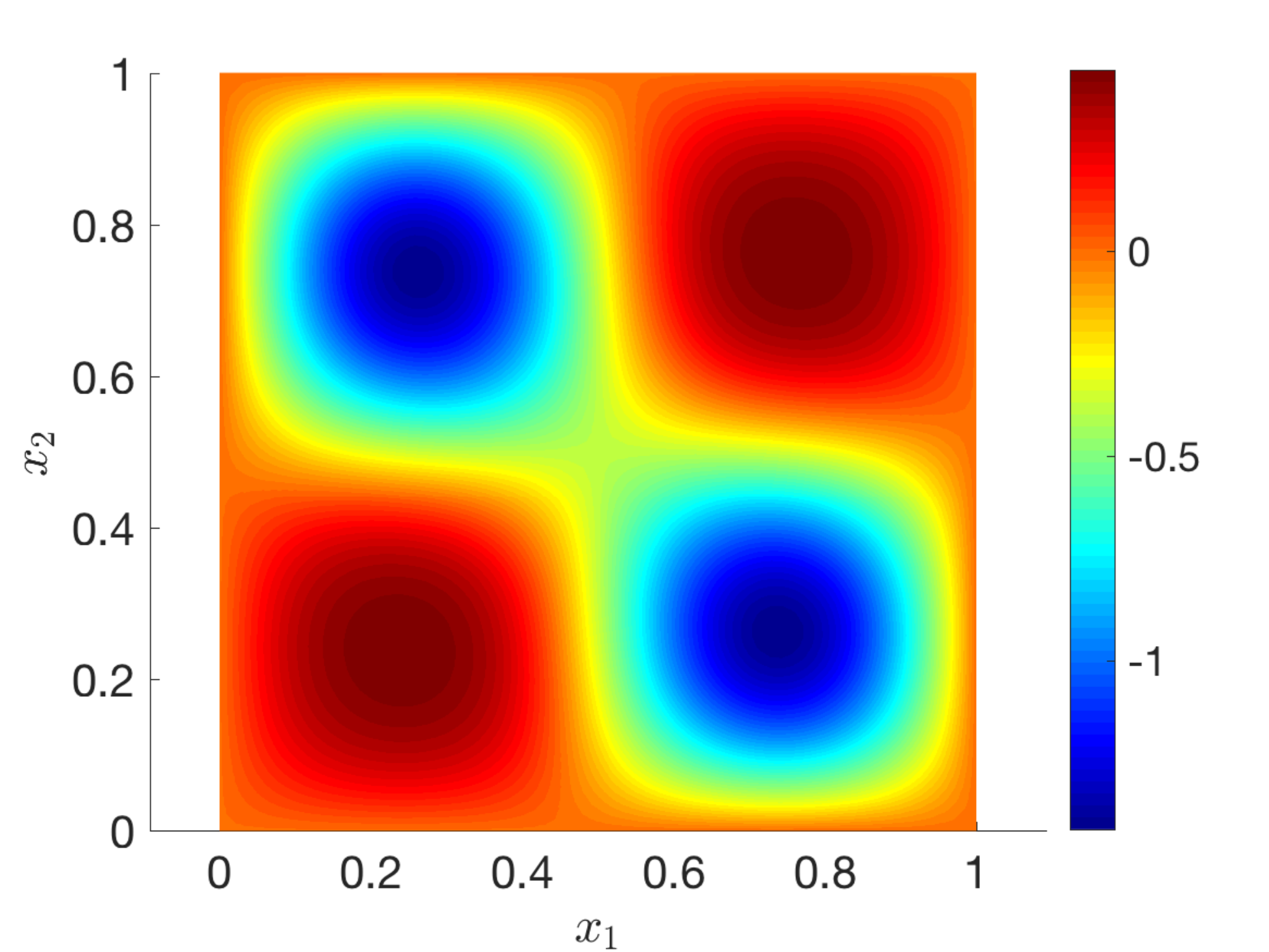}
\caption{$\paramComp=(10,10)$}
\label{fig_ex1_FEMsol_10_10}
\end{subfigure}
\caption{Heat equation, FOM solutions for different $\paramComp$ using the fine mesh. }\label{fig_ex1_FEMsols}
\end{figure}

\begin{table}[h!]
\center
\caption{Heat equation, parameters of global FOM discretization} \label{tab_ex1_globFEM_params}
{\begin{tabular}{|c||c|c|}
	\hline \rule{0pt}{2.5ex}
	&  ``Coarse'' mesh  &  ``Fine'' mesh  \\ [0.2ex]
	\hline \rule{0pt}{2.5ex} 
	Number of elements & 1600 & 6400  \\
	Number of nodes    & 1681 & 6561  \\
	Number of degrees of freedom $\ndof$			   & 1521 & 6241  \\
	\hline
\end{tabular}}
\end{table}

We first consider the model example introduced in Refs.
\cite{grepl2007efficient, chaturantabut2010nonlinear}. This is a parametric
nonlinear 2D heat problem which consists of computing
$\tempComp(\stateComp,\paramComp)$ with $\stateComp \equiv
(\stateCompArg{1},\stateCompArg{2}) \in \domain=[0,1]^2$ and
$\param\equiv(\paramCompArg{1},\paramCompArg{2}) \in
\paramDomain=[0.01,10]^2$ and homogeneous Dirichlet boundary condition on
$\boundary\equiv \partial \domain$ satisfying
\begin{equation}\label{eq_2DheatPDE}
-\nabla^2 \tempComp + \frac{\paramCompArg{1}}{\paramCompArg{2}}(e^{\paramCompArg{2} \tempComp}-1) = 100 \sin(2\pi \stateCompArg{1}) \sin(2\pi \stateCompArg{2}).
\end{equation}
This model can be interpreted as a 2D stationary diffusion problem with a
nonlinear interior heat source. The resulting solution exhibits a strongly nonlinear dependence 
on the parameters $\paramComp$.


For spatial discretization, we apply the finite-element method using two
meshes (which will be used to assess strong and weak scaling): a ``coase''
mesh and a ``fine'' mesh, characterized by 1600 (40$\times$40) and 6400
(80$\times$80) bilinear quadrilateral (Q1) elements, respectively.
Figure~\ref{fig_ex1_globMesh} depicts these meshes, while
Table~\ref{tab_ex1_globFEM_params} reports the corresponding parameters.
Figure~\ref{fig_ex1_FEMsols} plots the FOM reference solutions on the ``fine''
mesh with two different parameter values $\paramComp=(1,1)$ and
$\paramComp=(10,10)$. 
Applying these finite-element discretizations to Eq.~\eqref{eq_2DheatPDE} leads
to a parameterized system of nonlinear algebraic equations of the form \eqref{eq_originalAlgebraic}.

\subsubsection{Full-order model}

\begin{table}[h!]
\center
\caption{Heat equation, parameters for three FOM configurations} \label{tab_ex1_DDFEMconfig_params}
{\begin{tabular}{|c||c|c|c|}
	\hline \rule{0pt}{2.5ex}
	&  $2\times 2$ ``coarse''  &  $4\times 4$ ``fine''  &  $2\times 2$ ``fine''  \\ [0.2ex]
	\hline \rule{0pt}{2.5ex}
Number of subdomains	$\nsubdomains$  	 & 4   	& 16   & 4    \\
	Number of constraints $\nconstraints$ 	 & 172  & 1092 & 332  \\
	Number of ports $\nports$			 & 5    & 33   & 5    \\
	Number of interior DOFs $\ndofInteriorArg{1}$ & 441  & 441  & 1681
\\ \hline
	 & \multicolumn{2}{c|}{weak scaling} & \\ \hline
     & & \multicolumn{2}{c|}{strong scaling} \\ \hline
\end{tabular}}
\end{table}

\begin{table}[h!]
\center
\caption{Heat equation, parameters on each $\domaini$ of the $2\times 2$
	``fine'' configuration. In this case, there are $\nports=5$ total ports with
	$\ndofPortsArg{1} = 76$, $\ndofPortsArg{2} = 76$, 
$\ndofPortsArg{3} = 78$, $\ndofPortsArg{4} = 78$, $\ndofPortsArg{5} = 4$.
	} \label{tab_ex1_2x2config_params}
{\begin{tabular}{|c||c|c|c|c|}
	\hline \rule{0pt}{2.5ex}
	& $\domainArg{1}$ & $\domainArg{2}$ & $\domainArg{3}$  & $\domainArg{4}$  \\ [0.2ex]
	\hline \rule{0pt}{2.5ex} 
	$\nresi$  		 & 1521 & 1560 & 1560 & 1600  \\
	$\ndofInteriori$ & 1444 & 1482 & 1482 & 1521  \\
	$\ndofBoundaryi$ &  156 &  158 &  158 &  160  \\
	$\ndofi(=\ndofInteriori+\ndofBoundaryi)$ & 1600 & 1640 & 1640 & 1681 \\
	Number of subdomain ports $\card{\portsSubdomains{i}}$ &  3 &  3 &  3 &  3   \\ 
	\hline
\end{tabular}}
\end{table}

\begin{figure}[h!]
\centering
\begin{subfigure}[b]{0.45\textwidth}
\includegraphics[width=8.1cm]{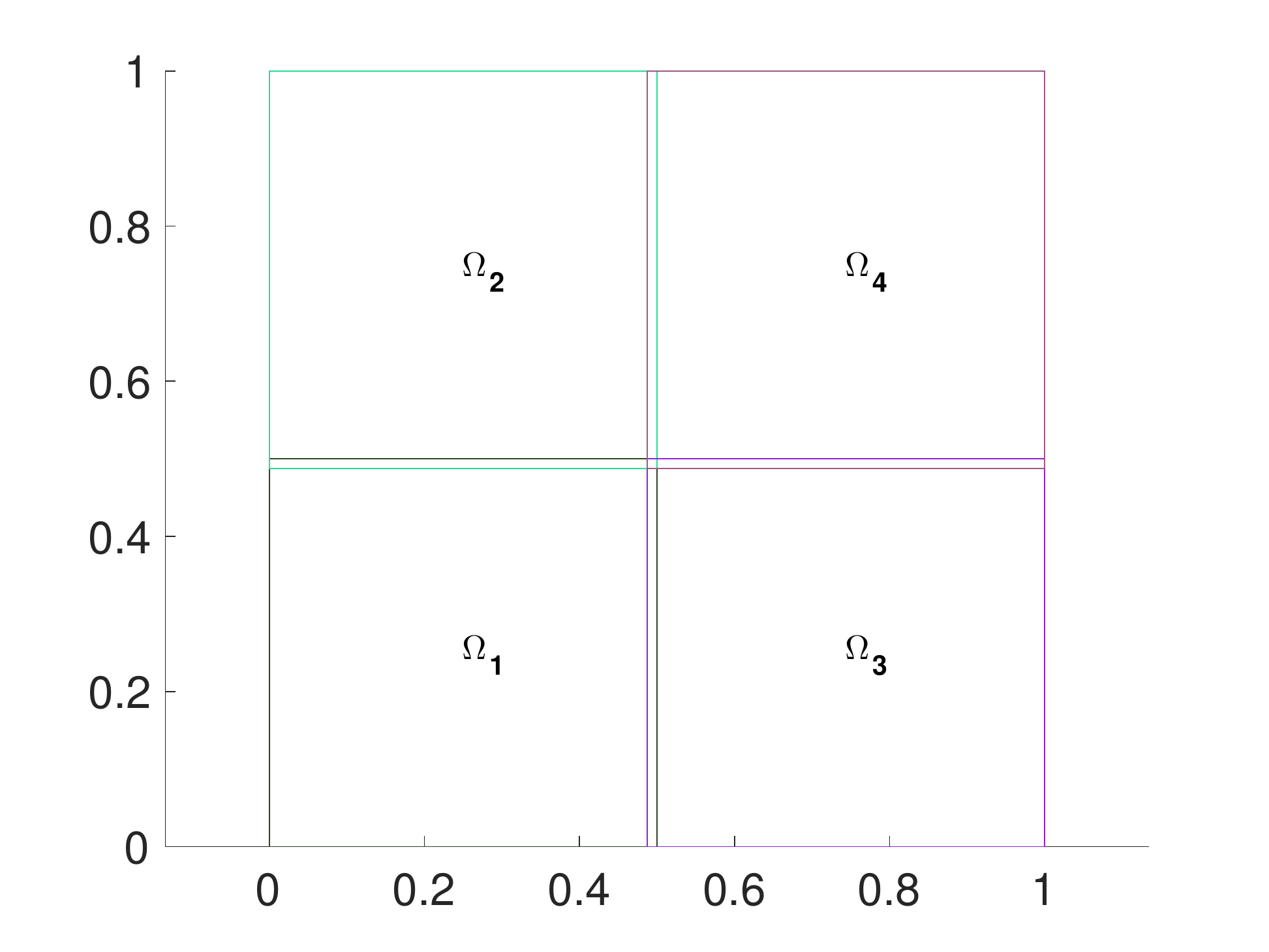}
\caption{$2\times 2$ configuration}
\label{fig_ex1_subdoms_2x2}
\end{subfigure}
~
\begin{subfigure}[b]{0.45\textwidth}
\includegraphics[width=8.1cm]{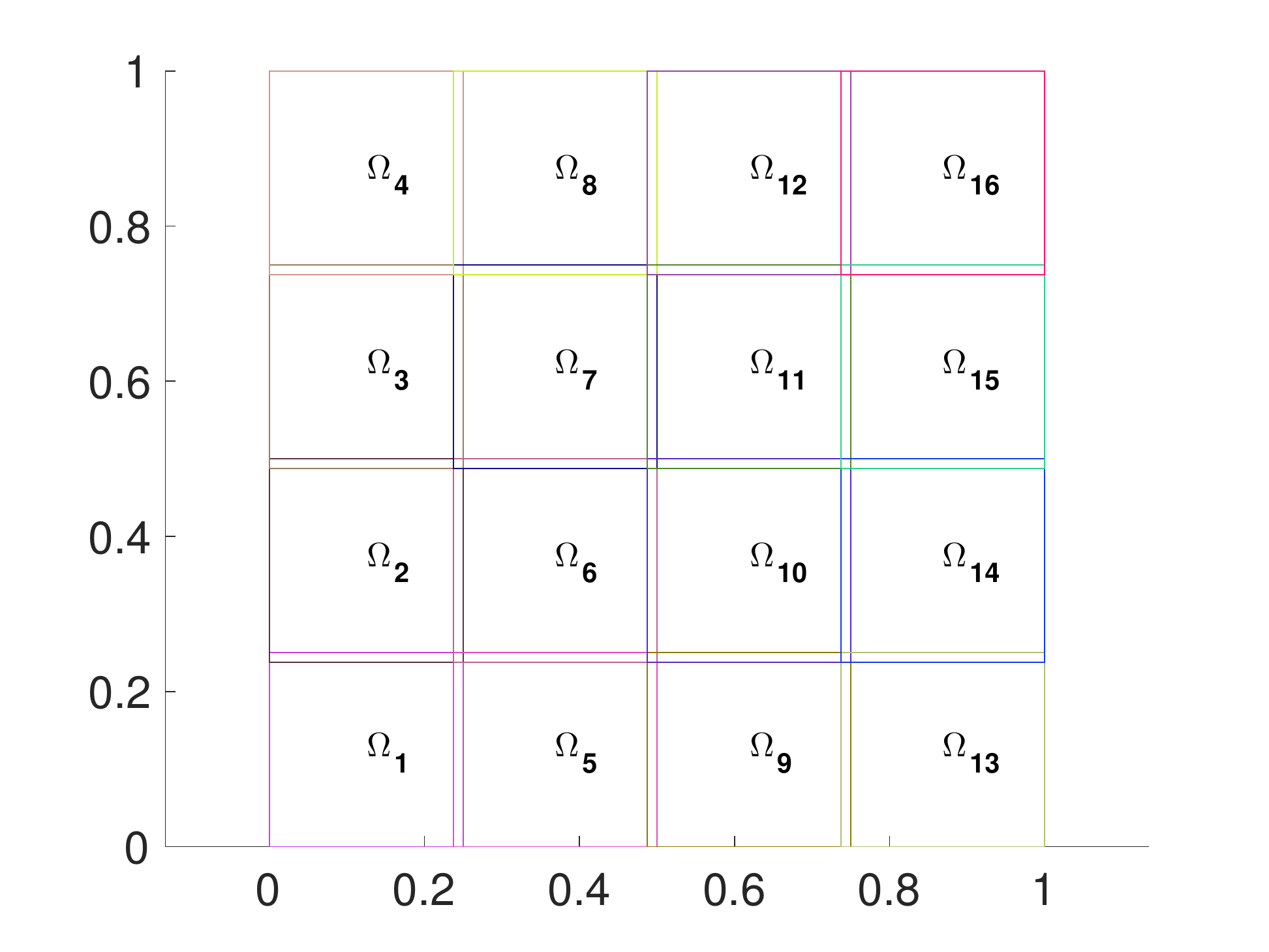}
\caption{$4\times 4$ configuration}
\label{fig_ex1_subdoms_4x4}
\end{subfigure}
\caption{Heat equation, two domain decomposition configurations based on the
	finite-element mesh. }\label{fig_ex1_2x2config_4x4config}
\end{figure}


After applying the finite-element discretization, we introduce the
algebraically non-overlapping decomposition of the problem described in
Section \ref{sec_DDFOM_formulation}. For this problem, the chosen algebraic
decomposition corresponds to a spatial domain decomposition in space. In particular, we employ decompositions into both $2\times 2$ (such that
$\nsubdomains=4$) and $4\times 4$ (such that $\nsubdomains=16$) configurations
as depicted in Figure \ref{fig_ex1_2x2config_4x4config}; note that these local
subdomains have one layer of elements overlapping (as explained in Figure
\ref{fig_divided_models}).  We apply the $2\times 2$ decomposition to the
``coarse'' mesh only, but apply both $4\times 4$ and $2\times 2$ decompositions to the ``fine'' mesh. Table \ref{tab_ex1_DDFEMconfig_params} lists
the parameters used for each of these configurations. The pairwise comparison of the
$2\times 2$ ``coarse'' and $4\times 4$ ``fine'' configurations is interpreted
as
weak scaling, while the pairwise comparision of the $2\times 2$ ``fine'' and $4\times 4$
``fine'' configuations interpreted as strong scaling, respectively. For reference, Table
\ref{tab_ex1_2x2config_params} reports the parameters characterizing each subdomain
$\domainArg{i}$, $i=1,\ldots, \nsubdomains$ of the $2\times 2$ ``fine''
configuration.

\subsubsection{DD-LSPG and DD-GNAT approximations: one online
computation}\label{sec:heatOneOnline}

To generate the reduced bases required for the reduced-order models, we solve
the FOM \eqref{eq_originalAlgebraic} for
$\param\in\trainSample\subset\paramDomain$. In our case, we define the
training-parameter set $\trainSample$ via a $20\times 20$ equispaced sampling of
the parameter domain $\paramDomain$, yielding $\numTrainSample=400$ samples. We
apply the methods described in Section \ref{sec_basis_construction} to create
port, skeleton, full-interface, and full-subdomain bases from these training
data. However, we recall that skeleton bases require full-system snapshots and
thus are not generally practical for decomposable systems that demand
``bottom-up'' training; we still include this approach for comparative purposes.
At each iteration of the Newton--Raphson algorithm used to solve the FOM
equations \eqref{eq_originalAlgebraic}, the residual vector is saved; the
resulting residual snapshots are employed to generate the residual bases
$\rbResi$, $i=1,\ldots,\nsubdomains$ employed by DD-GNAT via POD. Lastly, the
GNAT offline \YC{Algorithm}~\ref{alg_Greedy_GNAT} is performed to create sample meshes $\sampleSetOfSampleMesh{i}$ for all subdomains $\domainArg{i}$.

\begin{table}[h!]
	\small
\center
\caption{Heat equation, top-down training, $2\times 2$ ``fine'' configuration, ROM methods
	performance at point $\paramComp_{{\rm test}}=(5.005,5.005) \notin \trainSample$ for one
	online computation. Recall from Section
	\ref{sec_basis_construction} that $\energyCriterion\in[0,1]$ denotes the energy
	criterion employed by POD.} 
\label{tab_ex1_oneOnlineComp_inputParams}
{\begin{tabular}{|c||c|c||c|c||c|c||c|c|}
	\hline \rule{0pt}{1.5ex}
	constraint	& \multicolumn{8}{c|}{strong} \\	
	\hline \rule{0pt}{2.5ex}	
	basis	& \multicolumn{2}{c||}{port} & \multicolumn{2}{c||}{skeleton} &
	\multicolumn{2}{c||}{full-interface} & \multicolumn{2}{c|}{subdomain} \\ [0.3ex]
	\hline \rule{0pt}{2.5ex}
	method  & DD-LSPG  &  DD-GNAT & DD-LSPG  &  DD-GNAT & DD-LSPG  &  DD-GNAT &
	DD-LSPG  &  DD-GNAT  \\ [0.3ex]
	\hline \rule{0pt}{2.5ex} 
	$\energyCriterion$ for state & $1-10^{-5}$ & $1-10^{-5}$ & $1-10^{-5}$ & $1-10^{-5}$ & $1-10^{-5}$ & $1-10^{-5}$  & $1-10^{-5}$ & $1-10^{-5}$ \\
	$\energyCriterion$ for residual &   		  & $1-10^{-12}$ &   		  & $1-10^{-12}$ &   		  & $1-10^{-12}$  &  & $1-10^{-12}$ \\
	$\nsampleArg{i}/\nrbResi$		  &  & 2 &  & 2 &  & 2 &  &  2 \\
	\hline \rule{0pt}{2.5ex} \hspace{-2.5mm}	
	rel. error & 	0.0026 	  & 0.0012 & 0.0025 & 0.0019 & 0.6959 & 0.6667 & 1.0000 & 1.0000 \\
	speedup  	   &	3.87  	  & 8.86   & 3.88   & 8.82   & 3.91   & 8.85 & 13.61 & 30.98 \\
	\hline
\end{tabular}}
\end{table}

\begin{table}[h!]
\center
\caption{Heat equation, top-down training, $2\times 2$ ``fine'' configuration, ROM parameters on each $\domaini$ resulting from Table~\ref{tab_ex1_oneOnlineComp_inputParams}.}
\label{tab_ex1_oneOnlineComp_ROMparams}
{\begin{tabular}{|c||c|c|c|c||c|c|c|c||c|c|c|c||c|c|c|c|}
	\hline \rule{0pt}{2.5ex}
	basis	& \multicolumn{4}{c||}{port} & \multicolumn{4}{c||}{skeleton} &
	\multicolumn{4}{c||}{full-interface} & \multicolumn{4}{c|}{subdomain}  \\
	\hline \rule{0pt}{2.5ex}
	 & $\domainArg{1}$ & $\domainArg{2}$ & $\domainArg{3}$  & $\domainArg{4}$ & $\domainArg{1}$ & $\domainArg{2}$ & $\domainArg{3}$  & $\domainArg{4}$ & $\domainArg{1}$ & $\domainArg{2}$ & $\domainArg{3}$  & $\domainArg{4}$  & $\domainArg{1}$ & $\domainArg{2}$ & $\domainArg{3}$  & $\domainArg{4}$  \\ [0.2ex]
	\hline \rule{0pt}{2.5ex} \hspace{-2.5mm}
	$\nconstraintsROM$ 	& \multicolumn{4}{c||}{18} & \multicolumn{4}{c||}{9} & 
\multicolumn{4}{c||}{12}& \multicolumn{4}{c|}{12} \\	
\hline
	$\nrbInteriori$ 	&  4 &  2 &  2 &  4 & 4 & 2 & 2 & 4 & 4  &  2  &  2 & 4  & 4 & 2 & 2 & 4 \\
	$\nrbBoundaryi$ 	&  8 &  8 &  8 &  8 & 3 & 3 & 3 & 3 & 3  &  3  &  3 & 3  & 4 & 2 & 2 & 4 \\
	\hline \rule{0pt}{2.5ex}
	$\nrbPortArg{1}$ 	& 3 & 3 & 3 & 3  &    &    &    &    &    &    &    &  &  &  &  &  \\
	$\nrbPortArg{2}$ 	& 2 & 2 & 2 & 2  &    &    &    &    &    &    &    &  &  &  &  &   \\
	$\nrbPortArg{3}$ 	& 3 & 3 & 3 & 3  &    &    &    &    &    &    &    &  &  &  &  &   \\
	\hline \rule{0pt}{2.5ex}
	$\nsampleArg{i}$ 		& 102 & 310 & 310 & 104 & 102 & 310 & 310 & 104 & 102 & 310 & 310 & 104  & 102 & 310 & 310 & 104 \\
	$\nrbResi$ 			&  51 & 155 & 155 &  52 &  51 & 155 & 155&  52  &  51 & 155 & 155 &  52  & 51 & 155 & 155 & 52 \\ 
	\hline
\end{tabular}}
\end{table}

\begin{figure}[h!]
\centering
\begin{subfigure}[b]{0.3\textwidth}
\includegraphics[width=5.5cm]{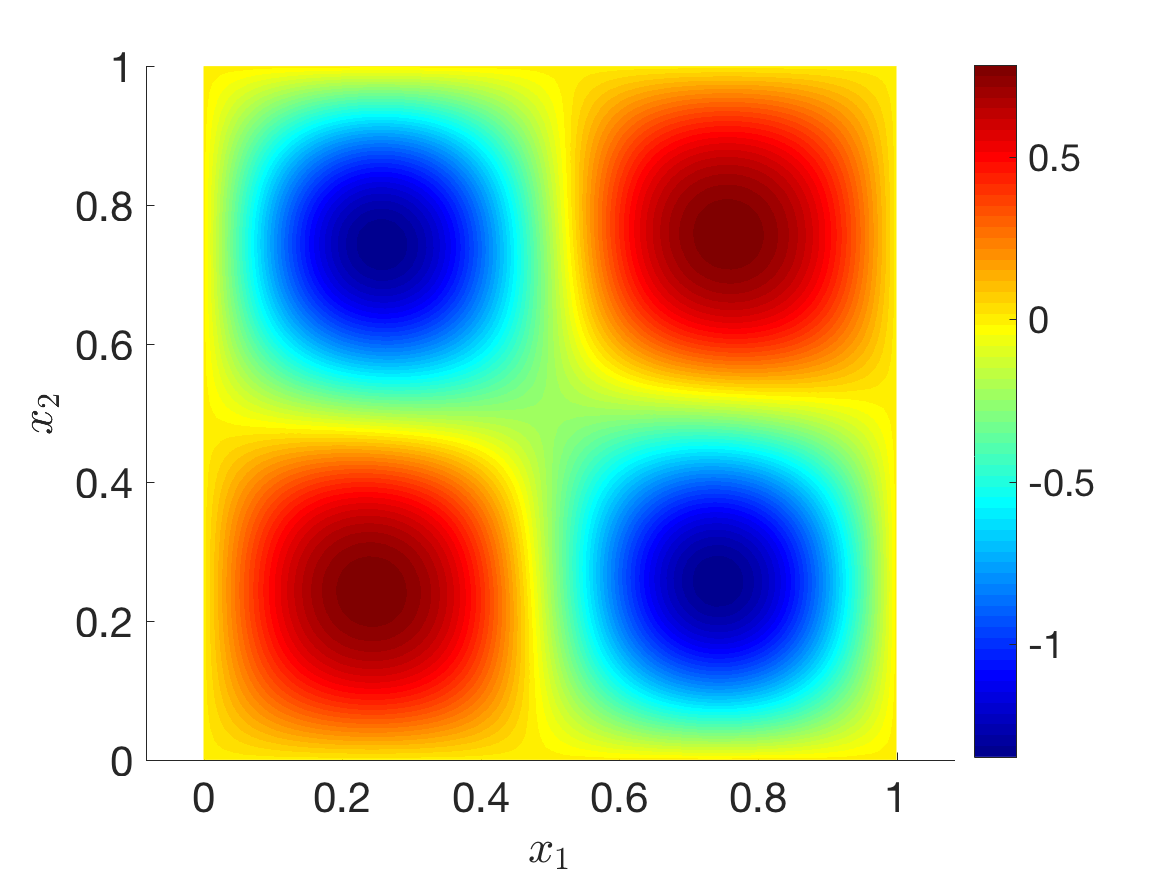}
\caption{global FEM solution}
\label{fig_ex1_heat22fn_sol_DDFEM_Omega}
\end{subfigure}
~
\begin{subfigure}[b]{0.3\textwidth}
\includegraphics[width=5.5cm]{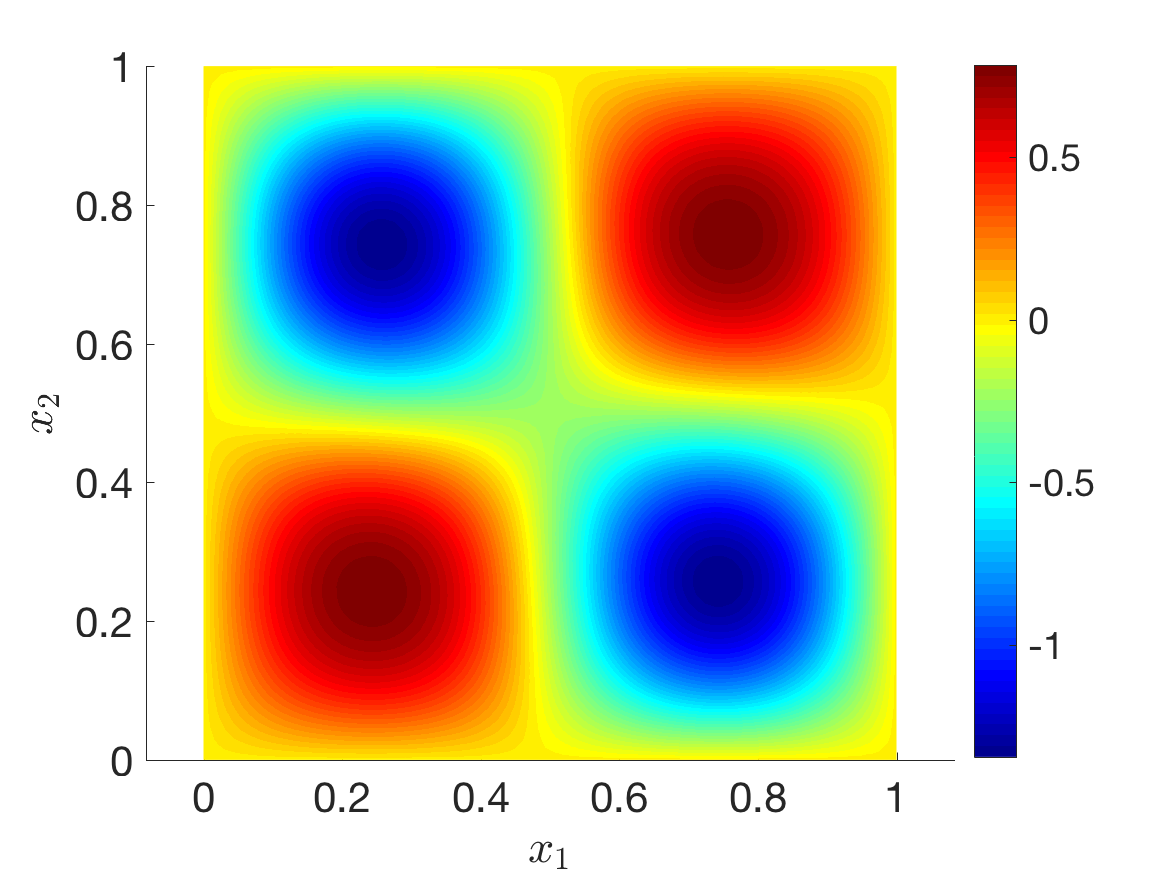}
\caption{DD-LSPG solution}
\label{fig_ex1_heat22fn_sol_DDROM_Omega}
\end{subfigure}
~
\begin{subfigure}[b]{0.3\textwidth}
\includegraphics[width=5.5cm]{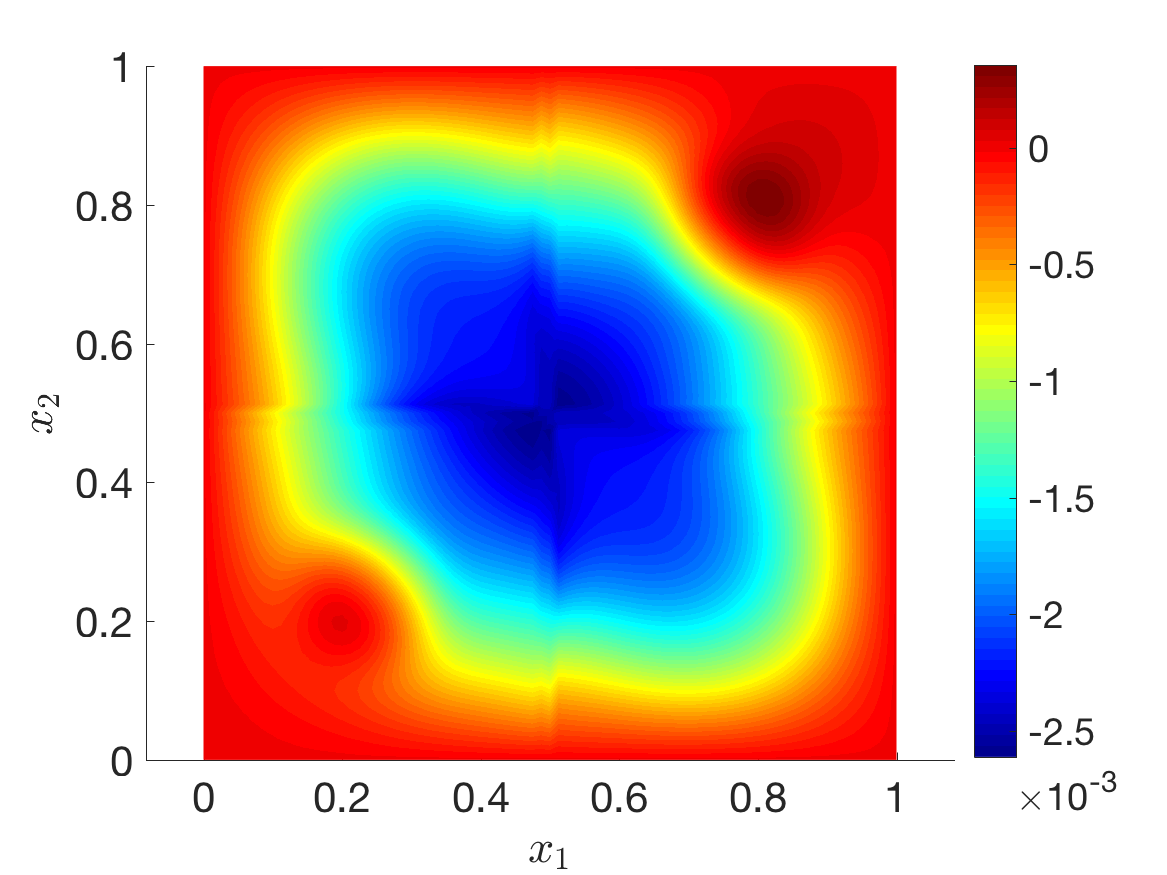}
\caption{DD-LSPG error}
\label{fig_ex1_heat22fn_sol_DDROMerr_Omega}
\end{subfigure}
~
\begin{subfigure}[b]{0.3\textwidth}
\includegraphics[width=5.5cm]{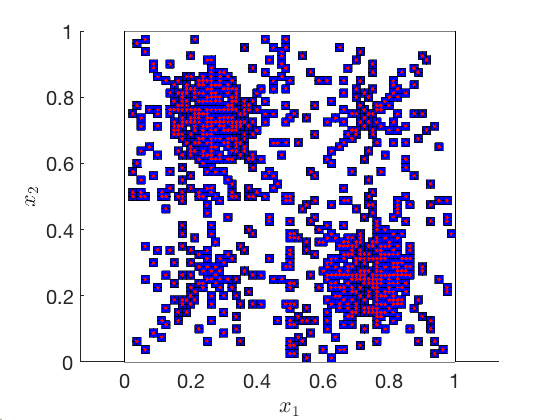}
\caption{Sample mesh}
\label{fig_ex1_heat22fn_sol_sampleMesh_Omega}
\end{subfigure}
~
\begin{subfigure}[b]{0.3\textwidth}
\includegraphics[width=5.5cm]{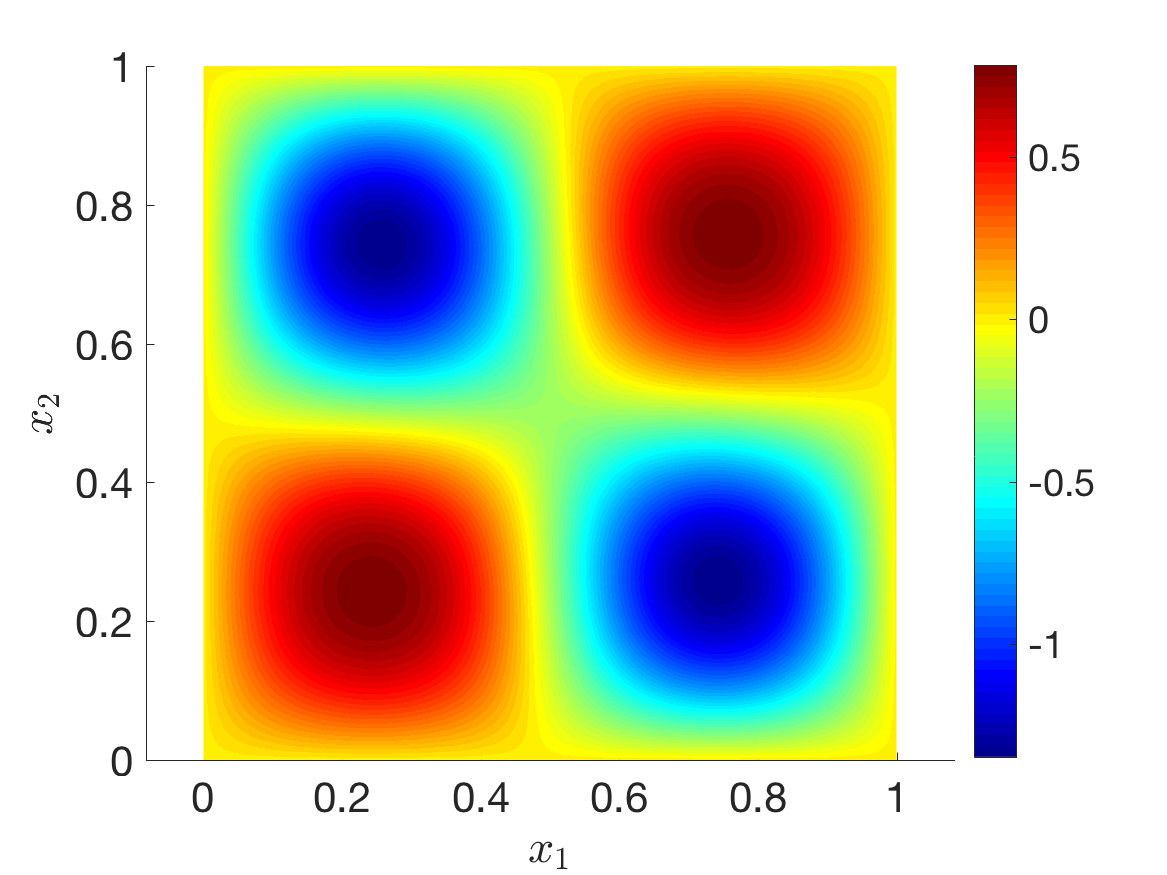}
\caption{DD-GNAT solution}
\label{fig_ex1_heat22fn_sol_DDGNAT_Omega}
\end{subfigure}
~
\begin{subfigure}[b]{0.3\textwidth}
\includegraphics[width=5.5cm]{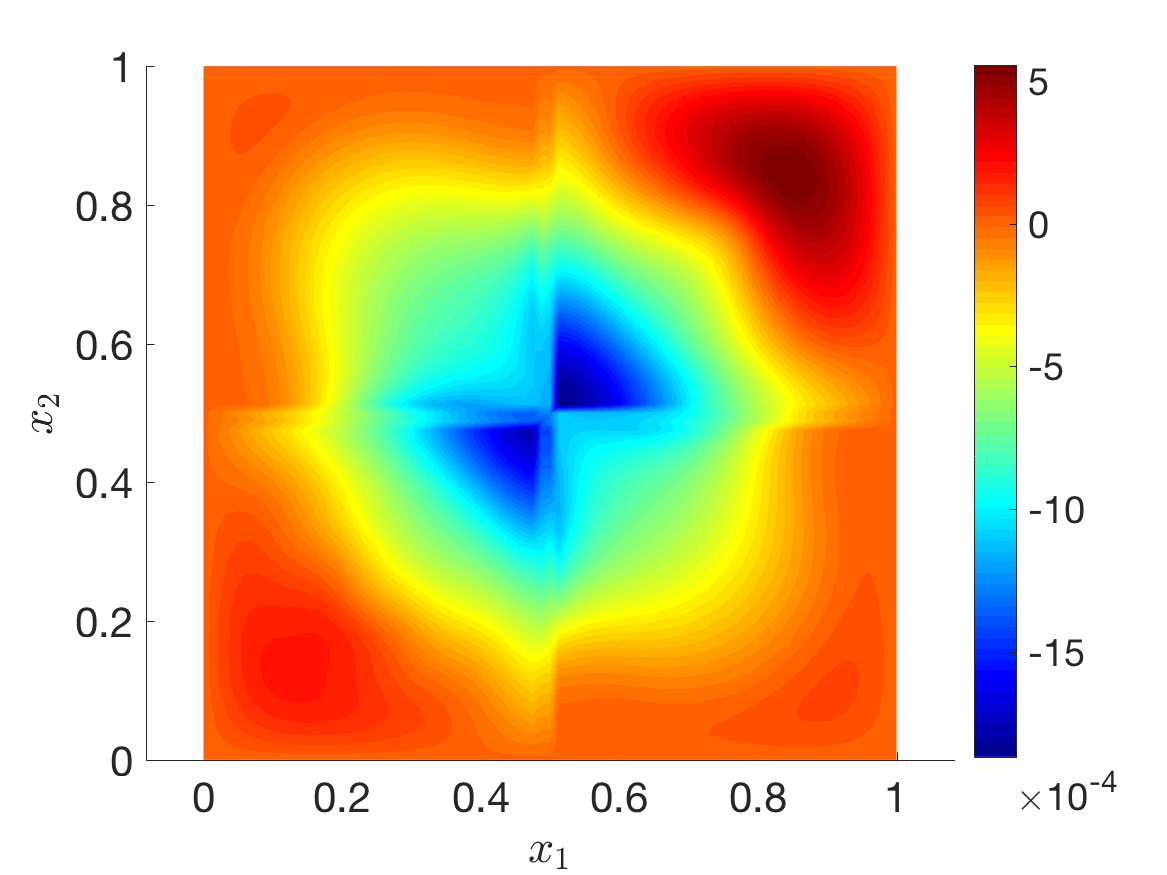}
\caption{DD-GNAT error}
\label{fig_ex1_heat22fn_sol_DDGNATerr_Omega}
\end{subfigure}	
\caption{Heat equation, top-down training, $2\times 2$ ``fine'' configuration, port bases in
	Table \ref{tab_ex1_oneOnlineComp_inputParams}, solutions visualized on $\domain$. }\label{fig_ex1_sol_vis_Omega}
\end{figure}

We now compare the DD-LSPG and DD-GNAT methods for fixed values of their parameters,
and for a single randomly selected online point $\paramComp_{{\rm test}}=(5.005,5.005) \notin
\trainSample$; results at other online points are qualitatively similar. Table~\ref{tab_ex1_oneOnlineComp_inputParams} reports the
chosen input parameters and associated performance of the methods, while the
resulting ROM parameters over each subdomain $\domaini$, $i=1,\ldots,
\nsubdomains$ are listed on Table~\ref{tab_ex1_oneOnlineComp_ROMparams}.
The results in
Table~\ref{tab_ex1_oneOnlineComp_inputParams} confirm the comments in Remark
\ref{rem:globalSol}, which suggested
that 
enforcing strong compatibility can yield poor results for 
full-interface and full-subdomain bases, and that only port and (generally
impractical) skeleton bases are
well-suited for strong compatibility constraints. Figure~\ref{fig_ex1_sol_vis_Omega} visualizes DD-LSPG and DD-GNAT solutions for the port-bases case: it shows that DD-LSPG and DD-GNAT yield accurate results for port bases with strong constraints as anticipated. \CH{Table~\ref{tab_ex1_oneOnlineComp_ROMparams} also shows specifically that $\nrbResi$ is about one order of magnitude larger than $\nrbi$, note that this is normal as $\nsampleArg{i} \ge \nrbResi \ge \nrbi$ is a necessary consistency condition to ensure the GNAT method works (see \cite{carlbergbou-mosleh2011,carlberg2013gnat}). Figure~\ref{fig_ex1_heat22fn_sol_sampleMesh_Omega} shows that the DD-GNAT method picks many sampling points near the center region of $\domainArg{i}$ to capture well the solution nonlinearity and hence ensures the desired solution accuracy. As a result, the DD-GNAT method provides almost similar solution to that of DD-LSPG, hence almost similar good accuracy (see Figures~\ref{fig_ex1_heat22fn_sol_DDROMerr_Omega} and \ref{fig_ex1_heat22fn_sol_DDGNATerr_Omega}). }

\subsubsection{DD-LSPG and DD-GNAT approximations: parameter study}\label{subsubsect_ex1_many_online_computations}

\begin{table}[h!]
\center
\caption{Heat equation, top-down training, ROM-method parameters limits for parameter study (skel.=skeleton, intf.=full-interface, subdom.=subdomain).
Recall from Section
	\ref{sec_basis_construction} that $\energyCriterion\in[0,1]$ denotes the energy
	criterion employed by POD.
	} \label{tab_ex1_manyOnlineComputation}
{\begin{tabular}{|c||c|c|}
	\hline \rule{0pt}{2.5ex}
	method  & DD-LSPG  &  DD-GNAT  \\ [0.2ex]
	\hline \rule{0pt}{2.5ex} 
	$\energyCriterion$ on $\domaini$ for interior/boundary bases  & $\{1-10^{-5}, 1-10^{-8}\}$ & $\{1-10^{-5}, 1-10^{-8}\}$  \\
$\energyCriterion$ on $\boundaryi$ for interior/boundary bases& $\{1-10^{-5}, 1-10^{-8}\}$ & $\{1-10^{-5}, 1-10^{-8}\}$  \\
 & & \\
	$\energyCriterion$ for full-subdomain bases & $\{1-10^{-3}, 1-10^{-5}, $ & $\{1-10^{-3}, 1-10^{-5}, $ \\
	 & $1-10^{-7}, 1-10^{-9}\}$ & $1-10^{-7}, 1-10^{-9}\}$ \\
 & & \\	 
	$\energyCriterion$ for $\residuali$ &  & $\{1-10^{-6}, 1-10^{-8}, $ \\
 &  & $1-10^{-10}, 1-10^{-12}\}$ \\
 & & \\ 
	$\nsampleArg{i}/\nrbResi$ 		    &  &   \{1, 1.5, 2, 4\}  \\
constraint type	   		    & \{1, 2, 3, 4, 5, strong\}	 &  \{1, 2, 3, 4, 5, strong\}  \\
	basis types	 & \{port, skel., intf., subdom.\} & \{port, skel., intf., subdom.\}  \\
	\hline
\end{tabular}}
\end{table}

\begin{figure}[h!]
\centering
\begin{subfigure}[b]{0.3\textwidth}
\includegraphics[width=5.5cm]{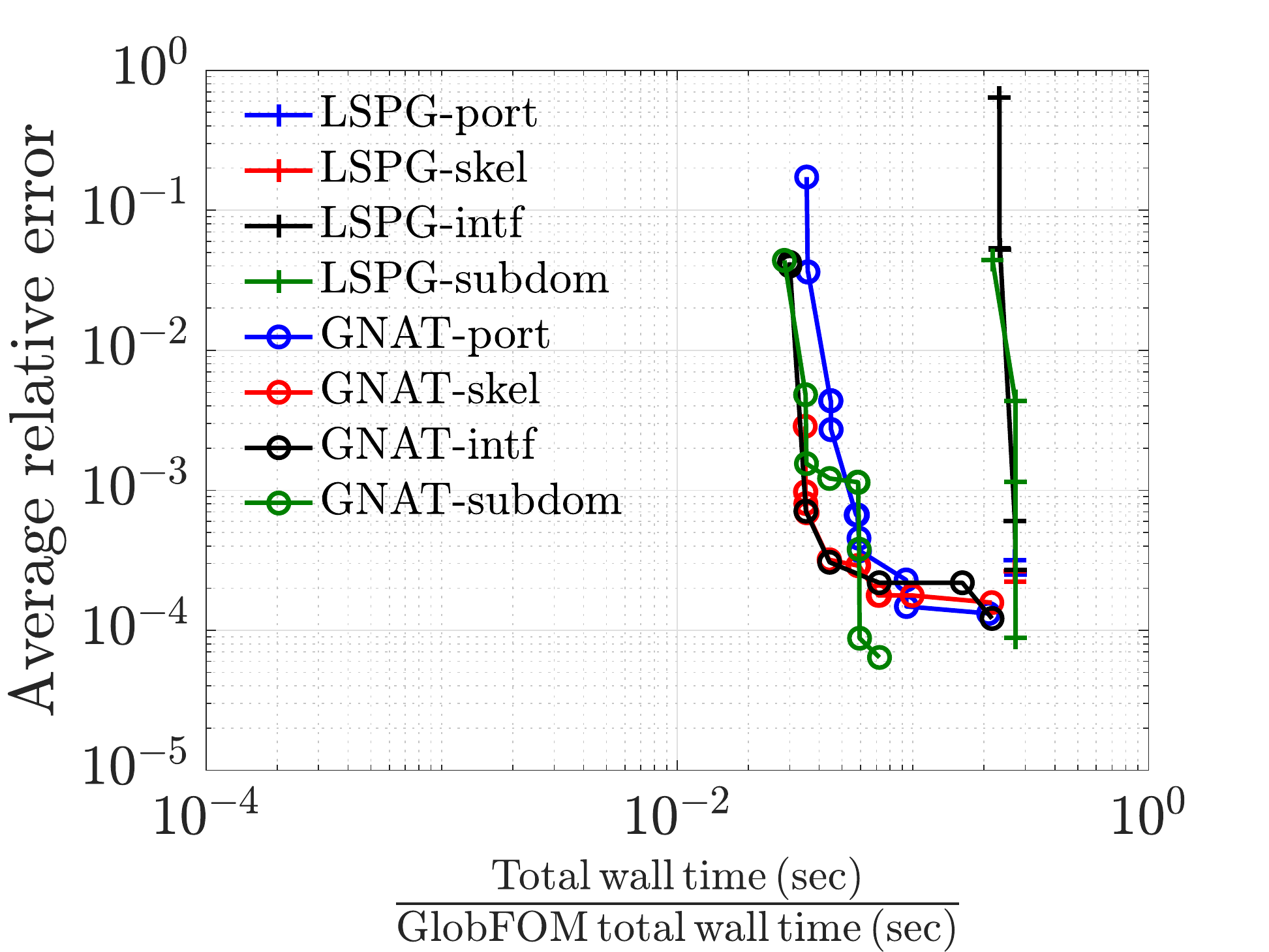}
\caption{Heat $2\times 2$ ``coarse''}
\label{fig_ex1_heat22ce_pareto_wallAll}
\end{subfigure}
~
\begin{subfigure}[b]{0.3\textwidth}
\includegraphics[width=5.5cm]{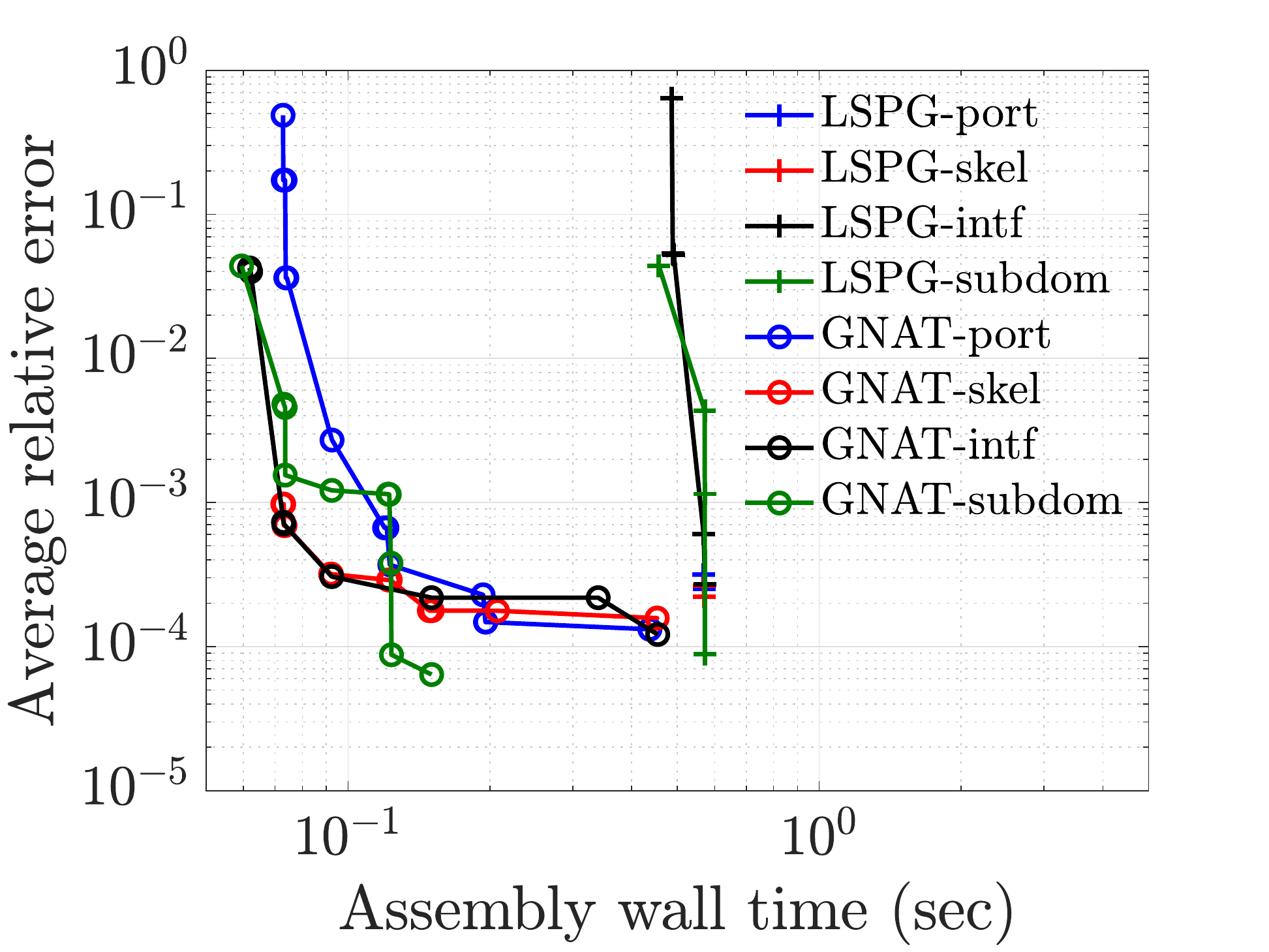}
\caption{Heat $2\times 2$ ``coarse''}
\label{fig_ex1_heat22ce_pareto_wallAsmb}
\end{subfigure}
~
\begin{subfigure}[b]{0.3\textwidth}
\includegraphics[width=5.5cm]{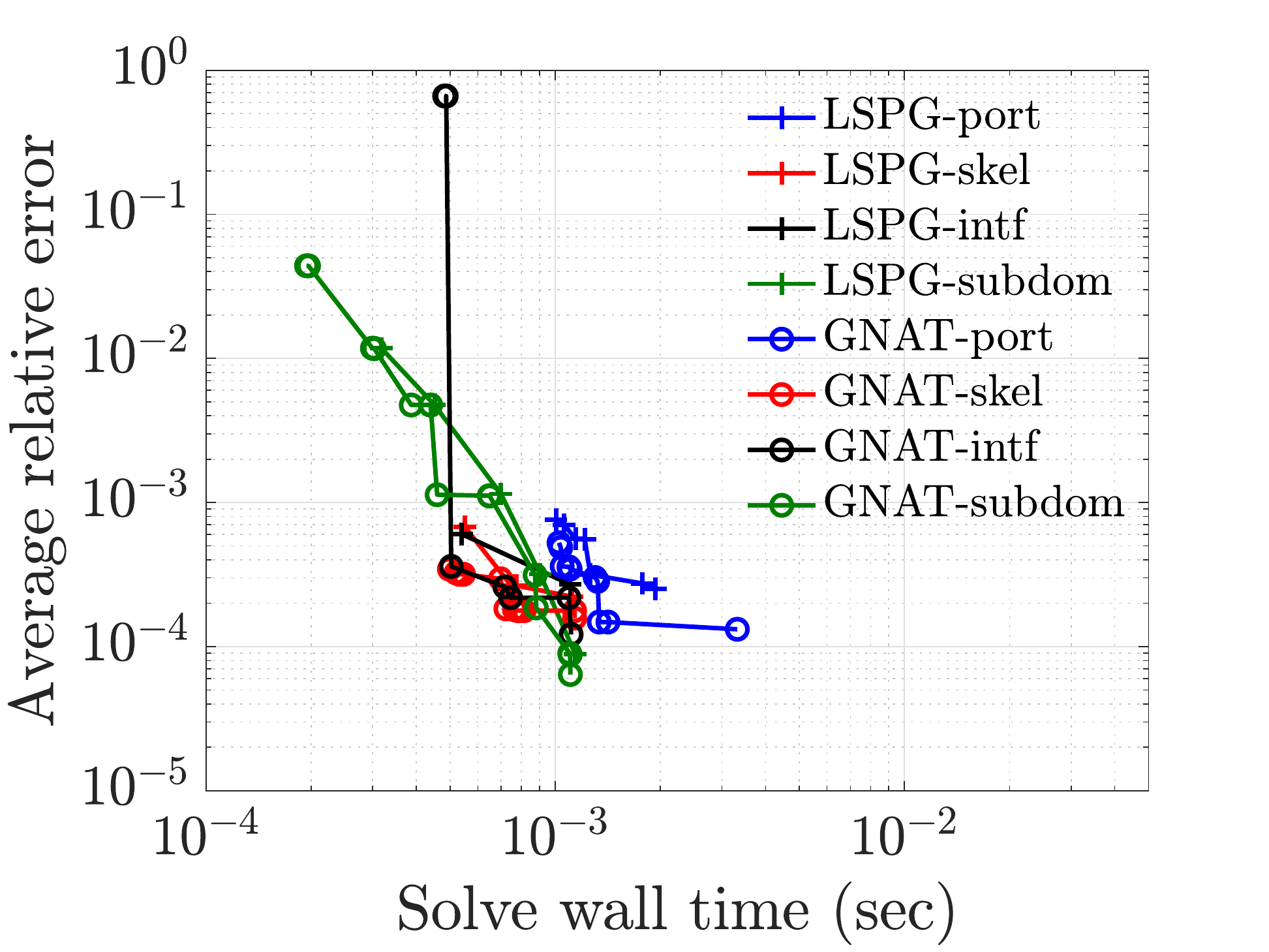}
\caption{Heat $2\times 2$ ``coarse''}
\label{fig_ex1_heat22ce_pareto_wallSolv}
\end{subfigure}
~
\begin{subfigure}[b]{0.3\textwidth}
\includegraphics[width=5.5cm]{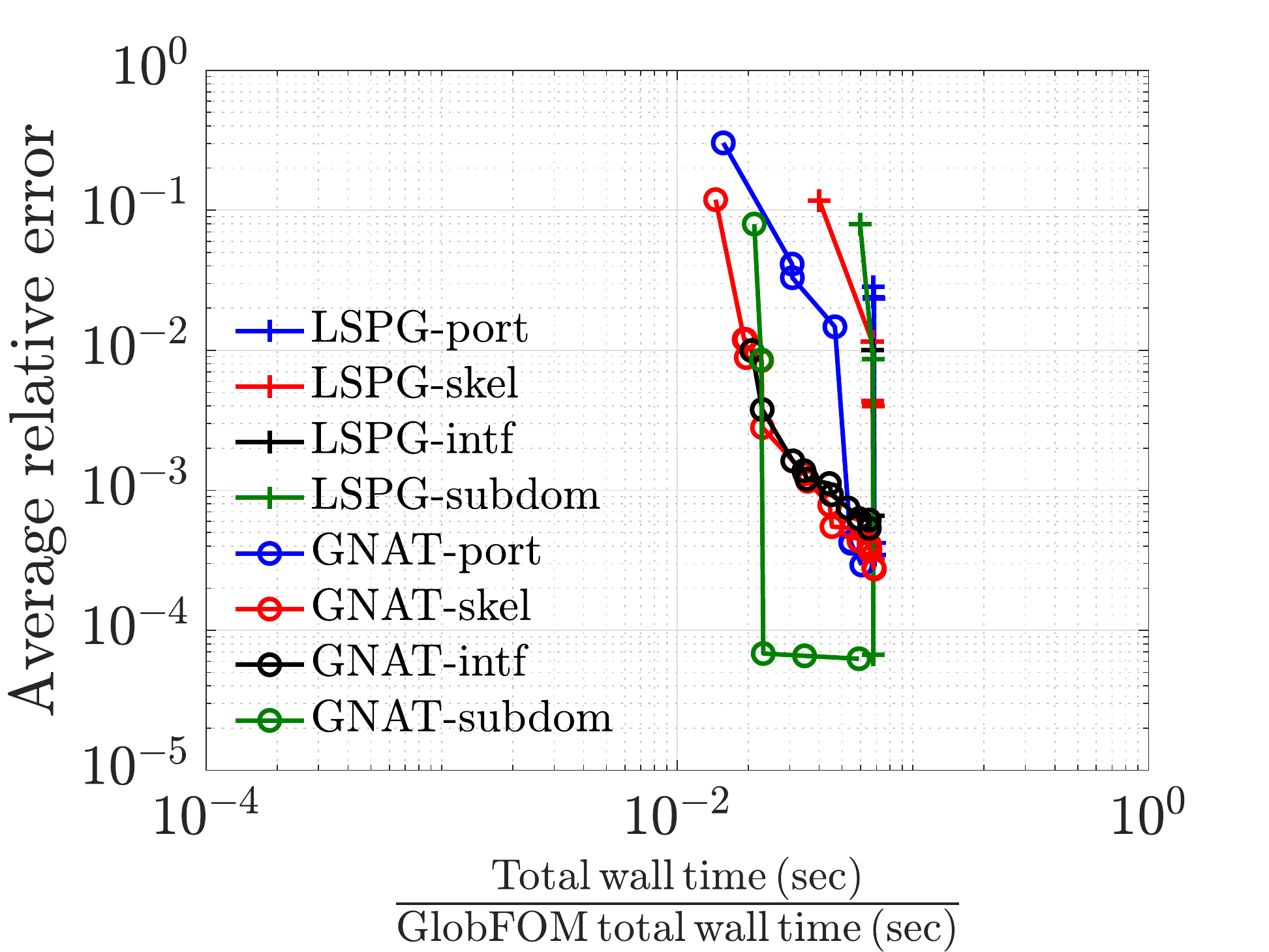}
\caption{Heat $4\times 4$ ``fine''}
\label{fig_ex1_heat44fn_pareto_wallAll}
\end{subfigure}
~
\begin{subfigure}[b]{0.3\textwidth}
\includegraphics[width=5.5cm]{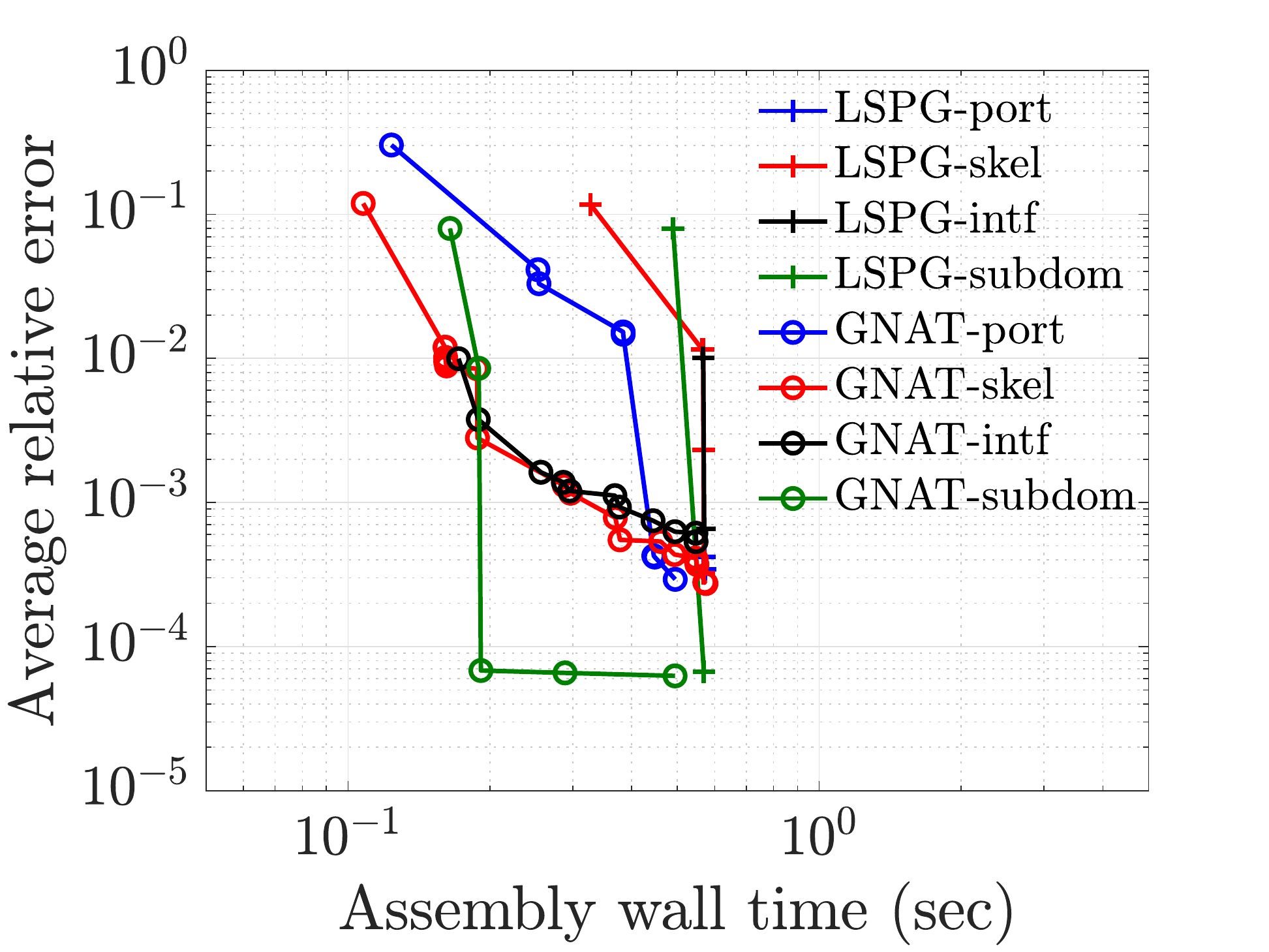}
\caption{Heat $4\times 4$ ``fine''}
\label{fig_ex1_heat44fn_pareto_wallAsmb}
\end{subfigure}
~
\begin{subfigure}[b]{0.3\textwidth}
\includegraphics[width=5.5cm]{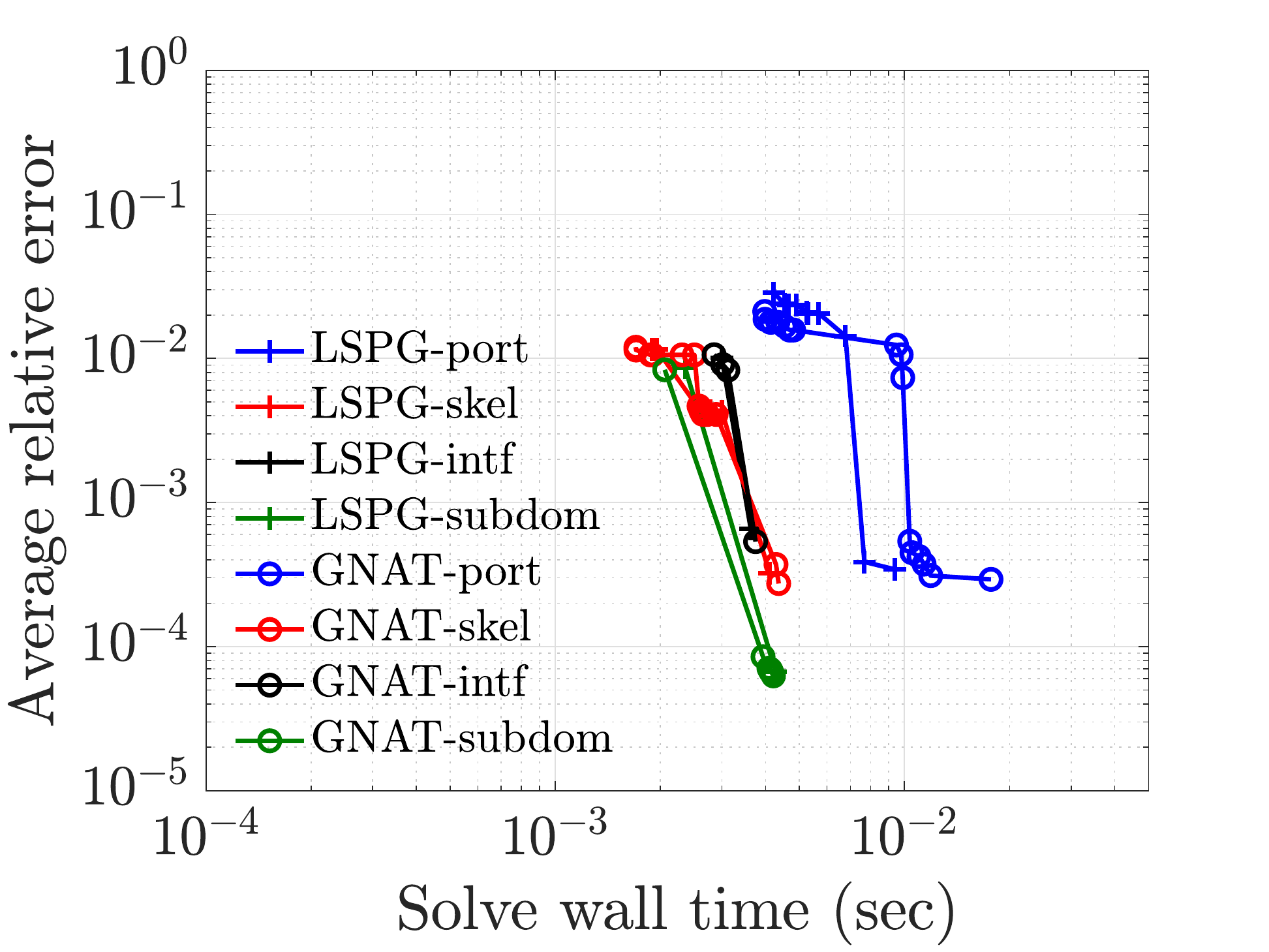}
\caption{Heat $4\times 4$ ``fine''}
\label{fig_ex1_heat44fn_pareto_wallSolv}
\end{subfigure}	
~
\begin{subfigure}[b]{0.3\textwidth}
\includegraphics[width=5.5cm]{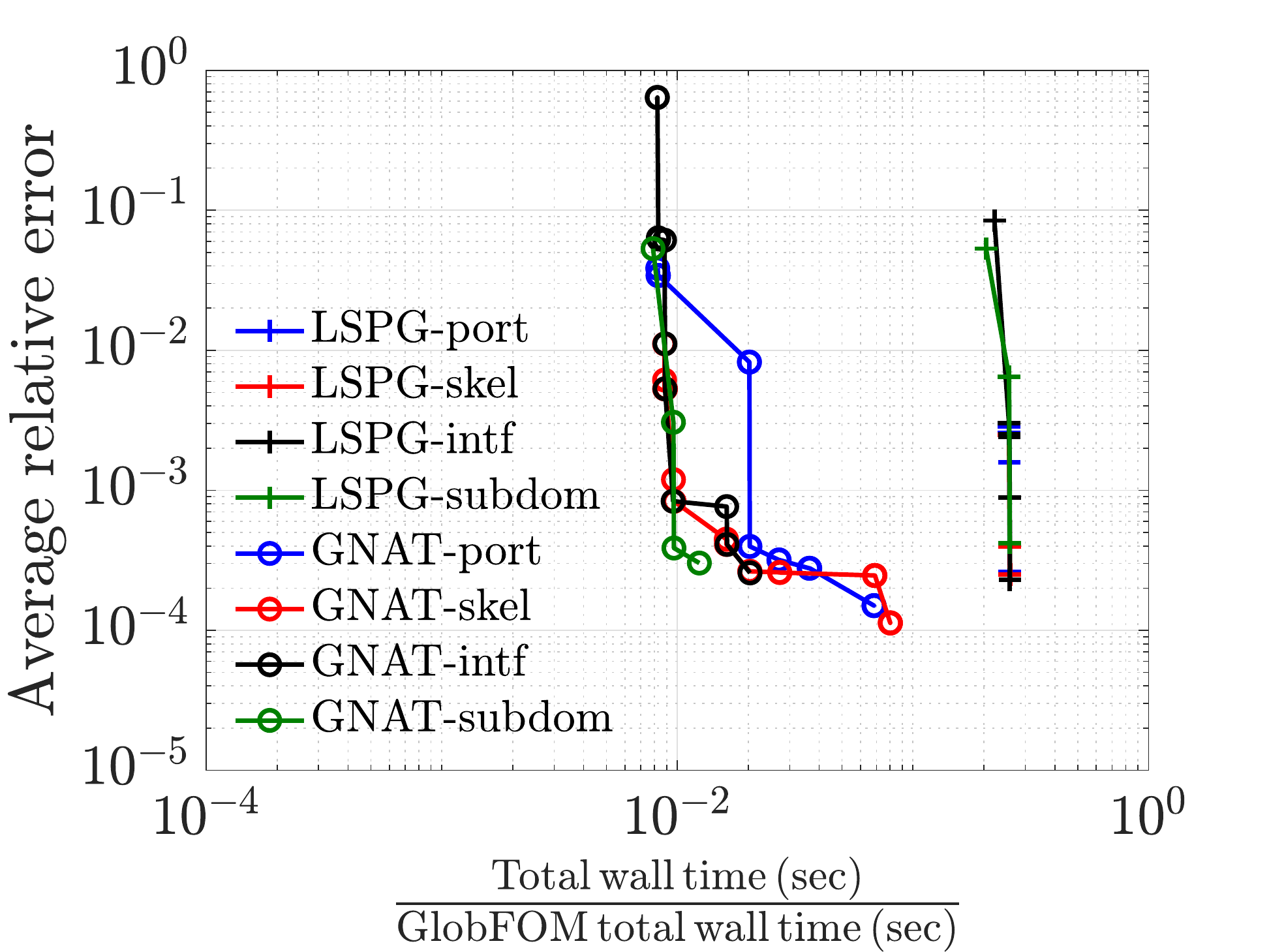}
\caption{Heat $2\times 2$ ``fine''}
\label{fig_ex1_heat22fn_pareto_wallAll}
\end{subfigure}
~
\begin{subfigure}[b]{0.3\textwidth}
\includegraphics[width=5.5cm]{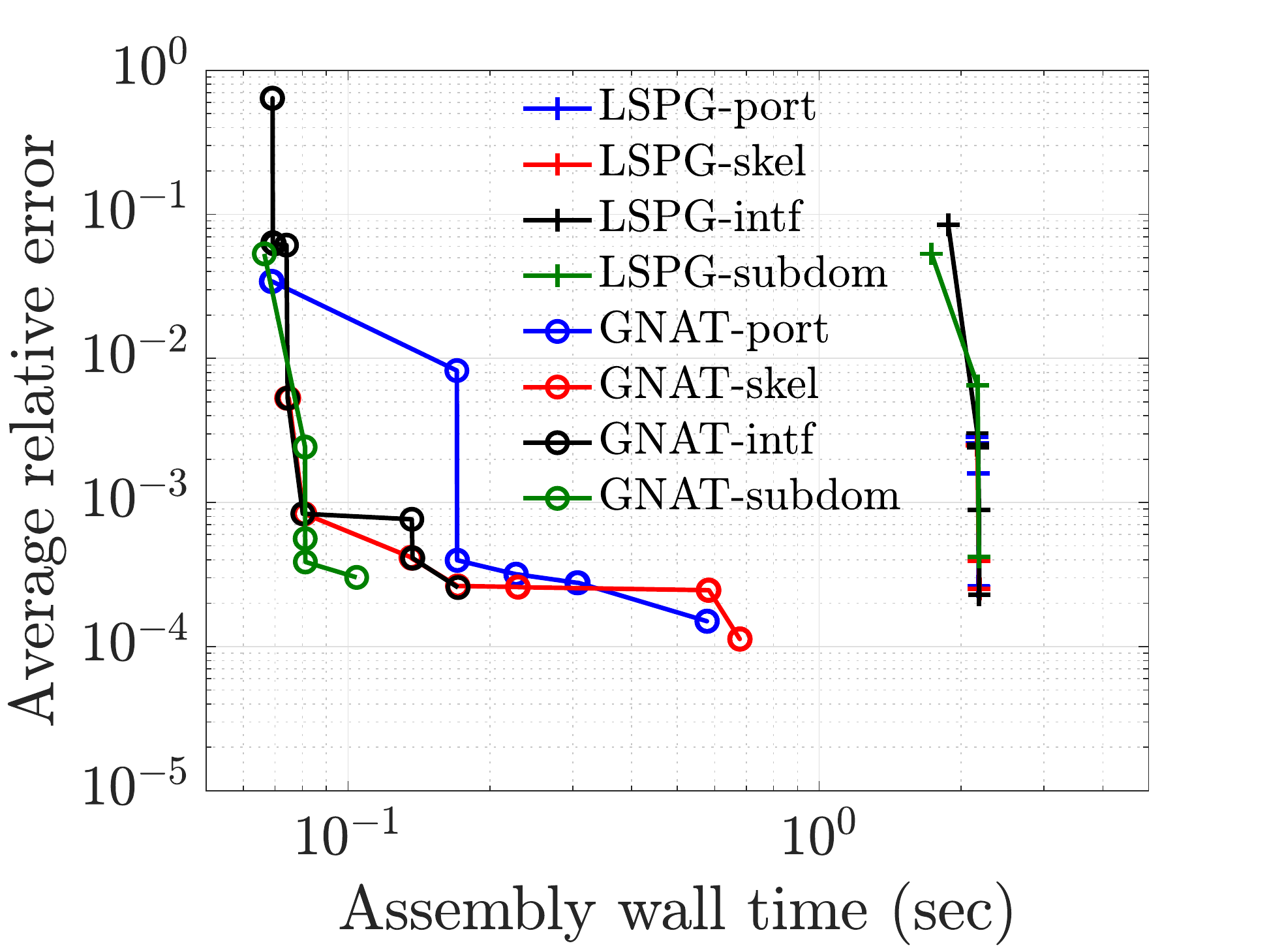}
\caption{Heat $2\times 2$ ``fine''}
\label{fig_ex1_heat22fn_pareto_wallAsmb}
\end{subfigure}
~
\begin{subfigure}[b]{0.3\textwidth}
\includegraphics[width=5.5cm]{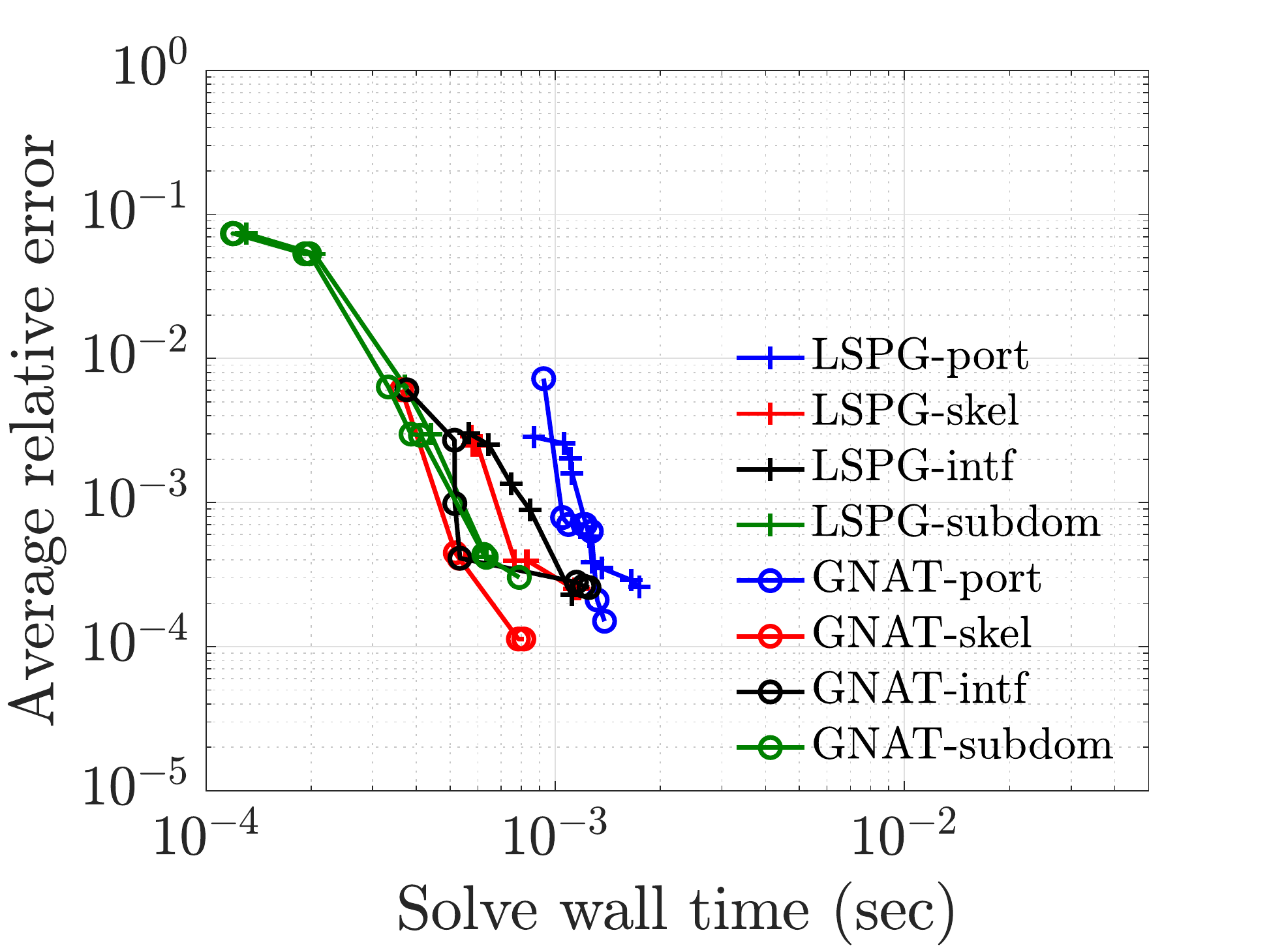}
\caption{Heat $2\times 2$ ``fine''}
\label{fig_ex1_heat22fn_pareto_wallSolv}
\end{subfigure}	
\caption{Heat equation, top-down training, Pareto front plots for wall-all (normalized with respect global FEM wall-all timing), wall-assemble and wall-solve timing of three different configurations for varying model parameters reported in Table~\ref{tab_ex1_manyOnlineComputation}. } \label{fig_ex1_all_pareto}
\end{figure}

\begin{figure}[h!]
\centering
\begin{subfigure}[b]{0.3\textwidth}
\includegraphics[width=5.1cm]{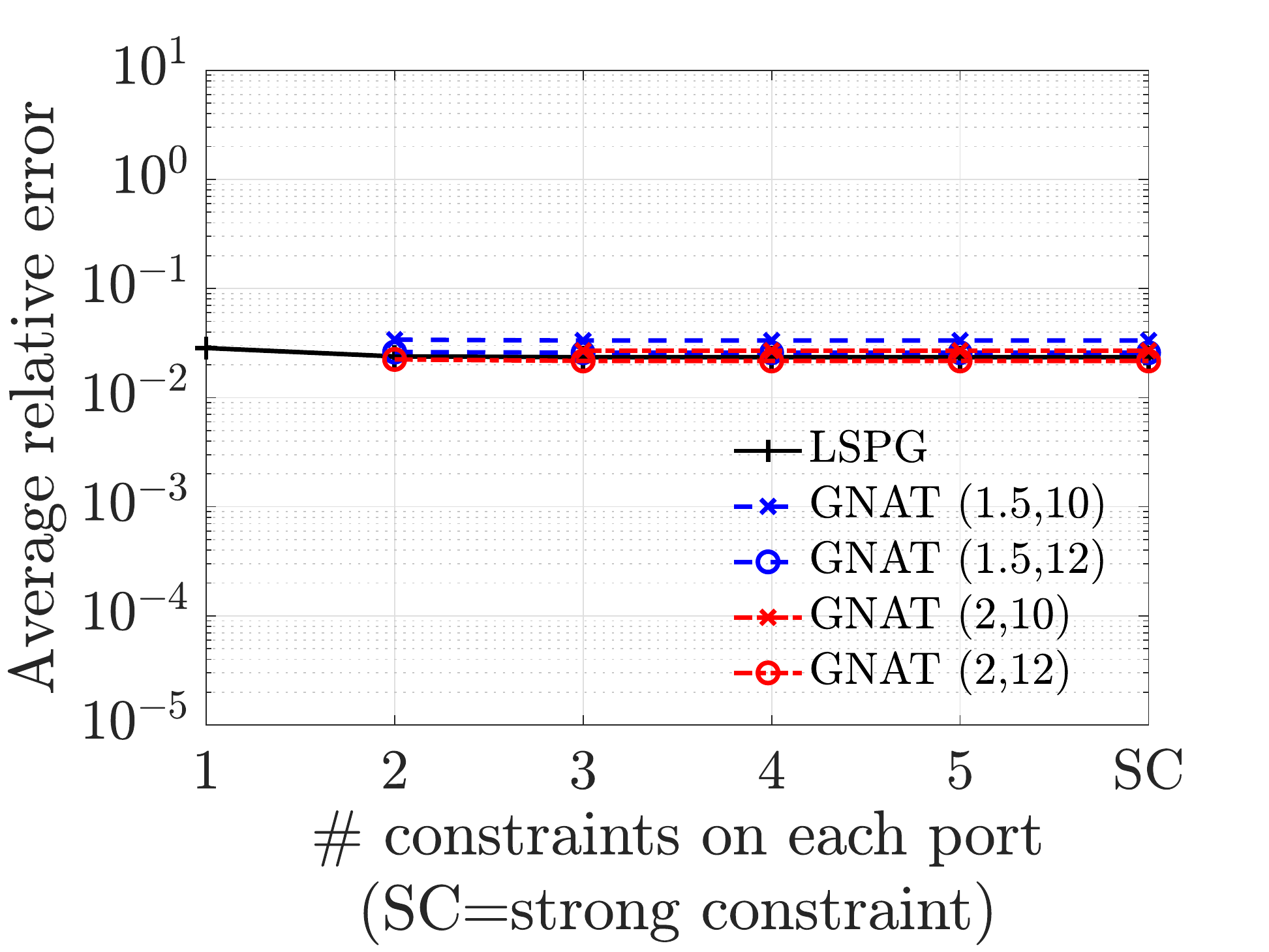}
\caption{port, 
	$\energyCriterion=10^{-5}$ on $\domaini$,
	$\energyCriterion=10^{-5}$ on $\boundaryi$}
\label{fig_ex1_4x4_rmsErr_portBF_55}
\end{subfigure}
~
\begin{subfigure}[b]{0.3\textwidth}
\includegraphics[width=5.1cm]{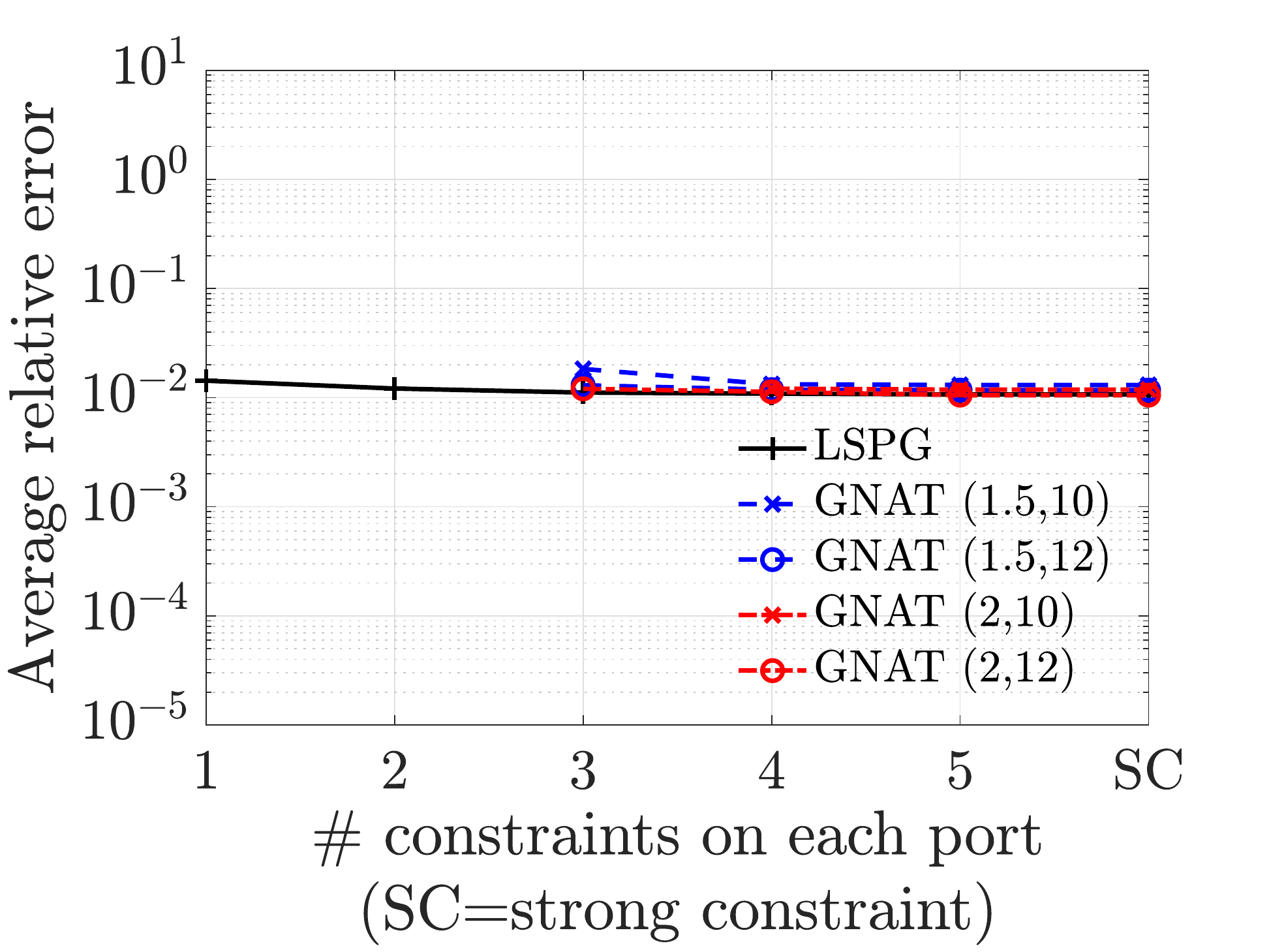}
\caption{port, 
$\energyCriterion=10^{-5}$ on $\domaini$,
	$\energyCriterion=10^{-8}$ on $\boundaryi$}
\label{fig_ex1_4x4_rmsErr_portBF_58}
\end{subfigure}
~
\begin{subfigure}[b]{0.3\textwidth}
\includegraphics[width=5.1cm]{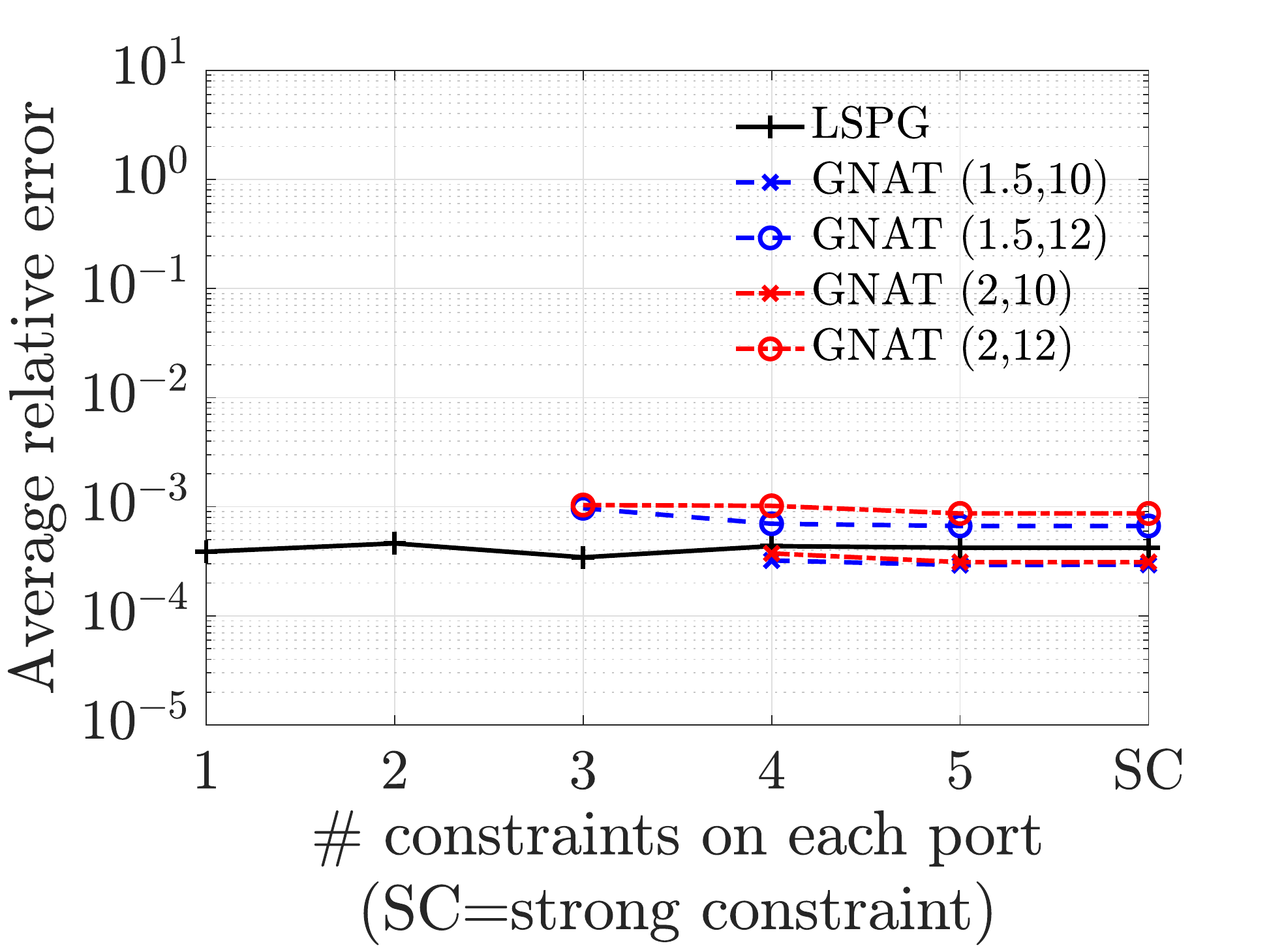}
\caption{port, $\energyCriterion=10^{-8}$ on $\domaini$,
	$\energyCriterion=10^{-8}$ on $\boundaryi$}
\label{fig_ex1_4x4_rmsErr_portBF_88}
\end{subfigure}
~
\begin{subfigure}[b]{0.3\textwidth}
\includegraphics[width=5.1cm]{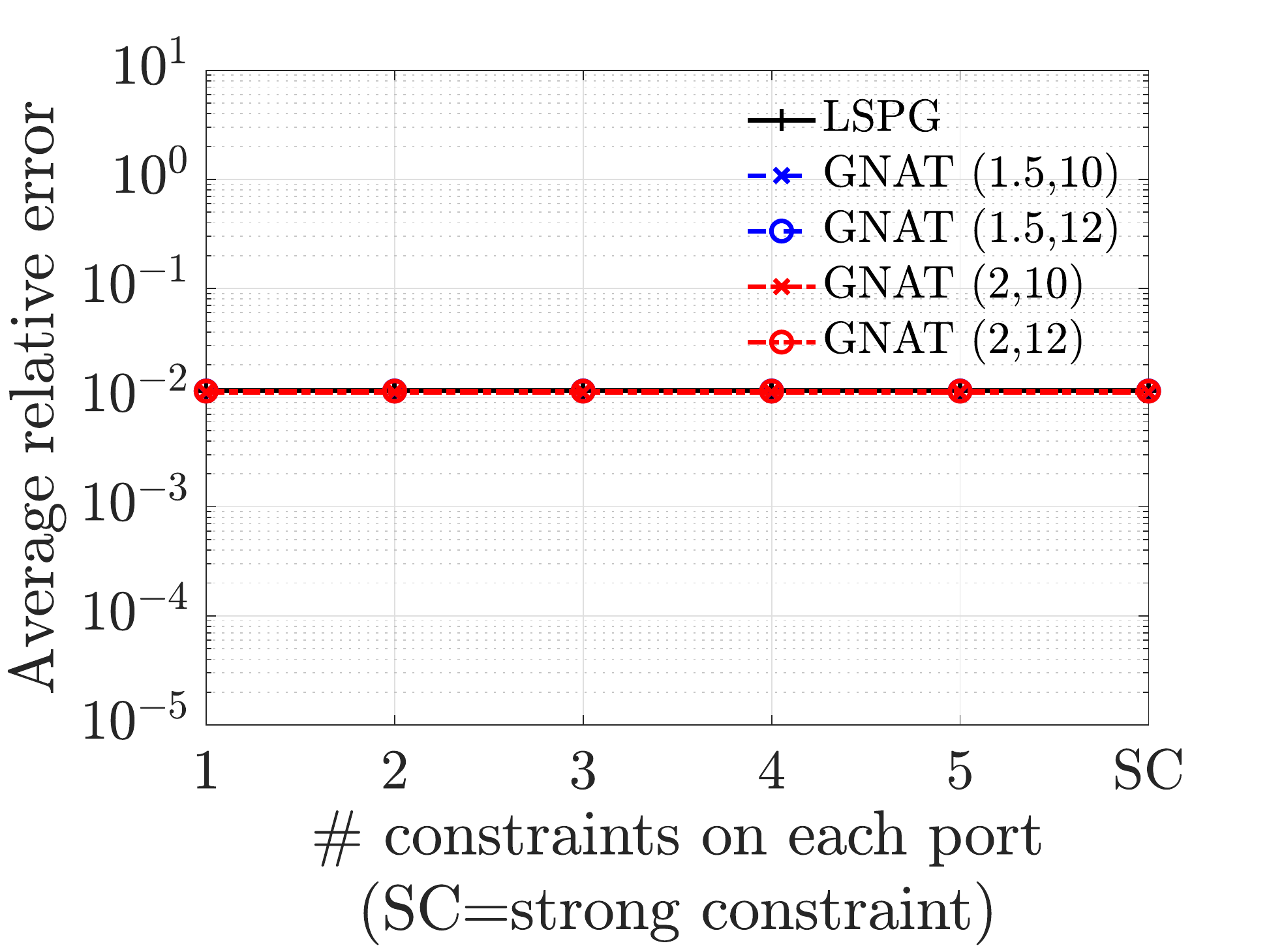}
\caption{skeleton, $\energyCriterion=10^{-5}$ on $\domaini$,
	$\energyCriterion=10^{-5}$ on $\boundaryi$}
\label{fig_ex1_4x4_rmsErr_skelBF_55}
\end{subfigure}
~
\begin{subfigure}[b]{0.3\textwidth}
\includegraphics[width=5.1cm]{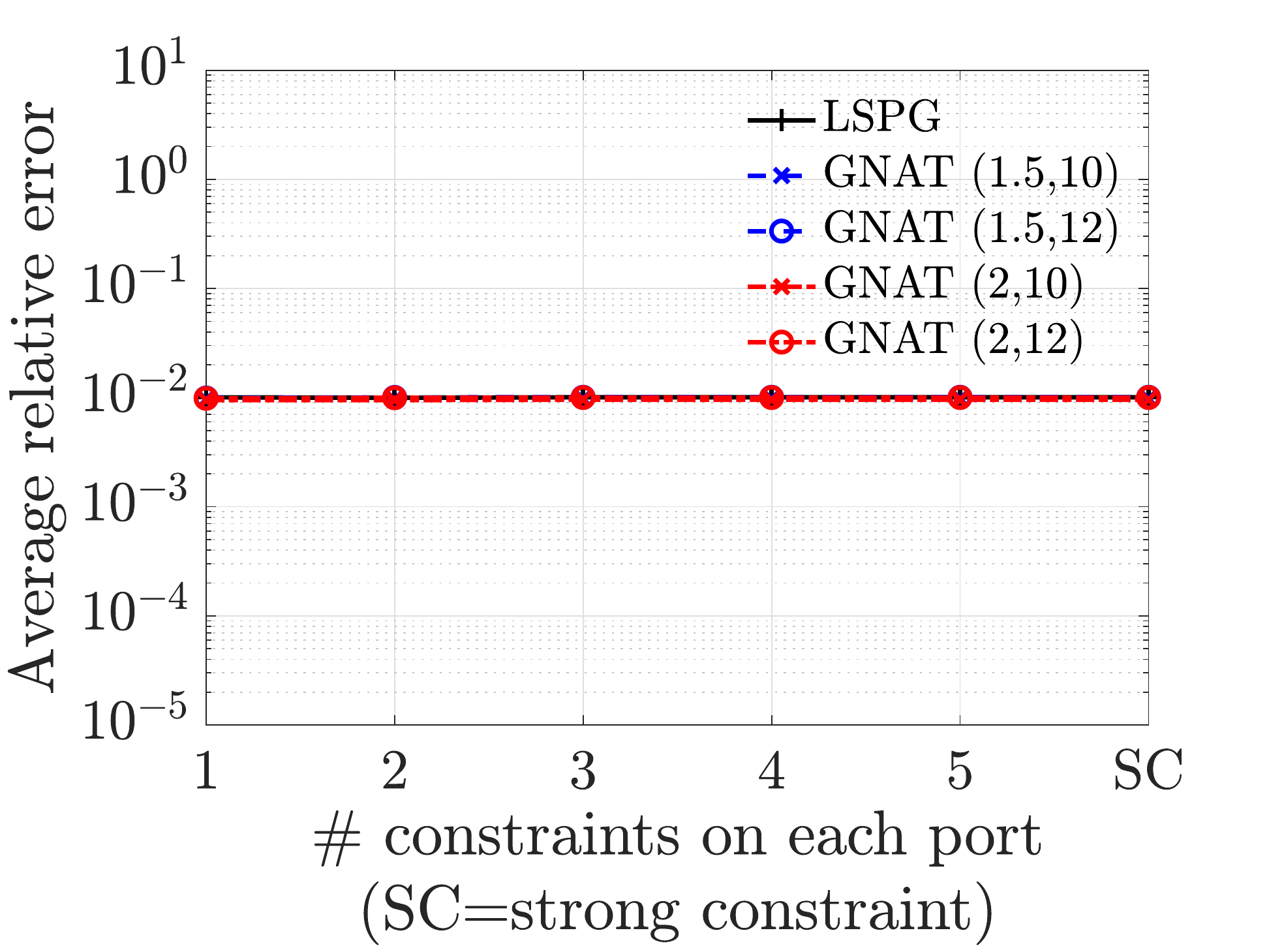}
\caption{skeleton, $\energyCriterion=10^{-5}$ on $\domaini$,
	$\energyCriterion=10^{-8}$ on $\boundaryi$}
\label{fig_ex1_4x4_rmsErr_skelBF_58}
\end{subfigure}
~
\begin{subfigure}[b]{0.3\textwidth}
\includegraphics[width=5.1cm]{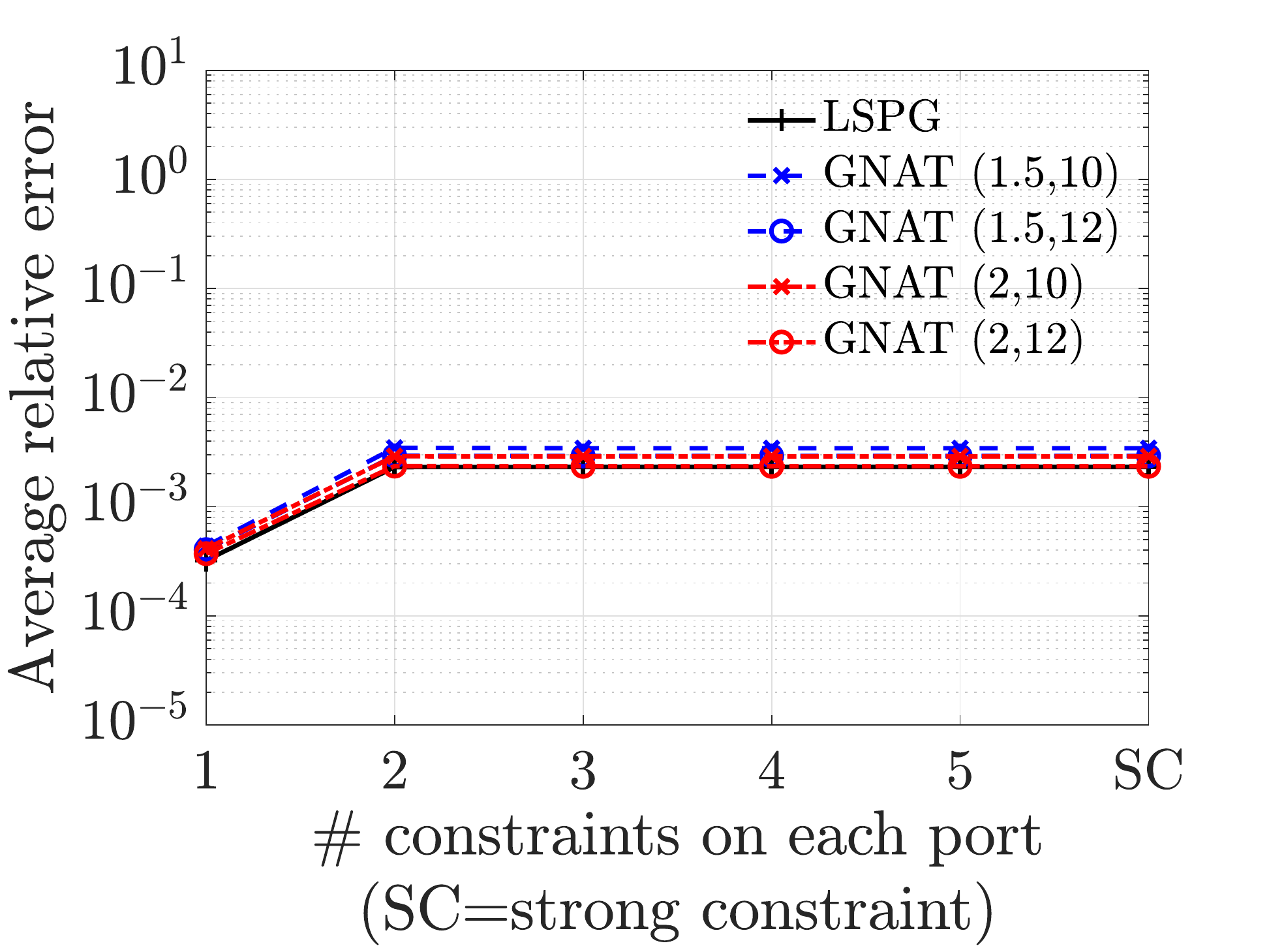}
\caption{skeleton, $\energyCriterion=10^{-8}$ on $\domaini$,
	$\energyCriterion=10^{-8}$ on $\boundaryi$}
\label{fig_ex1_4x4_rmsErr_skelBF_88}
\end{subfigure}
~
\begin{subfigure}[b]{0.3\textwidth}
\includegraphics[width=5.1cm]{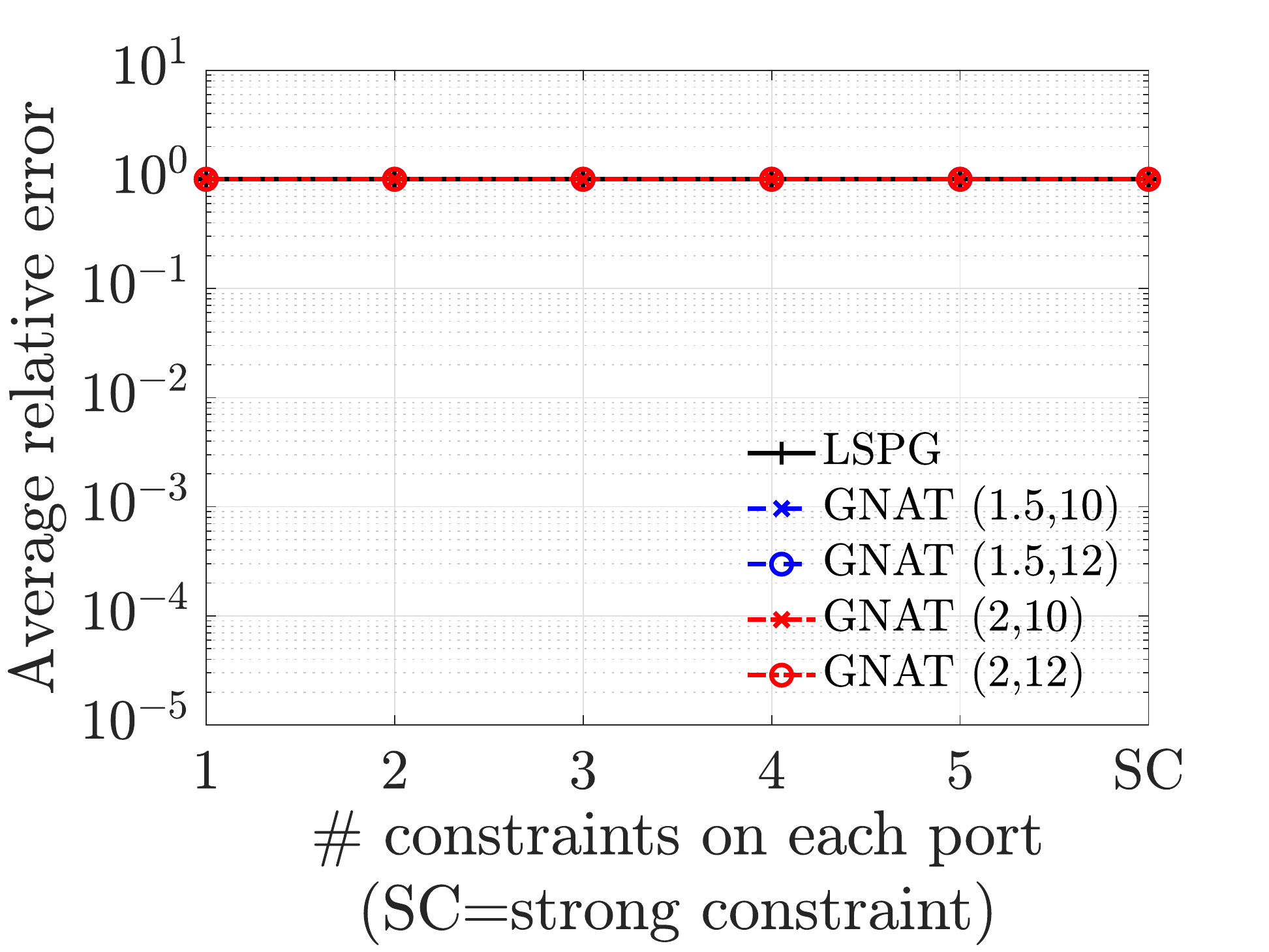}
\caption{full-interface, $\energyCriterion=10^{-5}$ on $\domaini$,
	$\energyCriterion=10^{-5}$ on $\boundaryi$}
\label{fig_ex1_4x4_rmsErr_intfBF_55}
\end{subfigure}
~
\begin{subfigure}[b]{0.3\textwidth}
\includegraphics[width=5.1cm]{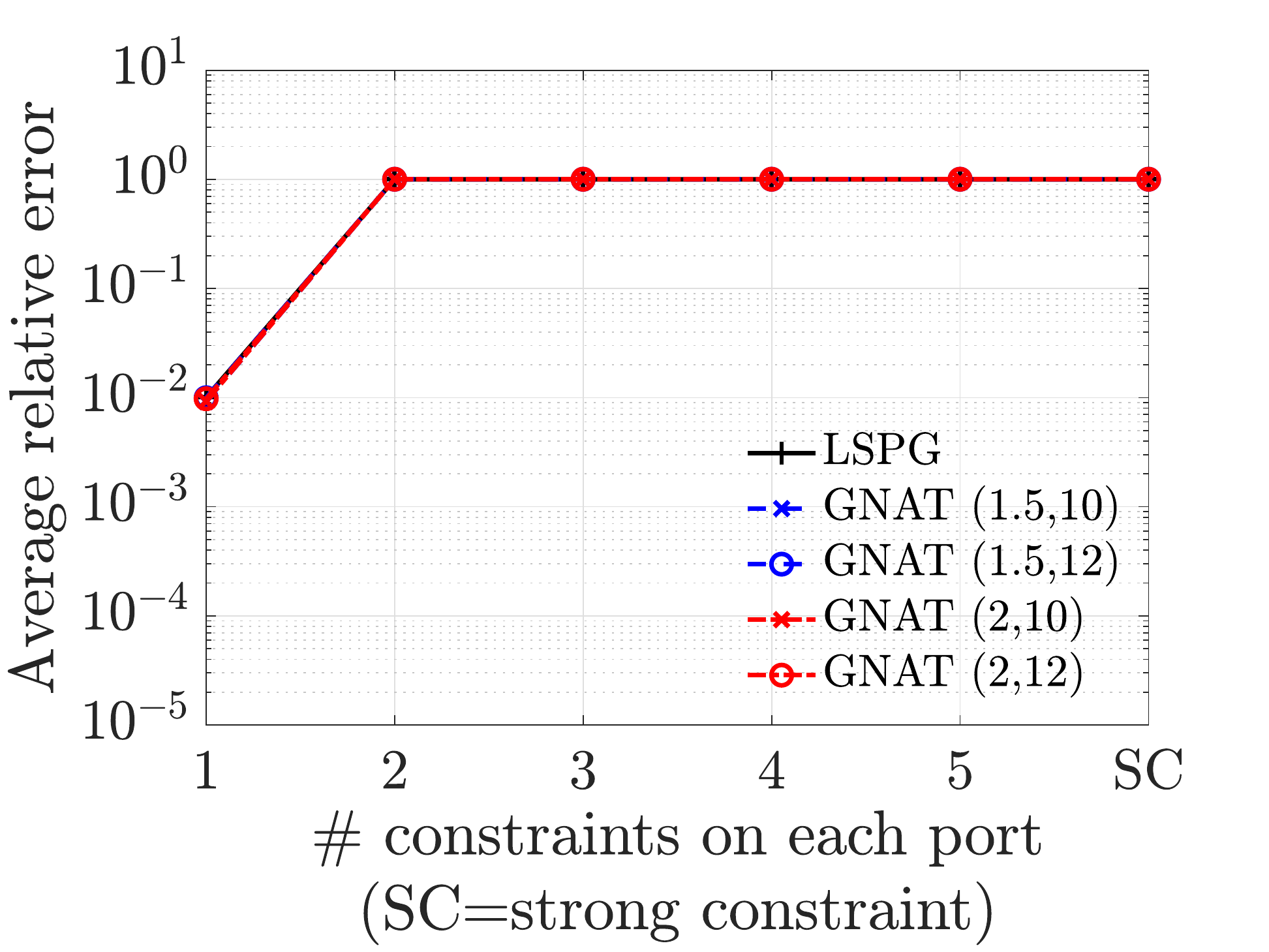}
\caption{full-interface, $\energyCriterion=10^{-5}$ on $\domaini$,
	$\energyCriterion=10^{-8}$ on $\boundaryi$}
\label{fig_ex1_4x4_rmsErr_intfBF_58}
\end{subfigure}	
~
\begin{subfigure}[b]{0.3\textwidth}
\includegraphics[width=5.1cm]{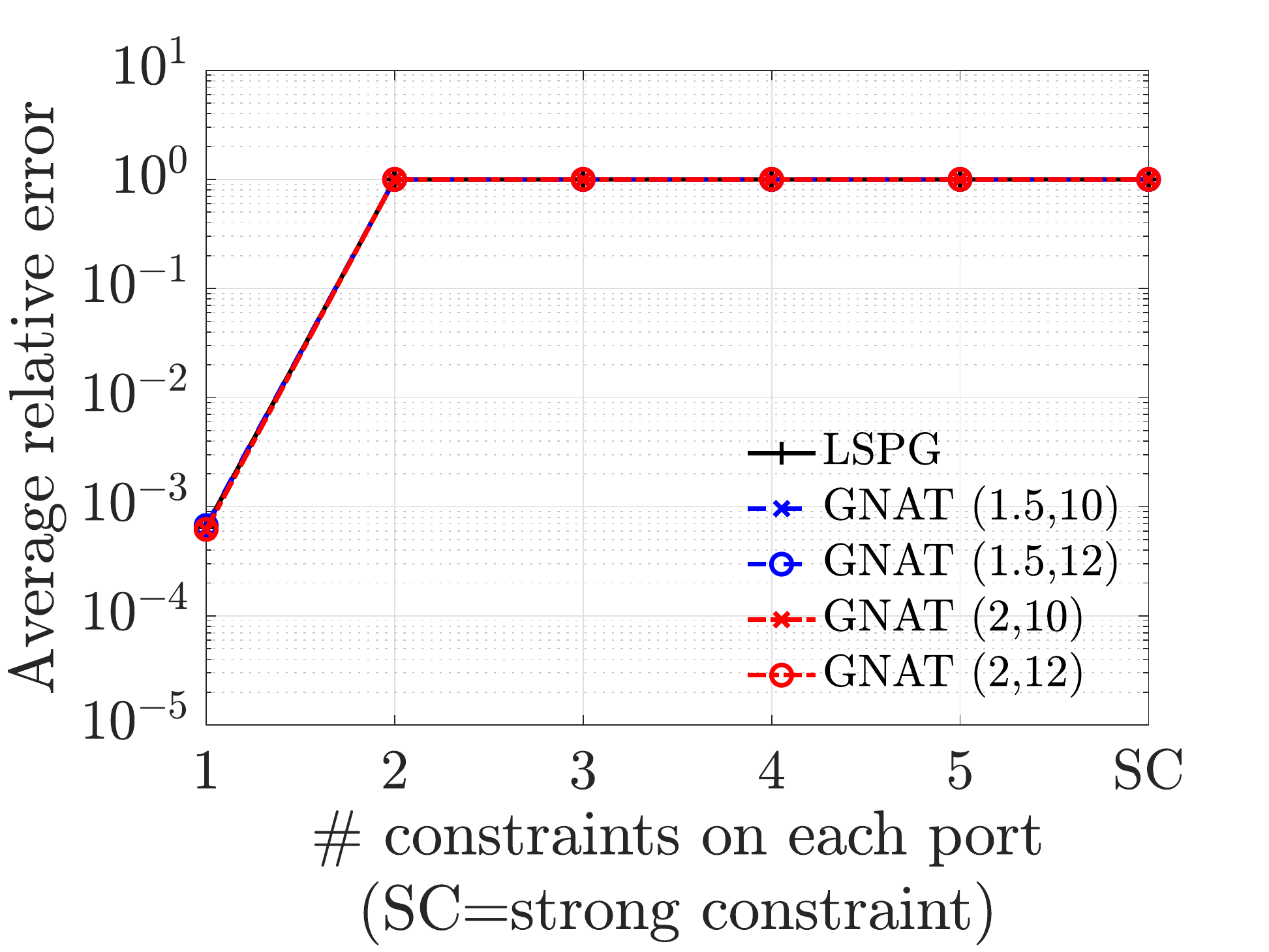}
\caption{full-interface, $\energyCriterion=10^{-8}$ on $\domaini$,
	$\energyCriterion=10^{-8}$ on $\boundaryi$}
\label{fig_ex1_4x4_rmsErr_intfBF_88}
\end{subfigure}	
~
\begin{subfigure}[b]{0.3\textwidth}
\includegraphics[width=5.1cm]{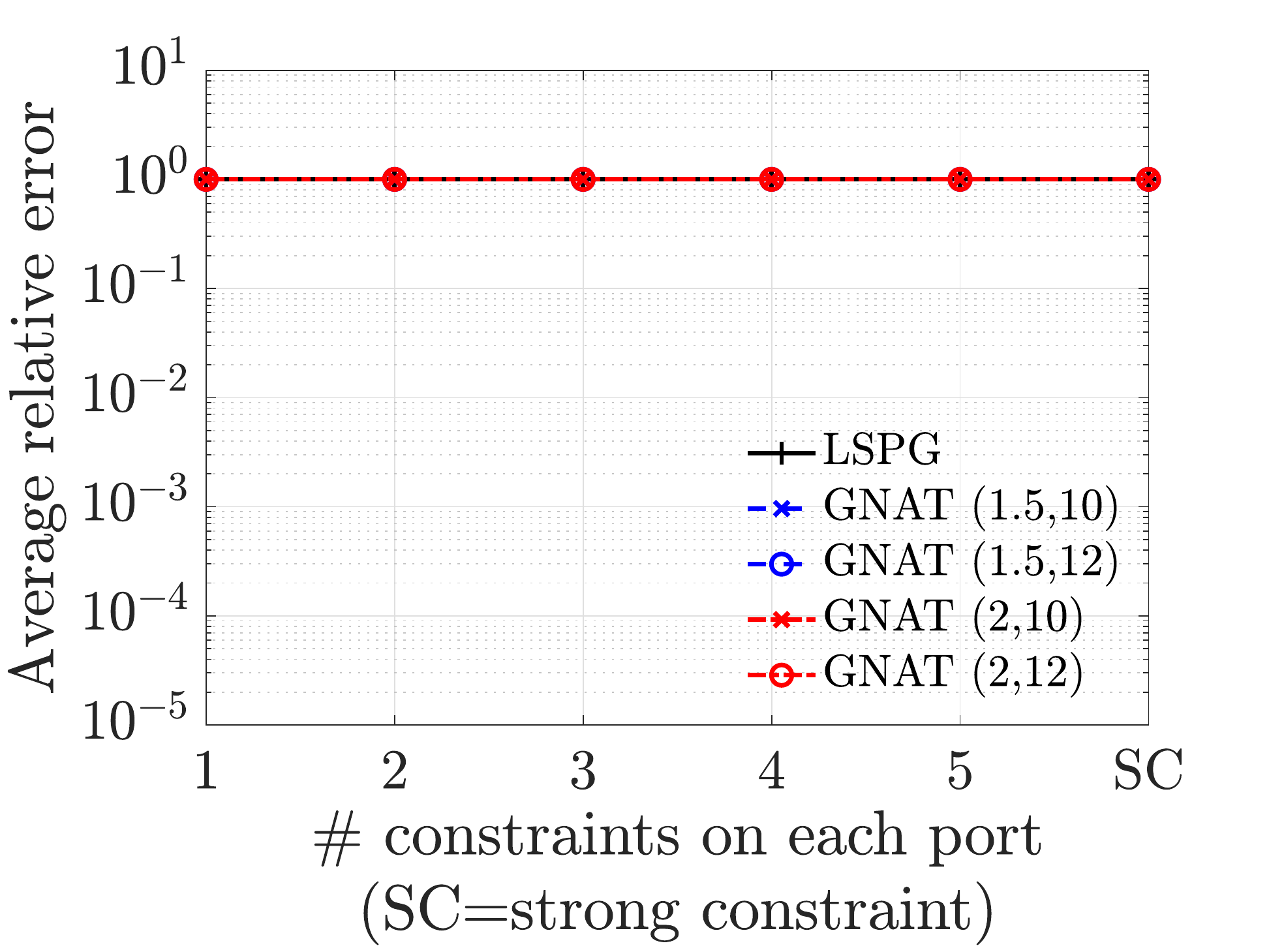}
	\caption{subdomain, $\energyCriterion=10^{-5}$}
\label{fig_ex1_4x4_rmsErr_subdomBF_5}
\end{subfigure}
~
\begin{subfigure}[b]{0.3\textwidth}
\includegraphics[width=5.1cm]{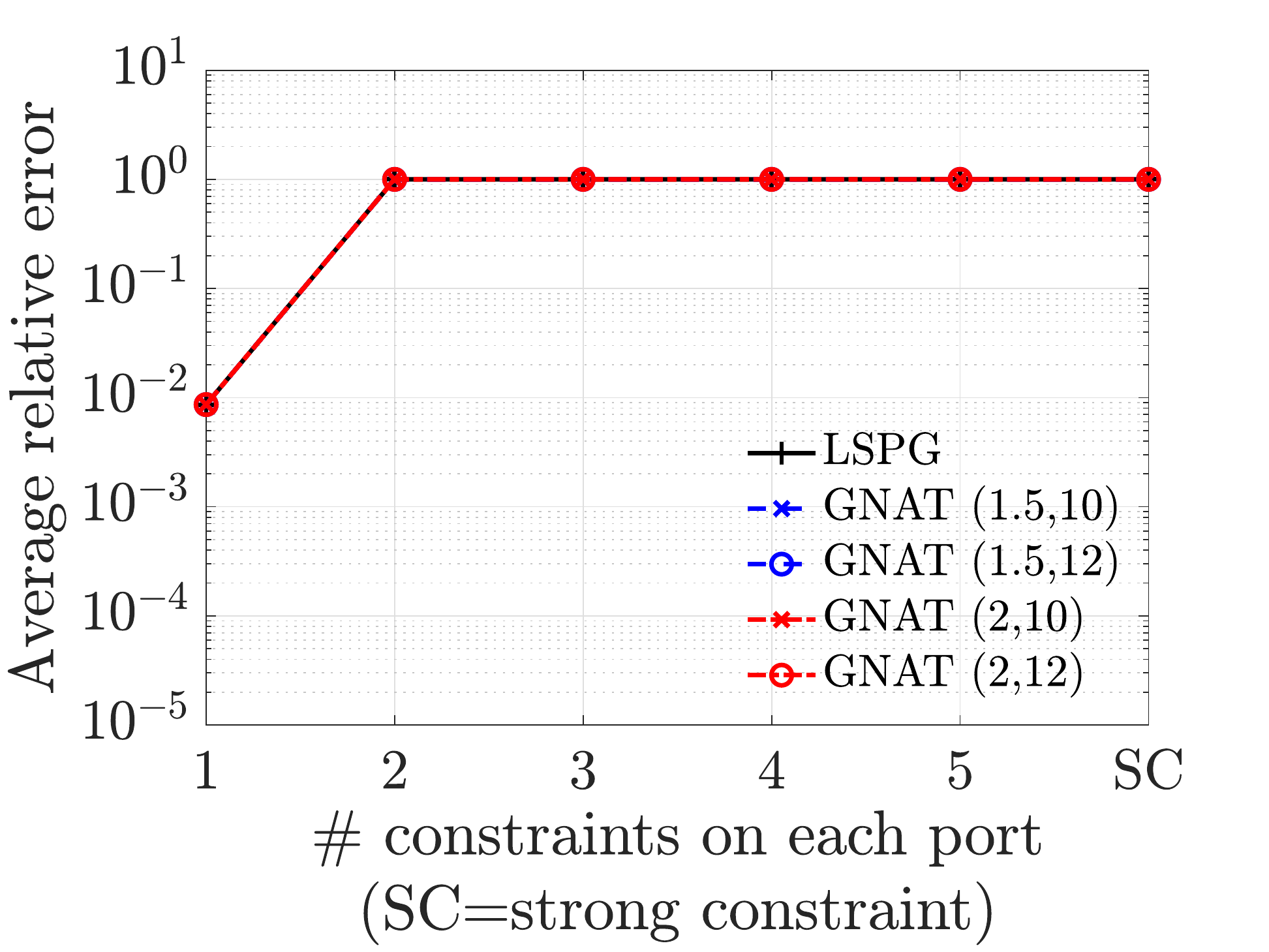}
\caption{subdomain, $\energyCriterion=10^{-7}$}
\label{fig_ex1_4x4_rmsErr_subdomBF_7}
\end{subfigure}
~
\begin{subfigure}[b]{0.3\textwidth}
\includegraphics[width=5.1cm]{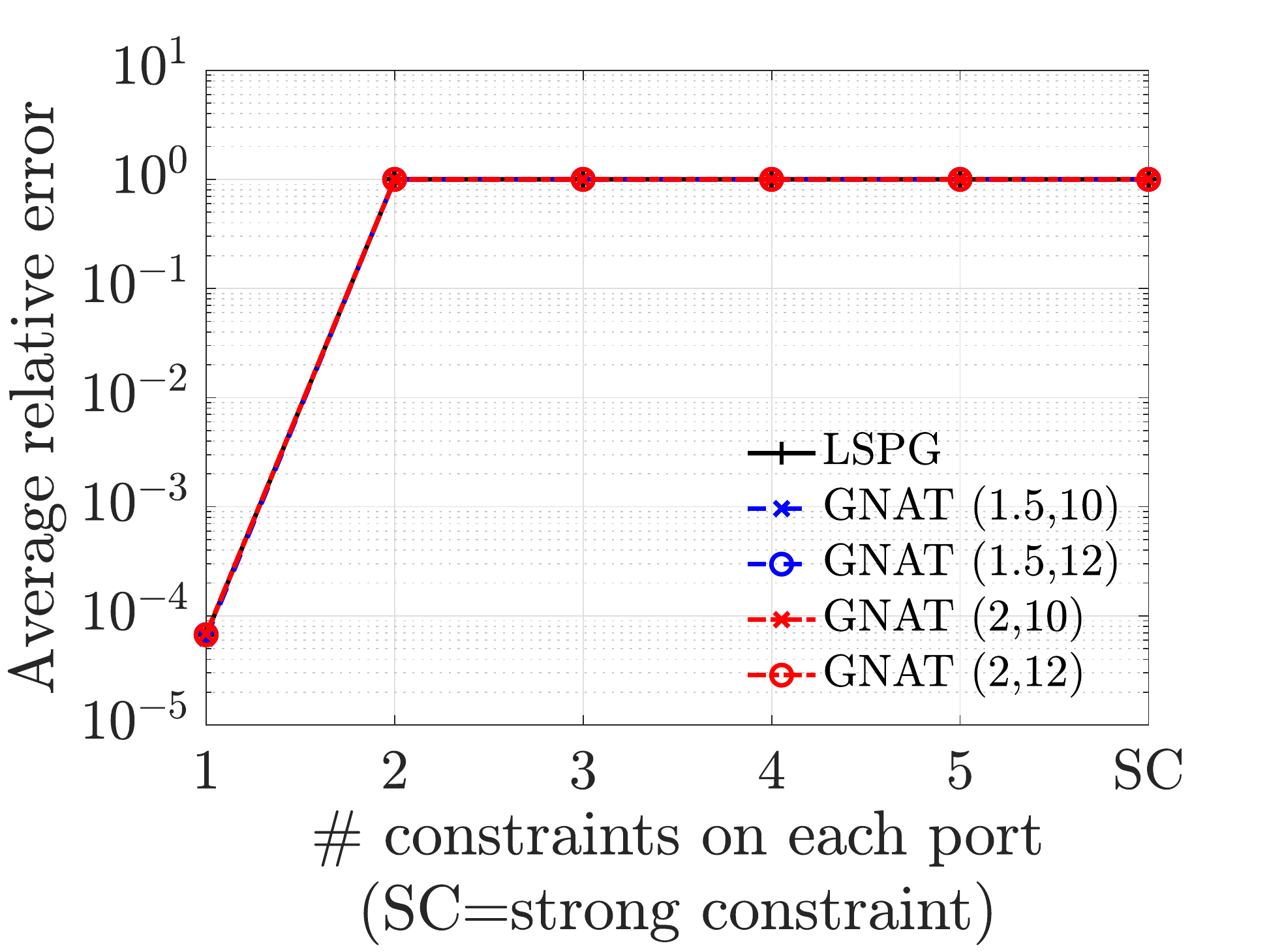}
\caption{subdomain, $\energyCriterion=10^{-9}$}
\label{fig_ex1_4x4_rmsErr_subdomBF_9}
\end{subfigure}
\caption{Heat equation, top-down training, $4\times 4$ ``fine'' configuration, average relative error (DD-LSPG and DD-GNAT) versus number of constraint per port for varying model parameters reported in Table~\ref{tab_ex1_manyOnlineComputation} (SC=strong constraints). In the legend $\text{GNAT}(x,y)$ implies the DD-GNAT model with $\nsampleArg{i}/\nrbResi=x$ and $\energyCriterion=10^{-y}$ for $\residuali$. }
\label{fig_ex1_heat44fn_rmsErr}
\end{figure}

Because assessing a given method's performance for a single instance of its
parameters does not lend insight into the model's complete error--cost
performance tradeoff, this section subjects each of the proposed ROMs to a
parameter study wherein each model parameter is varied between limits
specified in Table~\ref{tab_ex1_manyOnlineComputation}. \CH{Here, we choose several \textit{high} energy criteria for both the interior and boundary bases to obtain sufficient number of bases over all subdomains, thus ensure good solution accuracy. All other parameters are chosen straightforwardly.} 

For each weak compatibility constraint case, 
 we generate five sets of random matrices
$\testFunctionArg{j}$, $j=1,\ldots,\nports$, which in turn yields five different sets of
constraint matrices
$\constraintMatROMi$, $i=1\ldots, \nsubdomains$ as described in Section
\ref{subsect_constraint_types}. We use each
set to perform one ROM simulation, and record the associated timing and relative
error of that simulation. 
The reported timing and relative error comprises the 
average obtained over these five simulations. For the strong-constraint case, we perform only one simulation as $\constraintMatROMi=\constraintMatArg{i}$ defined uniquely. The recorded wall time for any parallel step is set to the largest wall time
incurred by any subdomain, while that for any serial step is simply set to
overall wall time incurred by the step (see Algorithms \ref{alg_psi_assembly}--\ref{alg_psi_solving}).

From these results, we then construct a Pareto front to characterize the 
error--cost for each method. This Pareto front is
characterized by the collection of method-parameter instances that yield simulation
results that are not dominated in both relative error or wall time by any other
method-parameter instance.
Figure~\ref{fig_ex1_all_pareto}
reports these Pareto fronts for the three configurations considered.
Figure~\ref{fig_ex1_heat44fn_rmsErr} plots the average relative error versus
number of constraints per port for the $4\times 4$ ``fine'' configuration.

We first analyze the solve wall time, and thus consider the subfigures
in the rightmost column of Figure~\ref{fig_ex1_all_pareto}. These plots yield
the following observations:
\begin{enumerate}[label=(\roman*)]
	\item For a given basis type, the solve costs for DD-LSPG and DD-GNAT are nearly
		the same, which is sensible because they yield SQP systems of the same
		dimension and structure.

	\item The $4\times 4$ ``fine'' configuration yields a costlier 
		solve among the three
		configurations, which is sensible because 
		the solve operation count is cubically proportional to $\nsubdomains$
		assuming fixed basis dimensions (see 
		Table~
		\ref{tab_DDGNATprduM_solving_cost}).

	\item For both DD-LSPG and DD-GNAT, the solve times for both the $2\times 2$
		``coarse'' and $2\times 2$ ``fine'' configurations was roughly the same,
		which is sensible given that these configurations are characterized by the
		similar basis dimensions and the same number of subdomains (see Table
		\ref{tab_DDGNATprduM_solving_cost}). 

  \item For a fixed error, the port basis type incurred the largest solve wall
    time compared with the other three \YC{basis types}; this is sensible
    because---on average---it has a larger basis dimension compared with the
    other three types. The skeleton, full-interface and subdomain types yield
    roughly the same solve cost for a fixed error.
\end{enumerate}

We now analyze assembly wall time\footnote{\CH{Note for all numerical experiments in this paper, the assembly stage is performed in a serial manner (not parallel) for simple implementation. However, the timing on each subdomain is recorded for properly post-assessment.}} and thus consider the middle column of Figure~\ref{fig_ex1_all_pareto}.
These figures illustrate the following trends:
\begin{enumerate}[label=(\roman*)]
	
	\item \CH{DD-LSPG assembly wall time of four basis types are almost similar, this is due to dominated time/cost of computing the residuals, Jacobians and Hessians as listed on Algorithm~\ref{alg_psi_assembly}.}
	
	\item For DD-LSPG, the $2\times 2$ ``fine'' configuration yields the largest
		assembly wall time because the number of DOFs per subdomain interior
		$\ndofInteriori$ and subdomain boundary $\ndofBoundaryi$ is the largest in this
		case, and $\nsampleResi=\nsampleInteriori=\ndofi$ and $\nsampleBoundaryi=
		\ndofBoundaryi$ for DD-LSPG (see Table~\ref{tab_ex1_DDFEMconfig_params}).
		On the other hand, the assembly wall time for DD-LSPG is similar in the
		$2\times 2$ ``coarse'' and $4\times 4$ ``fine'' configurations due to
		these numbers being roughly the same.

	\item Among 3 configurations, the relative performance improvement of
		DD-GNAT over DD-LSPG (with regard to the assembly wall time) is the smallest in the 
		$4\times 4$ ``fine'' configuration, while it is the largest for the 
		$2\times 2$ ``fine'' configuration. This occurs because when each subdomain becomes
		smaller in size with a modest number of DOFs per subdomain, the subdomain
		exhibiting the strongest solution nonlinearity requires a relatively large
		number of sample points
		(see
		Figure~\ref{fig_ex1_heat22fn_sol_sampleMesh_Omega}). 

	\item Conversely, the relative performance improvement of DD-GNAT over
		DD-LSPG (with regard to the assembly wall time) is the largest for the $2\times 2$ ``fine''
		configuration because this case corresponds to the largest number of
		degrees of freedom per subdomain interior and boundary.

	\item For a fixed accuracy, the port bases (blue curves) almost always
		yielded a higher assembly wall time than the other basis types; as before,
		this can be attributed to the larger basis dimensions that typically
		accompany this basis type.
\end{enumerate}

Next, we consider the overall wall time, and turn attention to the \YC{three}
subfigures in the leftmost column of Figure~\ref{fig_ex1_all_pareto}. From the
center and rightmost columns, we see that the assembly wall time dominates the
solve wall time in this example; thus, the assembly wall-time behavior is most
closely reflected in the overall wall-time performance of the different
methods. We see that the DD-GNAT methods can realize $>100\times$ wall-time speedups
relative to the FOM in this case, with DD-LSPG yielding only modest wall-time
speedups, which are $<10\times$ in both the $2\times 2$ ``coarse'' and ``fine''
configurations.

Figure~\ref{fig_ex1_heat44fn_rmsErr} shows the average relative error as a
function of the number of constraints per port with parameters reported in
Table~\ref{tab_ex1_manyOnlineComputation} for the $4\times 4$ ``fine''
configuration.  The subfigures in the first two rows of
Figure~\ref{fig_ex1_heat44fn_rmsErr} imply that strong compatibility
constraints yield better accuracy than weak compatibility for both port and
(generally impractical) skeleton basis types.  This is sensible, as these basis types ensure that
neighboring components have \textit{compatible} bases on shared ports, so weak
compatibility constraints lead to no benefit; see Remark \ref{rem:globalSol}.
On the other hand, the subfigures in the two last rows of
Figure~\ref{fig_ex1_heat44fn_rmsErr} show that weak constraint case with only
one constraint per port yield the best accuracy for full-interface and
full-subdomain basis types. As discussed in Remark \ref{rem:globalSol}, this result is expected because  neighboring components generally have
\textit{incompatible} bases on shared ports for these basis types. The next
section will lend additional insight into the behavior of the full-interface and
full-subdomain basis types.

\subsubsection{Full-interface basis: effect of weak and strong constraints}

\begin{figure}[h!]
\centering
\begin{subfigure}[b]{0.3\textwidth}
\includegraphics[width=5.5cm]{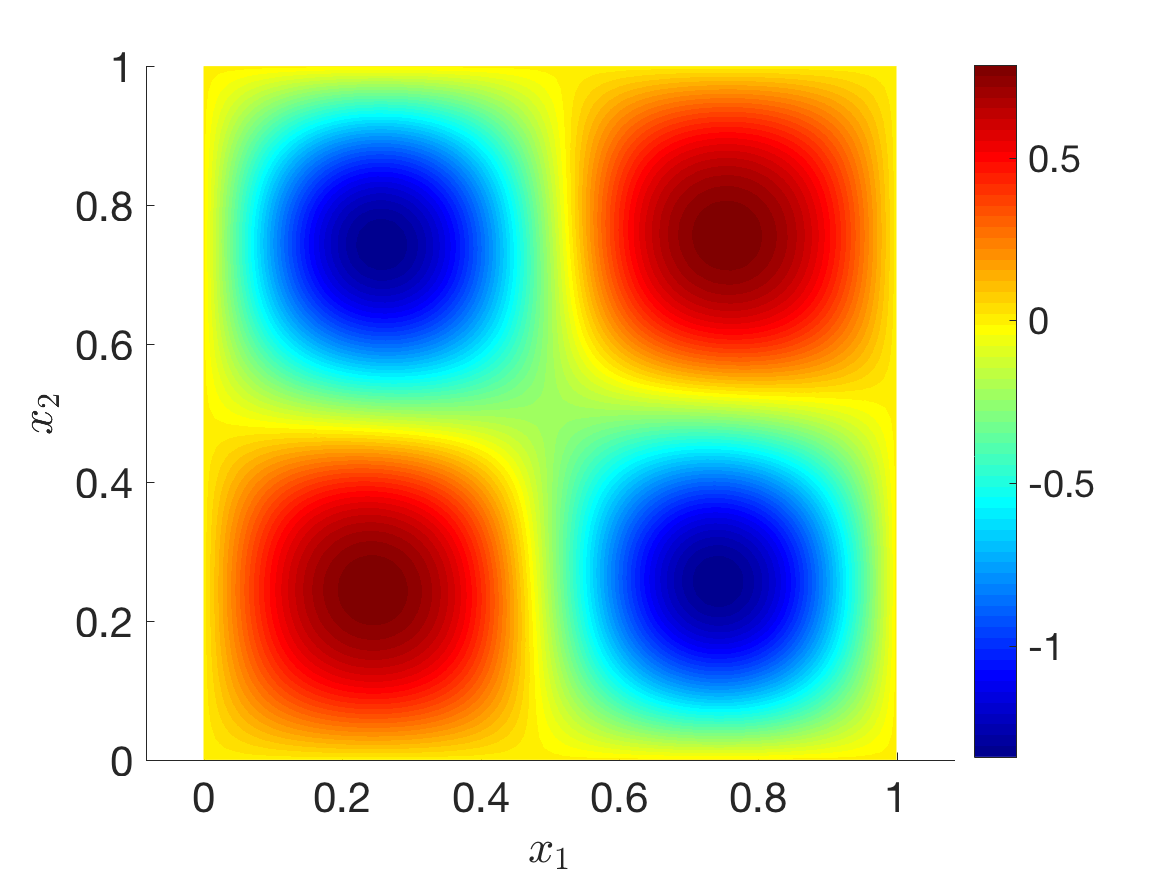}
\caption{DD-LSPG sol. on $\Omega$}
\label{fig_ex1_heat44fn_sol_DDROM_Omega_1constr}
\end{subfigure}
~
\begin{subfigure}[b]{0.3\textwidth}
\includegraphics[width=5.5cm]{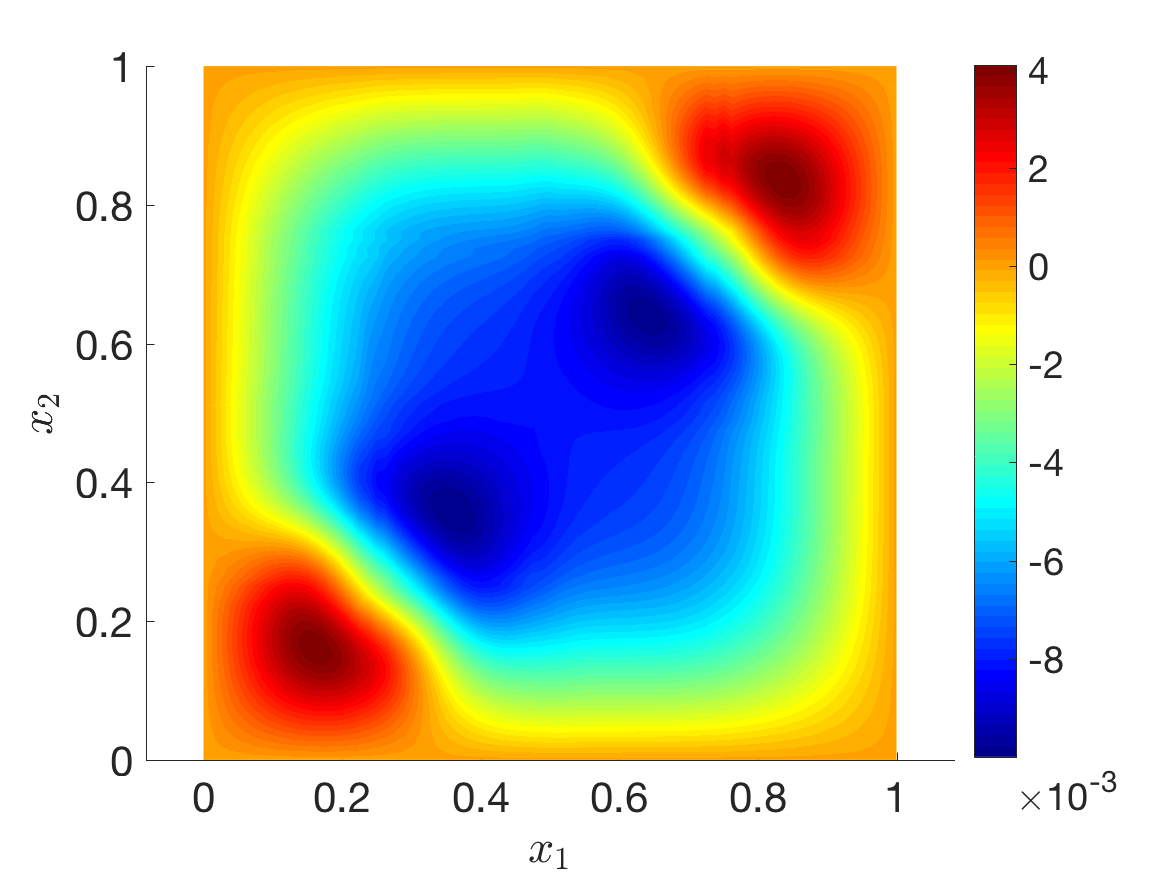}
\caption{DD-LSPG error on $\Omega$}
\label{fig_ex1_heat44fn_sol_DDROMerr_Omega_1constr}
\end{subfigure}
~
\begin{subfigure}[b]{0.3\textwidth}
\includegraphics[width=5.5cm]{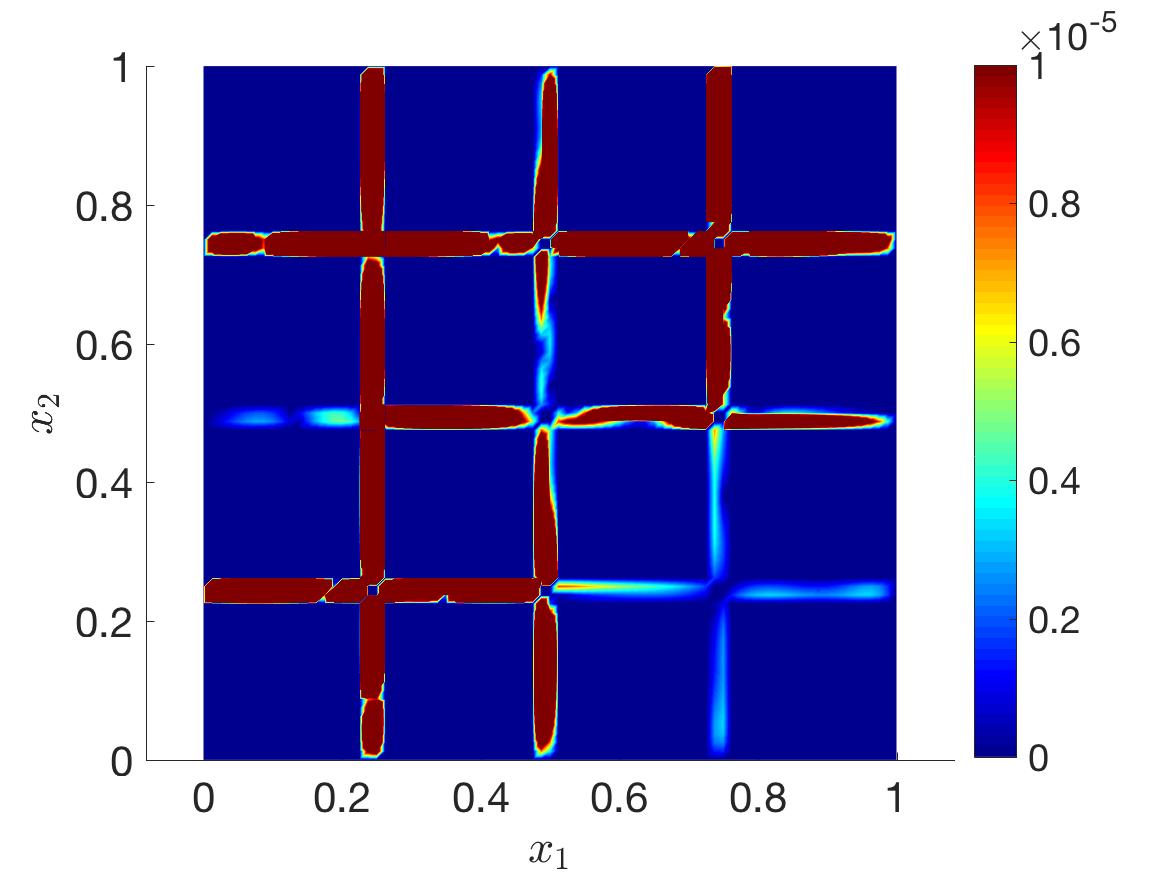}
\caption{Absolute discrepancy on $\Omega$}
\label{fig_ex1_heat44fn_sol_DDROM_discrepancy_1constr}
\end{subfigure}
~
\begin{subfigure}[b]{0.3\textwidth}
\includegraphics[width=5.5cm]{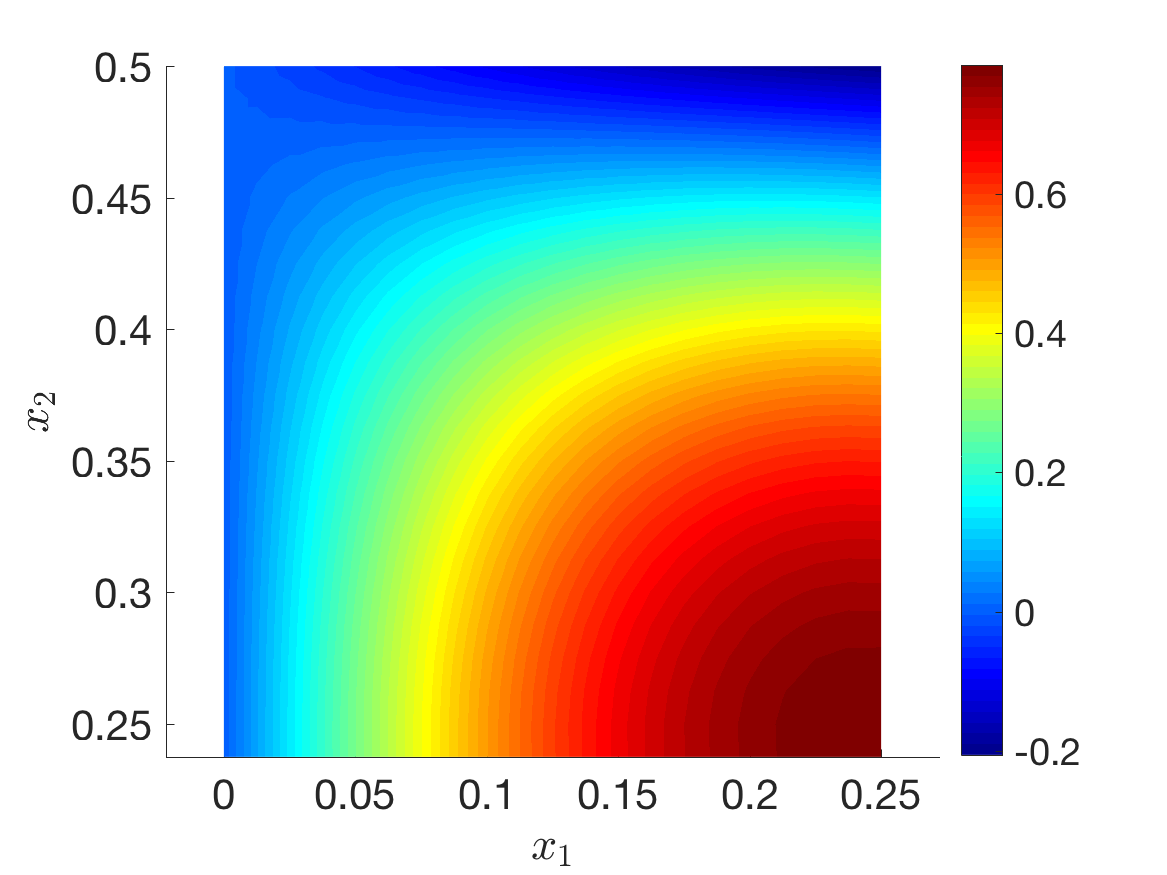}
\caption{DD-LSPG sol. on $\Omega_2$}
\label{fig_ex1_heat44fn_sol_DDROM_O2_1constr}
\end{subfigure}
~
\begin{subfigure}[b]{0.3\textwidth}
\includegraphics[width=5.5cm]{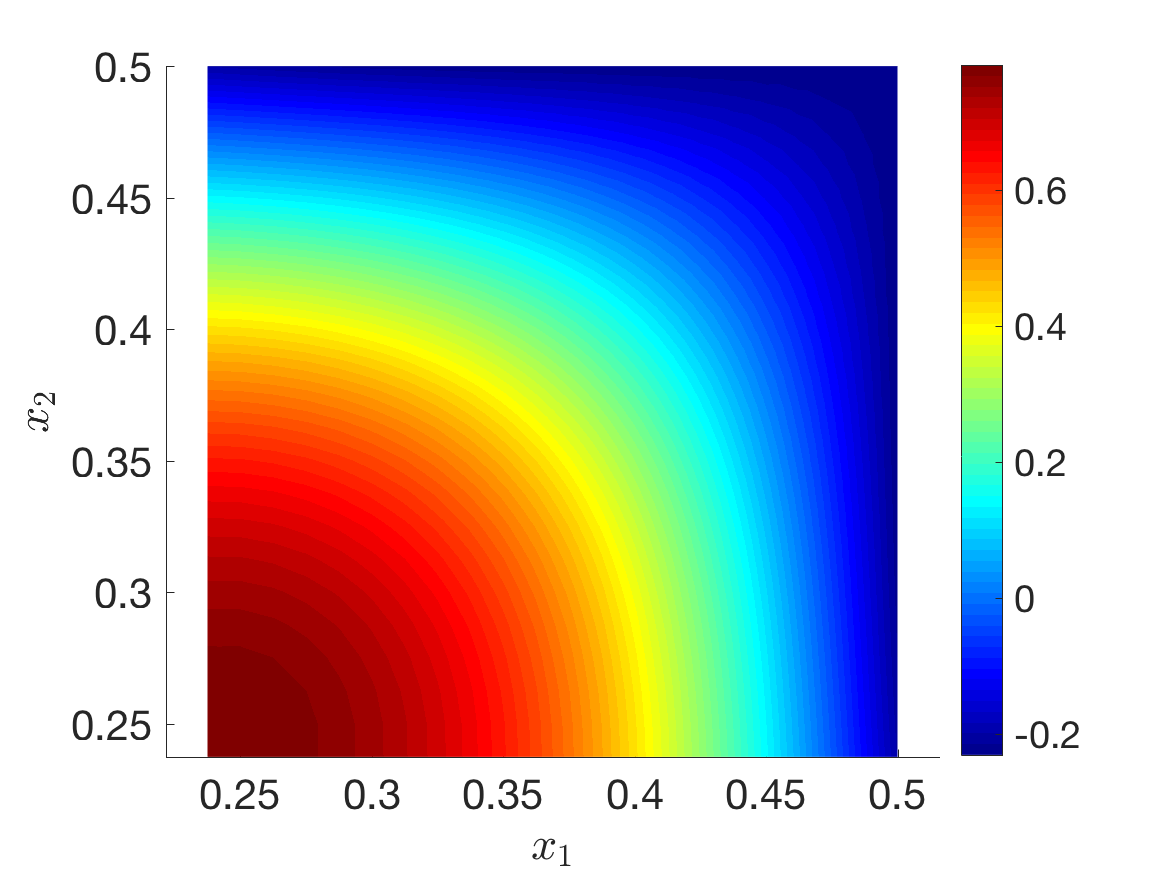}
\caption{DD-LSPG sol. on $\Omega_{6}$}
\label{fig_ex1_heat44fn_sol_DDROM_O6_1constr}
\end{subfigure}
~
\begin{subfigure}[b]{0.3\textwidth}
\includegraphics[width=5.5cm]{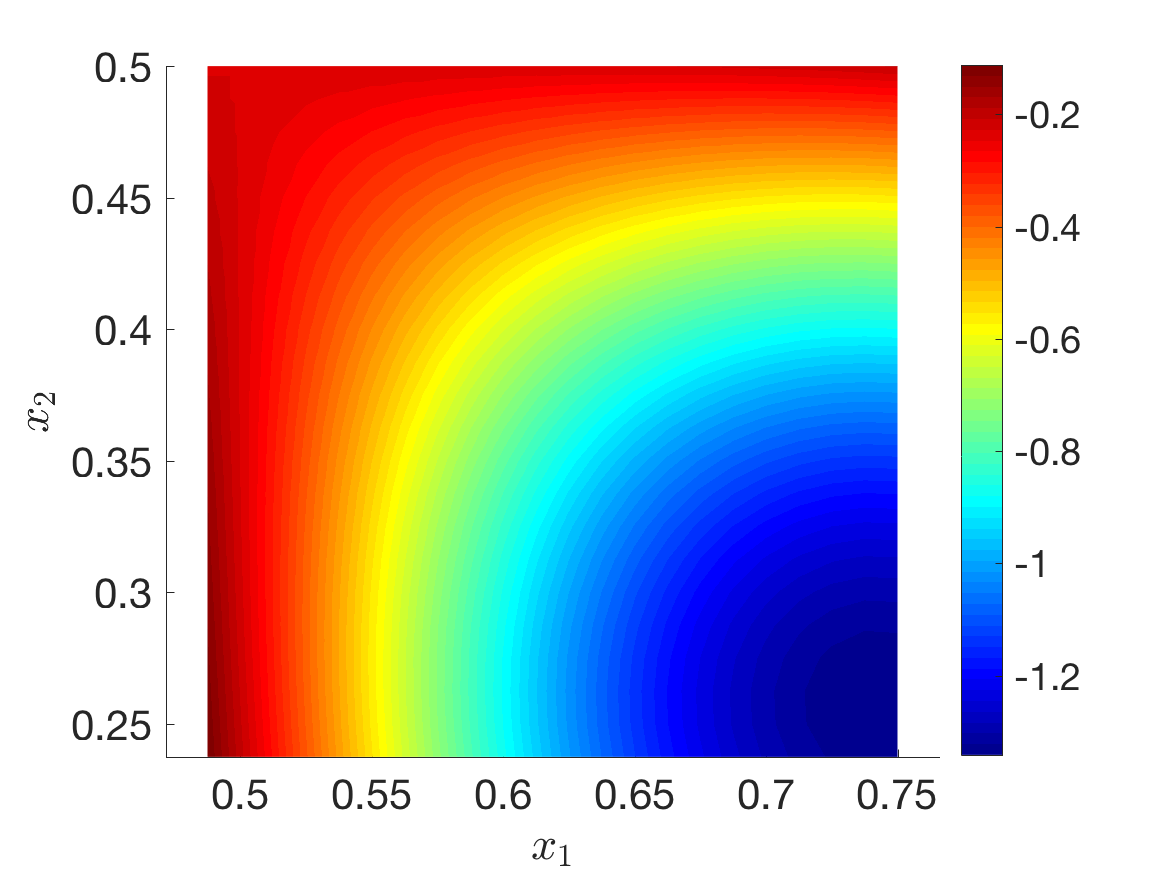}
\caption{DD-LSPG sol. on $\Omega_{10}$}
\label{fig_ex1_heat44fn_sol_DDROM_O10_1constr}
\end{subfigure}	
~\begin{subfigure}[b]{0.3\textwidth}
	\includegraphics[width=5.5cm]{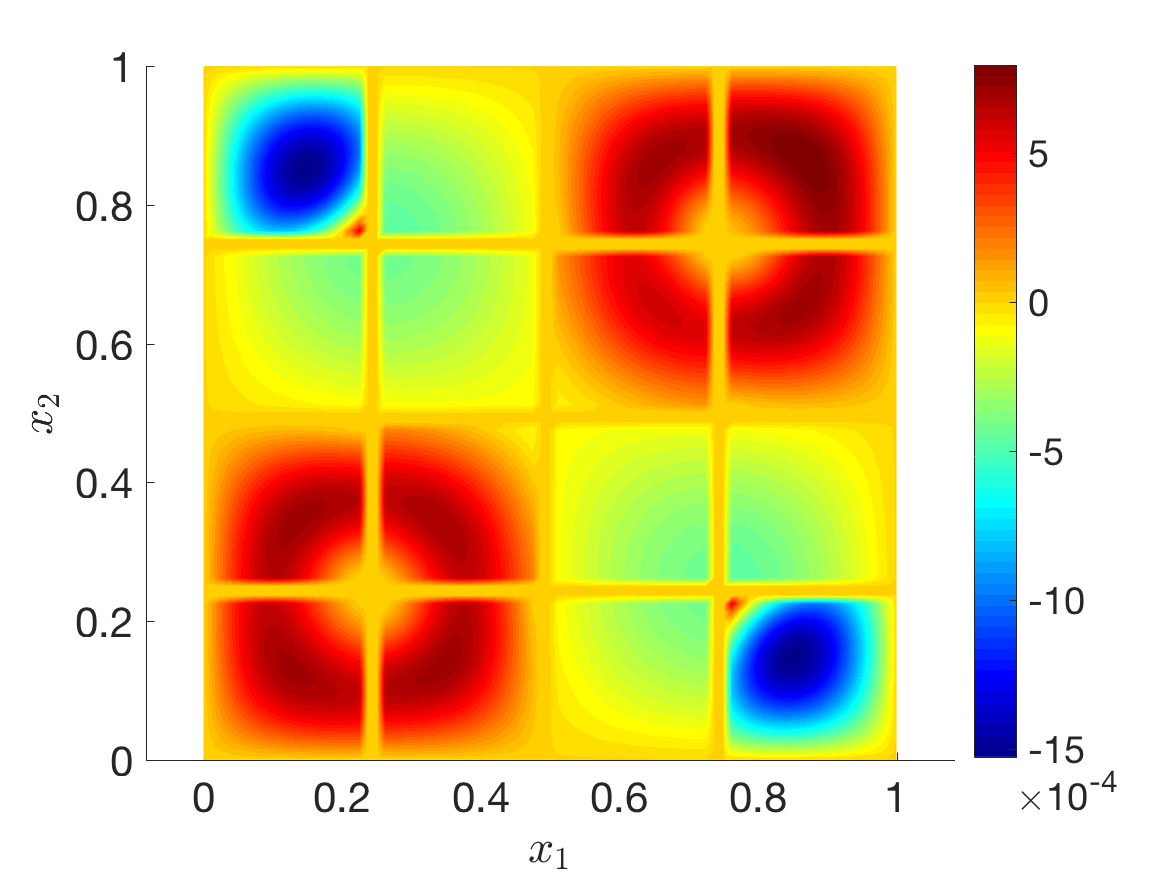}
	\caption{DD-LSPG sol. on $\Omega$}
	\label{fig_ex1_heat44fn_sol_DDROM_Omega_4constr}
\end{subfigure}
~
\begin{subfigure}[b]{0.3\textwidth}
	\includegraphics[width=5.5cm]{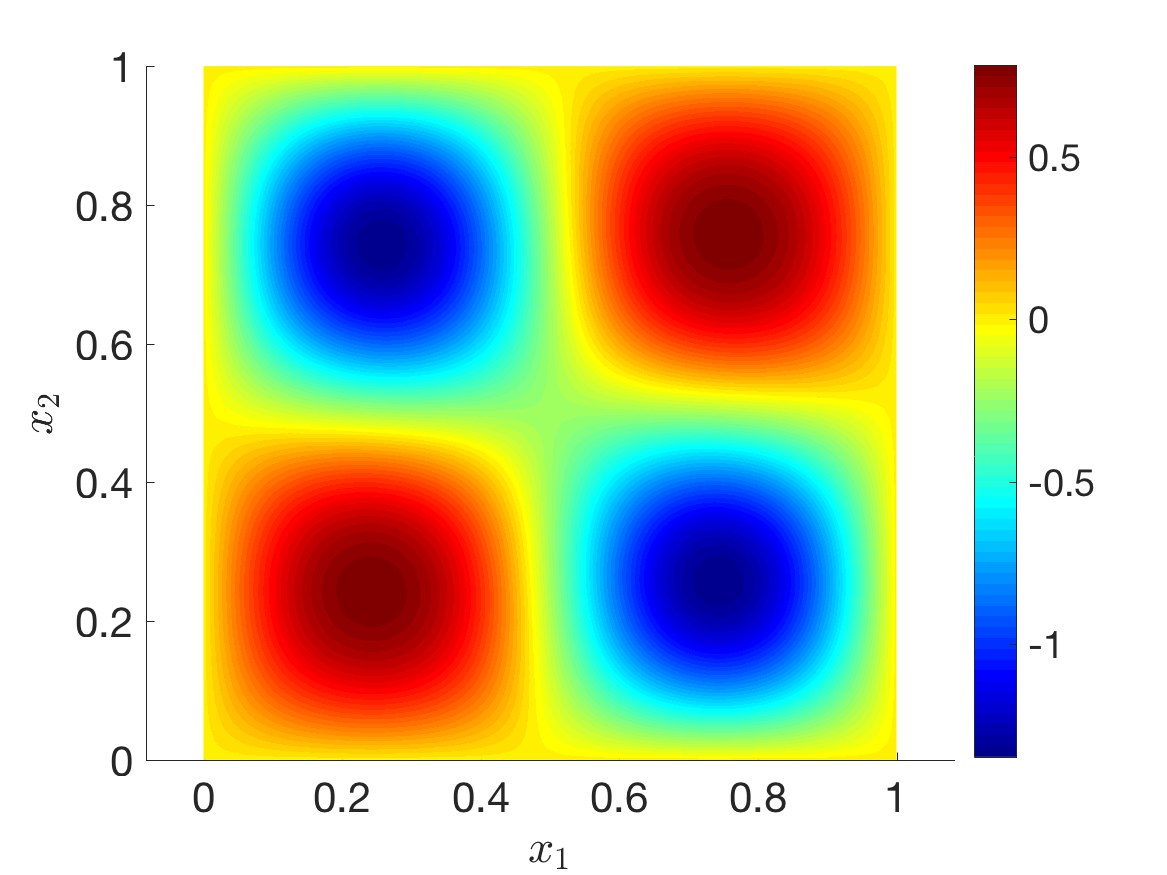}
	\caption{DD-LSPG error on $\Omega$}
	\label{fig_ex1_heat44fn_sol_DDROMerr_Omega_4constr}
\end{subfigure}
~
\begin{subfigure}[b]{0.3\textwidth}
	\includegraphics[width=5.5cm]{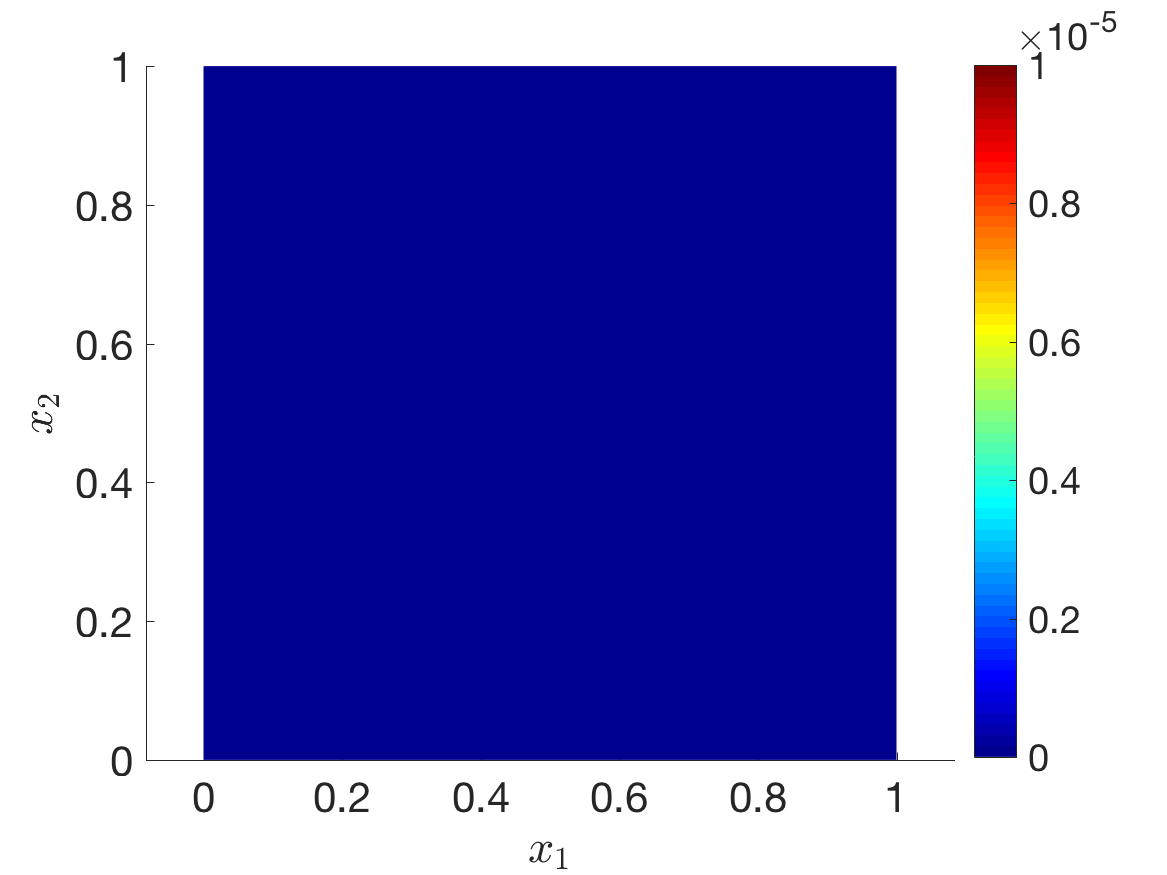}
	\caption{Absolute discrepancy on $\Omega$}
	\label{fig_ex1_heat44fn_sol_DDROM_discrepancy_4constr}
\end{subfigure}
~
\begin{subfigure}[b]{0.3\textwidth}
	\includegraphics[width=5.5cm]{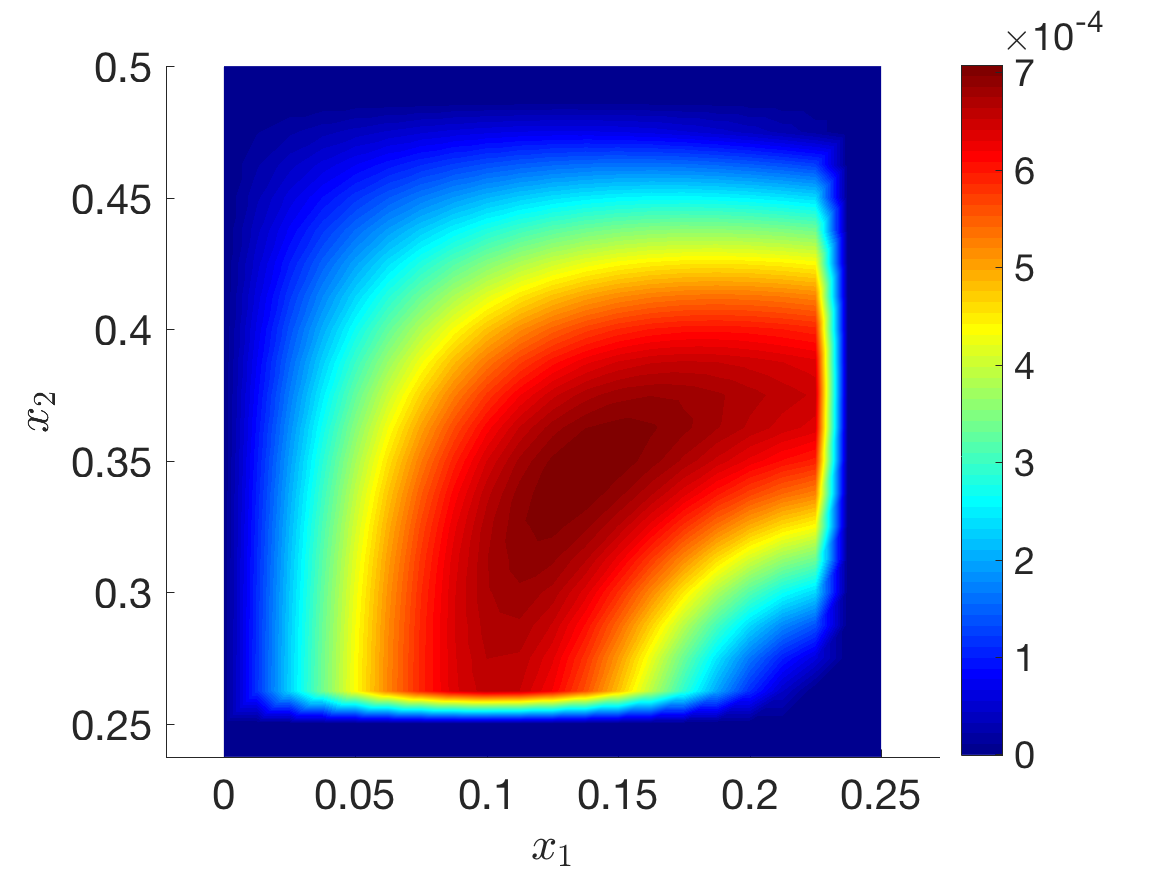}
	\caption{DD-LSPG sol. on $\Omega_2$}
	\label{fig_ex1_heat44fn_sol_DDROM_O2_4constr}
\end{subfigure}
~
\begin{subfigure}[b]{0.3\textwidth}
	\includegraphics[width=5.5cm]{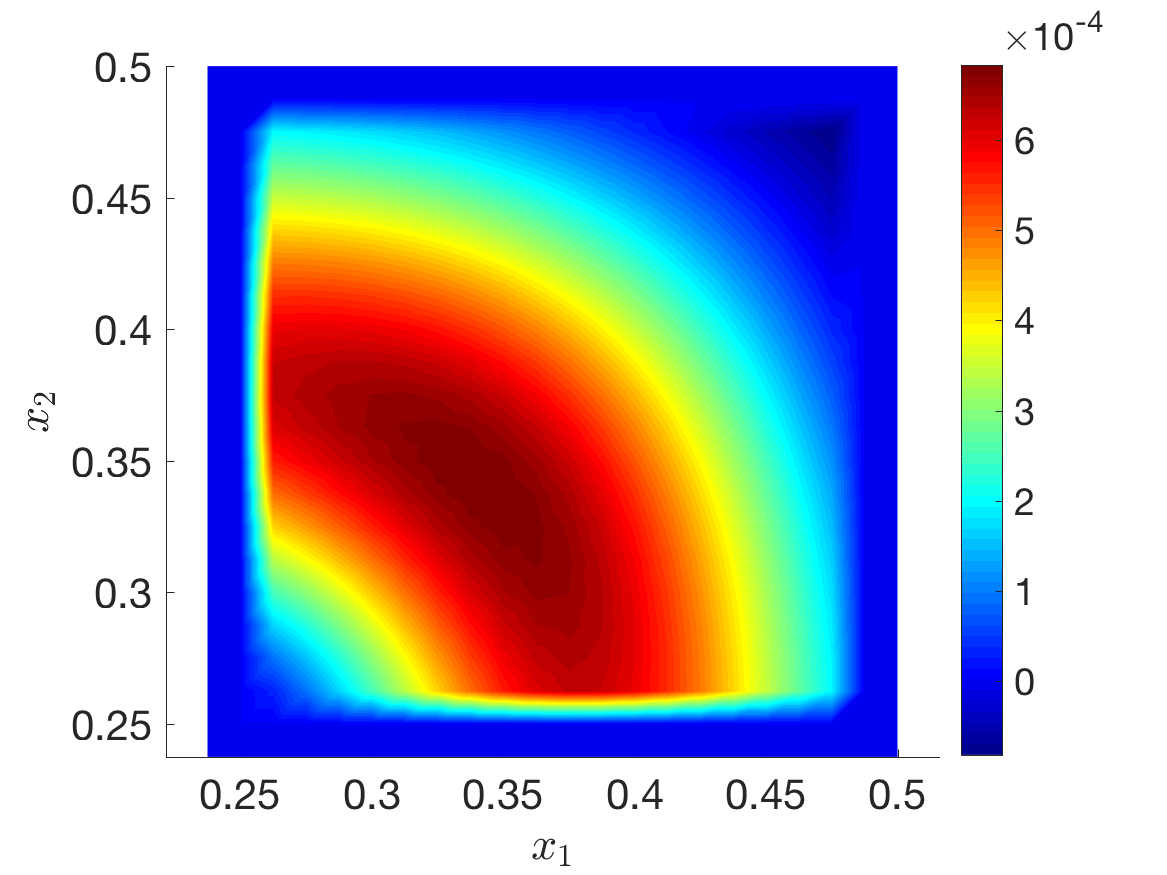}
	\caption{DD-LSPG sol. on $\Omega_{6}$}
	\label{fig_ex1_heat44fn_sol_DDROM_O6_4constr}
\end{subfigure}
~
\begin{subfigure}[b]{0.3\textwidth}
	\includegraphics[width=5.5cm]{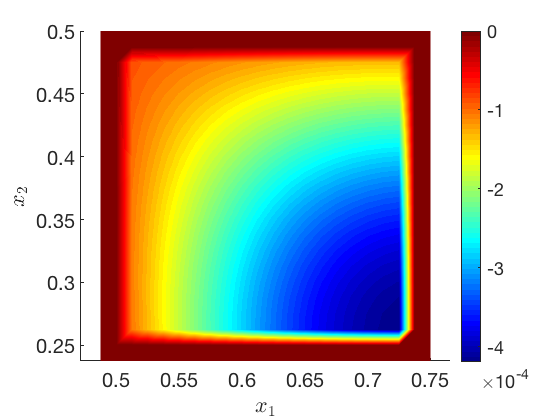}
	\caption{DD-LSPG sol. on $\Omega_{10}$}
	\label{fig_ex1_heat44fn_sol_DDROM_O10_4constr}
\end{subfigure}	
\caption{Heat equation, top-down training, $4\times 4$ ``fine'' configuration, solutions visualization on $\domain$ with full-interface bases and weak constraints: one constraint per port (top 2 rows) versus four constraints per port (last 2 rows). } \label{fig_ex1_sol_vis_1constr}
\end{figure}

For illustrative purposes, we investigate further the effects of weak versus
strong compatibility
constraints for the full-interface basis type \CH{(i.e., \textit{incompatible} bases)}. Due to its similarity,
only full-interface bases is discussed here. We consider the parameters for
DD-LSPG
simulation as follows: $4\times 4$ ``fine'' configuration,
$\paramComp=(5.005,5.005)$, full-interface bases, $\energyCriterion= 1-10^{-5}$ on $\domaini$, 
$\energyCriterion = 1-10^{-8}$ on $\boundaryi$, and weak constraint cases with one and four constraints per port. Figures~\ref{fig_ex1_sol_vis_1constr} visualizes the corresponding solutions of these
two cases. These figures verify visually our observations in the previous section and the comments in Remark \ref{rem:globalSol}: enforcing weak compatibility with only one constraint per port yields better global solutions despite a larger discrepancy in the solution computed by neighboring sudomains on the interface
(Figure~\ref{fig_ex1_heat44fn_sol_DDROM_Omega_1constr}--\ref{fig_ex1_heat44fn_sol_DDROM_O10_1constr}),
while enforcing
additional interface constraints (i.e, four constraints per port in this case)
imposes such a strict compatibility condition that the resulting interface
solution is simply the trivial solution, yielding significant overall errors
despite negligible discrepancies in the solutions computed on neighboring
subdomains (Figure~\ref{fig_ex1_heat44fn_sol_DDROM_Omega_4constr}--\ref{fig_ex1_heat44fn_sol_DDROM_O10_4constr}).

\subsubsection{Subdomain (or bottom-up) training for 4x4 ``fine'' configuration} \label{sect_subdom_training}

\textbf{Reproductive test}

\begin{figure}[h!]
\centering
\begin{subfigure}[b]{0.3\textwidth}
	\includegraphics[width=5.5cm]{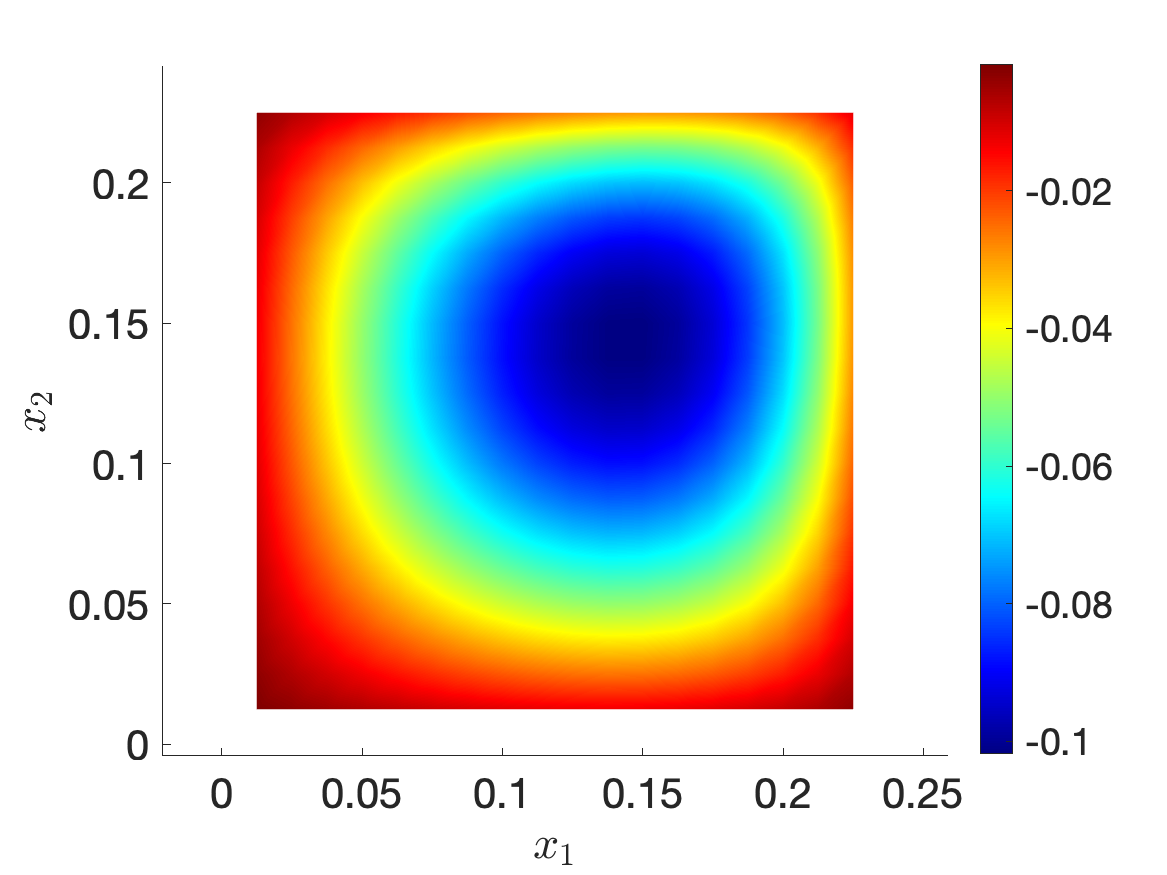}
	\caption{$\rbInteriorArg{1}$ on $\domainArg{1}$}
	\label{fig_ex1_heat44fn_btup_O1_PhiO1}
\end{subfigure}
~
\begin{subfigure}[b]{0.3\textwidth}
	\includegraphics[width=5.5cm]{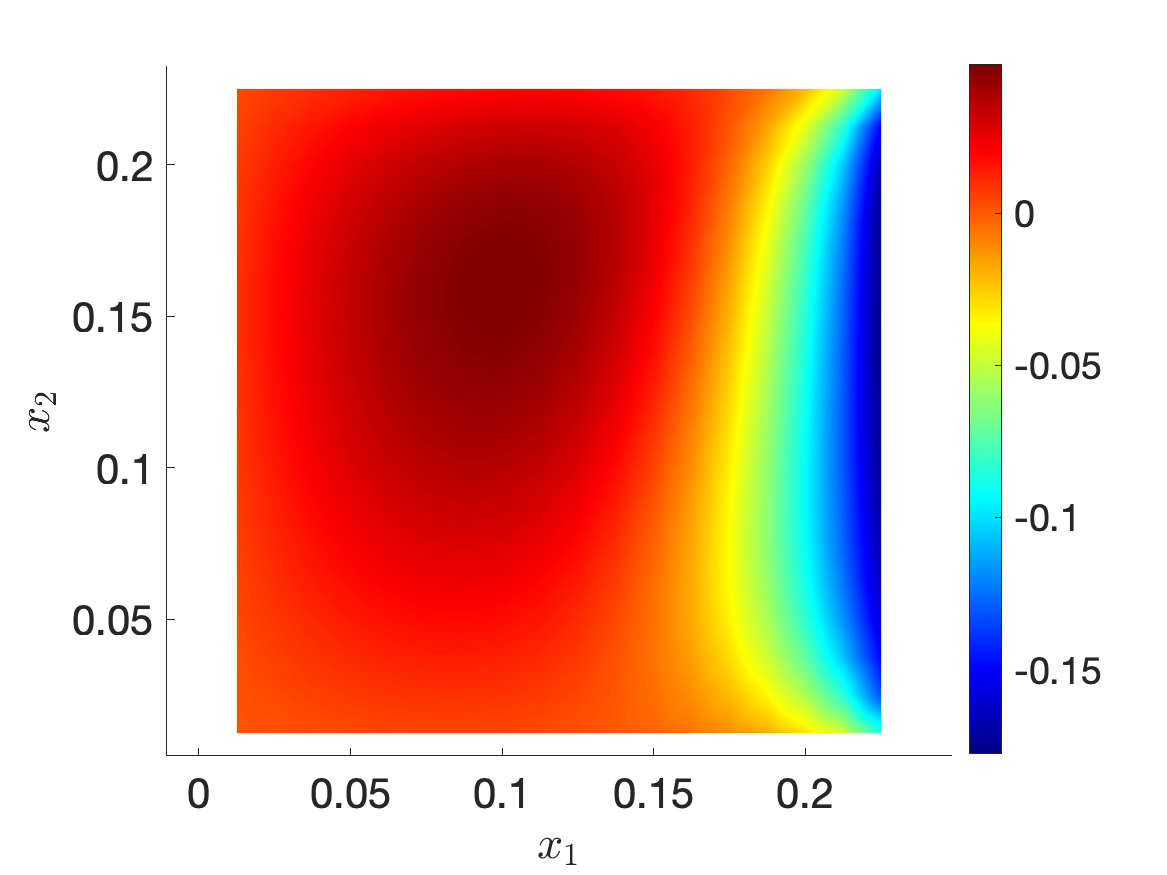}
	\caption{$\rbInteriorArg{2}$ on $\domainArg{1}$}
	\label{fig_ex1_heat44fn_btup_O1_PhiO2}
\end{subfigure}
~
\begin{subfigure}[b]{0.3\textwidth}
	\includegraphics[width=5.5cm]{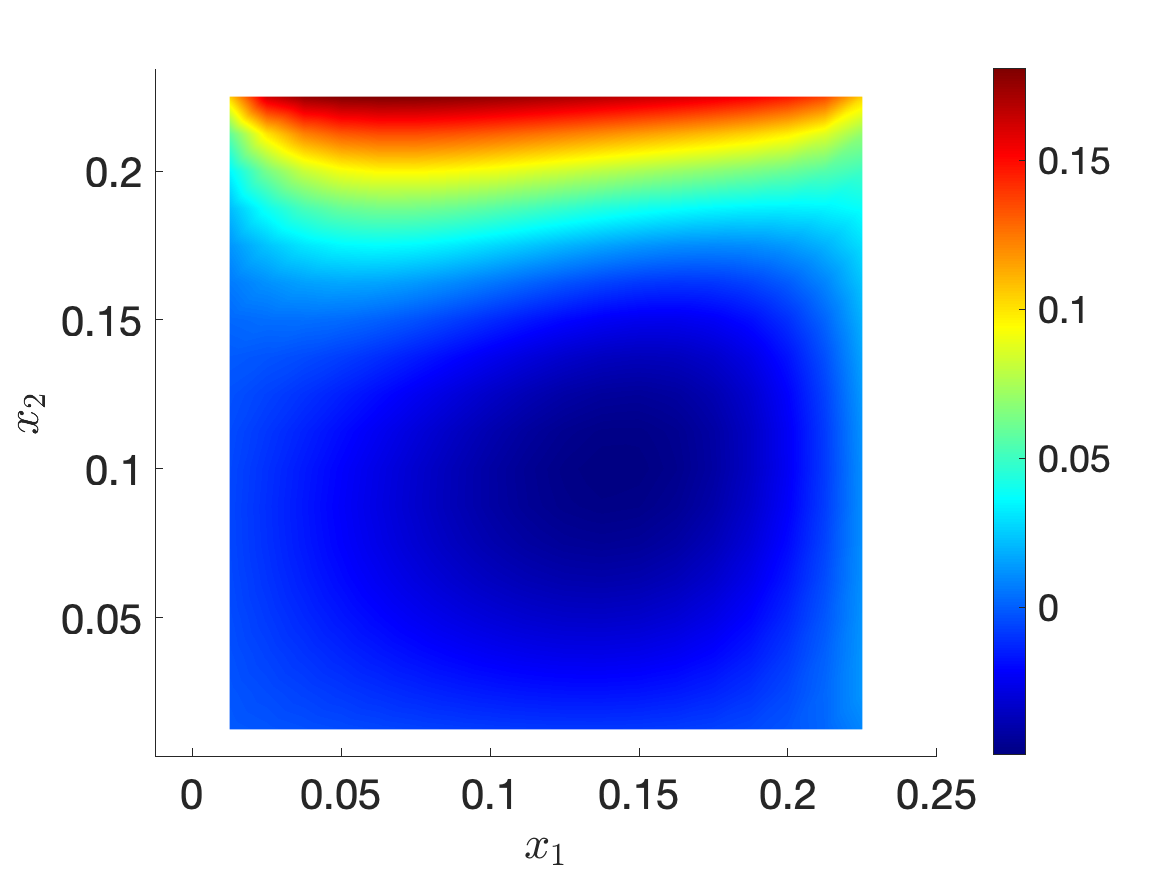}
	\caption{$\rbInteriorArg{3}$ on $\domainArg{1}$}
	\label{fig_ex1_heat44fn_btup_O1_PhiO3}
\end{subfigure}
~
\begin{subfigure}[b]{0.3\textwidth}
	\includegraphics[width=5.5cm]{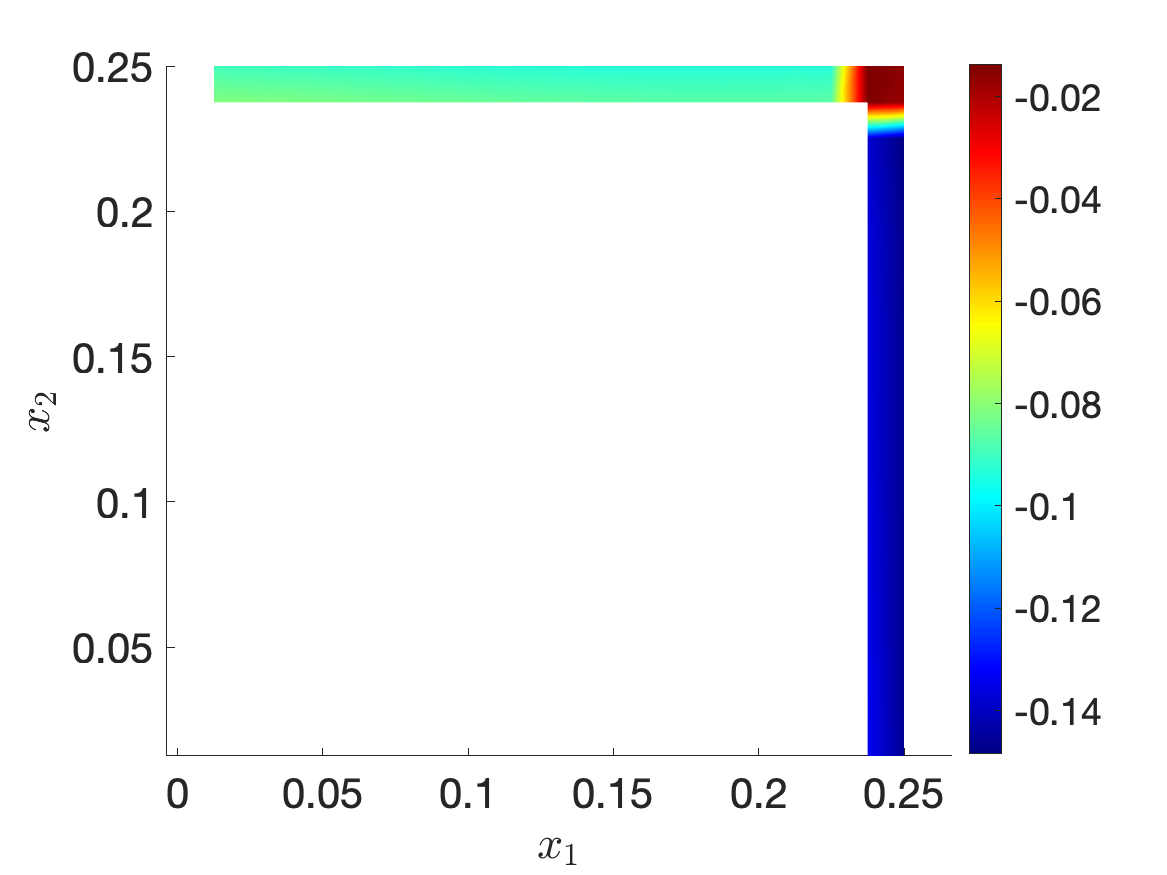}
	\caption{$\rbBoundaryArg{1}$ on $\domainArg{1}$}
	\label{fig_ex1_heat44fn_btup_O1_PhiG1}
\end{subfigure}
~
\begin{subfigure}[b]{0.3\textwidth}
	\includegraphics[width=5.5cm]{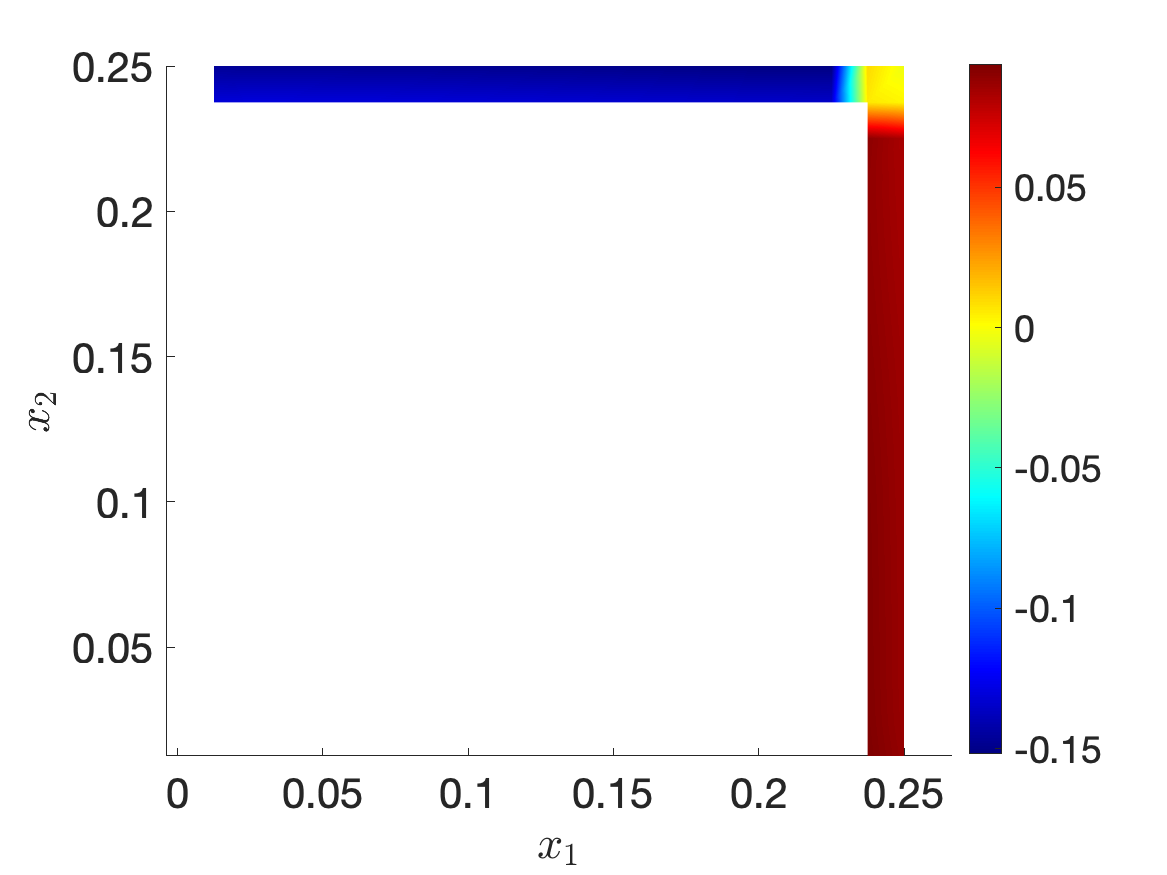}
	\caption{$\rbBoundaryArg{2}$ on $\domainArg{1}$}
	\label{fig_ex1_heat44fn_btup_O1_PhiG2}
\end{subfigure}
~
\begin{subfigure}[b]{0.3\textwidth}
	\includegraphics[width=5.5cm]{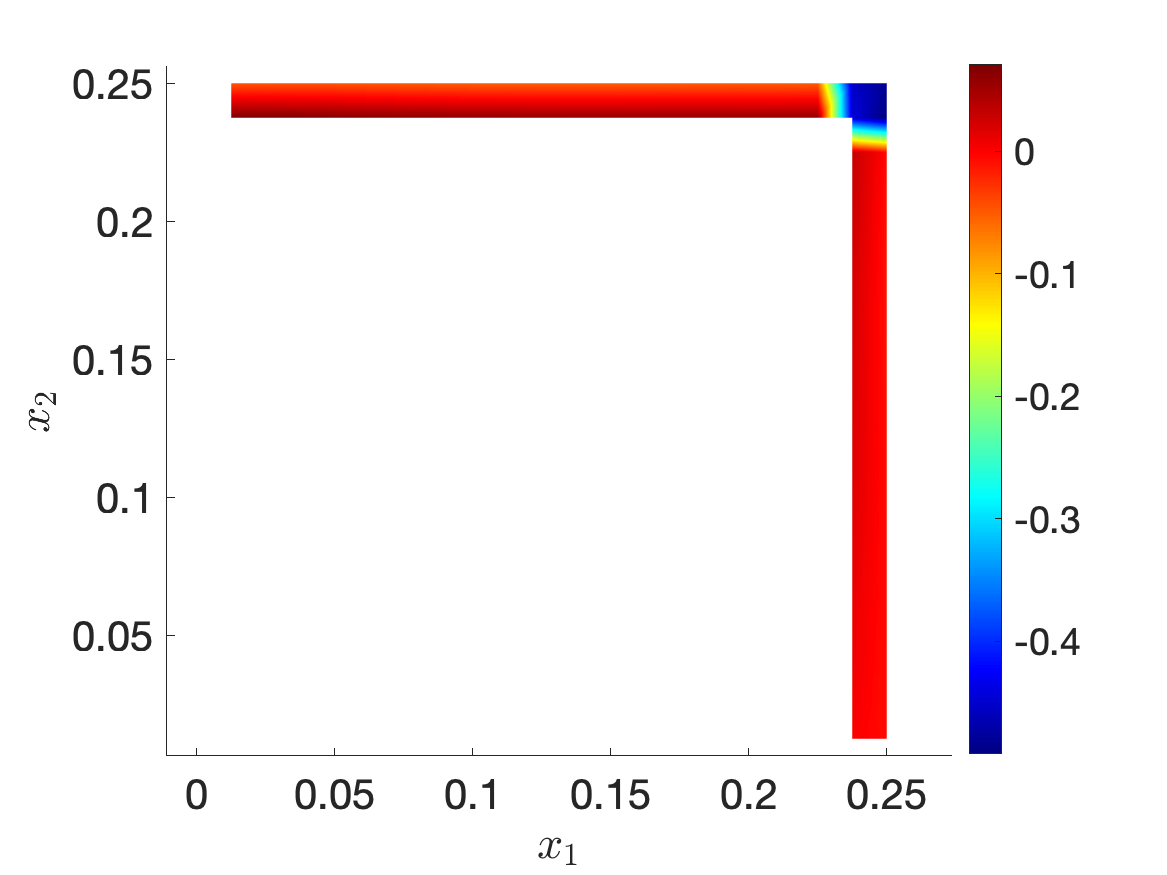}
	\caption{$\rbBoundaryArg{3}$ on $\domainArg{1}$}
	\label{fig_ex1_heat44fn_btup_O1_PhiG3}
\end{subfigure}	
~\begin{subfigure}[b]{0.3\textwidth}
	\includegraphics[width=5.5cm]{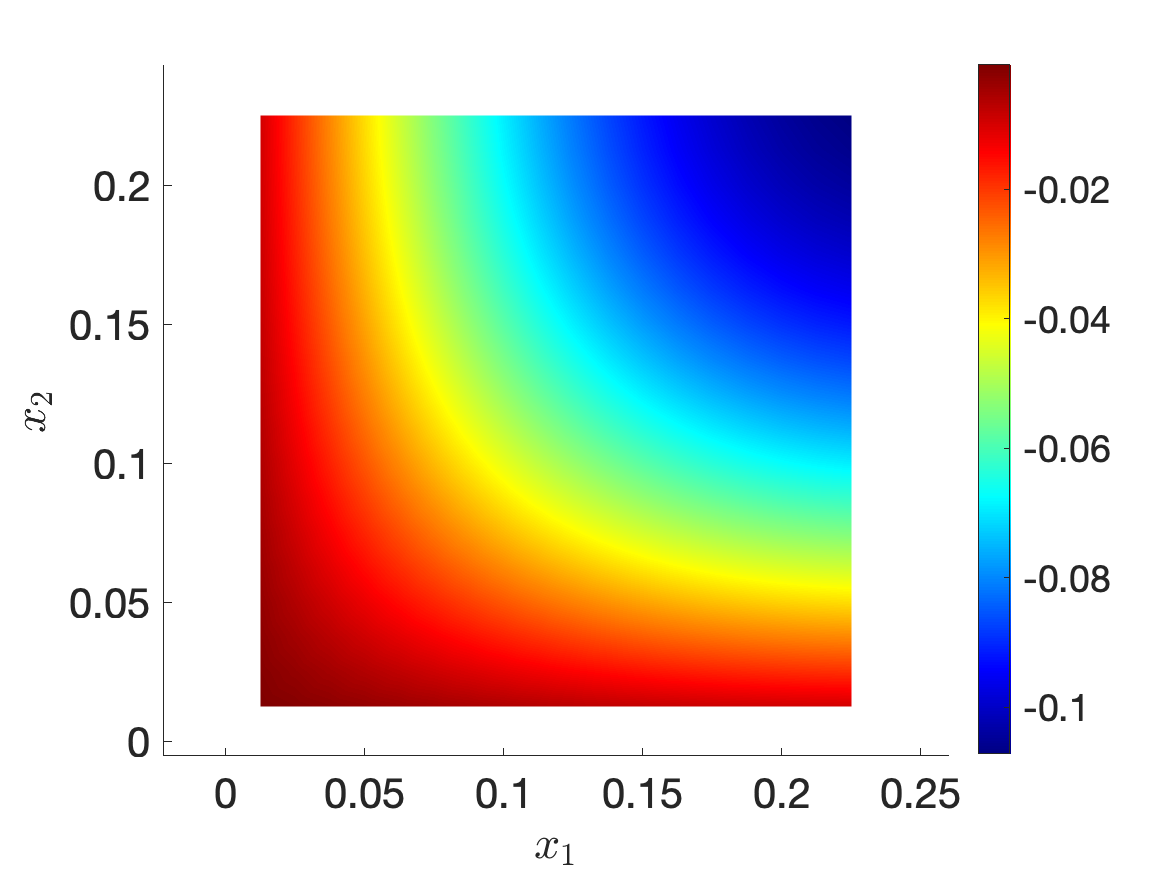}
	\caption{$\rbInteriorArg{1}$ on $\domainArg{1}$}
	\label{fig_ex1_heat44fn_topdown_O1_PhiO1}
\end{subfigure}
~
\begin{subfigure}[b]{0.3\textwidth}
	\includegraphics[width=5.5cm]{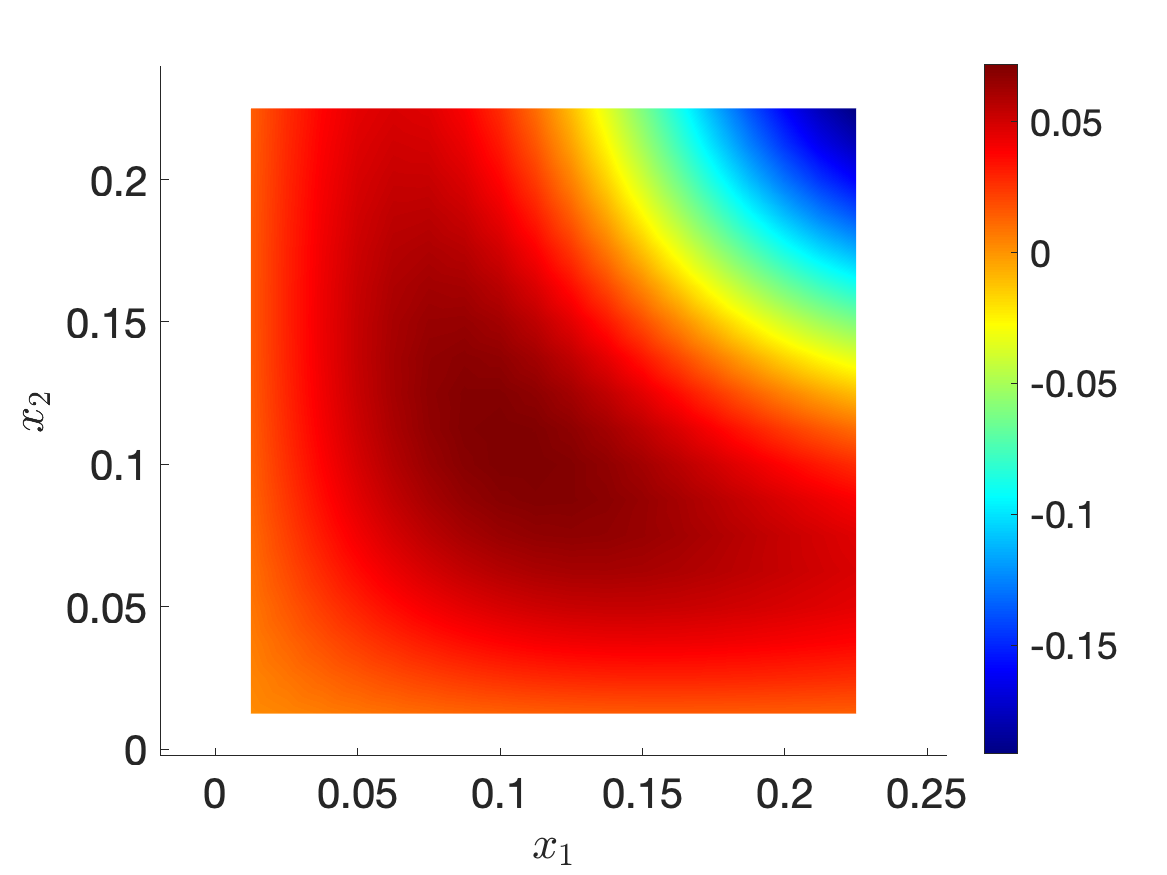}
	\caption{$\rbInteriorArg{2}$ on $\domainArg{1}$}
	\label{fig_ex1_heat44fn_topdown_O1_PhiO2}
\end{subfigure}
~
\begin{subfigure}[b]{0.3\textwidth}
	\includegraphics[width=5.5cm]{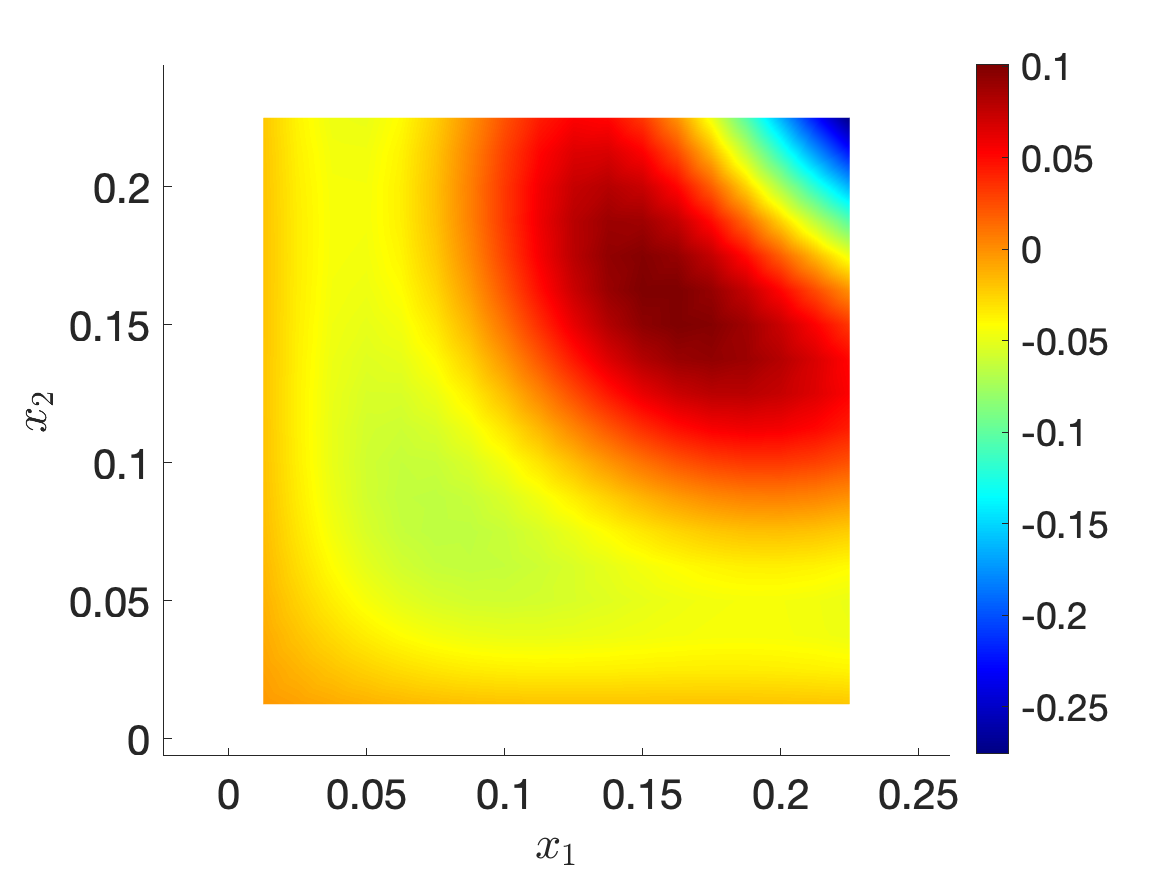}
	\caption{$\rbInteriorArg{3}$ on $\domainArg{1}$}
	\label{fig_ex1_heat44fn_topdown_O1_PhiO3}
\end{subfigure}
~
\begin{subfigure}[b]{0.3\textwidth}
	\includegraphics[width=5.5cm]{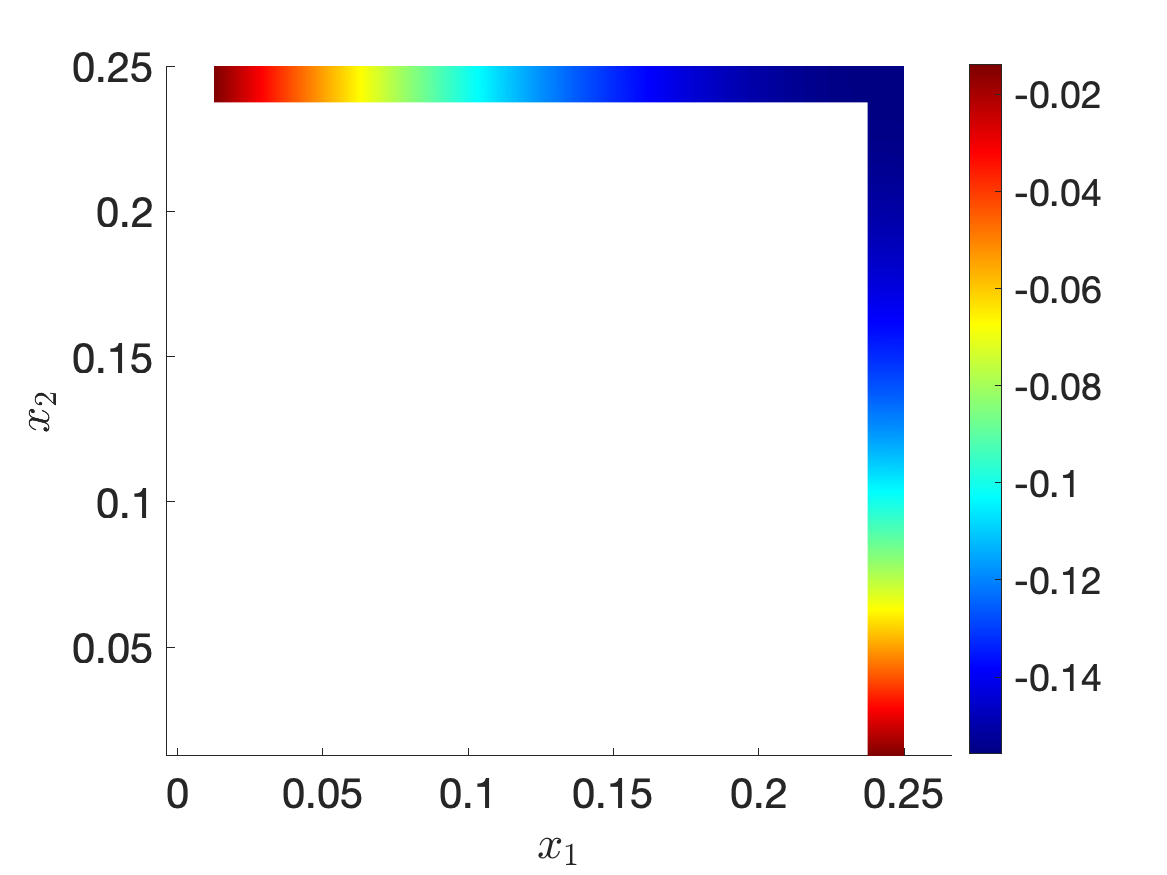}
	\caption{$\rbBoundaryArg{1}$ on $\domainArg{1}$}
	\label{fig_ex1_heat44fn_topdown_O1_PhiG1}
\end{subfigure}
~
\begin{subfigure}[b]{0.3\textwidth}
	\includegraphics[width=5.5cm]{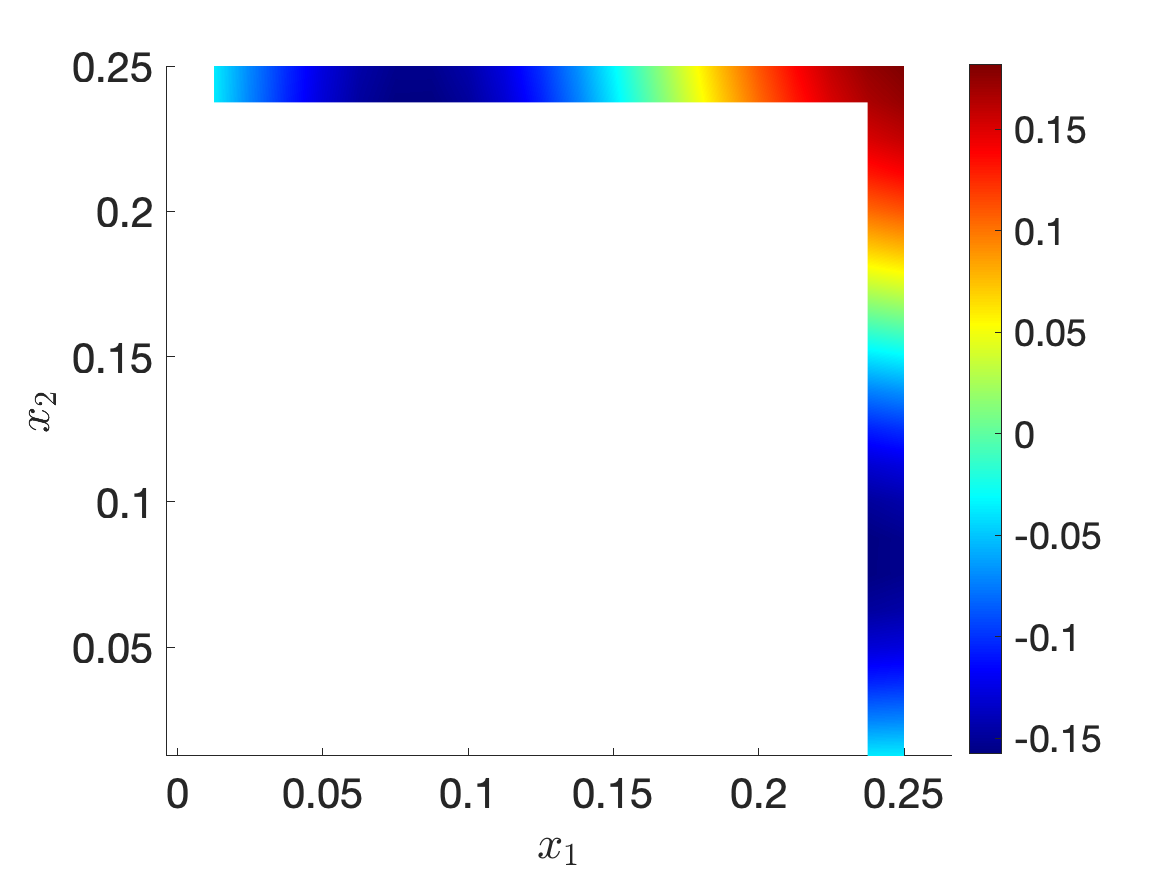}
	\caption{$\rbBoundaryArg{2}$ on $\domainArg{1}$}
	\label{fig_ex1_heat44fn_topdown_O1_PhiG2}
\end{subfigure}
~
\begin{subfigure}[b]{0.3\textwidth}
	\includegraphics[width=5.5cm]{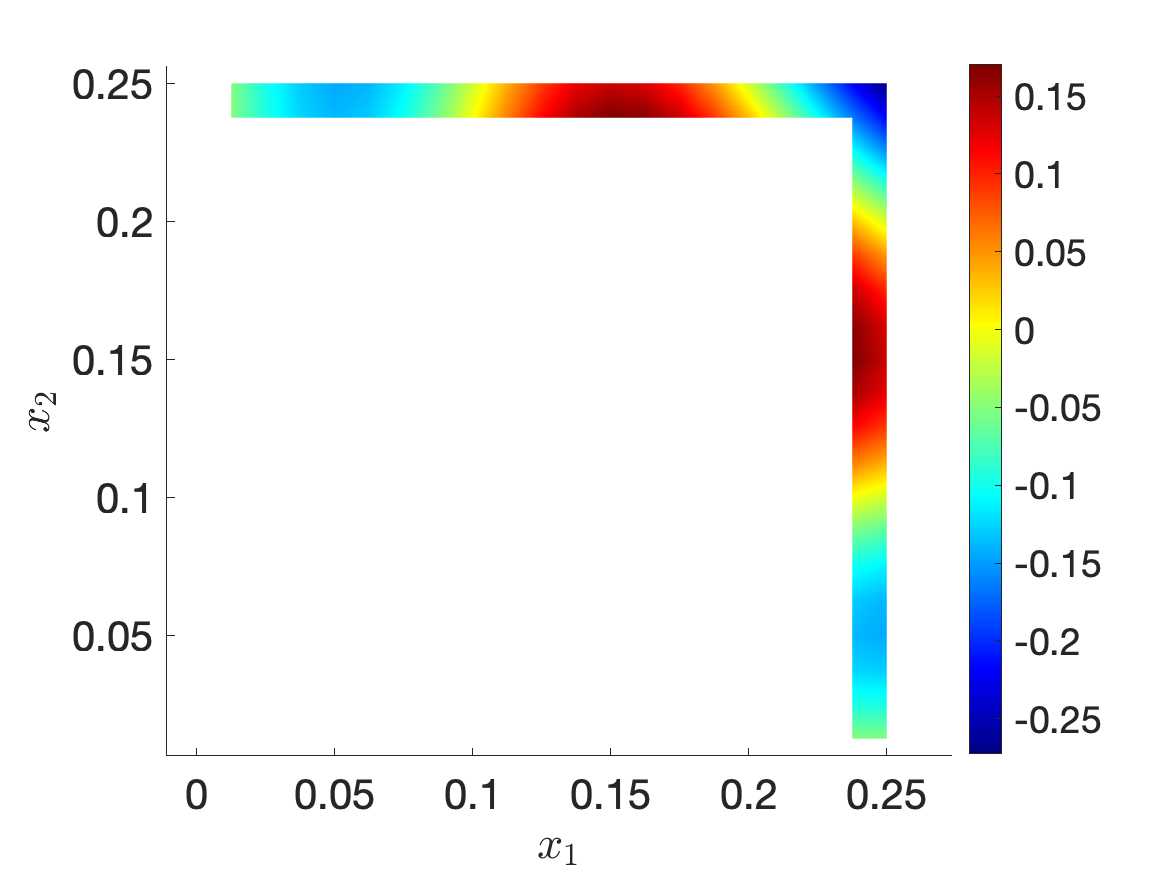}
	\caption{$\rbBoundaryArg{3}$ on $\domainArg{1}$}
	\label{fig_ex1_heat44fn_topdown_O1_PhiG3}
\end{subfigure}	
\caption{Heat equation, $4\times 4$ ``fine'' configuration, fixed train parameter, full-interface basis visualization on $\domainArg{1}$: bottom-up training (top 2 rows) versus top-down training (last 2 rows). } \label{fig_ex1_btup_topdown_intfBF}
\end{figure}

\begin{figure}[h!]
\centering
\begin{subfigure}[b]{0.3\textwidth}
	\includegraphics[width=5.5cm]{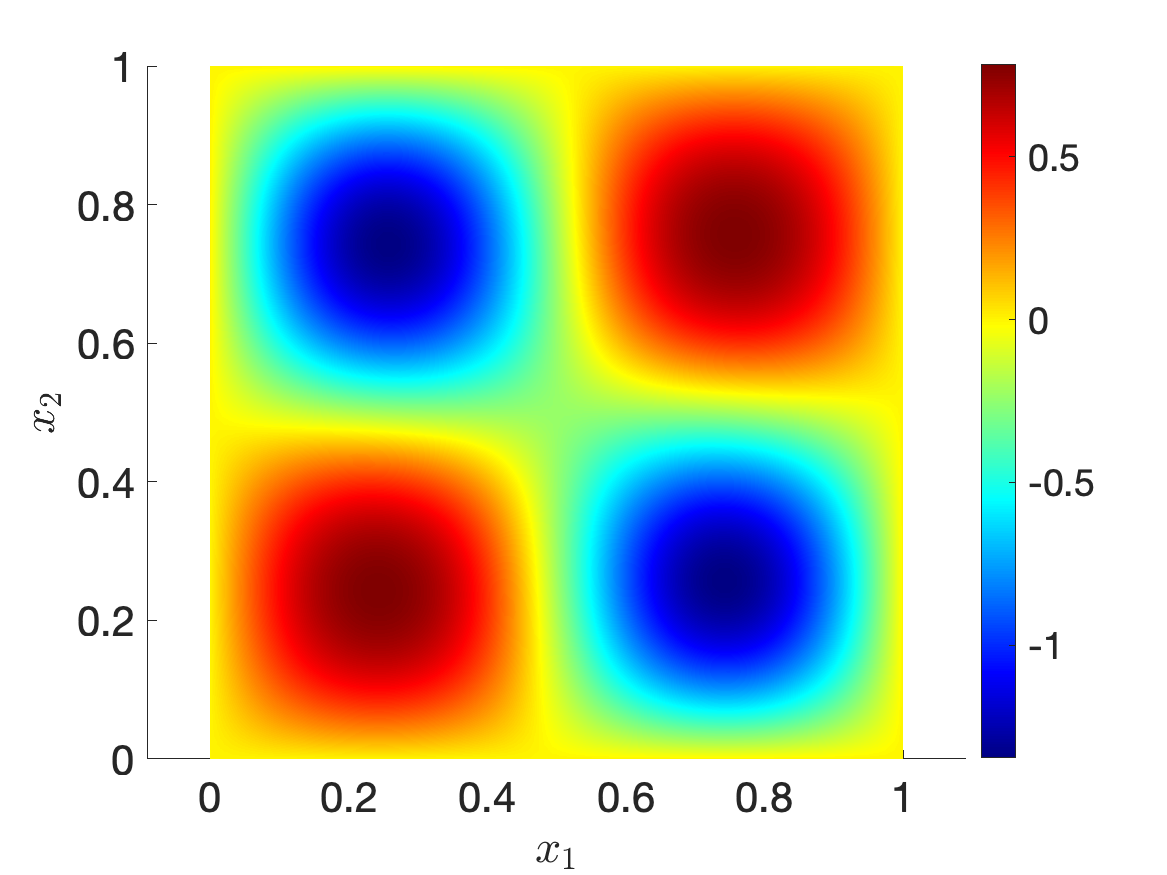}
	\caption{global FEM solution on $\domain$}
	\label{fig_ex1_heat44fn_sol_FEM_Omega}
\end{subfigure}
~
\begin{subfigure}[b]{0.3\textwidth}
	\includegraphics[width=5.5cm]{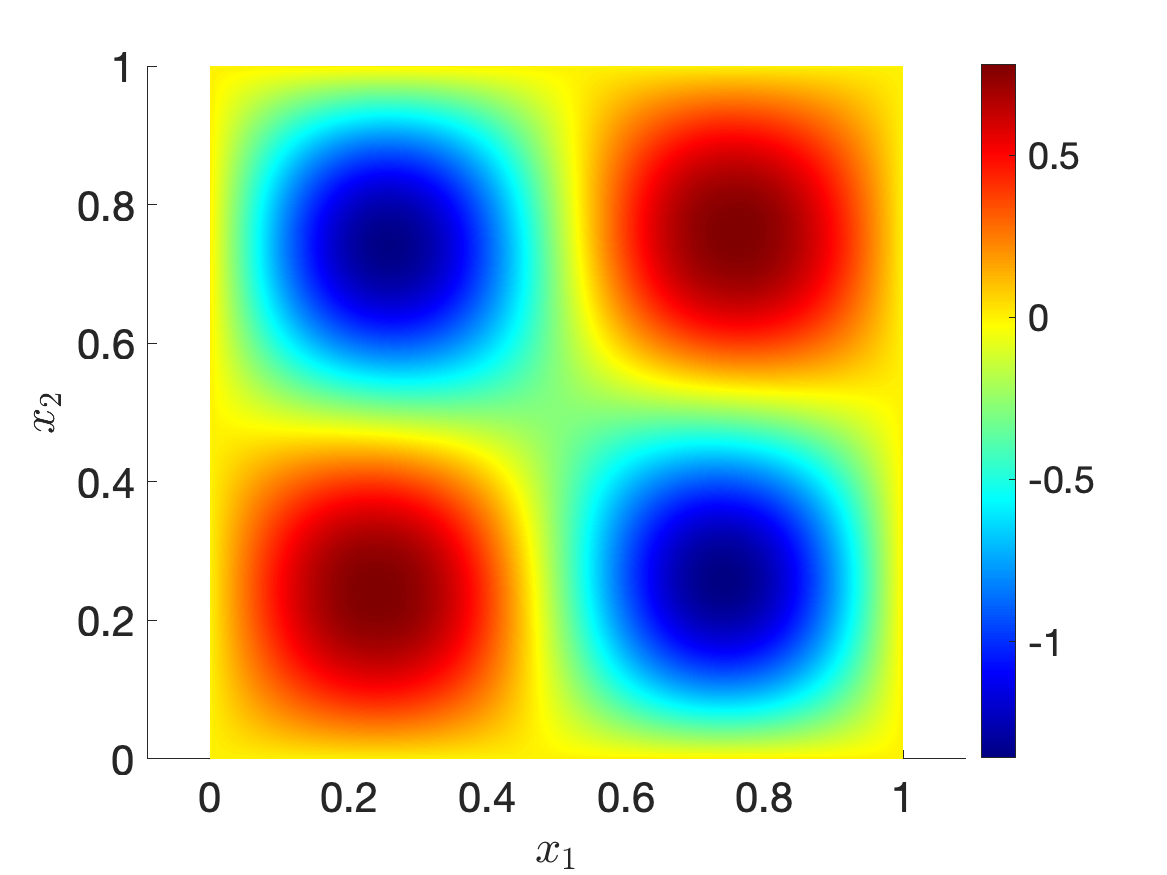}
	\caption{DD-LSPG solution on $\domain$}
	\label{fig_ex1_heat44fn_sol_DDROM_Omega}
\end{subfigure}
~
\begin{subfigure}[b]{0.3\textwidth}
	\includegraphics[width=5.5cm]{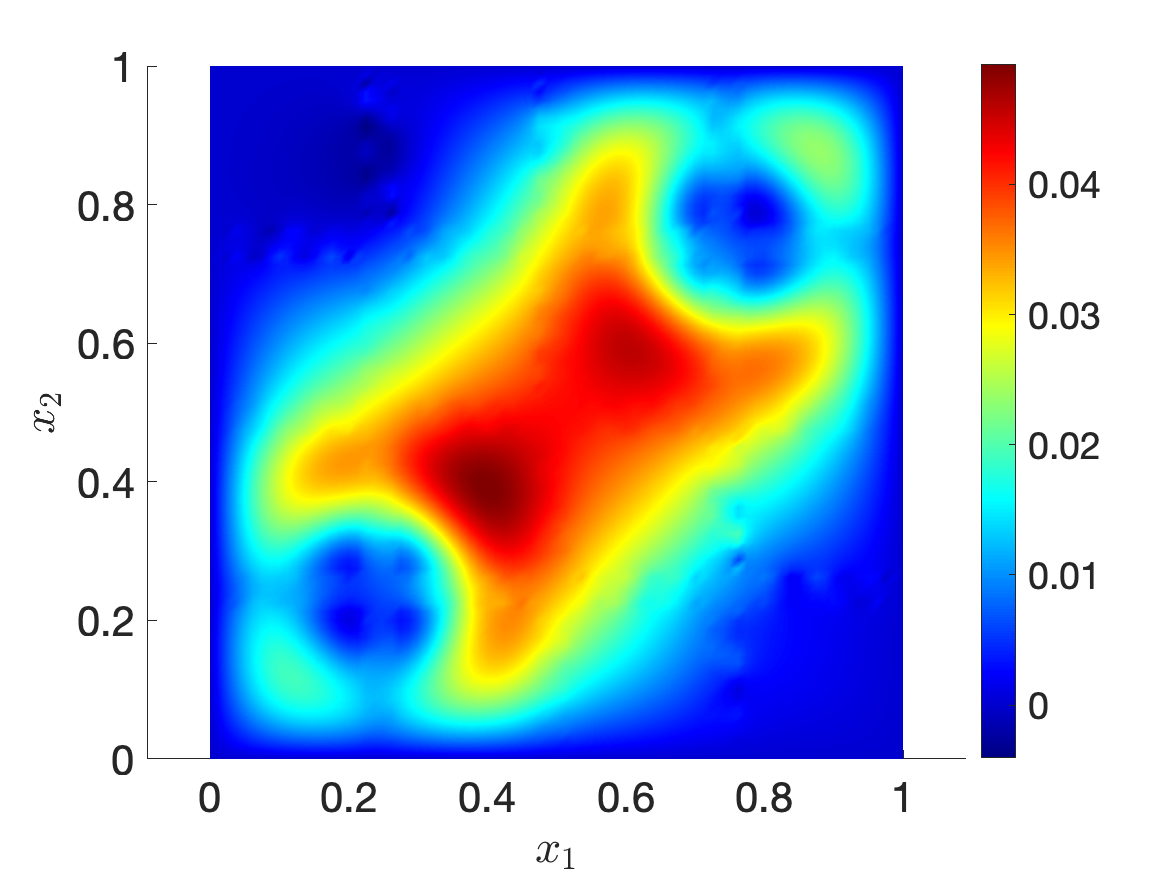}
	\caption{DD-LSPG error on $\domain$}
	\label{fig_ex1_heat44fn_sol_DDROMerr_Omega}
\end{subfigure}
~
\begin{subfigure}[b]{0.3\textwidth}
	\includegraphics[width=5.5cm]{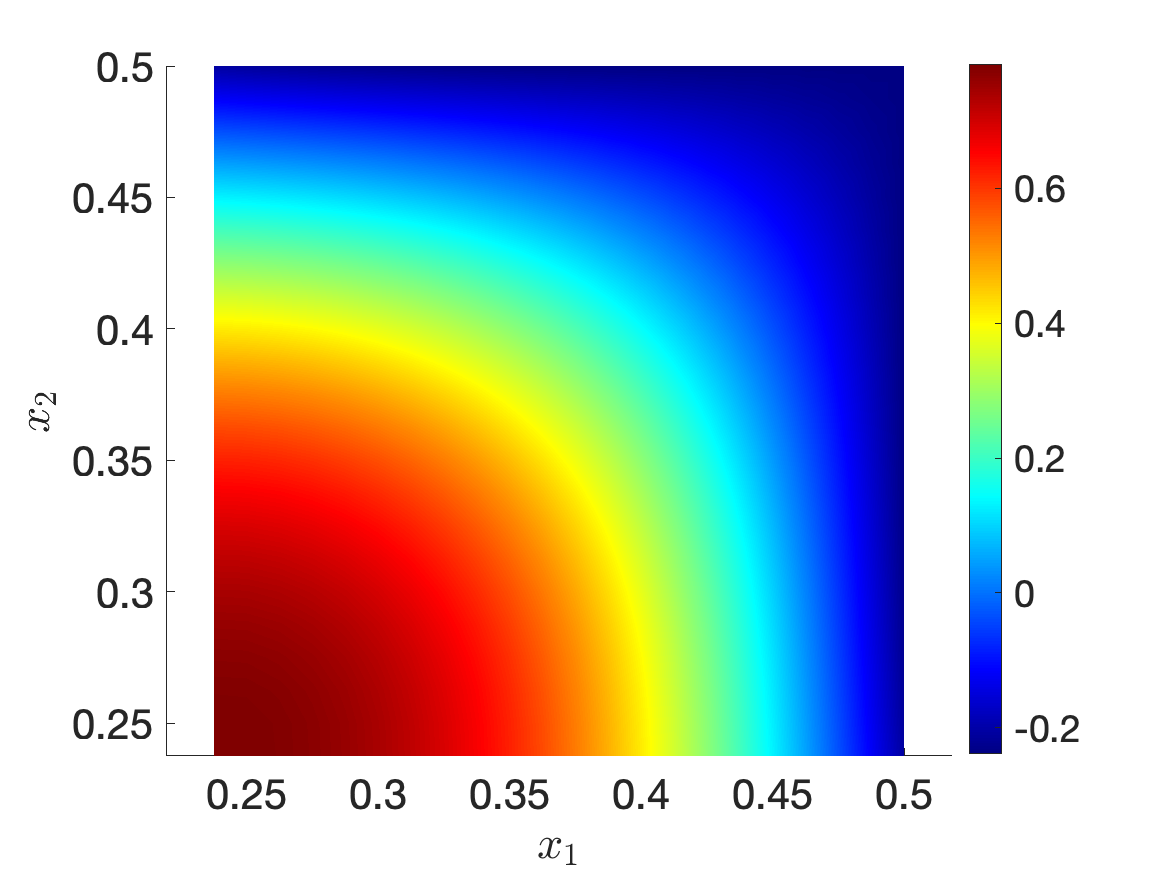}
	\caption{global FEM solution on $\domainArg{6}$}
	\label{fig_ex1_heat44fn_sol_FEM_O6}
\end{subfigure}
~
\begin{subfigure}[b]{0.3\textwidth}
	\includegraphics[width=5.5cm]{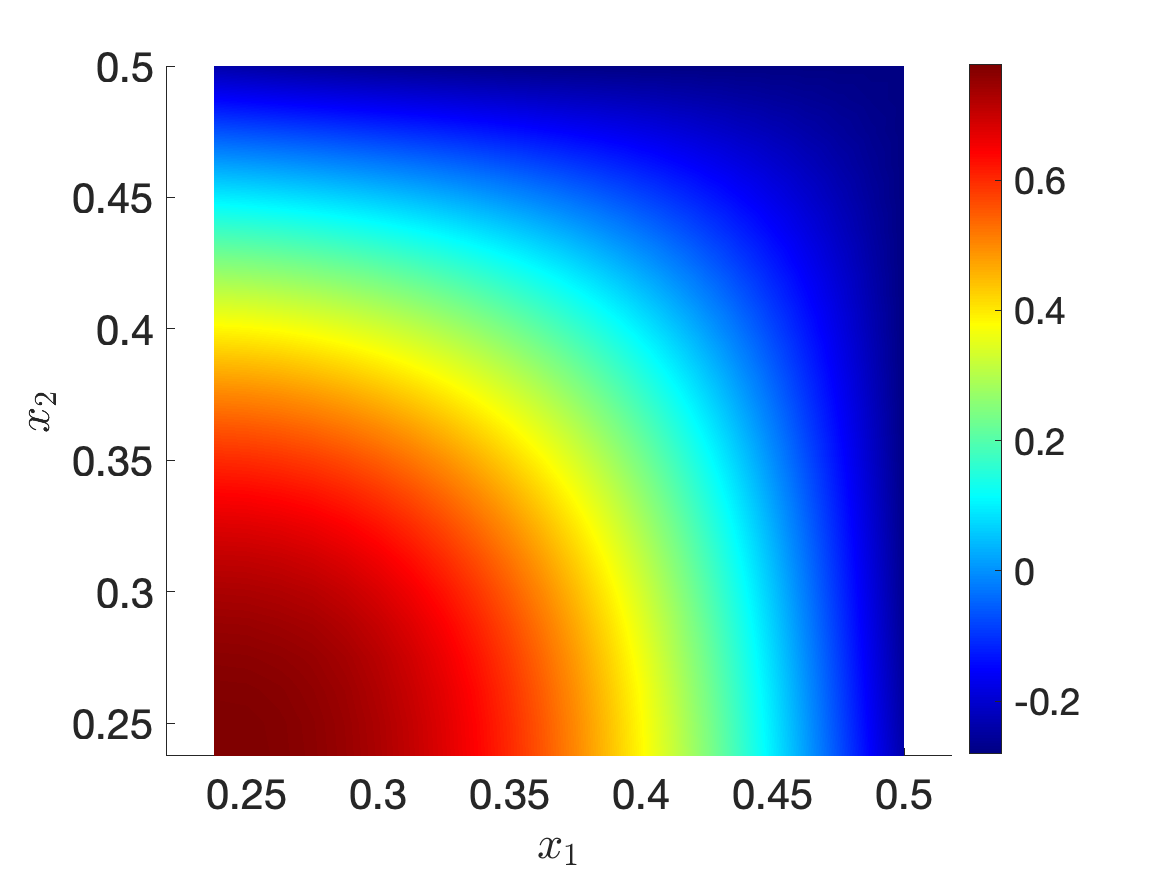}
	\caption{DD-LSPG solution on $\domainArg{6}$}
	\label{fig_ex1_heat44fn_sol_DDROM_O6}
\end{subfigure}
~
\begin{subfigure}[b]{0.3\textwidth}
	\includegraphics[width=5.5cm]{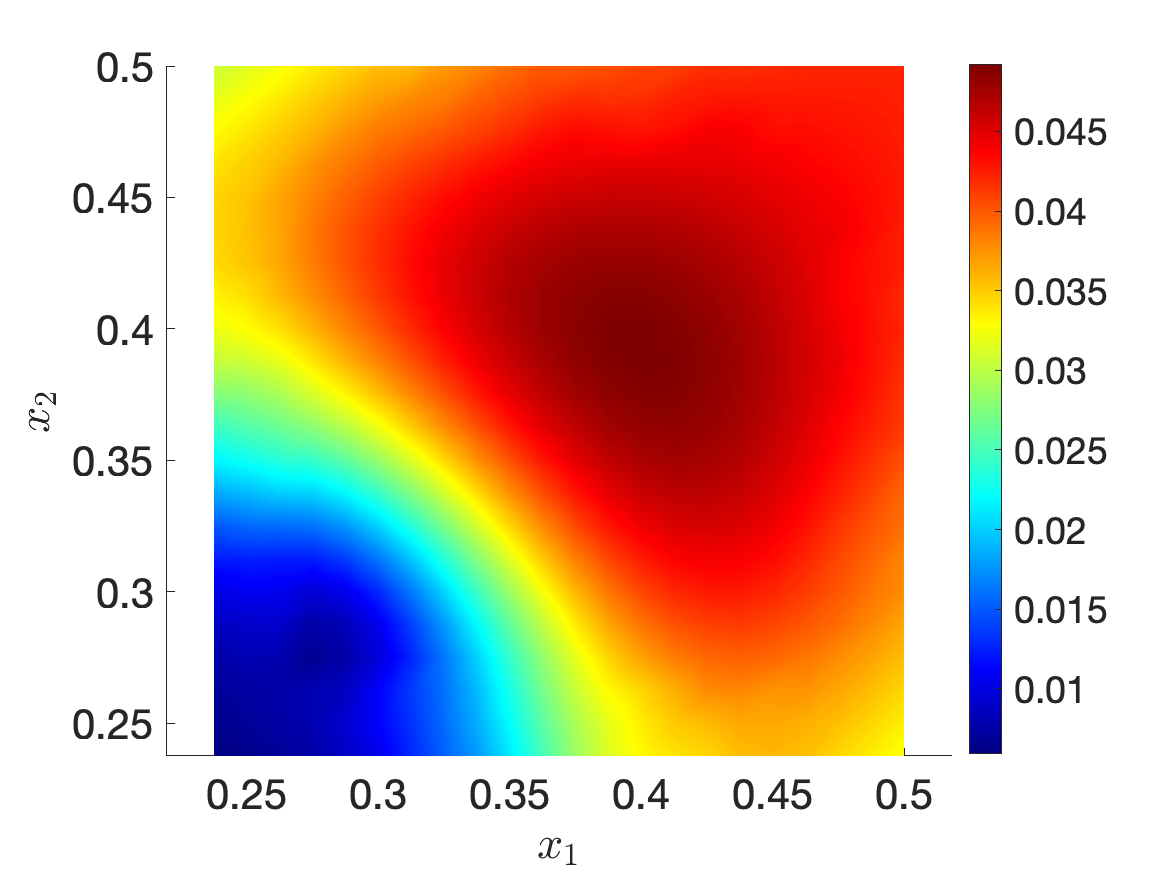}
	\caption{DD-LSPG error on $\domainArg{6}$}
	\label{fig_ex1_heat44fn_sol_DDROMerr_O6}
\end{subfigure}	
\caption{Heat equation, bottom-up training, $4\times 4$ ``fine'' configuration, reproductive test, full-interface bases ($\energyCriterion=1-10^{-12}$) in Table \ref{tab_ex1_oneOnlineComp_inputParams_btup}, solutions visualized on $\domain$ and $\domainArg{6}$. }\label{fig_ex1_btup_sol_vis}
\end{figure}

\begin{table}[h!]
\small
\center
\caption{Heat equation, bottom-up training, $4\times 4$ ``fine'' configuration, reproductive test, ROM methods
performance at point $\param_{{\rm test}}=(5.005,5.005) \notin \trainSample$ for one
online computation. Recall from Section
\ref{sec_basis_construction} that $\energyCriterion\in[0,1]$ denotes the energy
criterion employed by POD.} 
\label{tab_ex1_oneOnlineComp_inputParams_btup}
{\begin{tabular}{|c||c|c|c||c|c|c|}
	\hline \rule{0pt}{1.5ex}
	train parameter & \multicolumn{6}{c|}{$\trainParam{i} = (5,5) \; {\rm on} \; \domain_i, \quad 1 \le i \le \nsubdomains $ } 
	\\	[0.2ex]
	\hline \rule{0pt}{1.5ex}
	constraint	& \multicolumn{6}{c|}{strong} \\	
	\hline \rule{0pt}{2.5ex}	
	basis	& \multicolumn{3}{c||}{port} & 
	\multicolumn{3}{c|}{full-interface} 
	\\ [0.3ex]
	\hline \rule{0pt}{2.5ex}
	method  & DD-LSPG  & DD-LSPG  & DD-LSPG  &  
	DD-LSPG  &  DD-LSPG & DD-LSPG  
	\\ [0.3ex]
	\hline \rule{0pt}{2.5ex} 
	$\energyCriterion$ for state & $1-10^{-12}$ & $1-10^{-14}$ & $1-10^{-17}$ & $1-10^{-10}$ & $1-10^{-12}$ & $1-10^{-17}$  
	\\
	\hline \rule{0pt}{2.5ex} \hspace{-2.5mm}	number Newton iter.  & 9 & 7 & 7 & 8 & 7 & 7  
	\\	
	rel. error & 	$8.1960 \times 10^{-1}$ 	  & $5.9268 \times 10^{-2}$ & $1.9610 \times 10^{-7}$ & $3.8504 \times 10^{-1}$ & $4.7951 \times 10^{-2}$ & $1.9462 \times 10^{-7}$ 
	\\
	\hline
\end{tabular}}
\end{table}

\begin{table}[h!]
\center
\caption{Heat equation, bottom-up training, $4\times 4$ ``fine'' configuration, reproductive test, ROM parameters on first four $\domaini$ ($1 \le i \le 4$) resulting from
Table~\ref{tab_ex1_oneOnlineComp_inputParams_btup}.}
\label{tab_ex1_oneOnlineComp_ROMparams_btup}
{\begin{tabular}{|c||c|c|c|c||c|c|c|c|}
	\hline \rule{0pt}{2.5ex}
	basis	& \multicolumn{4}{c||}{port} & \multicolumn{4}{c|}{full-interface}   
	\\
	\hline \rule{0pt}{2.5ex}
	$\energyCriterion$ for state & \multicolumn{4}{c||}{$1-10^{-14}$} & \multicolumn{4}{c|}{$1-10^{-12}$}   
	\\ \hline \rule{0pt}{2.5ex}	
	subdomains & $\domainArg{1}$ & $\domainArg{2}$ & $\domainArg{3}$  & $\domainArg{4}$ & $\domainArg{1}$ & $\domainArg{2}$ & $\domainArg{3}$  & $\domainArg{4}$ 
	\\ [0.2ex]
	\hline \rule{0pt}{2.5ex} \hspace{-2.5mm}
	$\nconstraintsROM$ 	& \multicolumn{4}{c||}{984} & \multicolumn{4}{c|}{984}  \\	
	\hline \rule{0pt}{2.5ex} \hspace{-2.5mm}
	$\nrbInteriori$ ($\ndofInteriori$) 	& 25 (324) & 38 (324) & 36 (324) & 25 (342) & 17 (324) & 26 (324) & 25 (324) & 17 (342)
	\\
	$\nrbBoundaryi$ ($\ndofBoundaryi$)	&  76 (76) & 116 (116) & 116 (116) & 78 (78) & 72 (76) & 110 (116) & 110 (116) & 74 (78)
	\\ [0.2ex]
	\hline \rule{0pt}{2.5ex}
	$\nrbPortArg{1}$ ($\ndofPortsArg{1}$) & 36 (36) & 36 (36) & 36 (36) & 36 (36) &    &    &    &   
	\\
	$\nrbPortArg{2}$ ($\ndofPortsArg{2}$) & 4 (4) & 4 (4) & 4 (4) & 4 (4) &    &    &    & 
	\\
	$\nrbPortArg{3}$ ($\ndofPortsArg{3}$) & 36 (36) & 36 (36) & 36 (36) & 38 (38) &    &    &    &  
	\\
	$\nrbPortArg{4}$ ($\ndofPortsArg{4}$) &  & 4 (4) & 4 (4) &   &    &    &    &  
	\\	
	$\nrbPortArg{5}$ ($\ndofPortsArg{5}$) &  & 36 (36) & 36 (36) &   &    &    &    &  
	\\	
	\hline
\end{tabular}}
\end{table}

In this section, we investigate the possibility of subdomain (or bottom-up) training that is opposite to the top-down training in previous sections. The main goal is to create and use snapshots \textit{completely} at component/subdomain level (rather than system/global level like top-down training) to build corresponding reduced bases. In other words, the approach is \textit{completely} bottom-up similar to the SCRBE method \cite{huynh2013static_a, huynh2013static_b} except that our proposed approach solves nonlinear PDE, while SCRBE solves linear PDE only. We shall describe a specific approach that borrows ideas of SCRBE (algorithm 2, page 279 of \cite{eftang2013port}) to build such bottom-up reduced bases. 

We choose the 4x4 ``fine'' configuration and loop over all ports on all subdomains. On each port, we do the following: 
\begin{enumerate}
	\item Build a set of 1D Legendre polynomials that correspond with each edge of the port. For our particular FOM model, a port is a rectangle with $x$-line and $y$-line, thus we build 1D Legendre polynomials associated with these 2 lines. 
	
	\item Perform tensor product of these 1D Legendre polynomials to create 2D Legendre polynomials of that port. We denote these 2D Legendre polynomials $L^k_{m,j}$ where $k$ is the sequence index, $m$ denotes subdomain index, and $j$ is port index. (Note that we use same notation as algorithm 2, page 279 of \cite{eftang2013port}.) 
	
	\item Perform random linear combinations of these 2D Legendre polynomials to create associated boundary conditions on that port: 	
	\begin{equation} \label{eq_rand_lin_comb_Legendre}
	\state|_{P_j} = \sum_{k=1}^{\ndofPortjSubdomi{i}{j}} r \frac{1}{k^{\eta}} L^k_{i,j}, \qquad 1 \le i \le \nsubdomains, 
	\end{equation}
	where $r$ is a random variable with univariate uniform density over (-1,1), $\eta$ is a tuning parameter related to anticipated regularity. For our problem, we also choose $\eta=2$ following \cite{eftang2013port}. 
	
	\item Assemble all the above ``port boundary conditions'' to form associated boundary conditions on the interface of that subdomain: 
	\begin{equation} \label{eq_assemble_interface}
	\state|_{\boundaryi} = \bigcup_{j \in \portsSubdomains{i}} \state|_{P_j}, \qquad 1 \le i \le \nsubdomains. 
	\end{equation}			
	
	\item Solve the FOM problem of that subdomain $\domain_i$ with ``interface boundary conditions'' specified above to obtain a \textit{subdomain} FOM solution (or snapshot). Repeat this step with many random linear combinations to create many different interface boundary conditions to collect \textit{subdomain} snapshots, and store them to a \textit{subdomain} snapshots set. For our problem, we perform 200 random ``interface boundary conditions'' on each subdomain $\domain_i$, solve the subdomain FOM problems to collect 200 snapshots over each $\domain_i$, $1 \le i \le \nsubdomains$, respectively. Note also that we fix the input parameter $\paramComp_{{\rm train}}=(5,5)$ for all FOM solves over all subdomains $\domain_i, 1 \le i \le \nsubdomains$. 
	
	\item With available \textit{subdomain} snapshots set, implement sections~\ref{sec:constrIntBound} and/or \ref{sec:constFull} to form corresponding reduced bases. 
	
\end{enumerate}

With bottom-up reduced bases created above, we now compare the DD-LSPG method for fixed values of their parameters, and for the selected online point $\paramComp_{{\rm test}}=(5.005,5.005)$. Table~\ref{tab_ex1_oneOnlineComp_inputParams_btup} reports the chosen input parameters and associated performance of the DD-LSPG method, while the resulting ROM parameters over first four subdomains $\domainArg{i}, 1 \le i \le 4$ are listed on Table~\ref{tab_ex1_oneOnlineComp_ROMparams_btup}. Table~\ref{tab_ex1_oneOnlineComp_inputParams_btup} shows that the produced ROM solutions are converged and get more accurate with increasing number of bases.

Figure~\ref{fig_ex1_btup_topdown_intfBF} compares first three full-interface bases on $\domainArg{1}$ using bottom-up training (top 2 rows) and top-down training (last 2 rows). As observed from this figure, we see that the quality of bottom-up training bases are not as good as that of top-down training. (This is sensible because the subdomain \textit{completely} not know anything about the global solution.) As a result, bottom-up training uses much more interior and interfaces bases than top-down training does with same accuracy level (comparing Table~\ref{tab_ex1_oneOnlineComp_inputParams} with Table~\ref{tab_ex1_oneOnlineComp_inputParams_btup}, and Table~\ref{tab_ex1_oneOnlineComp_ROMparams} with Table~\ref{tab_ex1_oneOnlineComp_ROMparams_btup}; especially on Table~\ref{tab_ex1_oneOnlineComp_ROMparams_btup}, bottom-up training uses \textit{all} available interface bases). This leads to two important consequences: i) we obtain very little dimension reduction on the interface (although still big dimension reduction on the interior) for bottom-up training; ii) bottom-up training must use strong constraint to obtain good converged solutions. This is completely opposite to top-down training where number of interior and interface bases are small (compared with bottom-up training), hence weak constraints are necessary to obtain converged solutions (see Figure~\ref{fig_ex1_4x4_rmsErr_intfBF_88}). Figure~\ref{fig_ex1_btup_sol_vis} visualizes DD-LSPG solutions and error for the full-interface bases case: it shows that DD-LSPG yield accurate results for full-interface bases with strong constraints using bottom-up training approach. 


\textbf{Predictive test}

\begin{table}[h!]
\small
\center
\caption{Heat equation, bottom-up training, $4\times 4$ ``fine'' configuration, predictive test, ROM methods
performance at point $\param_{{\rm test}}=(5.005,5.005) \notin \trainSample$ for one
online computation. Recall from Section
\ref{sec_basis_construction} that $\energyCriterion\in[0,1]$ denotes the energy
criterion employed by POD.} 
\label{tab_ex1_oneOnlineComp_inputParams_btup_varyTrain}
{\begin{tabular}{|c||c|c||c|c|}
	\hline \rule{0pt}{1.5ex}
	train parameter & \multicolumn{4}{c|}{$\trainParam{i} = (\paramCompArg{t}^i,\paramCompArg{t}^i) \; {\rm on} \; \domain_i, {\rm where} \; \paramCompArg{t}^i=\displaystyle \frac{\paramCompArg{\rm max}-\paramCompArg{\rm min}}{\nsubdomains-1}(i-1) + \paramCompArg{\rm min}, \; 1 \le i \le \nsubdomains $ } 
	\\	[1.5ex]
	\hline \rule{0pt}{1.5ex}
	constraint	& \multicolumn{4}{c|}{strong} \\	
	\hline \rule{0pt}{2.5ex}	
	basis	& \multicolumn{2}{c||}{port} & 
	\multicolumn{2}{c|}{full-interface} 
	\\ [0.3ex]
	\hline \rule{0pt}{2.5ex}
	method  & DD-LSPG  & DD-LSPG  & DD-LSPG  &  
	DD-LSPG   
	\\ [0.3ex]
	\hline \rule{0pt}{2.5ex} 
	$\energyCriterion$ for state & $1-10^{-15}$ & $1-10^{-17}$ & $1-10^{-15}$ & $1-10^{-17}$   
	\\
	\hline \rule{0pt}{2.5ex} \hspace{-2.5mm}	number Newton iter.  & 7 & 7 & 7 & 7   
	\\	
	rel. error & $9.9537 \times 10^{-2}$ & $2.6104 \times 10^{-3}$ & $9.9537 \times 10^{-2}$ & $2.6104 \times 10^{-3}$ 
	\\
	\hline
\end{tabular}}
\end{table}

\begin{table}[h!]
\center
\caption{Heat equation, bottom-up training, $4\times 4$ ``fine'' configuration, predictive test, ROM parameters on first four $\domaini$ ($1 \le i \le 4$) resulting from
Table~\ref{tab_ex1_oneOnlineComp_inputParams_btup_varyTrain}.}
\label{tab_ex1_oneOnlineComp_ROMparams_btup_varyTrain}
{\begin{tabular}{|c||c|c|c|c||c|c|c|c|}
	\hline \rule{0pt}{2.5ex}
	basis	& \multicolumn{4}{c||}{port} & \multicolumn{4}{c|}{full-interface}   
	\\
	\hline \rule{0pt}{2.5ex}
	$\energyCriterion$ for state & \multicolumn{4}{c||}{$1-10^{-17}$} & \multicolumn{4}{c|}{$1-10^{-17}$}   
	\\ \hline \rule{0pt}{2.5ex}	
	subdomains & $\domainArg{1}$ & $\domainArg{2}$ & $\domainArg{3}$  & $\domainArg{4}$ & $\domainArg{1}$ & $\domainArg{2}$ & $\domainArg{3}$  & $\domainArg{4}$ 
	\\ [0.2ex]
	\hline \rule{0pt}{2.5ex} \hspace{-2.5mm}
	$\nconstraintsROM$ 	& \multicolumn{4}{c||}{984} & \multicolumn{4}{c|}{984}  \\	
	\hline \rule{0pt}{2.5ex} \hspace{-2.5mm}
	$\nrbInteriori$ ($\ndofInteriori$) 	& 200 (324) & 200 (324) & 200 (324) & 200 (342) & 200 (324) & 200 (324) & 200 (324) & 200 (342)
	\\
	$\nrbBoundaryi$ ($\ndofBoundaryi$)	&  76 (76) & 116 (116) & 116 (116) & 78 (78) & 72 (76) & 110 (116) & 110 (116) & 74 (78)
	\\ [0.2ex]
	\hline \rule{0pt}{2.5ex}
	$\nrbPortArg{1}$ ($\ndofPortsArg{1}$) & 36 (36) & 36 (36) & 36 (36) & 36 (36) &    &    &    &   
	\\
	$\nrbPortArg{2}$ ($\ndofPortsArg{2}$) & 4 (4) & 4 (4) & 4 (4) & 4 (4) &    &    &    & 
	\\
	$\nrbPortArg{3}$ ($\ndofPortsArg{3}$) & 36 (36) & 36 (36) & 36 (36) & 38 (38) &    &    &    &  
	\\
	$\nrbPortArg{4}$ ($\ndofPortsArg{4}$) &  & 4 (4) & 4 (4) &   &    &    &    &  
	\\	
	$\nrbPortArg{5}$ ($\ndofPortsArg{5}$) &  & 36 (36) & 36 (36) &   &    &    &    &  
	\\	
	\hline
\end{tabular}}
\end{table}

Finally, we perform a truly predictive test for our proposed framework using bottom-up training. We vary the training parameters on each subdomain so that they are all different from each other, and they also differ from the online testing parameter. In particular, we repeat the workflow 1--6 above to build the bottom-up bases except at step 5 we set $\trainParam{i}$ equidistant over subdomains $\domain_i$, i.e., $\trainParam{i} = (\paramCompArg{t}^i,\paramCompArg{t}^i) \; {\rm on} \; \domain_i, {\rm where} \; \paramCompArg{t}^i= \frac{\paramCompArg{\rm max}-\paramCompArg{\rm min}}{\nsubdomains-1}(i-1) + \paramCompArg{\rm min}, \; 1 \le i \le \nsubdomains$. (Namely, $\trainParam{1} \neq \trainParam{2} \neq \ldots \neq \trainParam{\nsubdomains} \neq \paramComp_{{\rm test}}$.)

We again compare the DD-LSPG method for fixed values of their parameters, and for the selected online point $\paramComp_{{\rm test}}=(5.005,5.005)$. Table~\ref{tab_ex1_oneOnlineComp_inputParams_btup_varyTrain} reports the chosen input parameters and associated performance of the DD-LSPG method, while the resulting ROM parameters over first four subdomains $\domainArg{i}, 1 \le i \le 4$ are listed on Table~\ref{tab_ex1_oneOnlineComp_ROMparams_btup_varyTrain}. Table~\ref{tab_ex1_oneOnlineComp_inputParams_btup_varyTrain} shows that the produced ROM solutions are converged and accurate. However, comparing Table~\ref{tab_ex1_oneOnlineComp_inputParams_btup} with \ref{tab_ex1_oneOnlineComp_inputParams_btup_varyTrain}, and Table~\ref{tab_ex1_oneOnlineComp_ROMparams_btup} with \ref{tab_ex1_oneOnlineComp_ROMparams_btup_varyTrain} show that predictive testing takes many more interior bases than reproductive testing with same accuracy level (and both testing cases use same number of interface bases). This is also sensible because the predictive testing case is usually more general and more challenging than the reproductive testing case.

%
%
%
%
%
%
%
%
%
%
%
%

\subsection{Parameterized Burgers' equation}\label{sec:burg}

\subsubsection{Exact solution and global FD discretization}


\begin{figure}[h!]
\centering
\begin{subfigure}[b]{0.45\textwidth}
	\includegraphics[width=8.1cm]{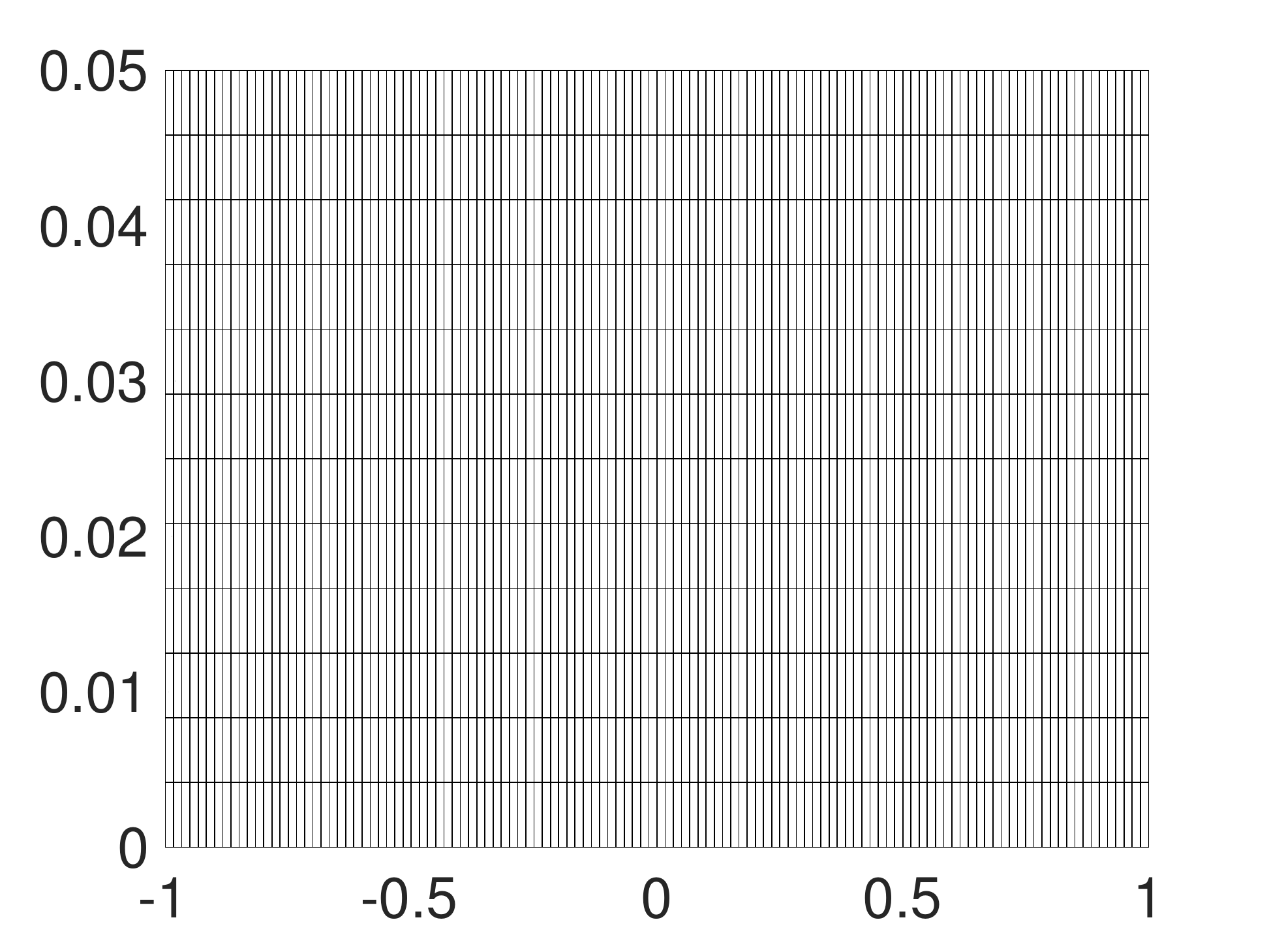}
	\caption{``Coarse'' 120x12 elements}
	\label{fig_ex2_globMesh_coarse}
\end{subfigure}
~
\begin{subfigure}[b]{0.45\textwidth}
	\includegraphics[width=8.1cm]{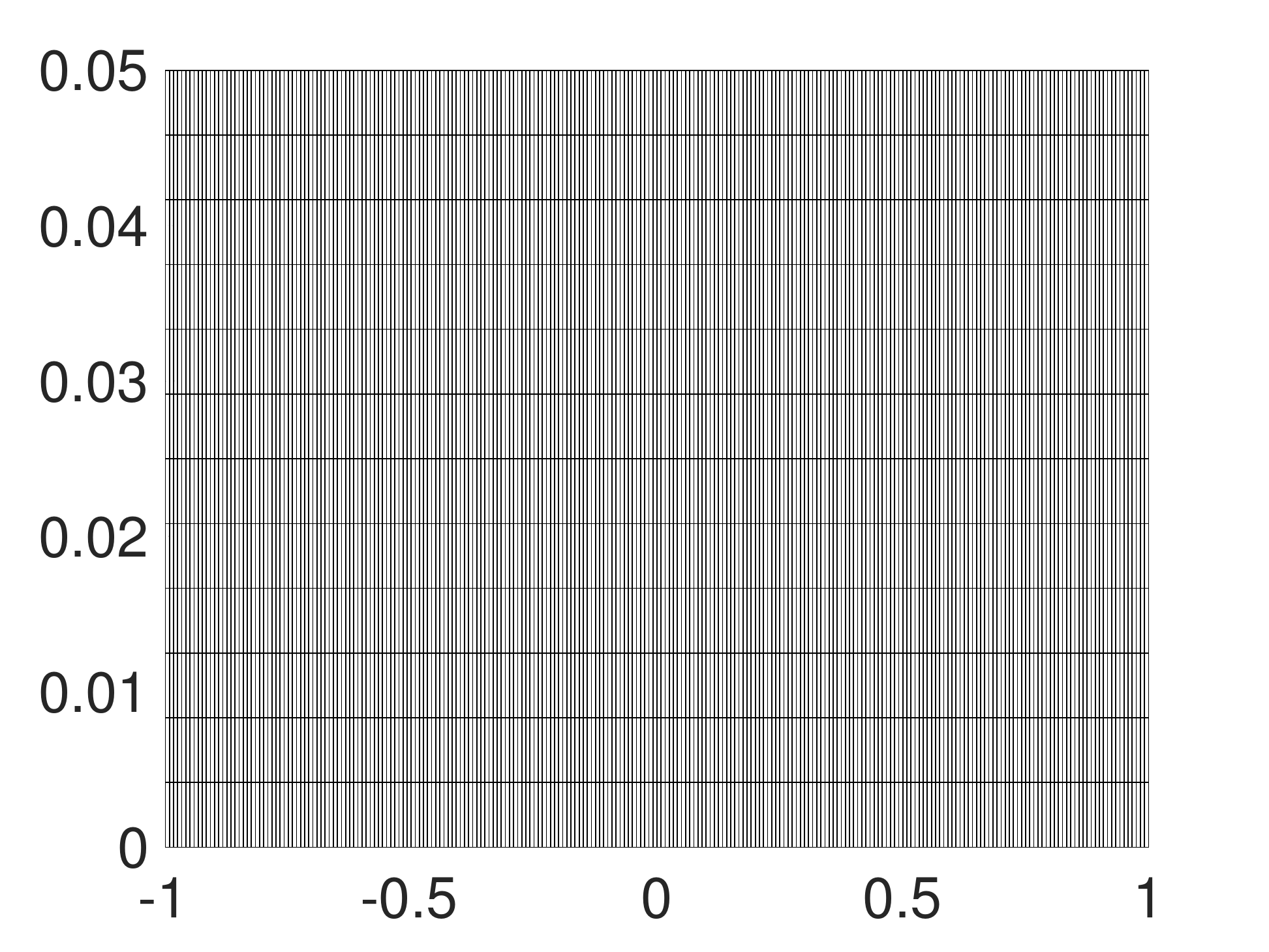}
	\caption{``Fine'' 240x12 elements}
	\label{fig_ex2_globMesh_fine}
\end{subfigure}
\caption{Burgers' equation. Two global FD mesh used for discretization.}\label{fig_ex2_globMesh}
\end{figure}


\begin{figure}[h!]
\centering
\begin{subfigure}[b]{0.45\textwidth}
	\includegraphics[width=8.1cm]{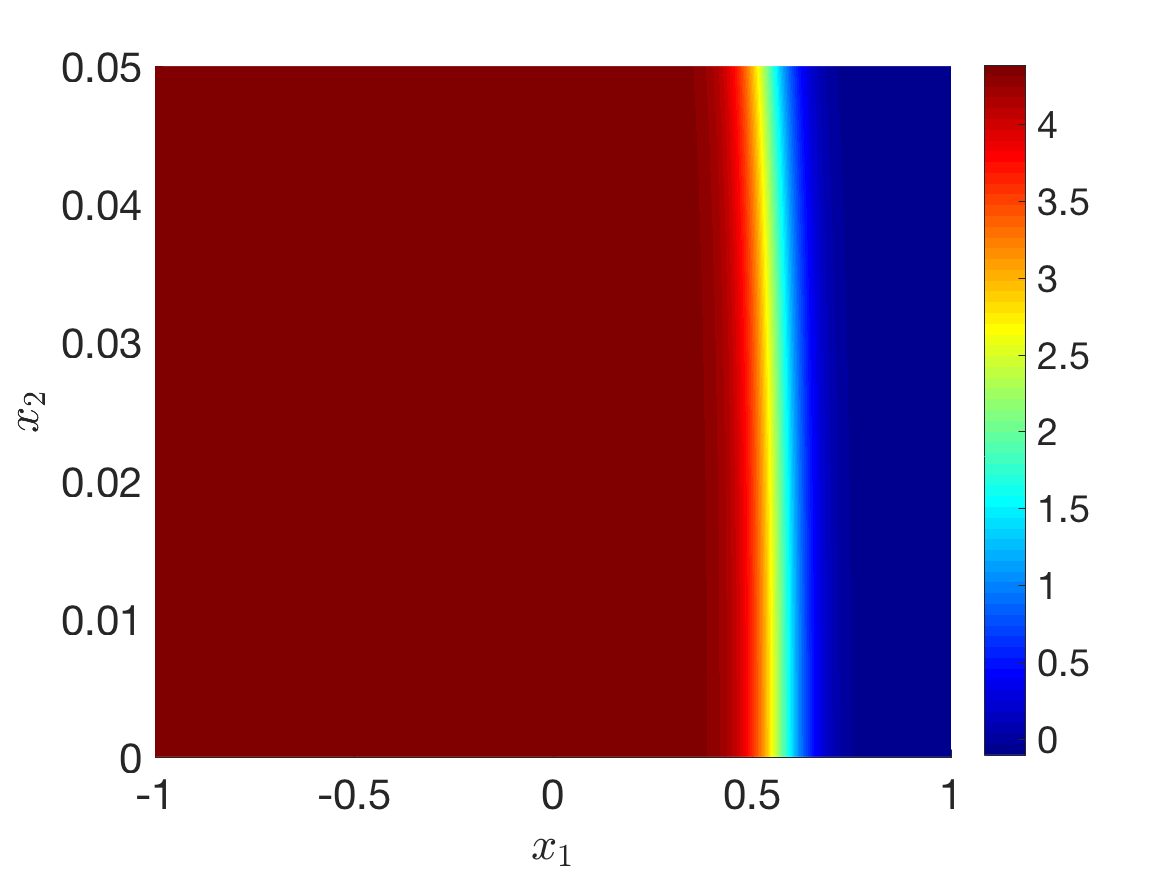}
	\caption{$u_1$-component}
	\label{fig_ex2_FDsol_mu152_u}
\end{subfigure}
~
\begin{subfigure}[b]{0.45\textwidth}
	\includegraphics[width=8.1cm]{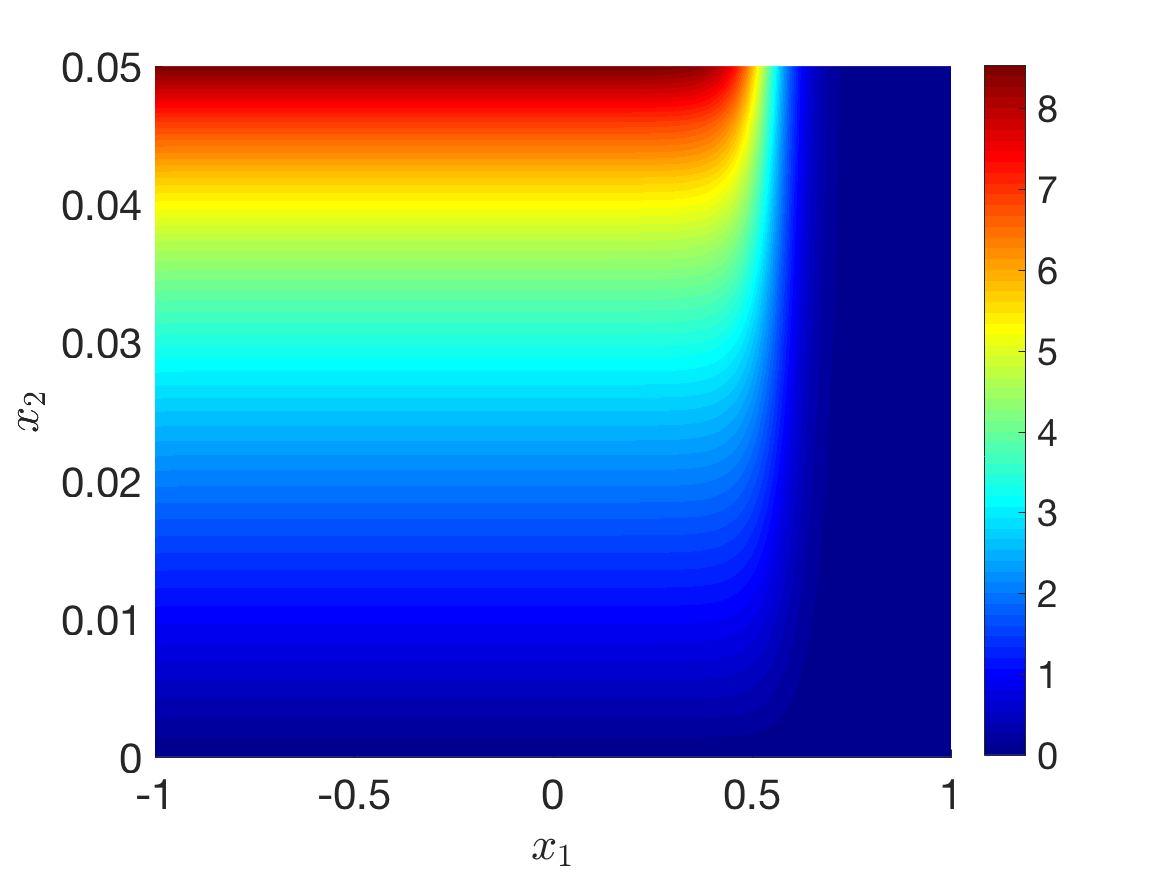}
	\caption{$u_2$-component}
	\label{fig_ex2_FDsol_mu152_v}
\end{subfigure}
\caption{Burgers' equation, global FD solution for $\mu=(7692.5384, 21.9230)$ using ``fine'' FD mesh 240x12 elements.}\label{fig_ex2_FDsols}
\end{figure}

\begin{table}[h!]
\center
\caption{Burgers equation, parameters for the exact solution} \label{tab_ex2_exactSol_params}
{\begin{tabular}{|c||c|c|c|c|c|c|c|c|}
	\hline \rule{0pt}{2.5ex}
	Parameter & $\viscosity$ &  $\stateCompArg{1}^0$  &  $a_2$  &  $a_3$  &  $a_4$  & $a_5$  & $a_1$  & $\lambda$  \\ [0.2ex]
	\hline \rule{0pt}{2.5ex}
	Value & 0.1 & 1  & $a_1$ & 0 & 0 & 1 & $\in[1,10000]$ & $\in[5,25]$ \\
	\hline
\end{tabular}}
\end{table}

\begin{table}[h!]
\center
\caption{Burgers equation, parameters for global FD discretization} 
\label{tab_ex2_globFD_params}
{\begin{tabular}{|c||c|c|}
	\hline \rule{0pt}{2.5ex}
	&  ``Coarse'' mesh  &  ``Fine'' mesh  \\ [0.2ex]
	\hline \rule{0pt}{2.5ex} \hspace{-2.5mm}
	Number of elements & 1440 & 2880  \\
	Number of nodes    & 1573 & 3133  \\
	$\ndof$			   & 2618 & 5258  \\
	\hline
\end{tabular}}
\end{table}



We now consider a parameterized 2D steady state Burgers' equation described in
Ref.~\cite{cfdblogvienna}. The problem consists of computing the velocity
field $\vel\equiv(\velCompArg{1},\velCompArg{2})$ that satisfies 
\begin{equation}\label{eq_BurgersPDE_original}
\vel \cdot \nabla \vel = \viscosity \nabla^2 \vel,
\end{equation}
where $\viscosity$ is the viscosity coefficient, $\stateComp =
(\stateCompArg{1},\stateCompArg{2}) \in \domain=[-1,1]\times[0,0.05]$. Nonhomogenous Dirichlet boundary conditions (on $\Gamma\equiv\partial\Omega$) for the numerical solutions are taken directly from the exact solution that is defined as follows 
\begin{equation}\label{eq_BurgersPDE_exactSol}
\begin{aligned}
\velComp_1 &= -2\viscosity \left[ a_2 + a_4 \stateCompArg{2} + \lambda a_5 \left( e^{\lambda(\stateCompArg{1}-\stateCompArg{1}^0)} + e^{-\lambda(\stateCompArg{1}-\stateCompArg{1}^0)} \right) \cos(\lambda \stateCompArg{2}) \right] / \Phi, 
\\
\velComp_2 &= -2\viscosity \left[ a_3 + a_4 \stateCompArg{2} - \lambda a_5 \left( e^{\lambda(\stateCompArg{1}-\stateCompArg{1}^0)} + e^{-\lambda(\stateCompArg{1}-\stateCompArg{1}^0)} \right) \sin(\lambda \stateCompArg{2}) \right] / \Phi,
\end{aligned}
\end{equation}
where $\Phi=a_1 + a_2 \stateCompArg{1} + a_3 \stateCompArg{2} + a_4
\stateCompArg{1} \stateCompArg{2} + a_5 \left(
e^{\lambda(\stateCompArg{1}-\stateCompArg{1}^0)} +
e^{-\lambda(\stateCompArg{1}-\stateCompArg{1}^0)} \right) \cos(\lambda
\stateCompArg{2})$, and $a_i, i=1,\ldots,5$, $\lambda$ and
$\stateCompArg{1}^0$ are given scalars. To parameterize the problem, these
parameters are given on Table~\ref{tab_ex2_exactSol_params}, and the input
parameter $\paramComp$ is defined as $\paramComp\equiv(\paramCompArg{1},\paramCompArg{2})=(a_1,\lambda) \in \paramDomain=[1,10000] \times [5,25]$. 




We use the finite-difference method \CH{with three-point centered difference scheme and uniform grid} to discretize
Eq.~\eqref{eq_BurgersPDE_original}. The exact solution on the boundary $\boundary$ is used as nonhomogeneous Dirichlet boundary condition to solve for the interior unknown nodes. Analogously to the previous example, we also employ a ``coarse'' and ``fine'' mesh, characterized by  1440 (120$\times$12) and 2880 (240$\times$12) quadrilateral elements, respectively. 
Figure~\ref{fig_ex2_globMesh} depicts
these meshes, while
Table~\ref{tab_ex2_globFD_params} reports the corresponding parameters. We
emphasize that this problem is characterized by two degrees of freedom per
node as opposed to the previous example.
Figure~\ref{fig_ex2_FDsols} plots the FD
reference solution on the ``fine'' mesh with $\paramComp=(7692.5384,21.9230)$.
As observed from Figure~\ref{fig_ex2_FDsols}, the solution presents a shock
which is characterized by $\paramComp=(a_1,\lambda)$, where $a_1$ relates to
the distance of the shock from the left edge and $\lambda$ relates to the steepness of the shock, respectively. 


\subsubsection{Full-order model}


\begin{figure}[h!]
\centering
\begin{subfigure}[b]{0.45\textwidth}
	\includegraphics[width=8.1cm]{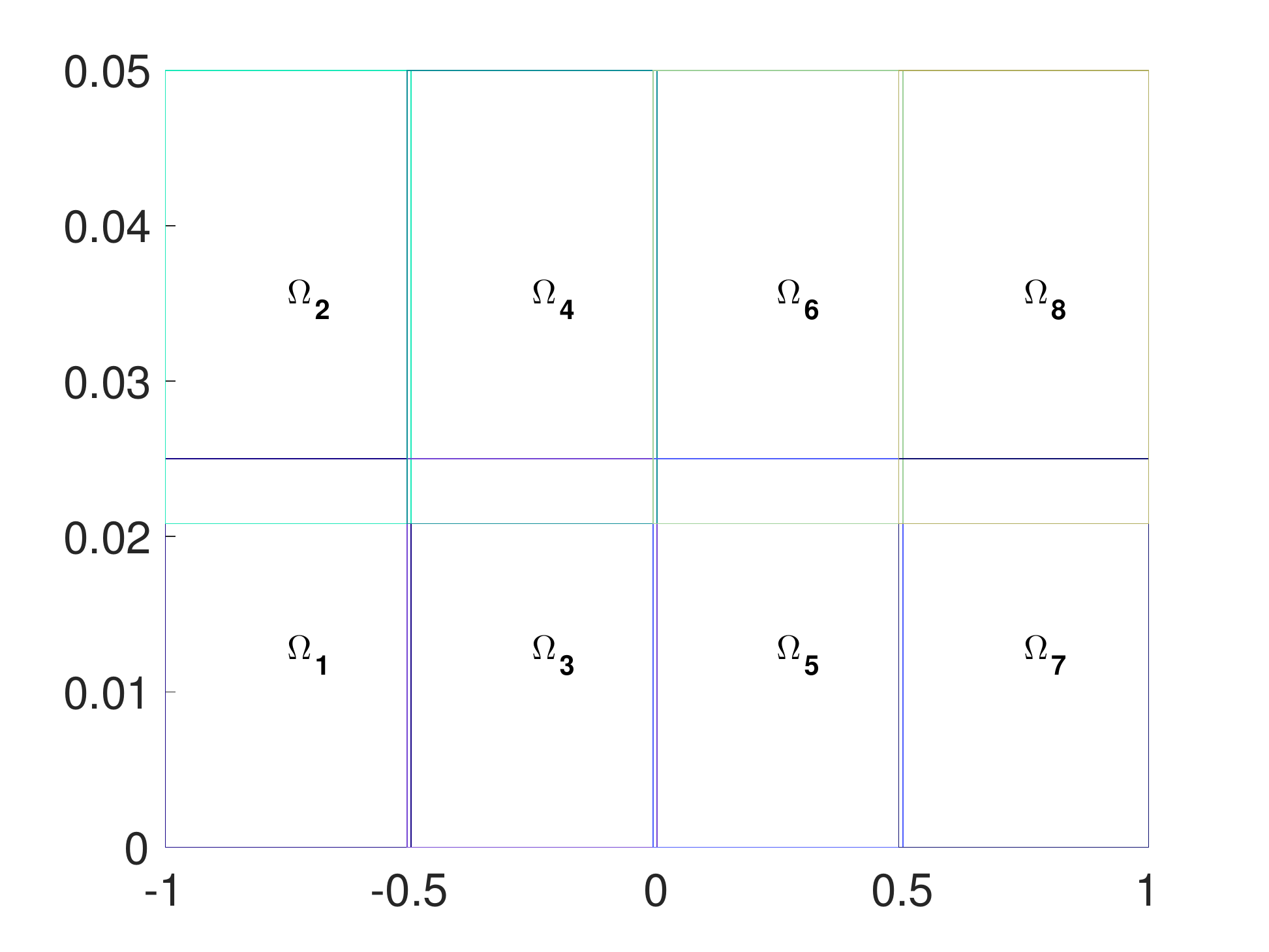}
	\caption{4x2 configuration}
	\label{fig_ex2_subdoms_4x2}
\end{subfigure}
~
\begin{subfigure}[b]{0.45\textwidth}
	\includegraphics[width=8.1cm]{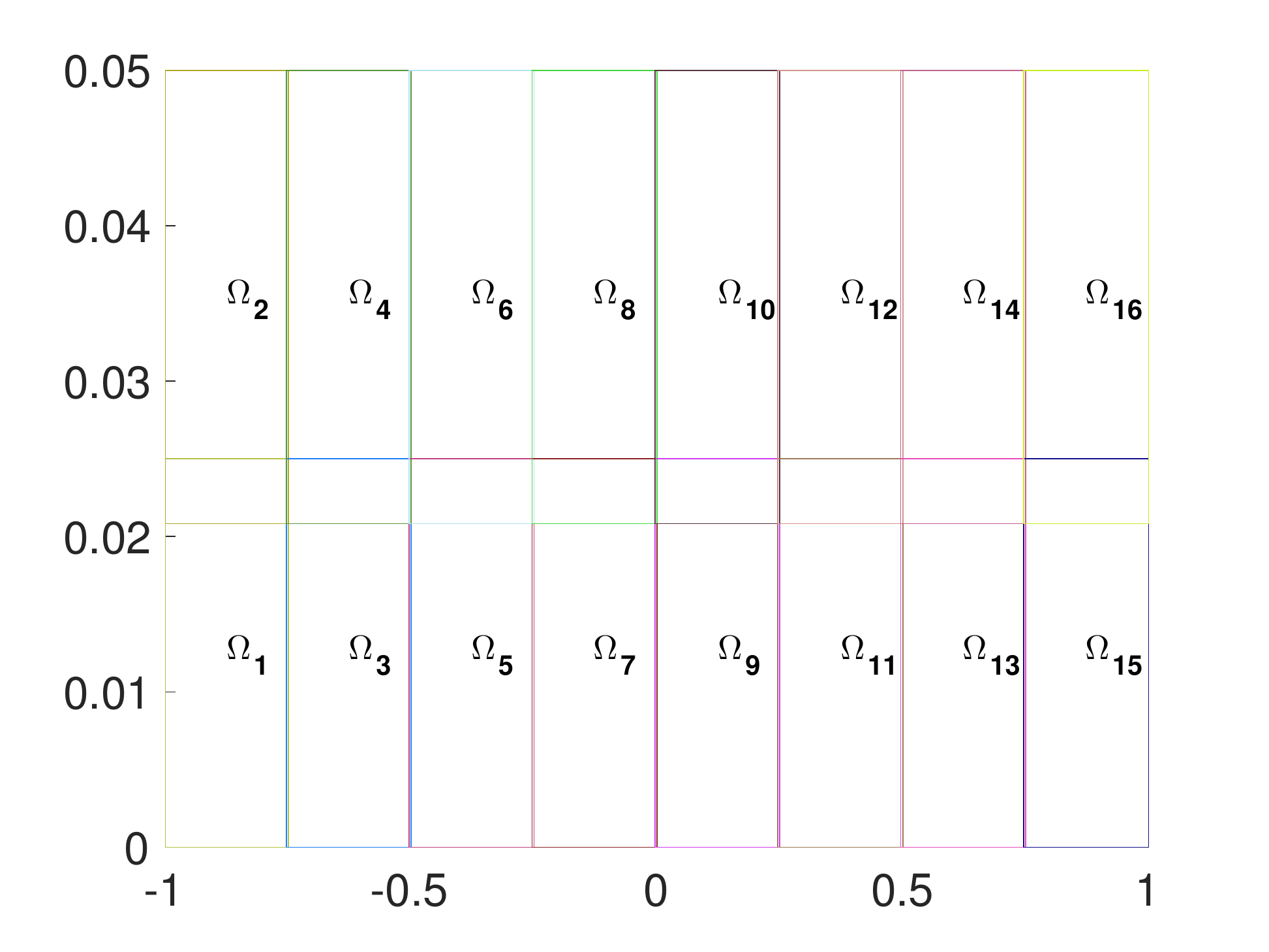}
	\caption{8x2 configuration}
	\label{fig_ex2_subdoms_8x2}
\end{subfigure}
\caption{Burgers' equation, two domain-decomposition configurations based on
	the finite-difference mesh.}\label{fig_ex2_4x2config_8x2config}
\end{figure}

\begin{table}[h!]
\center
\caption{Burgers equation, parameters used for three FOM configurations} \label{tab_ex2_DDFOMconfig_params}
{\begin{tabular}{|c||c|c|c|}
	\hline \rule{0pt}{2.5ex}
	&  4x2 ``coarse''  &  8x2 ``fine''  &  4x2 ``fine''  \\ [0.2ex]
	\hline \rule{0pt}{2.5ex} \hspace{-2.5mm}
	$\nsubdomains$  	 	 & 8   & 16   & 8    \\
	$\nconstraints$ 	 	 & 704 & 1488 & 1184 \\
	$\nports$			 	 & 13  & 29   & 13   \\
	\# DOFs  on $\domain_{1}$ & 434 & 434  & 854  \\ 
	\# nodes on $\domain_{1}$ & 217 & 217  & 427  \\ \hline
	 & \multicolumn{2}{c|}{weak scaling} & \\ \hline
	 & & \multicolumn{2}{c|}{strong scaling} \\ \hline	
\end{tabular}}
\end{table}

\begin{table}[h!]
\center
\caption{Burgers equation, parameters on each $\domaini$ for the 4x2 ``fine''
	configuration.
In this case, there are $\nports=13$ total ports with $\ndofPortsArg{1} = 16$, $\ndofPortsArg{2} = 232$, $\ndofPortsArg{3} = 20$, $\ndofPortsArg{4} = 16$, $\ndofPortsArg{5} = 232$, $\ndofPortsArg{6} = 20$, $\ndofPortsArg{7} = 16$, $\ndofPortsArg{8} = 232$, $\ndofPortsArg{9} = 20$, $\ndofPortsArg{10} = 236$, $\ndofPortsArg{11} = 8$, $\ndofPortsArg{12} = 8$, $\ndofPortsArg{13} = 8$. } \label{tab_ex2_4x2config_params}
{\begin{tabular}{|c||c|c|c|c|c|c|c|c|}
	\hline \rule{0pt}{2.5ex}
	& $\domainArg{1}$ & $\domainArg{2}$ & $\domainArg{3}$  & $\domainArg{4}$  & $\domainArg{5}$  & $\domainArg{6}$  & $\domainArg{7}$  & $\domainArg{8}$  \\ [0.2ex]
	\hline \rule{0pt}{2.5ex} \hspace{-2.5mm}
	$\nresi$  		 & 590 & 708 & 600 & 720 & 600 & 720 & 600 & 720  \\
	$\ndofInteriori$ & 464 & 580 & 464 & 580 & 464 & 580 & 472 & 590  \\
	$\ndofBoundaryi$ & 256 & 260 & 280 & 288 & 280 & 288 & 260 & 264  \\
	$\ndofi(=\ndofInteriori+\ndofBoundaryi)$ & 720 & 840 & 744 & 868 & 744 & 868 & 732 & 854  \\
	\hline \rule{0pt}{2.5ex}
	Number of subdomain ports $\card{\portsSubdomains{i}}$ &  3 &  3 &  5 &  5 &
	5 &  5 &  3 &  3   \\
	\hline
\end{tabular}}
\end{table}


After applying the finite-difference discretization, we introduce the
algebraically non-overlapping decomposition of the problem described in
Section \ref{sec_DDFOM_formulation}. As in the previous example, the chosen
algebraic decomposition corresponds to a spatial domain decomposition in
space. In particular, we employ decompositions into both $4\times 2$ (such
that $\nsubdomains=8$) and $8\times 2$ (such that $\nsubdomains=16$)
configurations as depicted in Figure \ref{fig_ex2_4x2config_8x2config}.
Table \ref{tab_ex2_DDFOMconfig_params}
lists the parameters used for each of these configurations.
The pairwise comparison of the
$4\times 2$ ``coarse'' and $8\times 2$ ``fine'' configurations is interpreted
as
weak scaling, while the pairwise comparision of the $4\times 2$ ``fine'' and $8\times 2$
``fine'' configuations interpreted as strong scaling, respectively.
For reference, Table
\ref{tab_ex2_4x2config_params} reports the parameters characterizing each subdomain
$\domainArg{i}$, $i=1,\ldots, \nsubdomains$ of the $4\times 2$ ``fine''
configuration.

\subsubsection{DD-LSPG and DD-GNAT approximations: one online computation} \label{sec:BurgersOneOnline}

\begin{table}[h!]
	\small
\center
\caption{ Burgers 4x2 ``fine'' configuration, ROM methods performance at point
	$\paramComp=(7692.5384,21.9230) \notin \trainSample$ for one online
	computation. Recall from Section
	\ref{sec_basis_construction} that $\energyCriterion\in[0,1]$ denotes the energy
	criterion employed by POD.} 
{\begin{tabular}{|c||c|c||c|c||c|c||c|c|}
	\hline \rule{0pt}{1.5ex}
	constraint	& \multicolumn{8}{c|}{strong} \\	
	\hline \rule{0pt}{2.5ex}	
	basis	& \multicolumn{2}{c||}{port} & \multicolumn{2}{c||}{skeleton} &
	\multicolumn{2}{c||}{full-interface} & \multicolumn{2}{c|}{subdomain} \\ [0.3ex]
	\hline \rule{0pt}{2.5ex}
	method  &  DD-LSPG  &  DD-GNAT &  DD-LSPG  &  DD-GNAT &  DD-LSPG  &  DD-GNAT
	&  DD-LSPG  &  DD-GNAT \\ [0.3ex]
	\hline \rule{0pt}{2.5ex} 
	$\energyCriterion$ for state   & $1-10^{-4}$ & $1-10^{-4}$ & $1-10^{-4}$ & $1-10^{-4}$ & $1-10^{-4}$ & $1-10^{-4}$ & $1-10^{-5}$ & $1-10^{-5}$ \\
	$\energyCriterion$ for residual &    & $1-10^{-8}$ &   & $1-10^{-8}$ &   	& $1-10^{-8}$  &  & $1-10^{-8}$ \\
	$\nsampleArg{i}/\nrbResi$		  &  & 2 &  & 2 &  & 2  &  & 2  \\
	\hline \rule{0pt}{2.5ex} \hspace{-2.5mm}	
	rel. error  & 0.0073 & 0.0077 & 0.0109 & 0.0108 & 0.7982 & 0.7917  & 1.0000 & 1.0000 \\
	speedup  	& 15.62  & 23.06  & 13.93  & 20.30  & 12.51  & 18.48  & 31.17  & 46.80 \\
	\hline
\end{tabular}} \\ [1ex]
\label{tab_ex2_oneOnlineComp_inputParams}
\end{table}

\begin{table}[h!]
\center
\caption{Burgers $4\times 2$ ``fine'' configuration, ROM parameters on first four $\domaini (1 \le i \le 4)$, resulting from Table~\ref{tab_ex2_oneOnlineComp_inputParams}.}
\label{tab_ex2_oneOnlineComp_ROMparams}
{\begin{tabular}{|c||c|c|c|c||c|c|c|c||c|c|c|c||c|c|c|c|}
	\hline \rule{0pt}{2.5ex}
	basis	& \multicolumn{4}{c||}{port} & \multicolumn{4}{c||}{skeleton} &
	\multicolumn{4}{c||}{full-interface} & \multicolumn{4}{c|}{subdomain}  \\
	\hline \rule{0pt}{2.5ex}
	& $\domainArg{1}$ & $\domainArg{2}$ & $\domainArg{3}$  & $\domainArg{4}$ & $\domainArg{1}$ & $\domainArg{2}$ & $\domainArg{3}$  & $\domainArg{4}$ & $\domainArg{1}$ & $\domainArg{2}$ & $\domainArg{3}$  & $\domainArg{4}$  & $\domainArg{1}$ & $\domainArg{2}$ & $\domainArg{3}$  & $\domainArg{4}$  \\ [0.2ex]
	\hline \rule{0pt}{2.5ex} \hspace{-2.5mm}
	$\nconstraintsROM$ 	& \multicolumn{4}{c||}{69} & \multicolumn{4}{c||}{84} & 
	\multicolumn{4}{c||}{48}& \multicolumn{4}{c|}{78} 
	\\	\hline
	$\nrbInteriori$ & 3 & 4 & 4 & 5 & 3 & 4 & 4 & 5 & 3 & 4 & 4 & 5 & 5 & 6 & 8 & 9 
	\\
	$\nrbBoundaryi$ & 9 & 10 & 17 & 19 & 12 & 12 & 12 & 12 & 3 & 4 & 5 & 6 & 5 & 6 & 8 & 9 
	\\  \hline \rule{0pt}{2.5ex}
	$\nrbPortArg{1}$ 	& 3 & 3 & 3 & 4  &    &    &    &    &    &    &    &  &  &  &  &  
	\\
	$\nrbPortArg{2}$ 	& 3 & 3 & 3 & 3  &    &    &    &    &    &    &    &  &  &  &  &   
	\\
	$\nrbPortArg{3}$ 	& 3 & 4 & 5 & 5  &    &    &    &    &    &    &    &  &  &  &  &   
	\\
	$\nrbPortArg{4}$ 	&   &   & 3 & 3  &    &    &    &    &    &    &    &  &  &  &  &   
	\\	
	$\nrbPortArg{5}$ 	&   &   & 3 & 4  &    &    &    &    &    &    &    &  &  &  &  &   
	\\	\hline \rule{0pt}{2.5ex}
	$\nsampleArg{i}$ & 80 & 88 & 108 & 124 & 80 & 88 & 108 & 124 & 80 & 88 & 108 & 124 & 80 & 88 & 108 & 124 
	\\
	$\nrbResi$ 		 & 40 & 44 &  54 &  62 & 40 & 44 &  54 &  62 & 40 & 44 &  54 &  62 & 40 & 44 &  54 &  62 
	\\  \hline
\end{tabular}}
\end{table}

\begin{figure}[h!]
\centering
\begin{subfigure}[b]{0.3\textwidth}
\includegraphics[width=5.5cm]{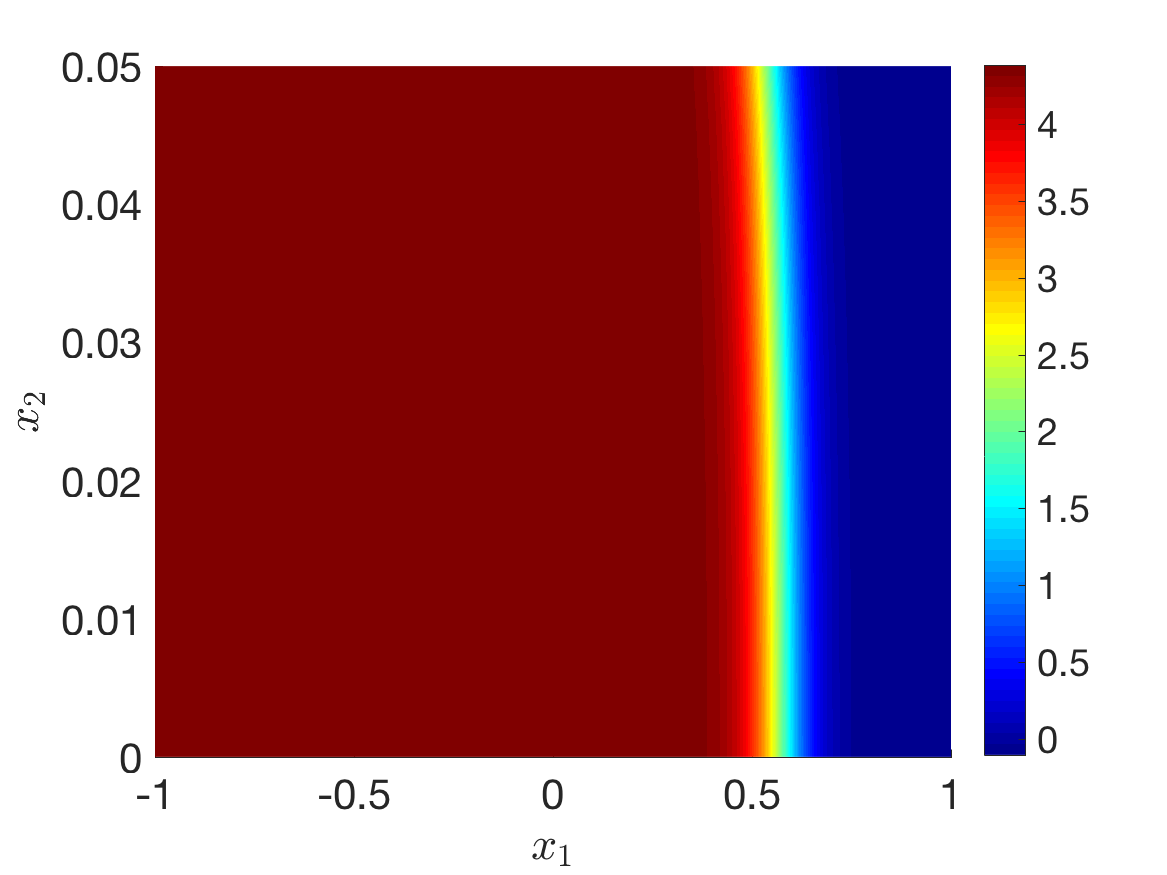}
\caption{FOM solution ($u_1$)}
\label{fig_ex2_bg42fn_solU_DDFOM_Omega}
\end{subfigure}
~
\begin{subfigure}[b]{0.3\textwidth}
\includegraphics[width=5.5cm]{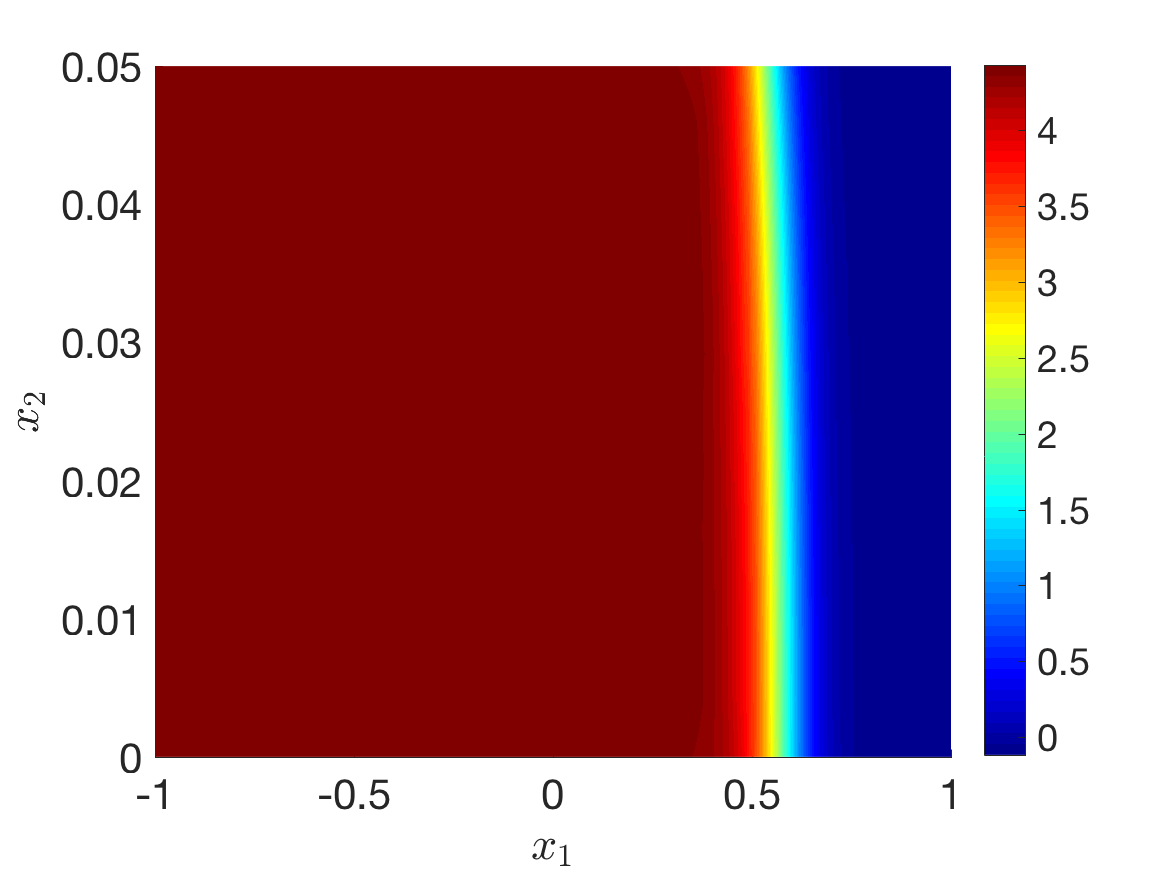}
\caption{DD-LSPG solution ($u_1$)}
\label{fig_ex2_bg42fn_solU_DDROM_Omega}
\end{subfigure}
~
\begin{subfigure}[b]{0.3\textwidth}
\includegraphics[width=5.5cm]{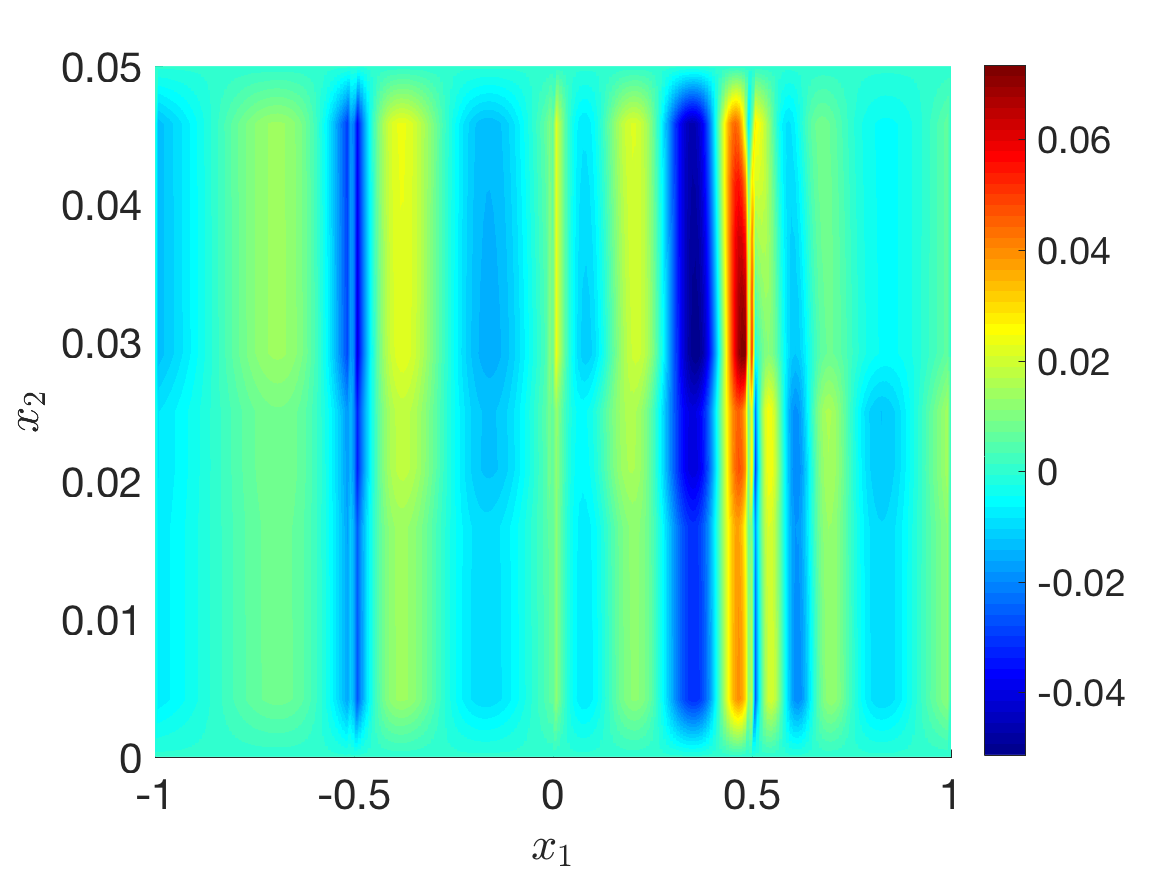}
\caption{DD-LSPG error ($u_1$)}
\label{fig_ex2_bg42fn_solU_DDROMerr_Omega}
\end{subfigure}
~
\begin{subfigure}[b]{0.3\textwidth}
\includegraphics[width=5.5cm]{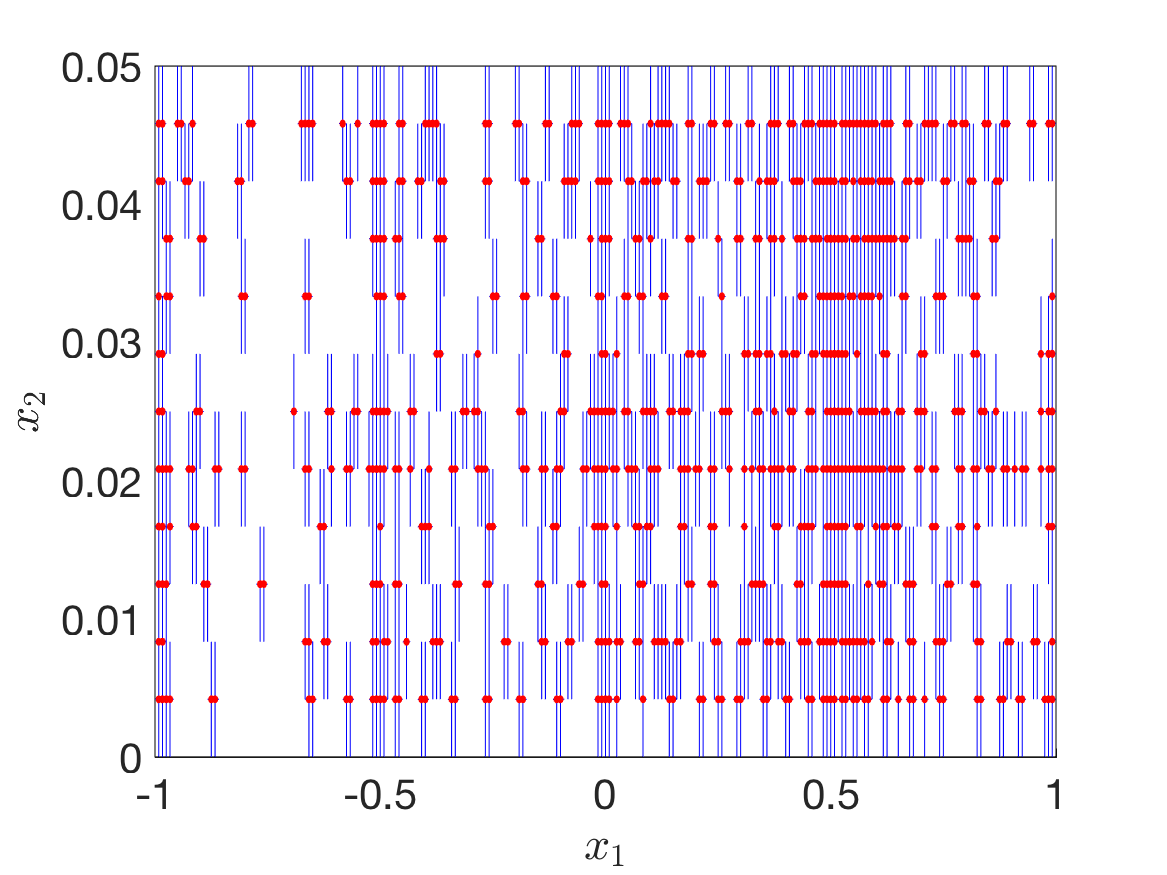}
\caption{Sample mesh ($u_1$)}
\label{fig_ex2_bg42fn_solU_sampleMesh_Omega}
\end{subfigure}
~
\begin{subfigure}[b]{0.3\textwidth}
\includegraphics[width=5.5cm]{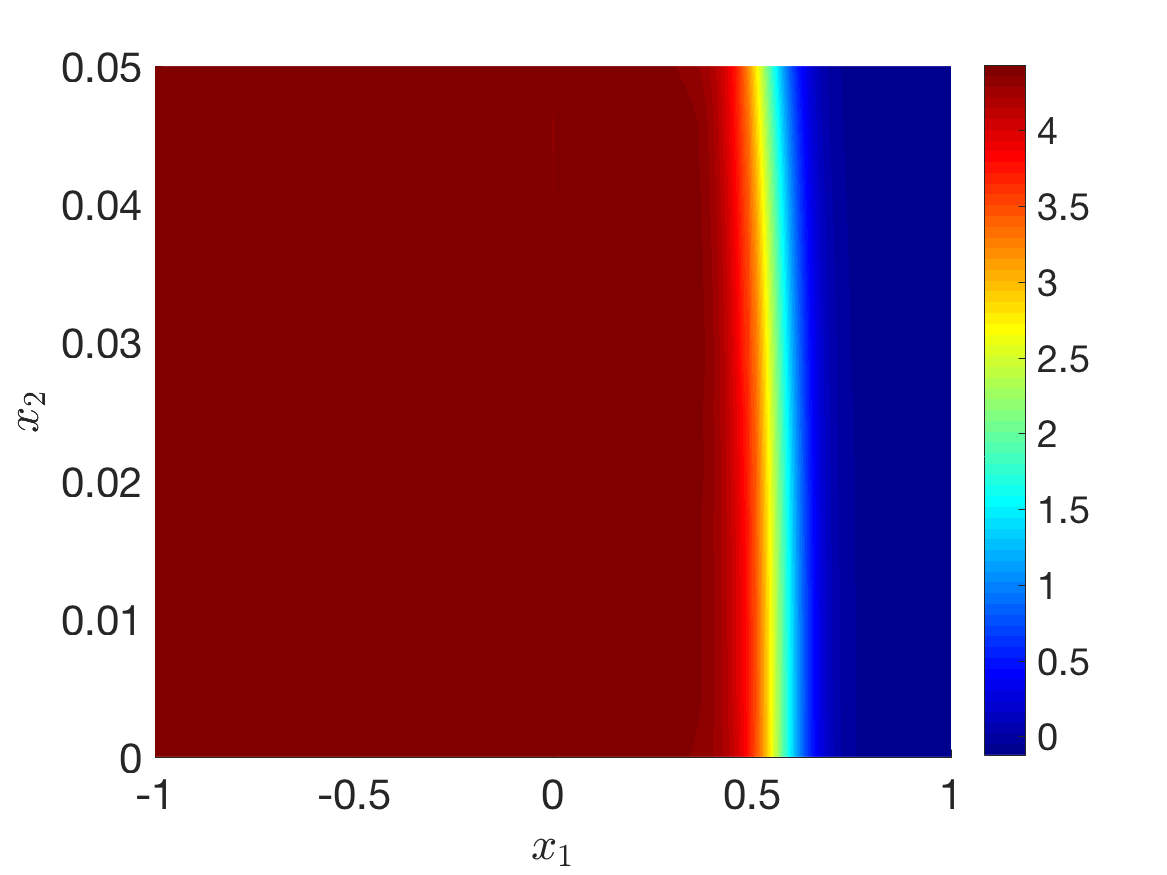}
\caption{DD-GNAT solution ($u_1$)}
\label{fig_ex2_bg42fn_solU_DDGNAT_Omega}
\end{subfigure}
~
\begin{subfigure}[b]{0.3\textwidth}
\includegraphics[width=5.5cm]{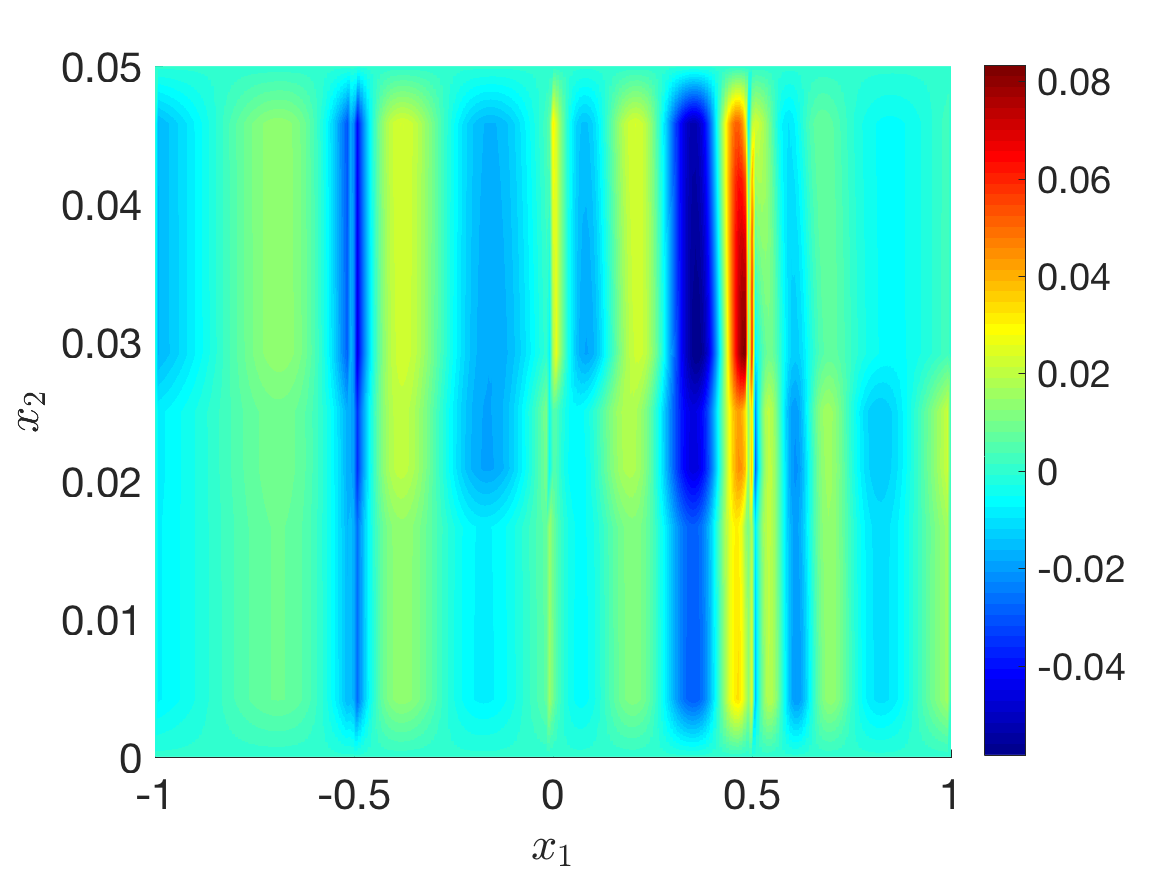}
\caption{DD-GNAT error ($u_1$)}
\label{fig_ex2_bg42fn_solU_DDGNATerr_Omega}
\end{subfigure}	
~
\begin{subfigure}[b]{0.3\textwidth}
\includegraphics[width=5.5cm]{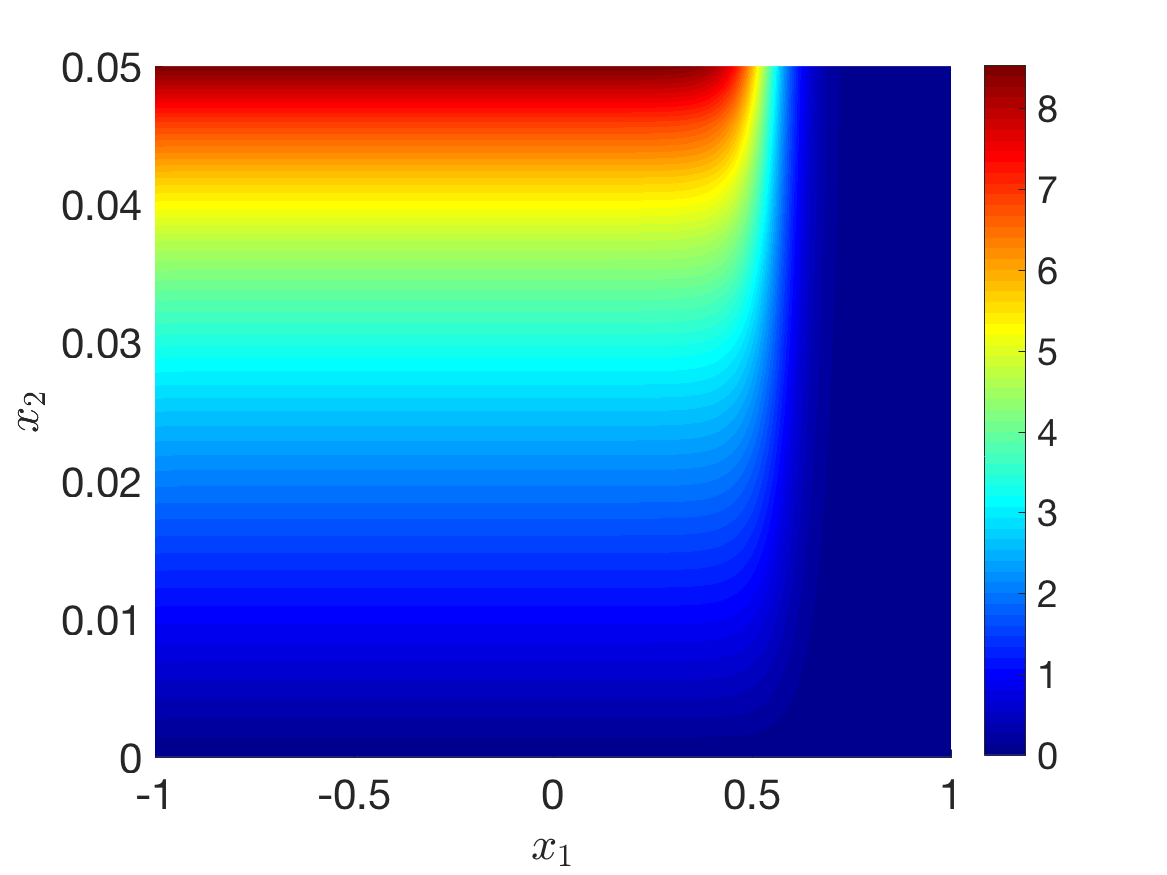}
\caption{FOM solution ($u_2$)}
\label{fig_ex2_bg42fn_solV_DDFOM_Omega}
\end{subfigure}
~
\begin{subfigure}[b]{0.3\textwidth}
\includegraphics[width=5.5cm]{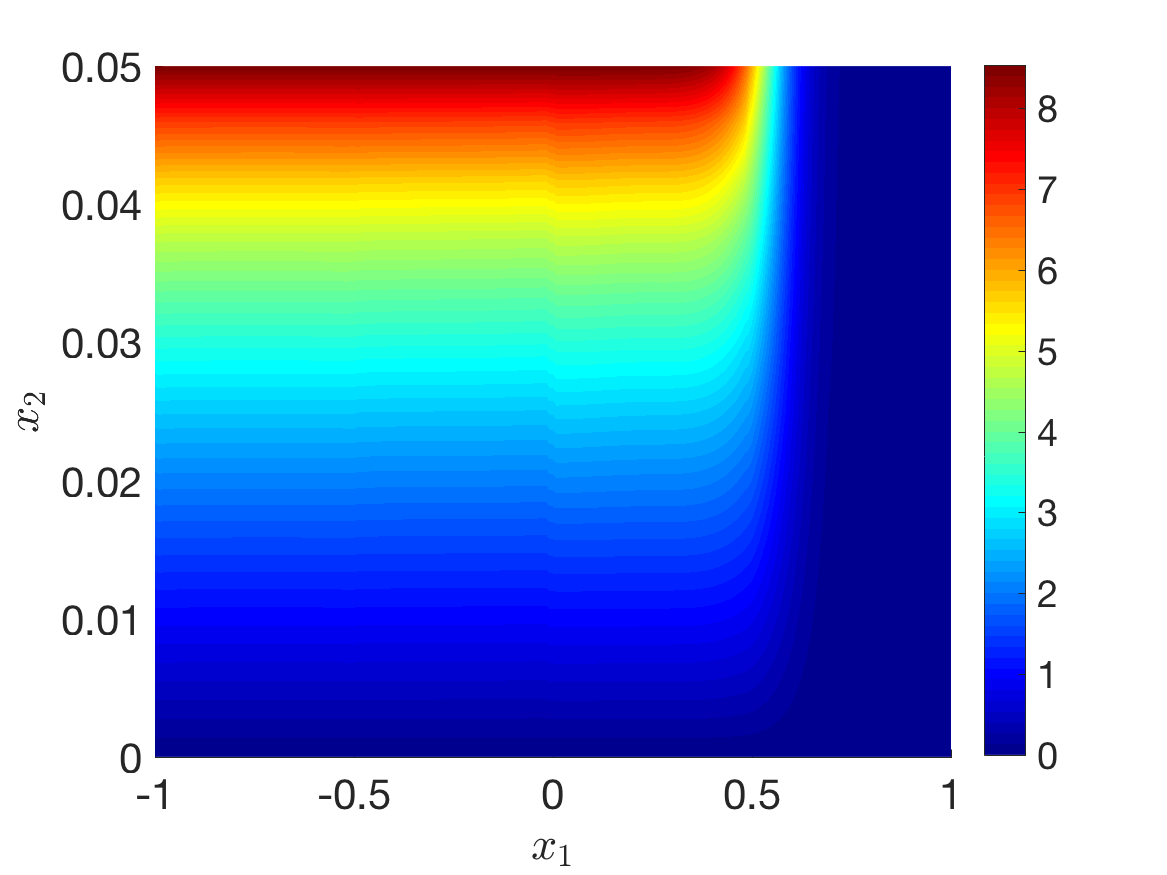}
\caption{DD-LSPG solution ($u_2$)}
\label{fig_ex2_bg42fn_solV_DDROM_Omega}
\end{subfigure}
~
\begin{subfigure}[b]{0.3\textwidth}
\includegraphics[width=5.5cm]{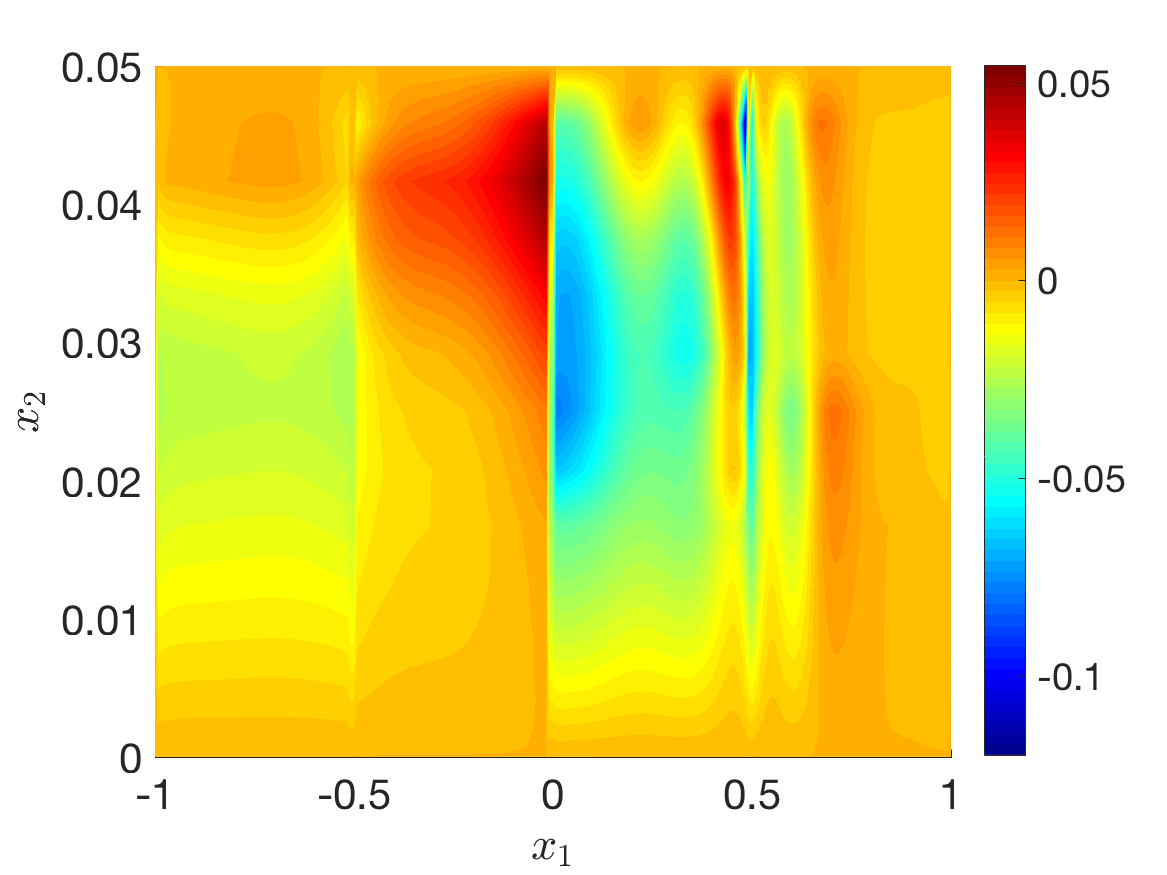}
\caption{DD-LSPG error ($u_2$)}
\label{fig_ex2_bg42fn_solV_DDROMerr_Omega}
\end{subfigure}	
~
\begin{subfigure}[b]{0.3\textwidth}
\includegraphics[width=5.5cm]{figures_1/bg42fn_sol_sampleMesh_Omega.png}
\caption{Sample mesh ($u_2$)}
\label{fig_ex2_bg42fn_solV_sampleMesh_Omega}
\end{subfigure}
~
\begin{subfigure}[b]{0.3\textwidth}
\includegraphics[width=5.5cm]{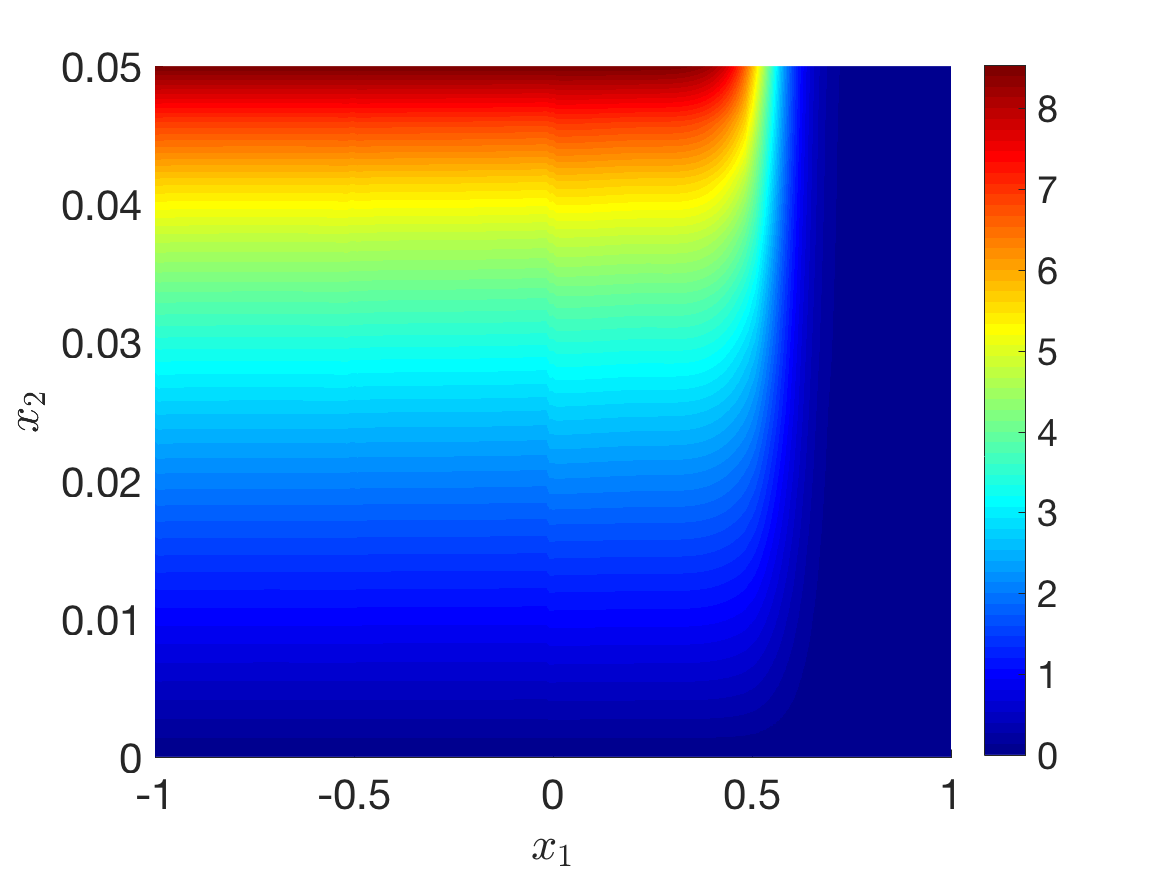}
\caption{DD-GNAT solution ($u_2$)}
\label{fig_ex2_bg42fn_solV_DDGNAT_Omega}
\end{subfigure}
~
\begin{subfigure}[b]{0.3\textwidth}
\includegraphics[width=5.5cm]{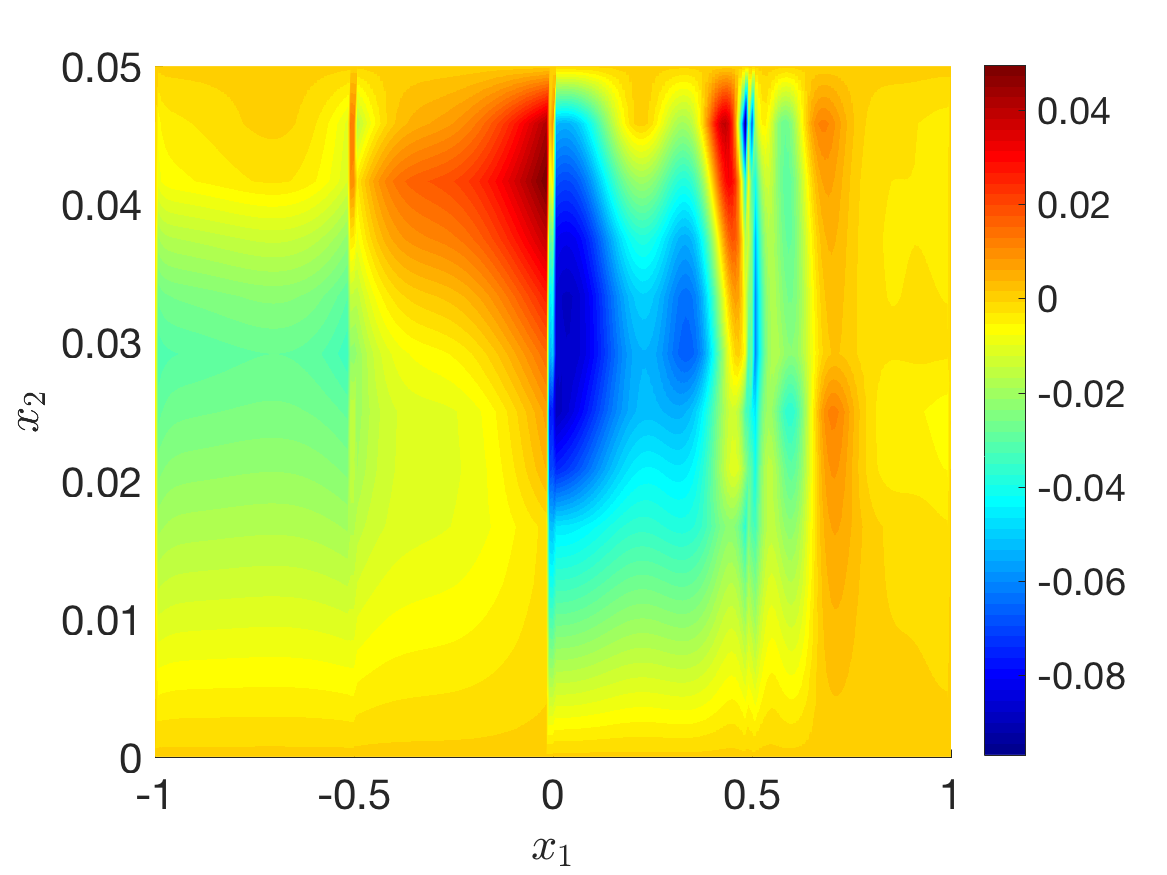}
\caption{DD-GNAT error ($u_2$)}
\label{fig_ex2_bg42fn_solV_DDGNATerr_Omega}
\end{subfigure}	
\caption{Burgers equation, 4x2 ``fine'' configuration, port bases in
Table~\ref{tab_ex2_oneOnlineComp_inputParams}, solutions visualized on $\domain$.} \label{fig_ex2_bg42fn_sol_Omega}
\end{figure}


We apply the same procedure described in Section \ref{sec:heatOneOnline} to generate the reduced bases required for the reduced-order models. In particular, we solve the FOM \eqref{eq_originalAlgebraic} for $\param\in\trainSample\subset\paramDomain$, where we again define the training-parameter set $\trainSample$ according to a $20\times 20$ equispaced sampling of the parameter domain $\paramDomain$, yielding $\numTrainSample=400$ samples. We then apply the methods described in Section \ref{sec_basis_construction} to create port, skeleton, full-interface, and full-subdomain bases from these training data. At each iteration of the Newton--Raphson algorithm used to solve the FOM equations \eqref{eq_originalAlgebraic}, the residual vector is saved, and the resulting residual snapshots are employed to generate the residual bases $\rbResi$, $i=1,\ldots,\nsubdomains$
that are used by DD-GNAT via POD. Lastly, the GNAT offline algorithm~\ref{alg_Greedy_GNAT} is performed to create sample meshes $\sampleSetOfSampleMesh{i}$ for all subdomains $\domainArg{i}$.


We now compare the methods DD-LSPG and GNAT for fixed values of their parameters, and
for the randomly selected online point $\paramComp=(7692.5384,21.9230) \notin
\trainSample$. Table~\ref{tab_ex2_oneOnlineComp_inputParams} reports the chosen input parameters and associated performance of the methods, \CH{while the resulting ROM parameters over first four subdomains $\domain_i$ are listed on Table~\ref{tab_ex2_oneOnlineComp_ROMparams}.} Again, the results on Table~\ref{tab_ex2_oneOnlineComp_inputParams} confirm the comments in Remark~\ref{rem:globalSol}, which suggested that enforcing strong compatibility can yield poor results for full-interface and full-subdomain bases, and that only port and skeleton bases are well-suited for strong compatibility constraints. Figure \ref{fig_ex2_bg42fn_sol_Omega} visualizes the DD-LSPG and DD-GNAT solutions for the port-bases case; it shows that DD-LSPG and DD-GNAT yield accurate results for port bases with strong constraints as anticipated.




\subsubsection{DD-LSPG and DD-GNAT approximations: parameter study}

\begin{table}[h!]
\center
\caption{Burgers
	equation, ROM-method parameters limits for parameter study (skel.=skeleton,
	intf.=full-interface, subdom.=subdomain). Recall from Section
	\ref{sec_basis_construction} that $\energyCriterion\in[0,1]$ denotes the energy
	criterion employed by POD.} \label{tab_ex2_manyOnlineComputation}
{\begin{tabular}{|c||c|c|}
	\hline \rule{0pt}{2.5ex}
	method  &  DD-LSPG  &  GNAT  \\ [0.2ex]
	\hline \rule{0pt}{2.5ex} 
	$\energyCriterion$ on $\domaini$ for interior/boundary bases   & $\{1-10^{-4}, 1-10^{-7}\}$ & $\{1-10^{-4}, 1-10^{-7}\}$  \\
	$\energyCriterion$ on $\boundaryi$ for interior/boundary bases & $\{1-10^{-4}, 1-10^{-7}\}$ & $\{1-10^{-4}, 1-10^{-7}\}$  \\
 & & \\	
		$\energyCriterion$ for full-subdomain bases & $\{1-10^{-5}, 1-10^{-6}, $ & $\{1-10^{-5}, 1-10^{-6}, $ \\
	& $1-10^{-7}, 1-10^{-9}\}$ & $1-10^{-7}, 1-10^{-9}\}$ \\
 & & \\	
		$\energyCriterion$ for $\residuali$ &  & $\{1-10^{-4}, 1-10^{-6}, $ \\
	&  & $1-10^{-8}, 1-10^{-10}\}$ \\
 & & \\	
	$\nsampleArg{i}/\nrbResi$ 		    &  &   \{1, 1.5, 2, 4\}  \\
	number of constraints	   		    & \{1, 2, 3, 4, 5, strong\}	 &  \{1, 2, 3, 4, 5, strong\}  \\
	basis types	 & \{port, skel., intf., subdom.\} & \{port, skel., intf., subdom.\}  \\
	\hline
\end{tabular}}
\end{table}

\begin{figure}[h!]
\centering
\begin{subfigure}[b]{0.3\textwidth}
\includegraphics[width=5.5cm]{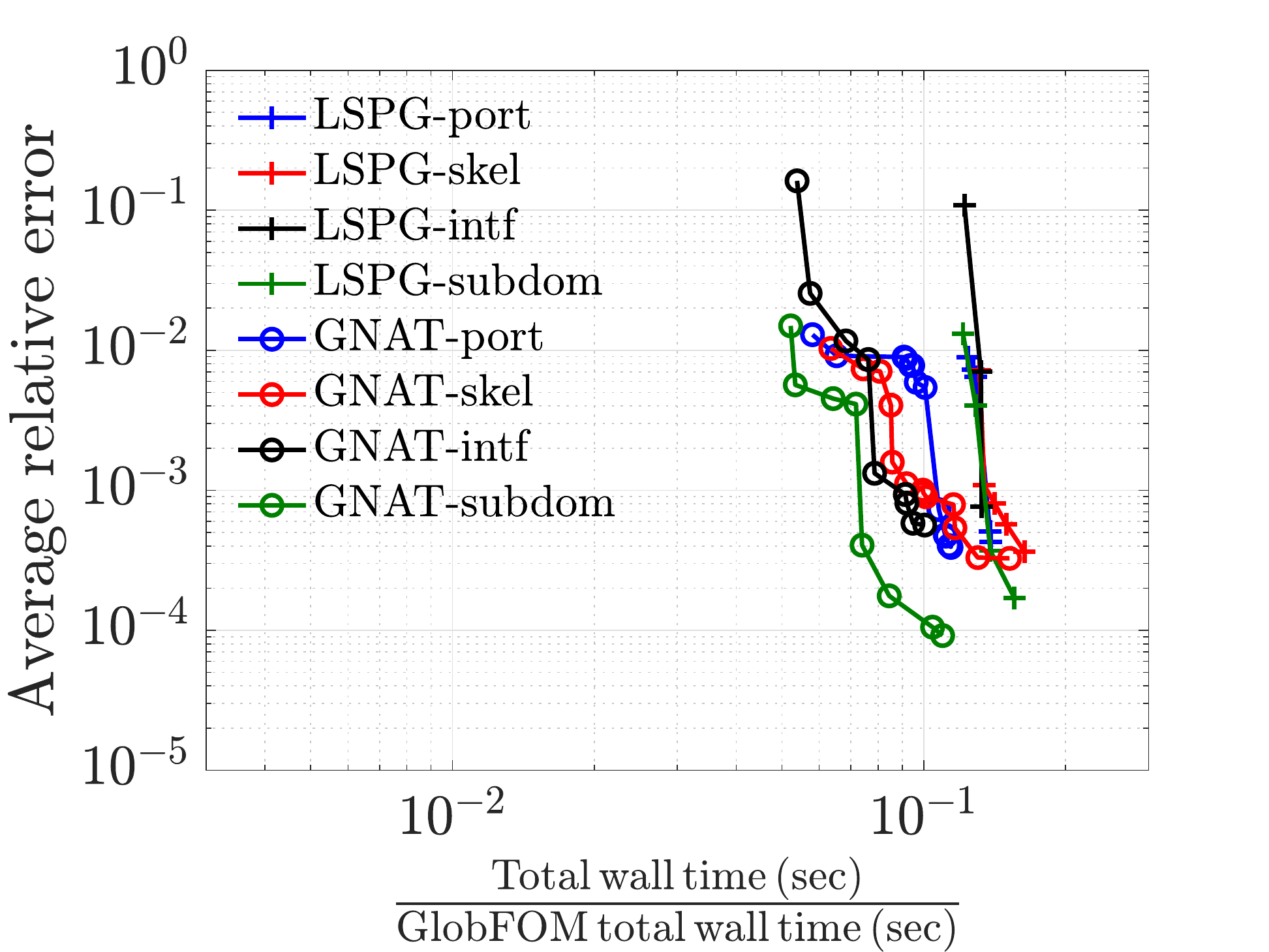}
\caption{Burger 4x2 ``coarse''}
\label{fig_ex2_bg42ce_pareto_wallAll}
\end{subfigure}
~
\begin{subfigure}[b]{0.3\textwidth}
\includegraphics[width=5.5cm]{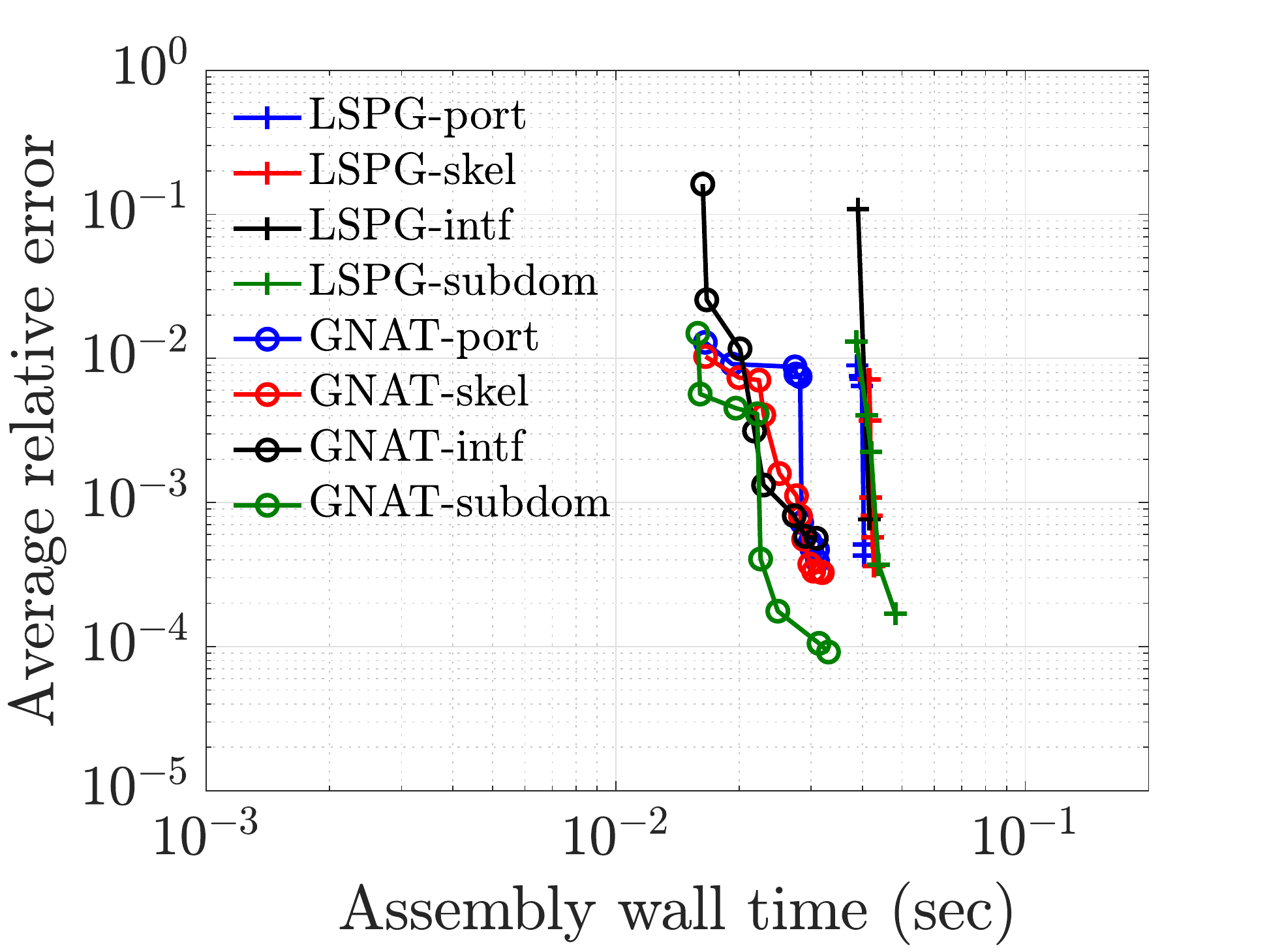}
\caption{Burger 4x2 ``coarse''}
\label{fig_ex2_bg42ce_pareto_wallAsmb}
\end{subfigure}
~
\begin{subfigure}[b]{0.3\textwidth}
\includegraphics[width=5.5cm]{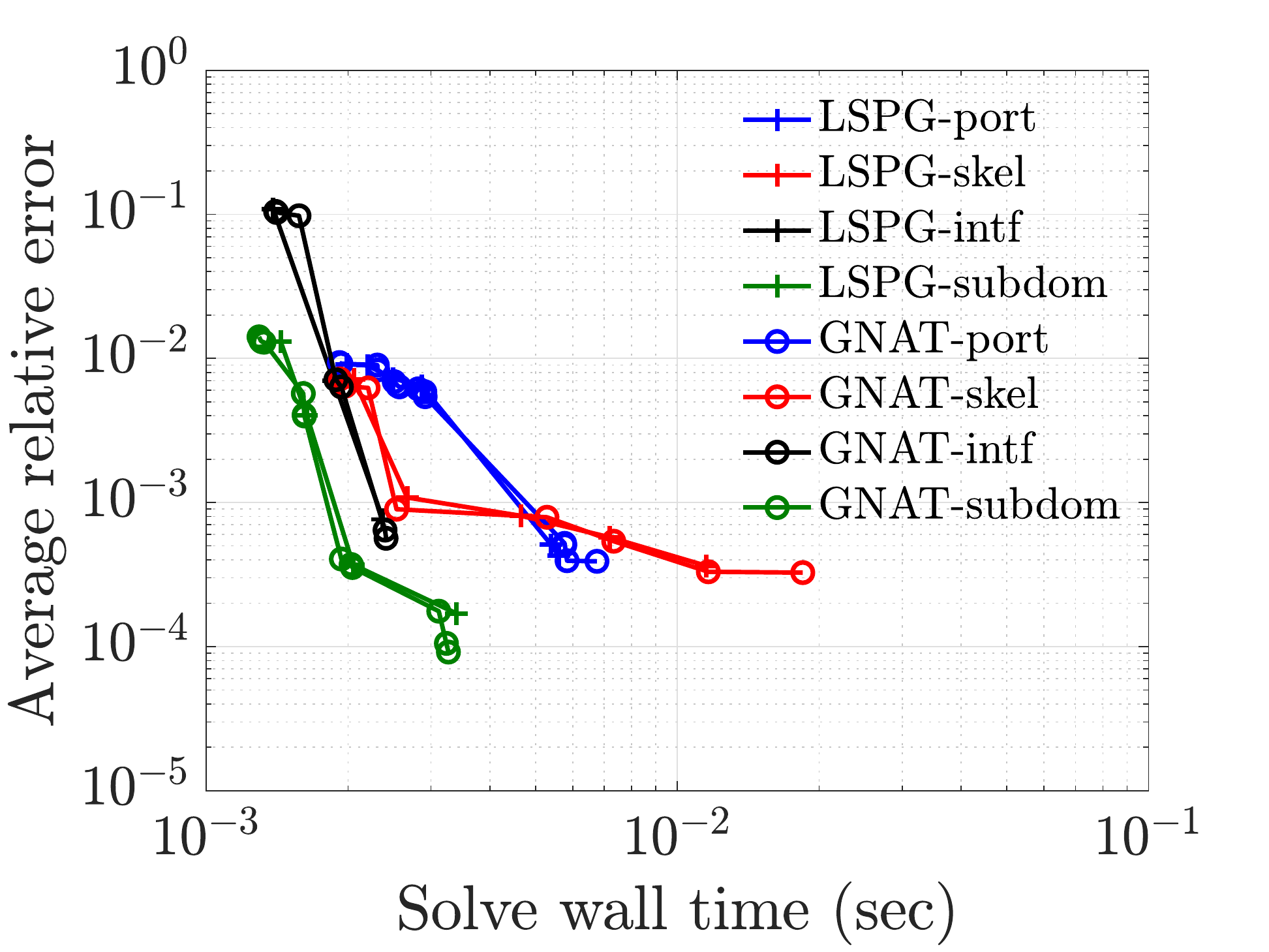}
\caption{Burger 4x2 ``coarse''}
\label{fig_ex2_bg42ce_pareto_wallSolv}
\end{subfigure}
~
\begin{subfigure}[b]{0.3\textwidth}
\includegraphics[width=5.5cm]{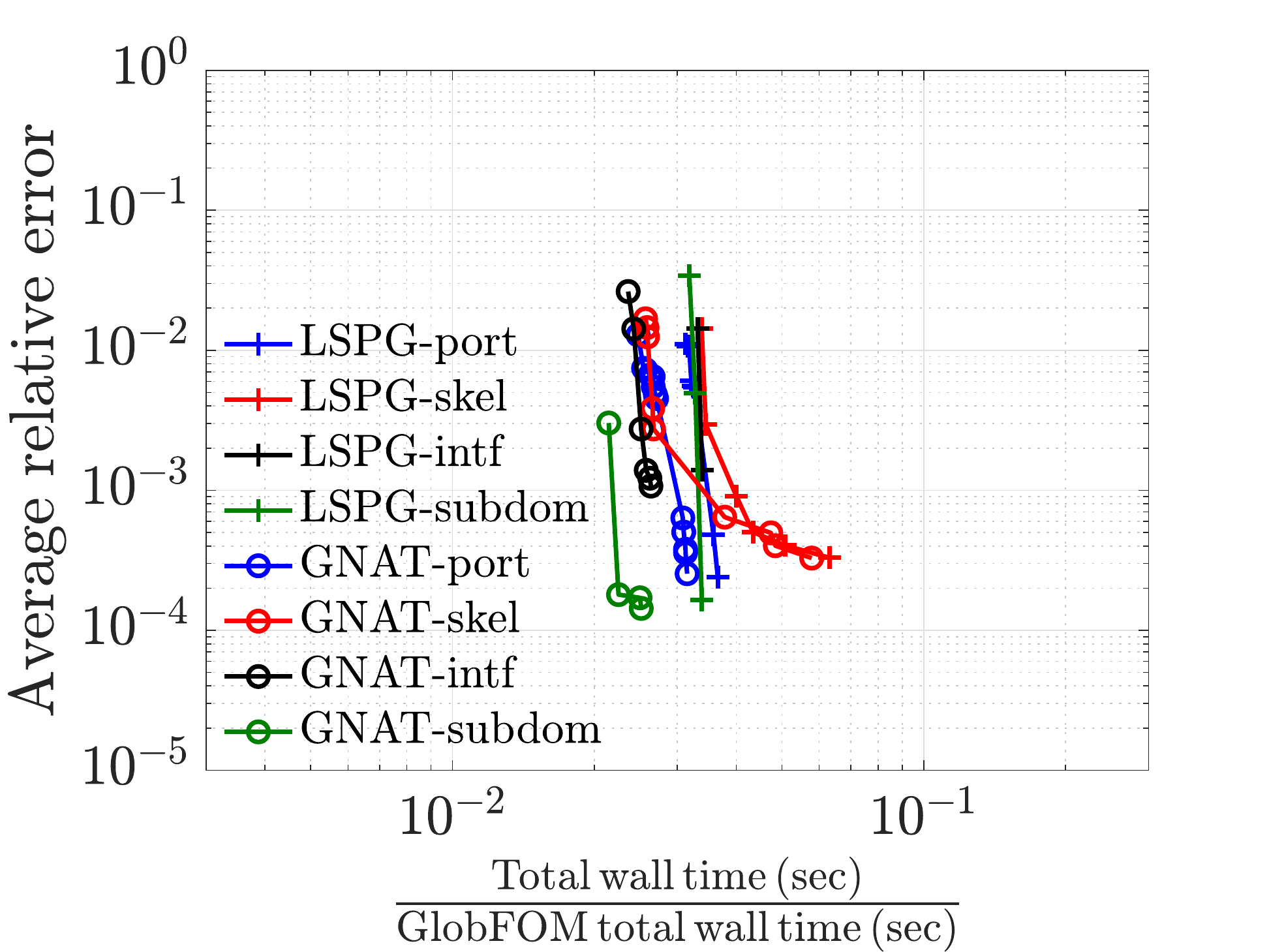}
\caption{Burger 8x2 ``fine''}
\label{fig_ex2_bg82fn_pareto_wallAll}
\end{subfigure}
~
\begin{subfigure}[b]{0.3\textwidth}
\includegraphics[width=5.5cm]{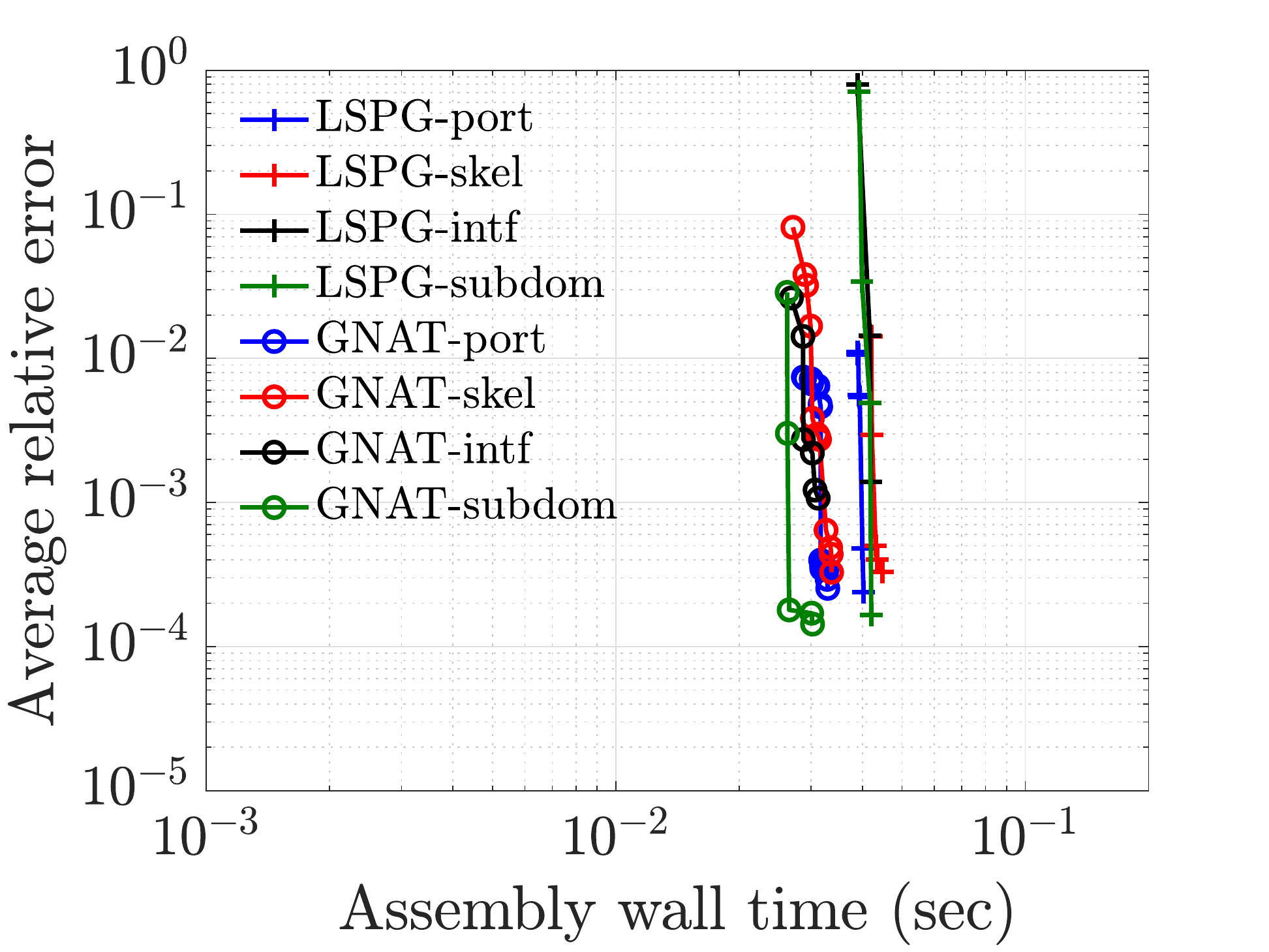}
\caption{Burger 8x2 ``fine''}
\label{fig_ex2_bg82fn_pareto_wallAsmb}
\end{subfigure}
~
\begin{subfigure}[b]{0.3\textwidth}
\includegraphics[width=5.5cm]{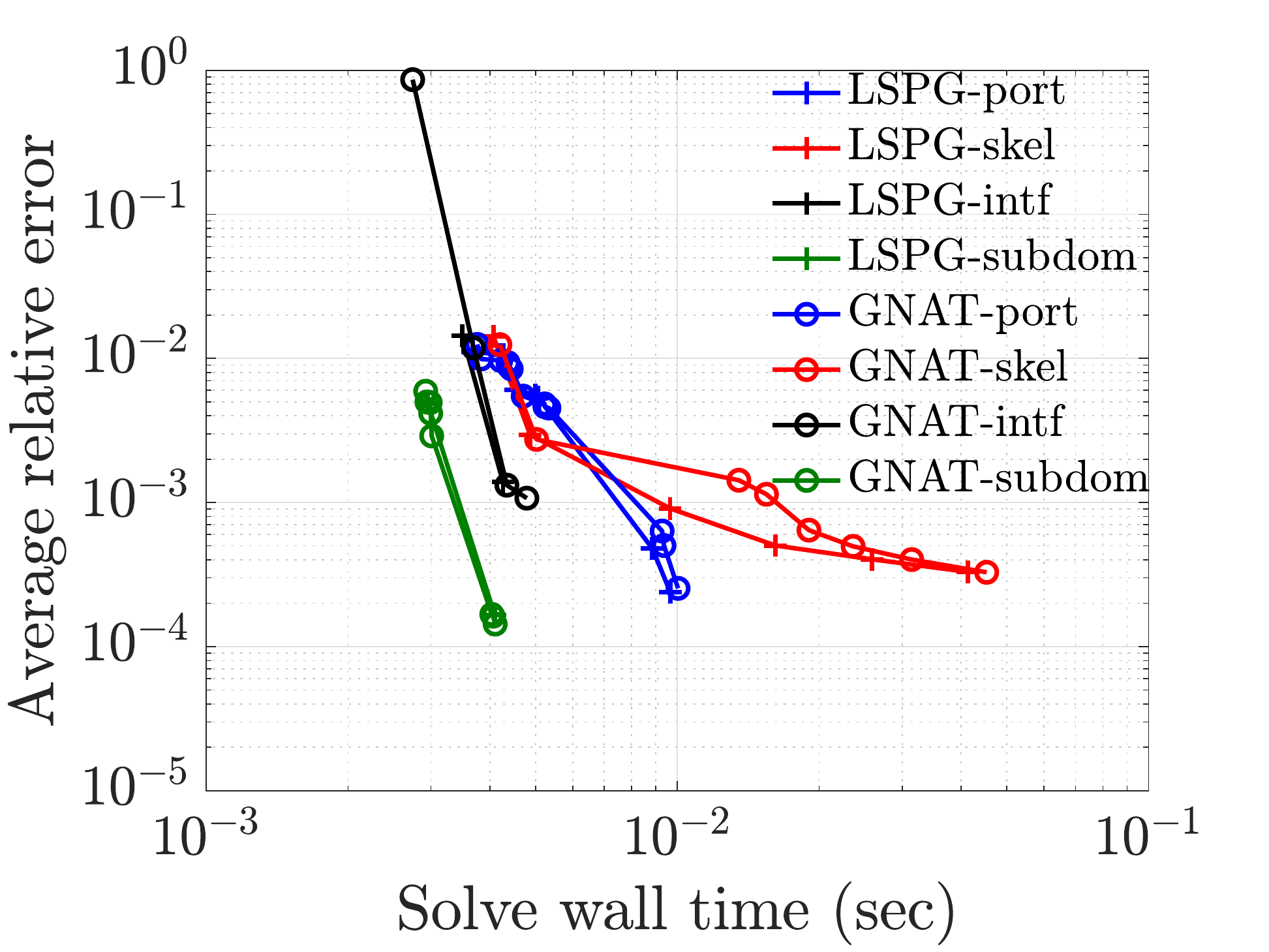}
\caption{Burger 8x2 ``fine''}
\label{fig_ex2_bg82fn_pareto_wallSolv}
\end{subfigure}	
~
\begin{subfigure}[b]{0.3\textwidth}
\includegraphics[width=5.5cm]{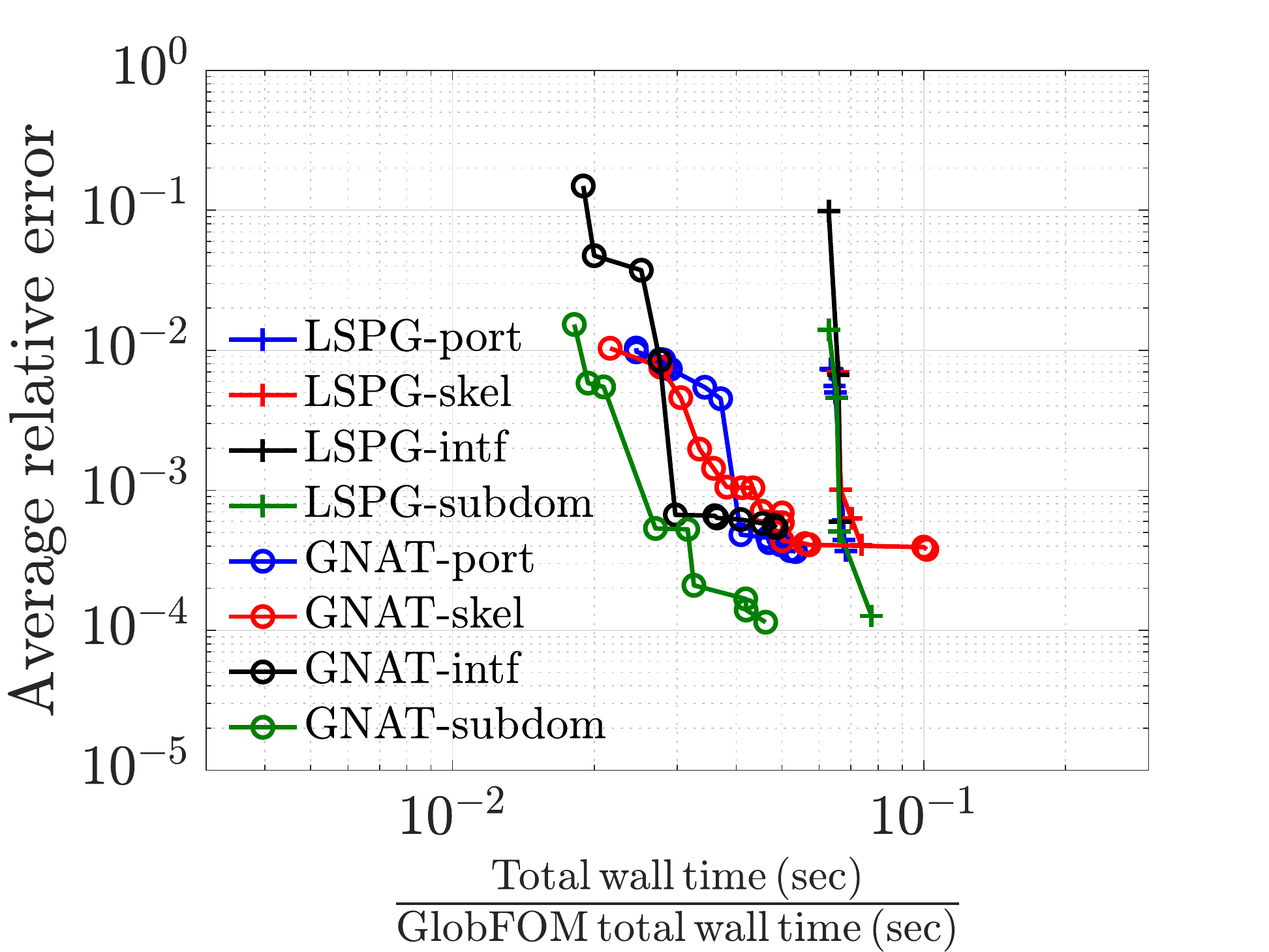}
\caption{Burger 4x2 ``fine''}
\label{fig_ex2_bg42fn_pareto_wallAll}
\end{subfigure}
~
\begin{subfigure}[b]{0.3\textwidth}
\includegraphics[width=5.5cm]{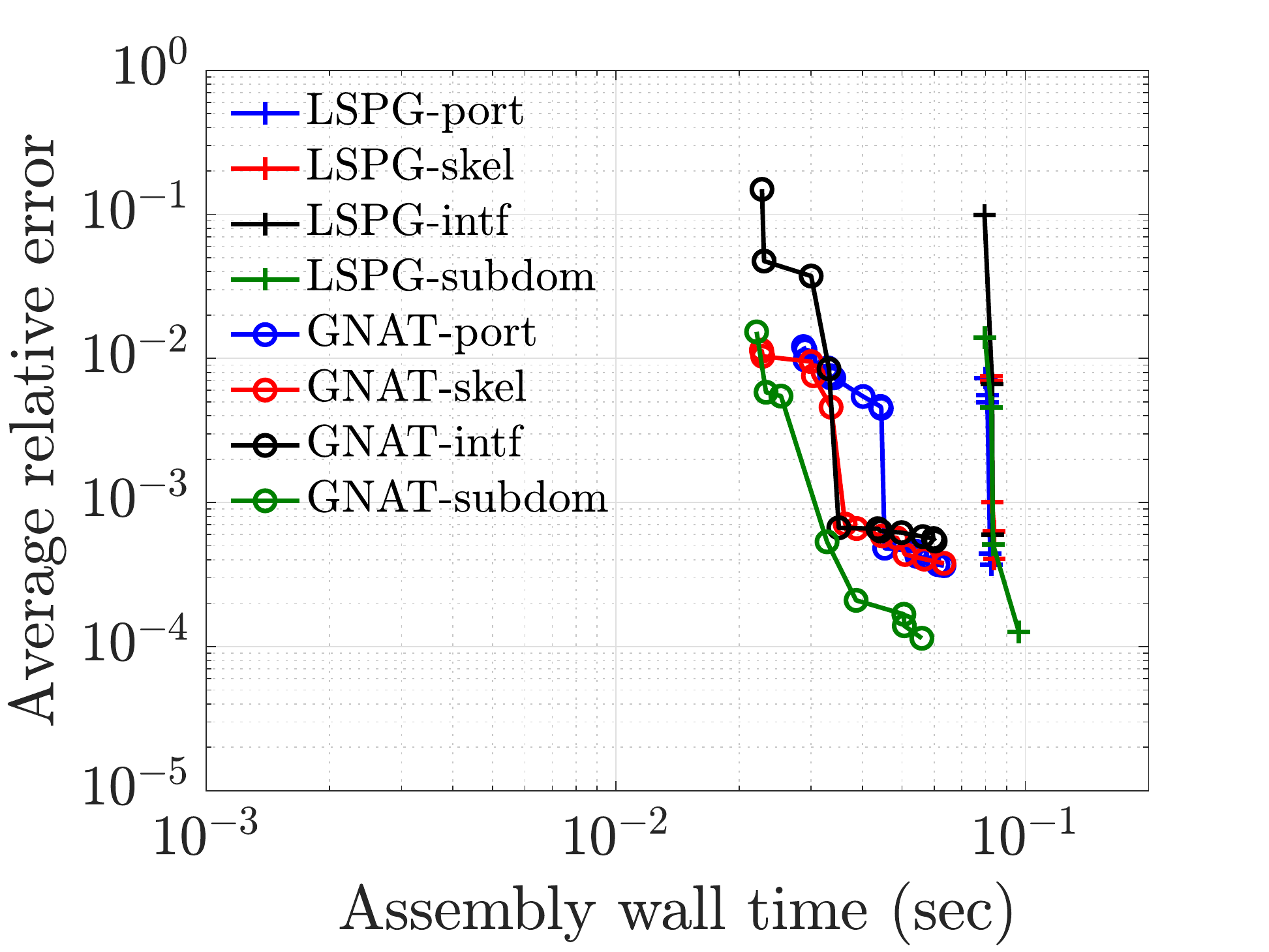}
\caption{Burger 4x2 ``fine''}
\label{fig_ex2_bg42fn_pareto_wallAsmb}
\end{subfigure}
~
\begin{subfigure}[b]{0.3\textwidth}
\includegraphics[width=5.5cm]{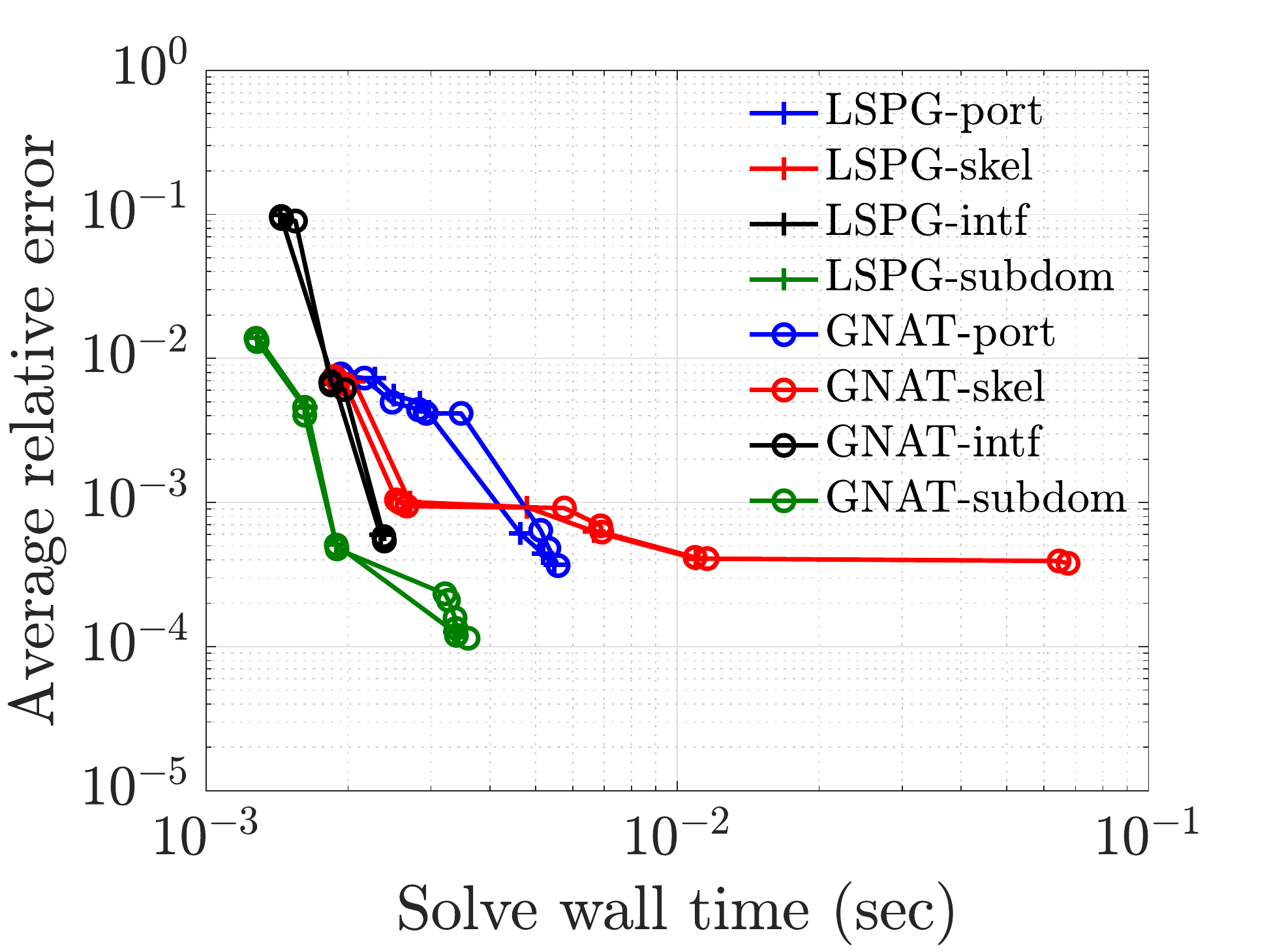}
\caption{Burger 4x2 ``fine''}
\label{fig_ex2_bg42fn_pareto_wallSolv}
\end{subfigure}	
\caption{Burgers' equation, Pareto front plots for wall-all (normalized with respect global FOM wall-all timing), wall-assemble and wall-solve timing of three different configurations for parameters reported in Table~\ref{tab_ex2_manyOnlineComputation}.} \label{fig_ex2_all_pareto}
\end{figure}

\begin{figure}[h!]
\centering
\begin{subfigure}[b]{0.3\textwidth}
\includegraphics[width=5.5cm]{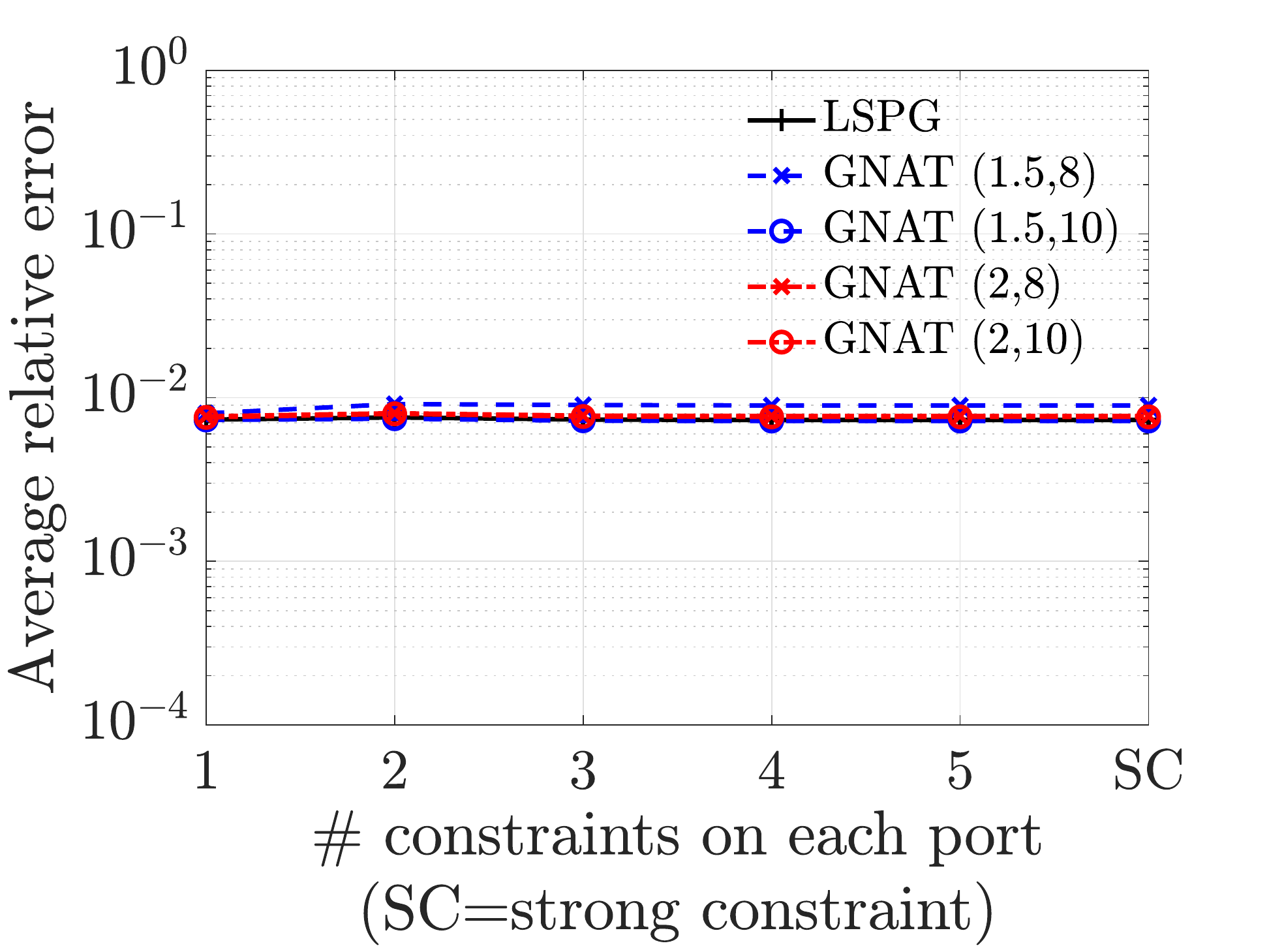}
\caption{port, 
$\energyCriterion=10^{-4}$ on $\domaini$,
	$\energyCriterion=10^{-4}$ on $\boundaryi$
}
\label{fig_ex2_4x2_rmsErr_portBF_55}
\end{subfigure}
~
\begin{subfigure}[b]{0.3\textwidth}
\includegraphics[width=5.5cm]{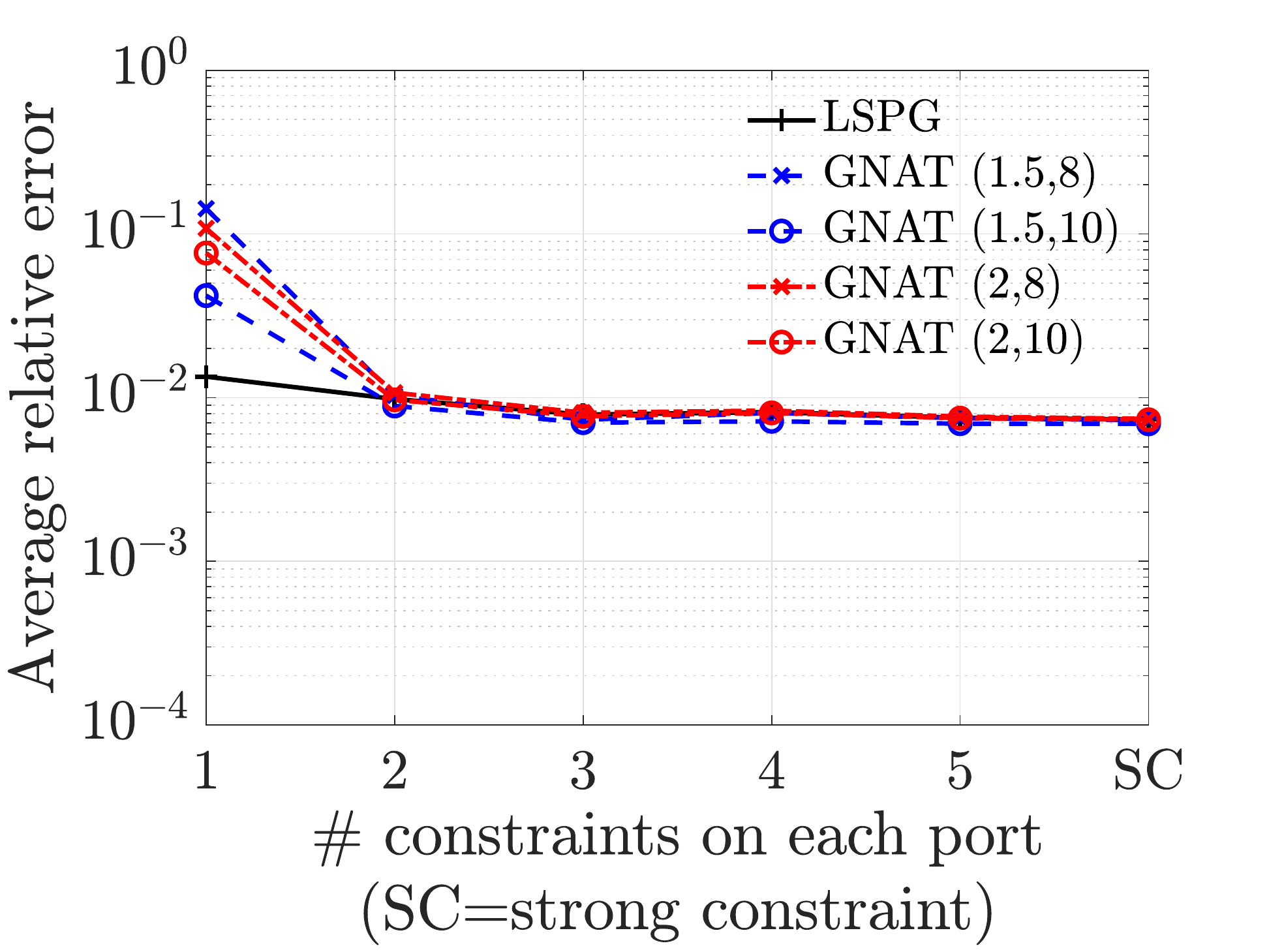}
\caption{port, 
$\energyCriterion=10^{-4}$ on $\domaini$,
	$\energyCriterion=10^{-7}$ on $\boundaryi$
	}
\label{fig_ex2_4x2_rmsErr_portBF_58}
\end{subfigure}
~
\begin{subfigure}[b]{0.3\textwidth}
\includegraphics[width=5.5cm]{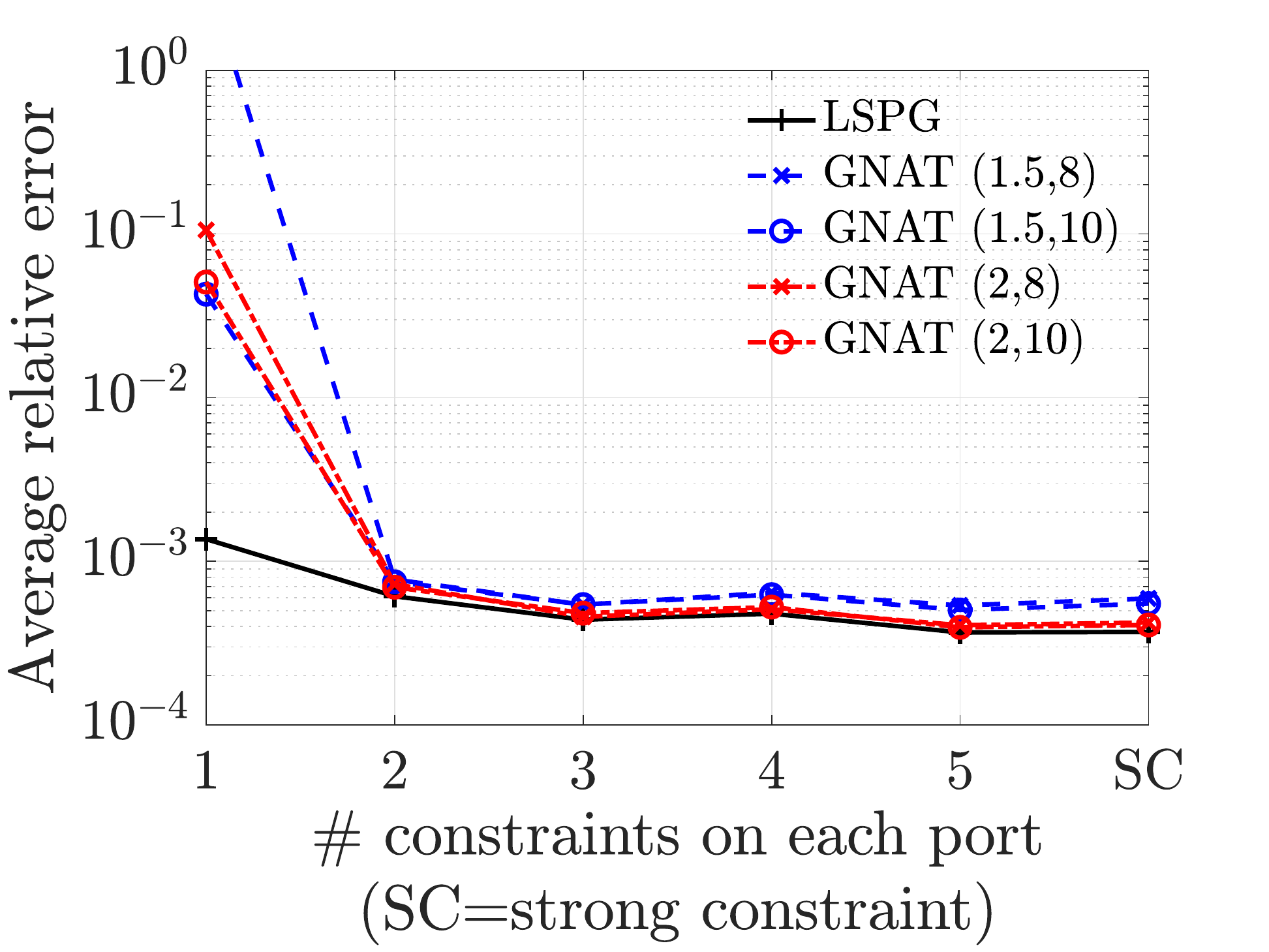}
\caption{port, 
$\energyCriterion=10^{-7}$ on $\domaini$,
	$\energyCriterion=10^{-7}$ on $\boundaryi$
	}
\label{fig_ex2_4x2_rmsErr_portBF_88}
\end{subfigure}
~
\begin{subfigure}[b]{0.3\textwidth}
\includegraphics[width=5.5cm]{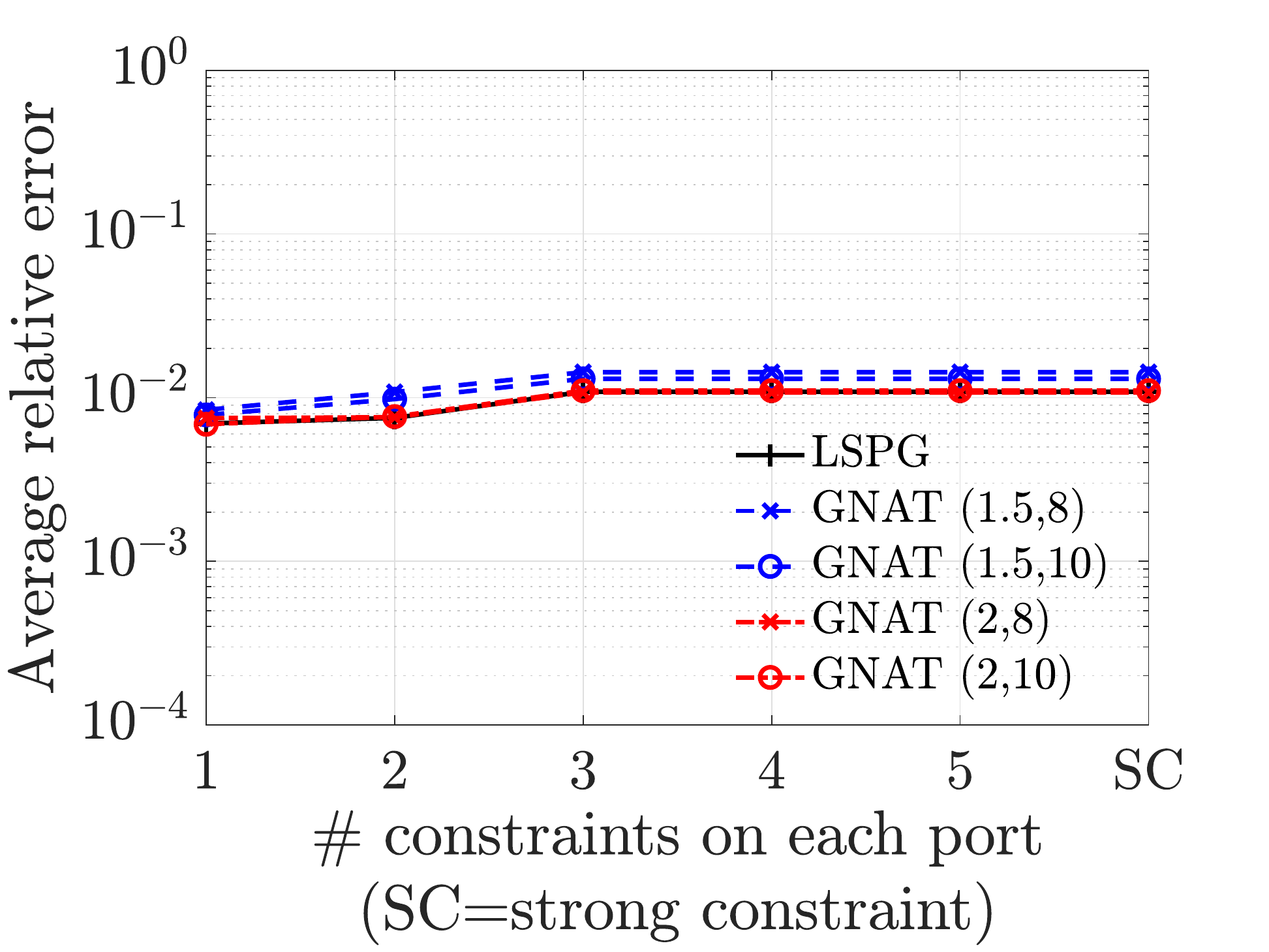}
\caption{skeleton, 
$\energyCriterion=10^{-4}$ on $\domaini$,
	$\energyCriterion=10^{-4}$ on $\boundaryi$
}
\label{fig_ex2_4x2_rmsErr_skelBF_55}
\end{subfigure}
~
\begin{subfigure}[b]{0.3\textwidth}
\includegraphics[width=5.5cm]{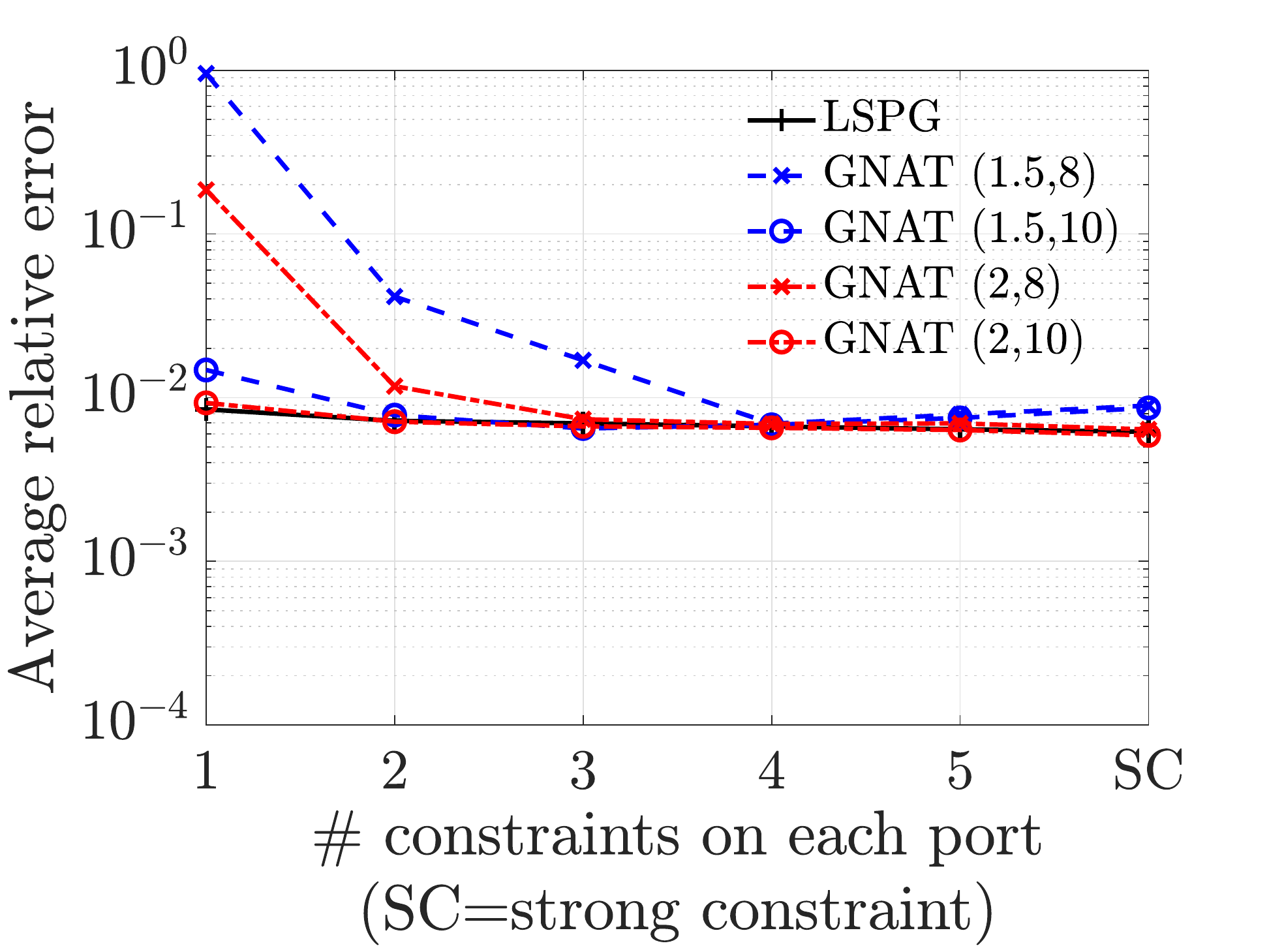}
\caption{skeleton, 
$\energyCriterion=10^{-4}$ on $\domaini$,
	$\energyCriterion=10^{-7}$ on $\boundaryi$
}
\label{fig_ex2_4x2_rmsErr_skelBF_58}
\end{subfigure}
~
\begin{subfigure}[b]{0.3\textwidth}
\includegraphics[width=5.5cm]{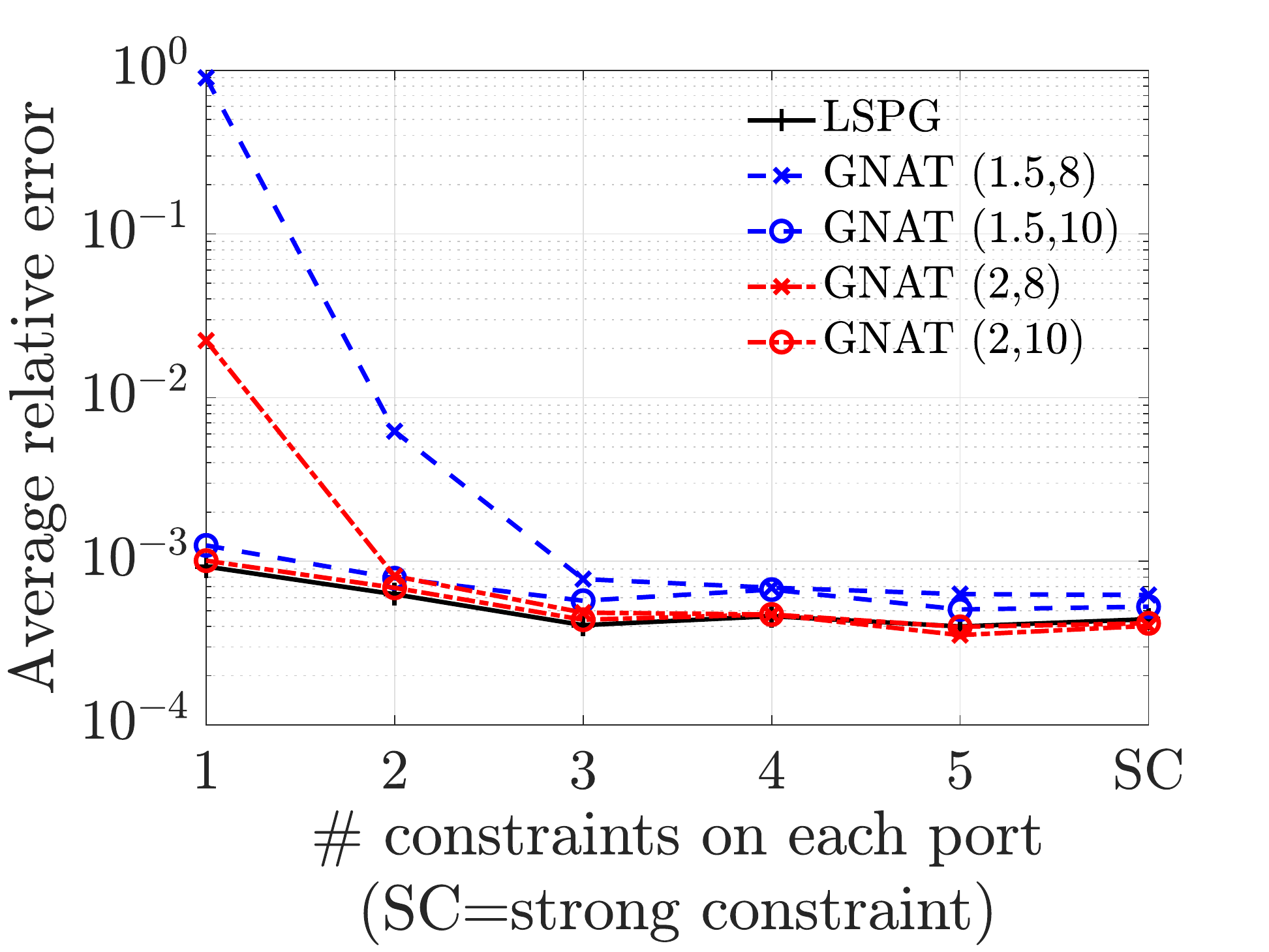}
\caption{skeleton, 
$\energyCriterion=10^{-7}$ on $\domaini$,
	$\energyCriterion=10^{-7}$ on $\boundaryi$
}
\label{fig_ex2_4x2_rmsErr_skelBF_88}
\end{subfigure}
~
\begin{subfigure}[b]{0.3\textwidth}
\includegraphics[width=5.5cm]{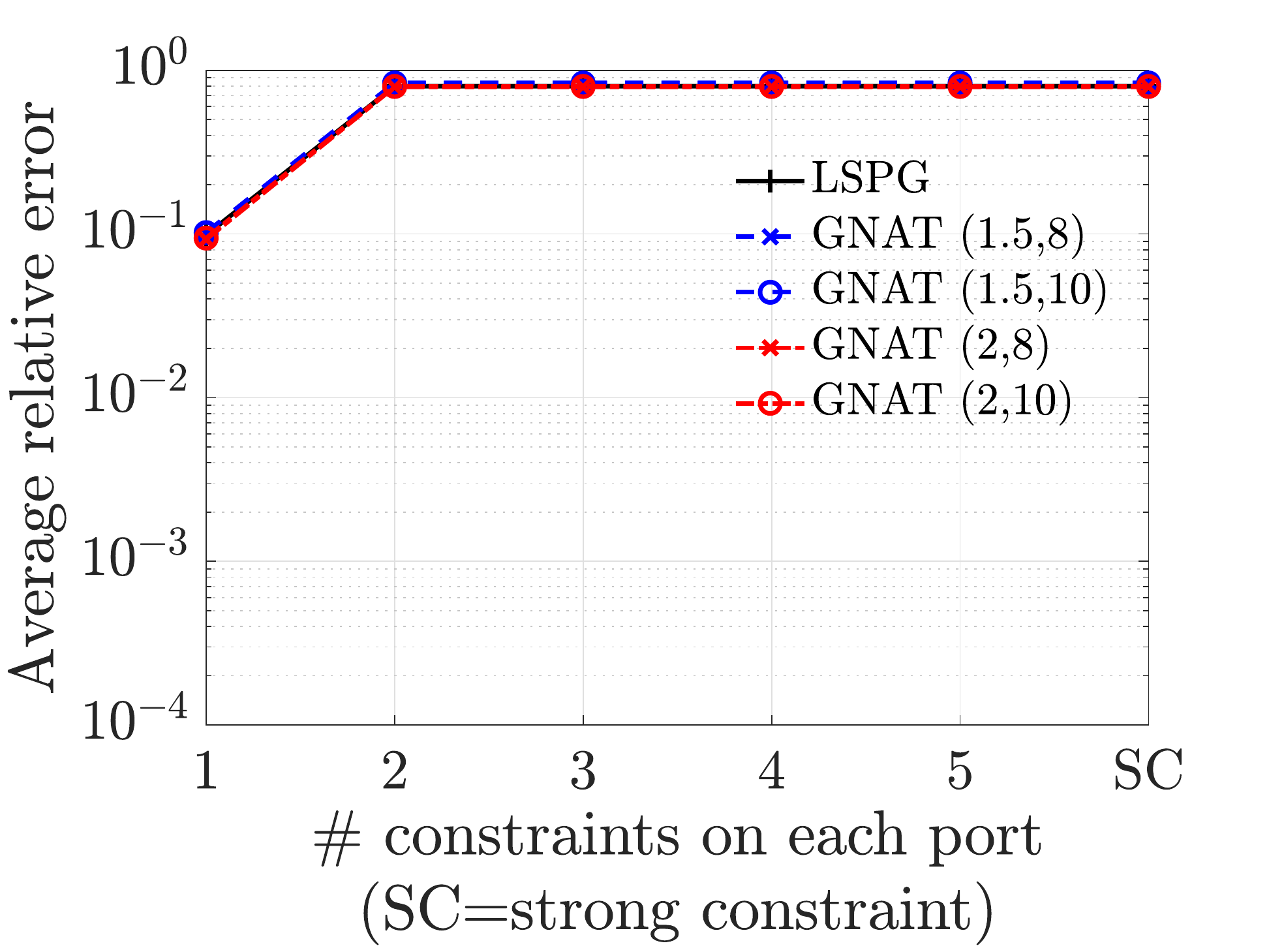}
\caption{full-interface, $\energyCriterion=10^{-4}$ on $\domaini$,
	$\energyCriterion=10^{-4}$ on $\boundaryi$}
\label{fig_ex2_4x2_rmsErr_intfBF_55}
\end{subfigure}
~
\begin{subfigure}[b]{0.3\textwidth}
\includegraphics[width=5.5cm]{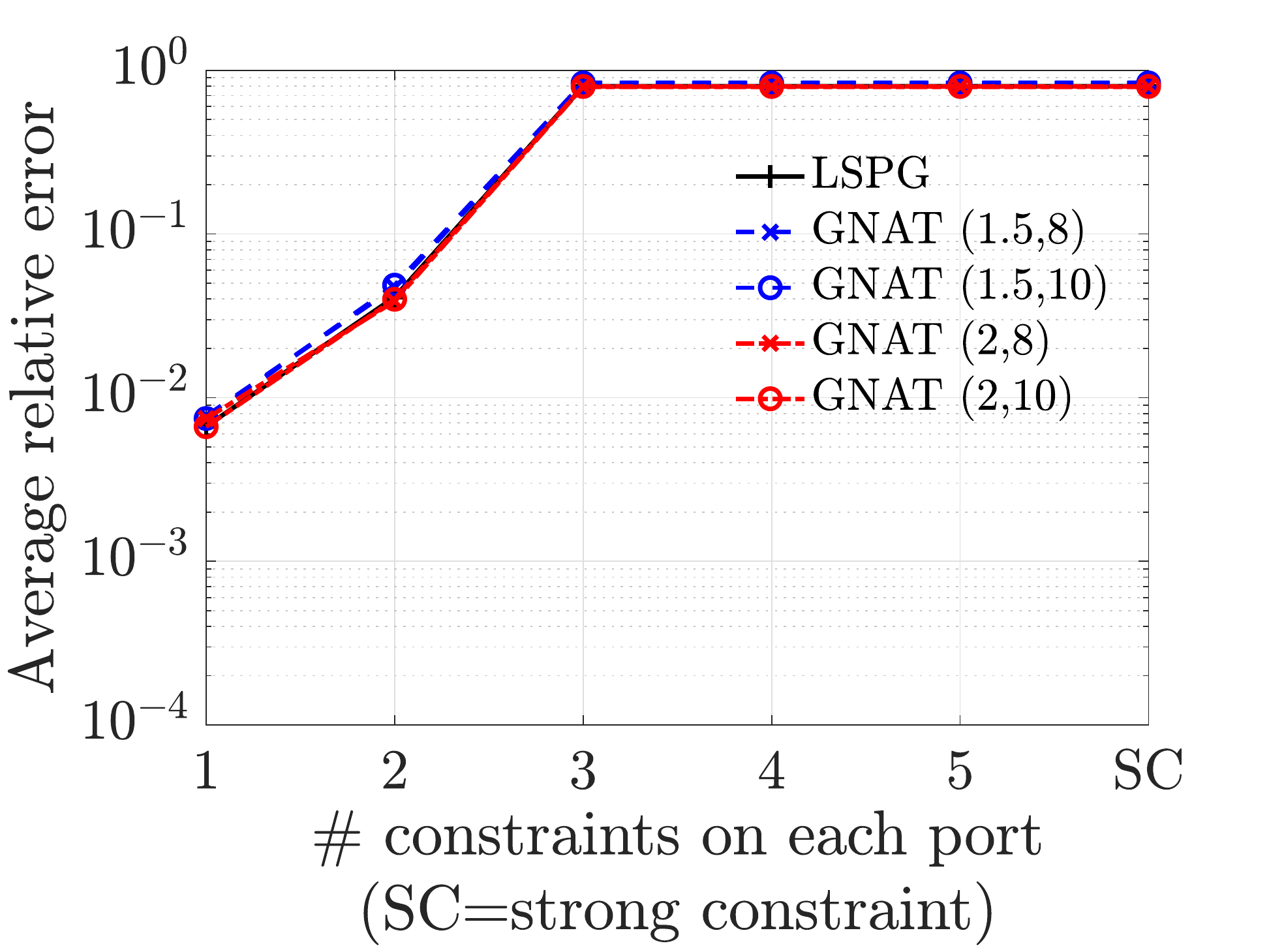}
\caption{full-interface, $\energyCriterion=10^{-4}$ on $\domaini$,
	$\energyCriterion=10^{-7}$ on $\boundaryi$}
\label{fig_ex2_4x2_rmsErr_intfBF_58}
\end{subfigure}	
~
\begin{subfigure}[b]{0.3\textwidth}
\includegraphics[width=5.5cm]{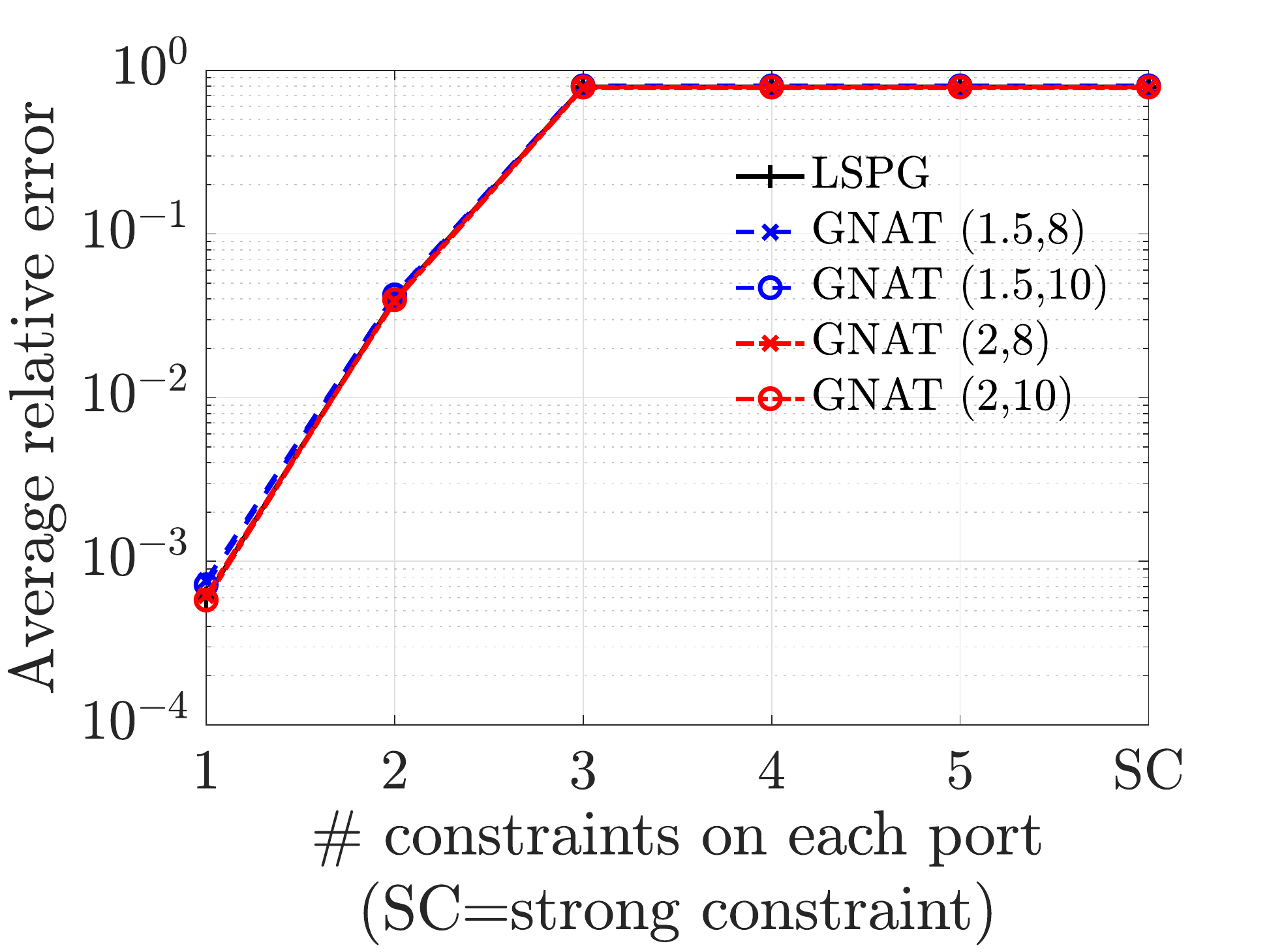}
\caption{full-interface, $\energyCriterion=10^{-7}$ on $\domaini$,
	$\energyCriterion=10^{-7}$ on $\boundaryi$}
\label{fig_ex2_4x2_rmsErr_intfBF_88}
\end{subfigure}	
~
\begin{subfigure}[b]{0.3\textwidth}
\includegraphics[width=5.5cm]{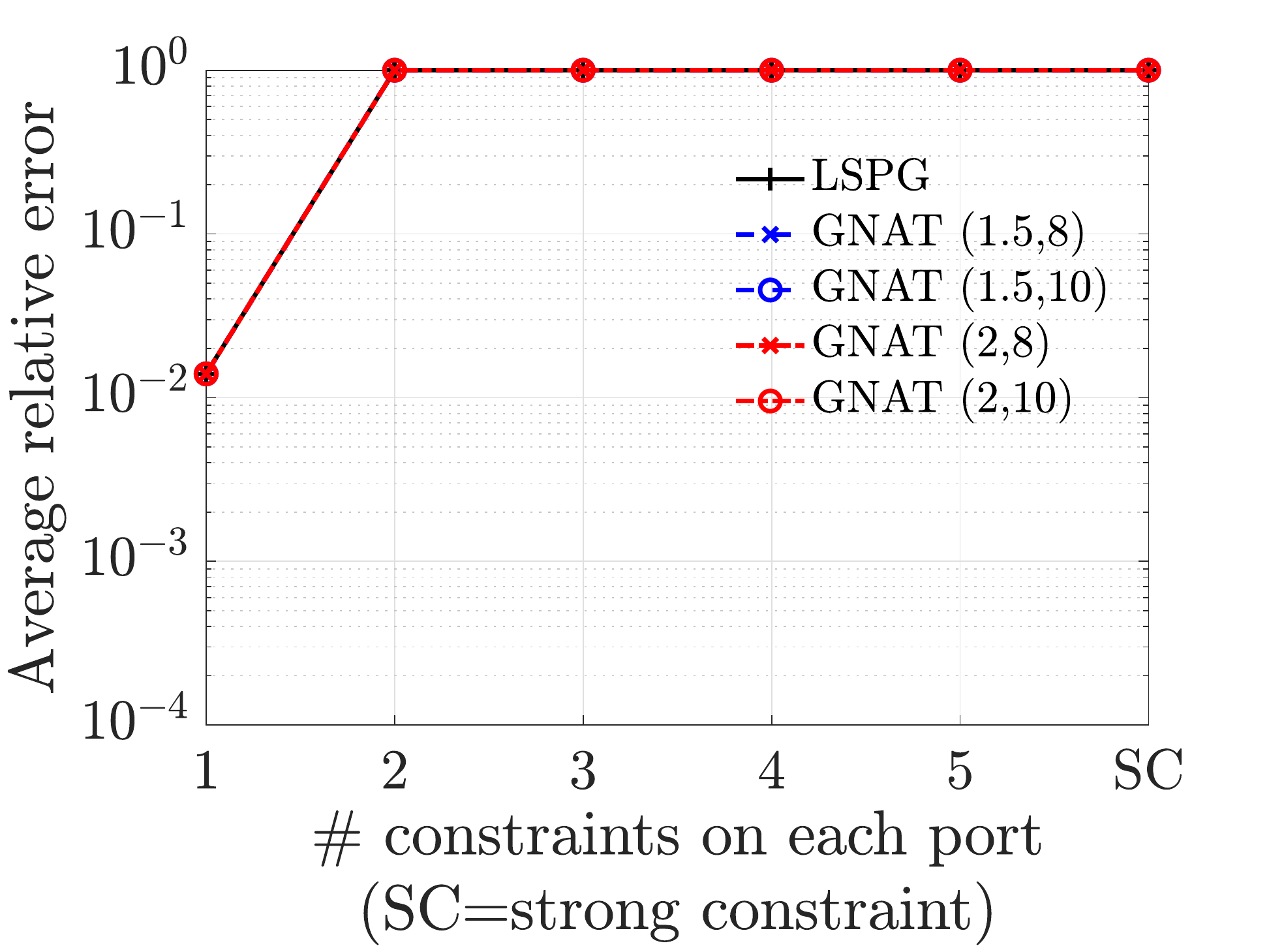}
	\caption{subdomain, $\energyCriterion=10^{-5}$}
\label{fig_ex2_4x2_rmsErr_subdomBF_5}
\end{subfigure}
~
\begin{subfigure}[b]{0.3\textwidth}
\includegraphics[width=5.5cm]{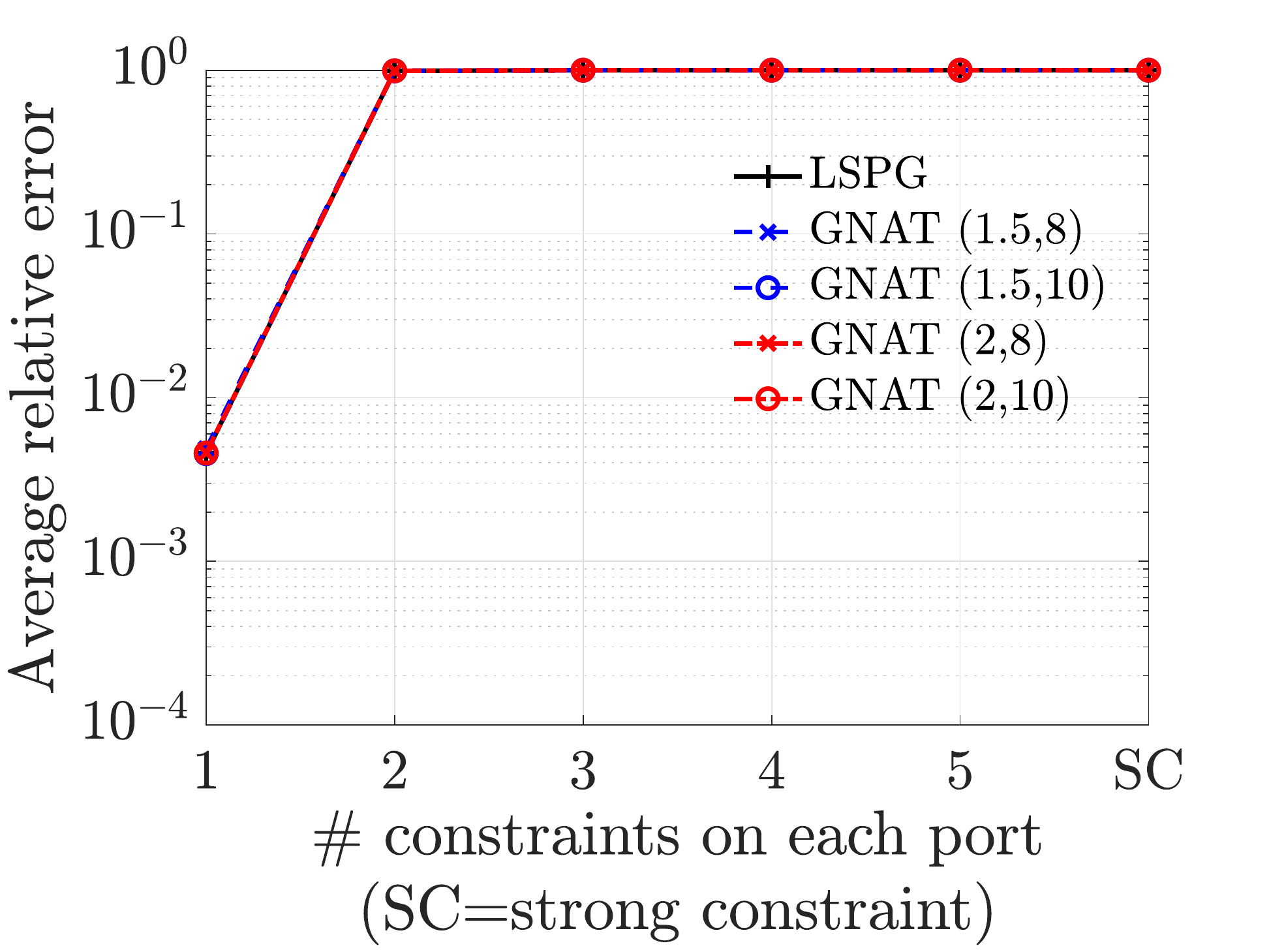}
\caption{subdomain, $\energyCriterion=10^{-6}$}
\label{fig_ex2_4x2_rmsErr_subdomBF_6}
\end{subfigure}
~
\begin{subfigure}[b]{0.3\textwidth}
\includegraphics[width=5.5cm]{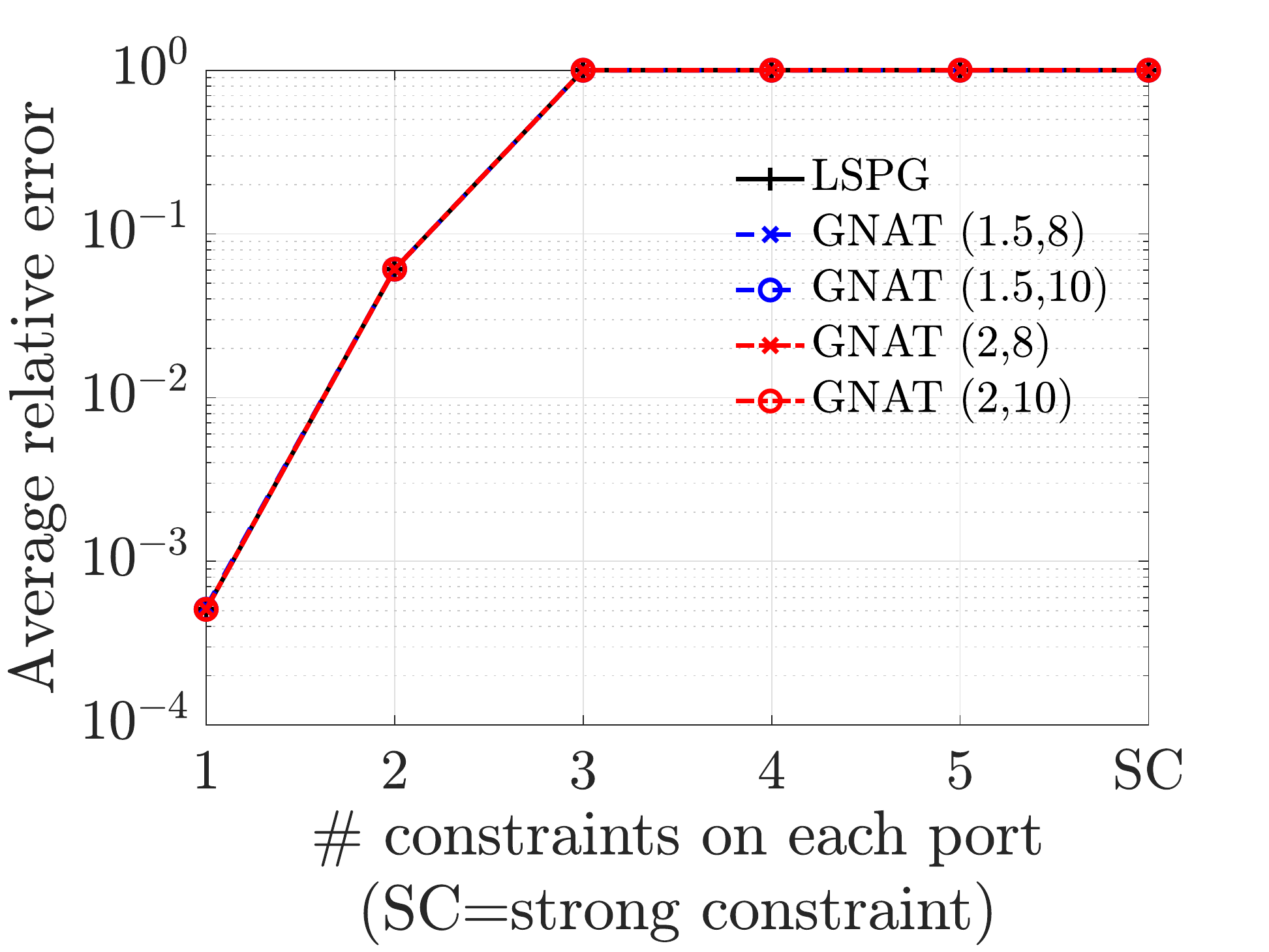}
\caption{subdom, $\energyCriterion=10^{-7}$}
\label{fig_ex2_4x2_rmsErr_subdomBF_7}
\end{subfigure}
\caption{Burgers equation, 4x2 ``fine'' configuration, average relative error
(DD-LSPG and GNAT) versus number of constraint per port for parameters reported in Table~\ref{tab_ex2_manyOnlineComputation}.
In the legend $\text{GNAT}(x,y)$ implies the GNAT model with
$\nsampleArg{i}/\nrbResi=x$ and
$\energyCriterion=10^{-y}$ for $\residuali$. } \label{fig_ex2_bg42fn_rmsErr}
\end{figure}



We again compare the performance of the ROM methods across a wide variation of
all method parameters. Table~\ref{tab_ex2_manyOnlineComputation} reports the
tested parameter values for each method. 
We employ the same approach to reporting wall times as previously described in
Section \ref{subsubsect_ex1_many_online_computations}.
Again, as described previously in Section
\ref{subsubsect_ex1_many_online_computations}, we then construct a Pareto
front for each method.  Figure~\ref{fig_ex2_all_pareto} reports these Pareto
fronts, while Figure~\ref{fig_ex2_bg42fn_rmsErr} plots the average relative
error versus number of constraints per port for the 4x2 ``fine''
configuration. 

Comparing Figures \ref{fig_ex2_all_pareto} and \ref{fig_ex1_all_pareto}
illustrates that nearly identical overall trends are apparent for the two
examples; we thus refer to the discussion in Section
\ref{subsubsect_ex1_many_online_computations} to provide the primary
interpretations for the current case. The primary difference between the
previous example and the current one is that the full-subdomain bases
outperform the full-interface bases in this case and thus yield
the best overall performance; further, the skeleton basis yields
worse wall-time performance for smaller errors compared with the port bases in
the present example. We emphasize that---as in the previous example---DD-GNAT
yields the best overall performance, achieving $>50\times$ speedup with $<1\%$ relative
error, and performs best with the full-subdomain
basis in this case.

Lastly, Figure~\ref{fig_ex2_bg42fn_rmsErr} reports the average relative error
as a function of the number of constraints per port with parameters reported
in Table~\ref{tab_ex2_manyOnlineComputation} for the 4x2 ``fine''
configuration.  Again, comparing Figures \ref{fig_ex2_bg42fn_rmsErr} and
\ref{fig_ex1_heat44fn_rmsErr} illuminate that nearly identical overall trends
are observed in this example as in the previous one.  In particular, the
subfigures in the top two rows of Figure~\ref{fig_ex2_bg42fn_rmsErr} imply
that strong compatibility constraints yield better accuracy than weak
compatibility for both port and skeleton basis types, while the two last rows
show that weak constraint case with a small number of constraints per port
yield the best accuracy for full-interface and full-subdomain basis types.
The discussion in Remark \ref{rem:globalSol} accounts for this behavior:
approaches that  ensure
neighboring components have \textit{compatible} bases on shared ports perform
best with strong compatibility constraints, while approaches that allow for
neighboring components to have \textit{incompatible} bases on shared ports
perform best with weak compatibility constraints.

\section{Conclusions}\label{sec_conclusions}

This work proposed the domain-decomposition least-squares Petrov--Galerkin
(DD-LSPG) model-reduction method applicable to parameterized systems of
nonlinear algebraic equations. In contrast to previous works, we adopt an
algebraically non-overlapping decomposition strategy, allowing it to be
applicable to multiple discretization techniques in the case of parameterized
PDEs; further, in constrast with previous DDROM methods for nonlinear systems,
it is a ``complete ROM'' approach rather than a hybrid ROM/FOM technique. We
equipped DD-LSPG with hyper-reduction, four different strategies
for constructing subdomain bases, supported both strong and weak
compatibility constraints, and proposed an SQP solver that exposes
parallelization. Further, we developed both \textit{a posteriori}
and \textit{a priori} error bounds for the technique. Numerical experiments
revealed several interesting performance attributes of the DD-LSPG
methodology:
\begin{enumerate} 
\item The best type of compatibility constraint is strongly dependent on the
	type of subdomain bases; in particular, subdomain bases that admit
		basis incompatibilities on shared interfaces (i.e., full-interface and
		full-subdomain bases) require weak compatibility constraints to avoid
		trivial interface solutions, while subdomain bases that guarantee
		shared-interface compatibility (i.e., port and skeleton bases) perform
		well with strong constraints.
\item Hyper-reduction is essential to keep assembly costs low when the number
	of DOFs per subdomain is large; this is evidenced by the substantial
		performance gains of DD-GNAT over DD-LSPG for such cases. 
\item The best overall performance was achieved by full-subdomain and
		full-interface bases that employed weak compatibility constraints, with
		the worst performance obtained by port bases, as the latter case generally
		yields a large number of interface DOFs compared with the other approaches.
		Skeleton bases generally yielded intermediate performance, but are
		impractical for truly extreme-scale problems or decomposable systems, as
		they require full-system snapshots to be constructed.
\item \CH{Bottom-up training is possible and promising with our proposed framework although more research needs to be done to make it more mature. }
\end{enumerate}

\YC{Our proposed DD-ROM method is less efficient than the monolithic ROM approach in the
online phase because the monolithic ROM approach can produce a smaller number of reduced
bases than our method. However, our method introduces a way of accomplishing a
domain-decomposition ROM that can be useful for truly large-scale problems where
the monolithic ROM may not be feasible due to expensive training phase, e.g.,
insufficient memory (thanks to bottom-up training).}
\YC{Indeed, while reduced-order models have demonstrated success in many
applications across computational science and engineering, they encounter
challenges when applied both to nonlinear extreme-scale models due to the
prohibitive cost of generating requisite training data, and to decomposable
systems due to many-query problems (e.g., design
\cite{amsallem2015design, choi2019accelerating, choi2020gradient,
white2020dual}) often requiring repeated reconfigurations of system components.
We believe that our current work is a step toward addressing these challenges.}

Future work will consider application to truly large-scale problems, alternative
parallel numerical solvers for DD-LSPG, more efficient ``bottom-up'' training
strategies that \YC{does not} require full-system snapshots and thus make the
approach directly amenable to extreme-scale and decomposable systems,
considering time-dependent problems, and supporting nonlinear trial manifolds
\cite{lee2020model,kim2020fast} for subdomains rather than strictly linear
subspaces spanned by reduced bases
\cite{hoang2015efficient,hoang2016hp,hoang2018fast}.  \YC{We will also consider
spatially distributed parameter-dependent problems because the DD-ROM should be
able to handle such a high dimensional parameter space efficiently. }

\section*{Acknowledgements}
This paper describes objective technical results and analysis. Any subjective views or opinions that might be expressed in the paper do not necessarily represent the views of the U.S. Department of Energy or the United States Government. Sandia National Laboratories is a multimission laboratory managed and operated by National Technology and Engineering Solutions of Sandia, LLC., a wholly owned subsidiary of Honeywell International, Inc., for the U.S. Department of Energy’s National Nuclear Security Administration under contract DE-NA-0003525.

\appendix

\section{Offline computational procedure for GNAT}\label{subsect_GNAT_offline}


\subsection{Formation of residual bases}
Given a residual-snapshot matrix $\setOfGlobalResidualSnapshots=\left[ \residual^1(\state;\trainParam{1}), \ldots, \residual^{\numNewtonIterationArg{1}}(\state;\trainParam{1}), \ldots, \residual^1(\state;\trainParam{\numTrainSample}), \ldots, \residual^{\numNewtonIterationArg{\numTrainSample}}(\state;\trainParam{\numTrainSample}) \right] \in \mathbb{R}^{\ndof \times \sum_{j=1}^{\numTrainSample} \numNewtonIterationArg{j}}$, where $\numNewtonIterationArg{j}$ is the number of Newton iterations associated with parameter $\trainParam{j}, 1 \le j \le \numTrainSample$. We will build the residual bases on each subdomain as: $\rbResArg{i} = \texttt{POD}(\projectionResArg{i}\setOfGlobalResidualSnapshots,\energyCriterion)$,
$i=1,\ldots,\nsubdomains$.



\subsection{Greedy algorithm} \label{subsubsect_GNAT_greedy_algorithm}


\begin{center}
\begin{minipage}[ht!]{16cm}
\vspace{0pt}
\begin{algorithm}[H]\label{alg_Greedy_GNAT}
\begin{algorithmic}[1]
	\caption{Greedy algorithm to construct spatial sample sets of all subdomains $\domain_i$ }
	\REQUIRE On each subdomain $\domainArg{i}$: residual basis $\rbResi \in \mathbb{R}^{\nresi \times \nrbResi}$, desired number of sample nodes $\nsampleArg{i}$, number of working columns of $\rbResi$ denoted by $\numWorkColumnsPhiRes{i} \le \min(\nrbResi,\numUnknownPerNode \nsampleArg{i})$, where $\numUnknownPerNode$ denotes the number of unknowns at a node.
	\ENSURE Spatial sample set $\sampleSetOfSampleMesh{i}$ on each subdomain $\domainArg{i}$, $i=1,\ldots,\nsubdomains$
	\FOR {$i=1:\nsubdomains$ \{drop subscript $i$ from this line for legibility\} }
	\STATE $\sampleSetOfSampleMesh{} \leftarrow$ \{corner nodes\}
	\STATE Compute the additional number of nodes to sample: $\additionalNumNodesOfSampleMesh{} = \nsampleArg{} - |\sampleSetOfSampleMesh{}|$
	\STATE Initialize counter for the number of working basis vectors used: $\counterOfNumWorkBasisVectorOfSampleMesh{} \leftarrow 0$
	\STATE Set the number of greedy iterations to perform: $\numGreedyIterOfSampleMesh{}=\min(\numWorkColumnsPhiRes{}, \additionalNumNodesOfSampleMesh{})$
	\STATE Compute the maximum number of right-hand sides in the least squares problems: $\maxNumRHSOfSampleMesh{} = {\rm ceil}(\numWorkColumnsPhiRes{} / \additionalNumNodesOfSampleMesh{}) $ 
	\STATE Compute the minimum number of working basis vectors per iteration: $\minNumWorkBasisVectorPerIterOfSampleMesh{} = {\rm floor}(\numWorkColumnsPhiRes{}/\numGreedyIterOfSampleMesh{})$
	\STATE Compute the minimum number of sample nodes to add per iteration: $\minNumSampleNodeToAddPerIterOfSampleMesh{} = {\rm floor}(\additionalNumNodesOfSampleMesh{}\maxNumRHSOfSampleMesh{}/\numWorkColumnsPhiRes{})$
	\FOR {$j=1,\ldots,\numGreedyIterOfSampleMesh{}$ \{greedy iteration loop\}} 
	\STATE Compute the number of working basis vectors for this iteration: $\numWorkBasisVectorForThisIterOfSampleMesh{} \leftarrow \minNumWorkBasisVectorPerIterOfSampleMesh{}$
	\STATE If ($j \le \numWorkColumnsPhiRes{} \mod \numGreedyIterOfSampleMesh{}$), then $\numWorkBasisVectorForThisIterOfSampleMesh{} \leftarrow \numWorkBasisVectorForThisIterOfSampleMesh{}+1$
	\STATE Compute the number of sample nodes to add during this iteration: $\numSampleNodeToAddForThisIterOfSampleMesh{} \leftarrow \minNumSampleNodeToAddPerIterOfSampleMesh{}$
	\STATE If ($\maxNumRHSOfSampleMesh{}=1$) and ($j \le \additionalNumNodesOfSampleMesh{} \mod \numWorkColumnsPhiRes{}$), then $\numSampleNodeToAddForThisIterOfSampleMesh{} \leftarrow \numSampleNodeToAddForThisIterOfSampleMesh{}+1$
	\IF {$j=1$} 
	\STATE {$[R^1 \ldots R^{\numWorkBasisVectorForThisIterOfSampleMesh{}}] \leftarrow [\phi^1_{r} \ldots \phi^{\numWorkBasisVectorForThisIterOfSampleMesh{}}_r]$}
	\ELSE { 
		\STATE \textbf{for} $q=1,\ldots, \numWorkBasisVectorForThisIterOfSampleMesh{}$ \{basis vector loop\}
		\STATE $R^q \leftarrow
		\phi^{\counterOfNumWorkBasisVectorOfSampleMesh{}+q}_r - [\phi^1_r \ldots
		\phi^{\counterOfNumWorkBasisVectorOfSampleMesh{}}_r]\alpha $, with $\alpha
		= \argmin{\gamma \in
		\mathbb{R}^{\counterOfNumWorkBasisVectorOfSampleMesh{}}} \|[\sampleMat\phi^1_r
		\ldots \sampleMat\phi^{\counterOfNumWorkBasisVectorOfSampleMesh{}}_r]\gamma -
		\sampleMat\phi^{\counterOfNumWorkBasisVectorOfSampleMesh{}+q}_r\|_2$
		\STATE \textbf{end for}	}
	\ENDIF
	\FOR{$k=1,\ldots,\numSampleNodeToAddForThisIterOfSampleMesh{}$ \{sample node loop\} }
	\STATE Choose node with largest average error: $n \leftarrow {\rm arg} \max\limits_{l \notin \sampleSetOfSampleMesh{}} \sum_{q=1}^{\numWorkBasisVectorForThisIterOfSampleMesh{}}\left( \sum_{j\in\delta(l)} (R^q_j)^2 \right)$, where $\delta(l)$ denotes the degrees of freedom associated with node $l$.
	\STATE $\sampleSetOfSampleMesh{} \leftarrow \sampleSetOfSampleMesh{} \cup \{n\}$
	\ENDFOR	
	\STATE $\counterOfNumWorkBasisVectorOfSampleMesh{} \leftarrow \counterOfNumWorkBasisVectorOfSampleMesh{} + \numWorkBasisVectorForThisIterOfSampleMesh{}$
	\ENDFOR
	\ENDFOR
\end{algorithmic}
\end{algorithm}
\end{minipage}%
\end{center}

We adopt and adjust the original greedy algorithm developed earlier \cite{carlberg2013gnat} to build the sample mesh for each subdomain $\domainArg{i}, i=1,\ldots,\nsubdomains$. Algorithm~\ref{alg_Greedy_GNAT} presents
the modified greedy algorithm in which we drop the subscript $i$ of subdomains
for avoiding cumbersomeness, and note that $[\phi^1_r \ldots \phi^{\nrbResArg{i}}_r]
= \rbResArg{i}$ is the residual bases and $\sampleMat = \sampleMatArg{i}$ is the sampling matrix of each subdomain $\domainArg{i}$.

In comparison to the original greedy algorithm in \cite{carlberg2013gnat}, Algorithm~\ref{alg_Greedy_GNAT} has two modifications: (i) there is an outer ``for loop'' that loops over all subdomains (line 1, algorithm~\ref{alg_Greedy_GNAT}), and (ii) we include the \CH{``corner'' nodes (i.e., any interface node)} into the sample mesh before the first greedy iteration (line 2, algorithm~\ref{alg_Greedy_GNAT}). The latter modification ensures that there is at least one interface node be included in the sample mesh, as otherwise $\redStateBoundaryi$ (and hence $\stateApproxBoundaryi$) will not be updated through Newton iterations (since $\searchDirBoundaryik$ is always zero in \eqref{eq_prdumono_update}). Namely, there may have no connection between one subdomain with surrounding neighbor subdomains, or that subdomain is completely isolated. This phenomenon is called ``digraph connecting condition'' \cite{DigraphConnectingCondition} (see figure \ref{fig_divided_models} for an example of corner nodes of subdomains). We also note that the offline GNAT procedure is performed only once as it only depends on the residual bases of each subdomain, and completely do \textit{not} depend on basis types, constraint types and solver types of the DD-LSPG problem.

\bibliography{references_1}
\bibliographystyle{aiaa}
\end{document}